\documentclass[a4paper,11pt]{article}

\usepackage[utf8x]{inputenc}
\usepackage[T1]{fontenc}

\usepackage{makeidx}
\usepackage[colorlinks]{hyperref}  \hypersetup{linkcolor=blue, citecolor=red, urlcolor=teal}

\usepackage{amsthm}
\usepackage{amssymb}
\usepackage{stmaryrd}
\usepackage{bm}
\usepackage{nccrules}
\usepackage[nottoc]{tocbibind}
\usepackage[intlimits,leqno]{amsmath}
\usepackage{mathrsfs}
\usepackage{xspace}
\usepackage{tikz-cd}
\usetikzlibrary{shapes.geometric}

\usepackage{paralist} 

\usepackage{geometry}
\DeclareSymbolFont{bbold}{U}{bbold}{m}{n}
\DeclareSymbolFontAlphabet{\mathbbold}{bbold}

\newcommand{\Z}{\ensuremath{\mathbb{Z}}}
\newcommand{\Q}{\ensuremath{\mathbb{Q}}}
\newcommand{\R}{\ensuremath{\mathbb{R}}}
\newcommand{\C}{\ensuremath{\mathbb{C}}}

\newcommand{\A}{\ensuremath{\mathbb{A}}}
\newcommand{\Sph}{\ensuremath{\mathbb{S}}}	
\newcommand{\Gal}[1]{\ensuremath{\mathrm{Gal}(#1)}}	

\newcommand{\gr}{\ensuremath{\mathrm{gr}}}

\newcommand{\Aut}{\ensuremath{\mathrm{Aut}}}
\newcommand{\Tr}{\ensuremath{\mathrm{tr}\xspace}}

\newcommand{\Resprod}{\ensuremath{{\prod}'}}
\newcommand{\Sym}{\ensuremath{\mathrm{Sym}}}

\newcommand{\dd}{\ensuremath{\;\mathrm{d}}}

\newcommand{\angles}[1]{\ensuremath{\langle #1 \rangle}}
\newcommand{\mes}{\ensuremath{\mathrm{mes}}}

\newcommand{\Stab}{\ensuremath{\mathrm{Stab}}}

\newcommand{\identity}{\ensuremath{\mathrm{id}}}

\newcommand{\Hom}{\ensuremath{\mathrm{Hom}}}

\newcommand{\rightiso}{\ensuremath{\stackrel{\sim}{\rightarrow}}}
\newcommand{\leftiso}{\ensuremath{\stackrel{\sim}{\leftarrow}}}

\newcommand{\Ker}{\ensuremath{\mathrm{ker}\xspace}}

\newcommand{\Image}{\ensuremath{\mathrm{Im}\xspace}}

\newcommand{\Lie}{\ensuremath{\mathrm{Lie}\xspace}}
\newcommand{\Ad}{\ensuremath{\mathrm{Ad}\xspace}}
\newcommand{\ad}{\ensuremath{\mathrm{ad}\xspace}}

\newcommand{\mult}{\ensuremath{\mathrm{mult}}}

\newcommand{\Gm}{\ensuremath{\mathbb{G}_\mathrm{m}}}
\newcommand{\Ga}{\ensuremath{\mathbb{G}_\mathrm{a}}}

\newcommand{\Supp}{\ensuremath{\mathrm{Supp}}}

\newcommand{\GL}{\ensuremath{\mathrm{GL}}}
\newcommand{\SO}{\ensuremath{\mathrm{SO}}}

\newcommand{\Spin}{\ensuremath{\mathrm{Spin}}}

\newcommand{\SU}{\ensuremath{\mathrm{SU}}}

\newcommand{\Sp}{\ensuremath{\mathrm{Sp}}}

\newcommand{\Mp}{\ensuremath{\widetilde{\mathrm{Sp}}}}
\newcommand{\MMp}[1]{\ensuremath{\widetilde{\mathrm{Sp}}^{(#1)}}}
\newcommand{\inv}{\ensuremath{\mathrm{inv}}}
\newcommand{\Ind}{\ensuremath{\mathrm{Ind}}}

\newcommand{\Lgrp}[1]{\ensuremath{{}^{\mathrm{L}} #1}}

\newcommand{\WD}{\ensuremath{\mathrm{WD}}}
\newcommand{\We}{\ensuremath{\mathrm{W}}}

\theoremstyle{plain}
\newtheorem{proposition}{Proposition}[subsection]
\newtheorem{lemma}[proposition]{Lemma}
\newtheorem{theorem}[proposition]{Theorem}
\newtheorem{corollary}[proposition]{Corollary}
\theoremstyle{definition}
\newtheorem{definition}[proposition]{Definition}
\newtheorem{definition-theorem}[proposition]{Definition-Theorem}
\newtheorem{definition-proposition}[proposition]{Definition-Proposition}
\newtheorem{hypothesis}[proposition]{Hypothesis}

\theoremstyle{remark}
\newtheorem{remark}[proposition]{Remark}

\newcommand{\bmu}{\ensuremath{\bm{\mu}}}
\newcommand{\bdelta}{\ensuremath{\bm{\delta}}}
\newcommand{\bomega}{{\ensuremath{\bm{\omega}}}}
\newcommand{\noyau}{\ensuremath{\bm{\varepsilon}}} 
\newcommand{\rev}{\ensuremath{\mathbf{p}}} 
\newcommand{\asp}{\ensuremath{\dashrule[.7ex]{2 2 2 2}{.4}}} 
\newcommand{\Endo}{\ensuremath{!}}  
\newcommand{\EndoE}{\ensuremath{\mathcal{E}}}  
\newcommand{\tildering}[1]{\ensuremath{\tilde{\mathring{#1}}}}	

\newcommand{\orbI}{\ensuremath{\mathcal{I}}}	
\newcommand{\Icusp}{\ensuremath{\mathcal{I}_\mathrm{cusp}}}	
\newcommand{\Iasp}{\ensuremath{\mathcal{I}_{\asp}}}	
\newcommand{\Iaspcusp}{\ensuremath{\mathcal{I}_{\substack{\asp \\ \mathrm{cusp}}}}}	

\renewcommand{\Re}{\ensuremath{\mathrm{Re}\xspace}}
\renewcommand{\Im}{\ensuremath{\mathrm{Im}\xspace}}

\usepackage{lmodern}

\title{Spectral transfer for metaplectic groups. I.\\ Local character relations}
\author{Wen-Wei Li}
\date{}

\makeindex

\begin{document}

\maketitle

\begin{abstract}
  Let $\Mp(2n)$ be the metaplectic covering of $\Sp(2n)$ over a local field of characteristic zero. The core of the theory of endoscopy for $\Mp(2n)$ is the geometric transfer of orbital integrals to its elliptic endoscopic groups. The dual of this map, called the spectral transfer, is expected to yield endoscopic character relations which should reveal the internal structure of $L$-packets. As a first step, we characterize the image of the collective geometric transfer in the non-archimedean case, then reduce the spectral transfer to the case of cuspidal test functions by using a simple stable trace formula. In the archimedean case, we establish the character relations and determine the spectral transfer factors by rephrasing the works by Adams and Renard.
\end{abstract}

{\scriptsize
	\begin{tabular}{ll}
		\textbf{MSC (2010)} & Primary 22E50; Secondary 11F72 \\
		\textbf{Keywords} & endoscopy, character relation, metaplectic group, Arthur-Selberg trace formula
\end{tabular}}

\tableofcontents

\section{Introduction}\label{sec:intro}

Our aim is to initiate the study of local spectral transfer for Weil's metaplectic covering $\Mp(W) \twoheadrightarrow \Sp(W)$ over a local field $F$ of characteristic zero. To begin with, let us explain the ideas by reviewing the case of reductive groups. 

\paragraph{Review of the reductive case}
Under simplifying assumptions, the theory of endoscopy for a connected reductive $F$-group $G$ provides a collective transfer map
\begin{align*}
  \mathcal{T}^\EndoE: \orbI(G) & \longrightarrow \bigoplus_{G^\Endo \in \EndoE_\text{ell}(G)} S\orbI(G^\Endo) \\
  f_G & \longmapsto f^\EndoE = (f^\Endo)_{G^\Endo \in \EndoE_\text{ell}(G)}
\end{align*}
where
\begin{itemize}
  \item $\orbI(G)$ stands for the space of strongly regular semisimple orbital integrals on $G$, viewed as functions on the set $\Gamma_\text{reg}(G)$ of strongly regular semisimple conjugacy classes;
  \item $\EndoE_\text{ell}(G)$ is the set of equivalence classes of elliptic endoscopic data $(G^\Endo, \mathcal{G}^\Endo, s, \hat{\xi})$ for $G$ with $G^\Endo$ being the corresponding endoscopic group of $G$, and we often identify such a datum with $G^\Endo$ by abusing notations;
  \item $S\orbI(G^\Endo)$ stands for the space of stable strongly regular semisimple orbital integrals on $G^\Endo$, viewed as functions on the stable avatar $\Delta_\text{reg}(G^\Endo)$ of $\Gamma_\text{reg}(G^\Endo)$.
\end{itemize}
For every endoscopic datum there are:
\begin{inparaenum}[(i)]
  \item a correspondence between $\Gamma_\text{reg}(G)$ and $\Delta_\text{reg}(G^\Endo)$, conveniently denoted as $\gamma \leftrightarrow \delta$;
  \item a notion of compatible Haar measures for comparing the orbital integrals along $\delta$ and the stable orbital integrals along $\gamma$, for various $\gamma \leftrightarrow \delta$;
  \item a transfer factor $\Delta(\gamma, \delta)$ for $(\gamma, \delta) \in \Delta_\text{reg}(G^\Endo) \times \Gamma_\text{reg}(G)$ satisfying $\gamma \leftrightarrow \delta$; the choice is unique up to $\C^\times$.
\end{inparaenum}  
The transfer map $\mathcal{T}^\EndoE$ is characterized by requiring that for every $f_G \in \orbI(G)$, the component $f^\Endo \in S\orbI(G^\Endo)$ of $f^\EndoE$ satisfies
$$ f^\Endo(\gamma) = \sum_{\substack{\delta \in \Gamma_\text{reg}(G) \\ \gamma \leftrightarrow \delta}} \Delta(\gamma, \delta) f_G(\delta) $$
for $\gamma \in \Delta_\text{reg}(G^\Endo)$ in general position, called the $G$-regular ones; here we use compatible Haar measures to define $f_G(\delta)$ and $f^\Endo(\gamma)$. For this reason $\mathcal{T}^\EndoE$ is deemed the \emph{geometric transfer}, or the collective geometric transfer since we considered all $G^\Endo$ at once. Here we omit some subtleties such as $z$-extensions, etc.

This theory is invented by Langlands and his collaborators in order to stabilize the Arthur-Selberg trace formula; we refer to \cite{CHLN11} for a detailed survey. The bottleneck turns out to be the existence of transfer, the so-called Fundamental Lemma in the unramified case and its weighted version, which are established in a series of works by Waldspurger, B. C. Ngô, Chaudouard and Laumon. By dualizing $f_G \mapsto f^\Endo$, we may transfer stable distributions on $G^\Endo(F)$ (i.e.\ elements of $S\orbI(G^\Endo)^\vee$) to invariant distributions on $G(F)$, for various $G^\Endo$. This is expected to yield endoscopic local character relations. Granting the local Langlands conjecture, the tempered stable characters of $G^\Endo$ should be mapped to certain virtual tempered characters on $G$ whose coefficients should be explicitly given in terms of the \emph{internal structure} of the $L$-packets on $G$ --- whence the name ``endoscopy''. These issues are discussed in depth in Arthur's articles \cite{Ar89-unip,Ar96}.


The local Langlands correspondence and the endoscopic spectral transfers have been settled in \cite{Ar13} for many classical groups, including the split $\SO(2n+1)$ that we will need. A general treatment on the spectral transfer for real groups can be found in \cite{Sh10,Sh08}.

\paragraph{Metaplectic coverings}
As is well-known, certain coverings of connected reductive groups play a prominent rôle in the study of automorphic representations and arithmetic. Let $F$ be local as before and fix an additive character $\psi$ of $F$. Our basic example is Weil's metaplectic covering of $G = \Sp(W)$, where $(W, \angles{\cdot|\cdot})$ is a symplectic $F$-vector space of dimension $2n$. It is a central extension
\begin{gather*}
  1 \to \bmu_8 \to \tilde{G} \xrightarrow{\rev} G(F) \to 1, \\
  \bmu_8 := \{z \in \C^\times: z^8 = 1\}
\end{gather*}
of locally compact groups. It is customary to use $\bmu_2 = \{\pm 1\}$ instead, but the enlarged version (obtained by push-forward via $\bmu_2 \hookrightarrow \bmu_8$) in preferred in this article. Its global avatar is crucial in understanding the Siegel modular forms of half-integral weights. In a harmonic analyst's eyes, the upshot is to investigate
\begin{itemize}
  \item the irreducible admissible representations $\pi$ of $\tilde{G}$ which are genuine, meaning that $\pi(\noyau)=\noyau\cdot\identity$ for all $\noyau \in \bmu_8$;
  \item the anti-genuine test functions $f$ on $\tilde{G}$, meaning that $f(\noyau\tilde{x}) = \noyau^{-1} f(\tilde{x})$;
  \item the orbital integrals of anti-genuine test functions.
\end{itemize}
Furthermore, we have to consider the inverse images under $\rev$ of Levi subgroups of $G(F)$ as well, hereafter called Levi subgroups of $\tilde{G}$. Thanks to the choice of eightfold coverings, they take the form $\tilde{M} = \prod_{i \in I} \GL(n_i) \times \Mp(W^\flat)$ with $2\sum_i n_i + \dim_F W^\flat = \dim_F W$, called the coverings of metaplectic type. For the rudiments of harmonic analysis on general coverings, see \cite{Li12b}.

Our long-term goal is to bring $\tilde{G}$ into Langlands' formalism. The first requirement seems to be a reasonable theory of endoscopy for $\tilde{G}$, together with transfer, fundamental lemma and a partial stabilization of the trace formula as a reality check. Based upon prior works by Adams and Renard in the case $F=\R$, such a theory has been proposed in \cite{Li11}. Roughly speaking, folk wisdom suggests a close relation between $\tilde{G}$ and the split $\SO(2n+1)$, as exemplified by the $\theta$-correspondence for the dual pair $(O(V,q), \Sp(W))$ with $(V,q)$ being a $(2n+1)$-dimensional quadratic $F$-space. Accordingly, one may imagine that the dual group of $\tilde{G}$ ``is'' $\Sp(2n,\C)$. These definitions have evident generalizations to coverings of metaplectic type. The precise formalism will be reviewed in \S\ref{sec:endoscopy}.

For the impatient reader, the recipe is simply to replace $G$ by $\tilde{G}$ (passage to coverings) and $\orbI(G)$ by $\Iasp(\tilde{G})$ (passage to anti-genuine orbital integrals) in all the foregoing discussions. The elliptic endoscopic data are in bijection with pairs $(n', n'') \in \Z_{\geq 0}^2$ with $n'+n''=n$, with $G^\Endo = \SO(2n'+1) \times \SO(2n''+1)$. The transfer of orbital integrals and the fundamental lemma still hold true in this setup. The stabilization of the elliptic terms of the Arthur-Selberg trace formula is established in \cite{Li15}, which in turn relies on the invariant trace formula established in \cite{Li14b} for a more general class of coverings.


\paragraph{Main ingredients}
The next stage is to undertake the spectral transfer for $\tilde{G}$. We shall take our lead from \cite{Ar96}. The first step is to set up the spectral parameters and regard $f_{\tilde{G}}$, $f^\EndoE$ as functions on the relevant spaces. On the endoscopic side, we form the space of endoscopic spectral parameters $T^\EndoE_\text{ell}(\tilde{G})$ by taking the disjoint union of discrete series $L$-parameters of each $G^\Endo$. Doing the same construction for every Levi subgroup $\tilde{M}$, we define
$$ T^\EndoE(\tilde{G}) = \bigsqcup_{M/\text{conj}} T^\EndoE_\text{ell}(\tilde{M})/W^G(M) $$
where $W^G(M) = N_G(M)(F)/M(F)$. As to the side of $\tilde{G}$, a reasonable choice is to take the genuine tempered spectrum $\Pi_{\text{temp},-}(\tilde{G})$. However, Arthur's formalism in \cite{Ar93} is to work with a space
$$ T_-(\tilde{G}) = \bigsqcup_{M/\text{conj}} T_{\text{ell},-}(\tilde{M})/W^G(M) $$
of virtual genuine tempered characters of $\tilde{G}$ derived from Knapp-Stein theory, that is better-suited for the study of orthogonality questions; cf.\ \cite[\S 5.4]{Li12b} for the case of coverings. Note that there are parallel constructions for the geometric parameters, i.e.\ conjugacy and stable conjugacy classes.

To begin with, assume $F$ non-archimedean. Let $\orbI^\EndoE(\tilde{G})$ stand for the image of $\mathcal{T}^\EndoE$. By the trace Paley-Wiener theorem, elements in $\Iasp(\tilde{G})$ (resp. $\orbI^\EndoE(\tilde{G})$) may be regarded as certain $\C$-valued functions on $T_-(\tilde{G})$ (resp. $T^\EndoE(\tilde{G})$); these spaces of functions are called Paley-Wiener spaces. The desideratum is an equality of the form
\begin{gather}\label{eqn:desideratum}
  f^\EndoE(\phi) = \sum_{\tau \in T_-(\tilde{G})} \Delta(\phi, \tau) f_{\tilde{G}}(\tau), \quad \phi \in T^\EndoE(\tilde{G})
\end{gather}
for suitably \emph{spectral transfer factors} $\Delta(\phi, \tau)$, where $f^\EndoE = \mathcal{T}^\EndoE(f_{\tilde{G}}) \in \orbI^\EndoE(\tilde{G})$. This is our Main Theorem \ref{prop:character-relation}, reinterpreted as in Remark \ref{rem:character-relation-equiv}. The approach in \cite{Ar94} may be rephrased as follows.
\begin{enumerate}
  \item An element $f_{\tilde{G}} \in \Iasp(\tilde{G})$ is called cuspidal if it vanishes off the elliptic locus. There is a similar notion on the endoscopic side, and geometric transfer preserves cuspidality.
  \item Characterize the image of $\Iaspcusp(\tilde{G})$ under $\mathcal{T}^\EndoE$; this is a technical \textit{tour de force}.
  \item It is straightforward to obtain \eqref{eqn:desideratum} for cuspidal $f_{\tilde{G}}$ and spectral parameters $(\phi, \tau) \in T^\EndoE_\text{ell}(\tilde{G}) \times T_{\text{ell},-}(\tilde{G})$, with uniquely determined spectral transfer factors $\Delta(\phi, \tau)$. The upshot is to extend \eqref{eqn:desideratum} to all $f_{\tilde{G}}$; the case of non-elliptic $(\phi, \tau)$ will then follow by the compatibility of geometric transfer with parabolic induction, together with recurrence on $\dim_F W$. As a consequence, we can also characterize the full image $\orbI^\EndoE(\tilde{G})$ of $\mathcal{T}^\EndoE$ (Corollary \ref{prop:transfer-surj}).
\end{enumerate}

In the last step we invoke a local-global argument based on the simple stable trace formula for $\tilde{G}$ in \cite{Li15}. Equation \eqref{eqn:desideratum} may also be seen as a description of the image of $\phi$ (as a linear functional on $\orbI^\EndoE(\tilde{G})$) under the dual of $\mathcal{T}^\EndoE$. This is what we mean by local character relations.

In this Introduction, we cannot explain in depth the basic ideas of proof, or even the relevant definitions. Nonetheless, we point out below a few differences between the present work and Arthur's.
\begin{enumerate}
  \item For the covering $\tilde{G}$, the geometric transfer $\mathcal{T}^\EndoE$, fundamental lemma and the simple stable trace formula are surely highly non-trivial. Fortunately, they are already dealt with in \cite{Li11,Li15}.
  \item To characterize the geometric transfer of cuspidal genuine orbital integrals, we proceed as in \cite{Ar96} via Harish-Chandra's technique of descent. For $\tilde{G}$ this is the technical core of \cite{Li11}, and eventually we are reduced to a similar characterization of the images of
    \begin{compactitem}
      \item standard endoscopic transfer for classical groups,
      \item non-standard endoscopic transfer for the datum $(\Spin(2n+1), \Sp(2n), \cdots)$,
    \end{compactitem}
    both on the level of Lie algebras. For the first case, we may reuse Arthur's results in \cite{Ar96} with minor improvements. The second case is even simpler, since non-standard endoscopic transfer is bijective by its very definition. Cf. the proof of Theorem \ref{prop:transfer-cusp-surj}.
  \item Recall that the spectral parameters in $T^\EndoE(\tilde{G})$ are built as a disjoint union over conjugacy classes of Levi subgroups of $\tilde{G}$: they are essentially the discrete series $L$-parameters of elliptic endoscopic groups $M^\Endo$ of $\tilde{M}$. But one can also go the other way around: consider $M^\Endo$ as a Levi subgroup of some elliptic endoscopic group $G^\Endo$ of $\tilde{G}$. The situation is conveniently represented by
    $$ \begin{tikzcd}
    G^\Endo \arrow[-,dashed]{r}[above]{\text{ell.}}[below]{\text{endo.}} & \tilde{G} \\
    M^\Endo \arrow[-,dashed]{r}[above]{\text{ell.}}[below]{\text{endo.}} \arrow[hookrightarrow]{u}[left]{\text{Levi}} & \tilde{M} \arrow[hookrightarrow]{u}[right]{\text{Levi}}
    \end{tikzcd} $$
    Such endoscopic data $G^\Endo$ together with embeddings $M^\Endo \hookrightarrow G^\Endo$ (up to conjugacy) can be parametrized by a finite set $\EndoE_{M^\Endo}(\tilde{G}) \neq \varnothing$; we will write $G^\Endo = G[s]$ if it is parametrized by $s \in \EndoE_{M^\Endo}(\tilde{G})$. In contrast with the case of reductive groups, passing from the route $M^\Endo \dashrightarrow \tilde{M} \dashrightarrow \tilde{G}$ to $M^\Endo \dashrightarrow G^\Endo \dashrightarrow \tilde{G}$ introduces a twist by some $z[s] \in Z_{M^\Endo}(F)$ in geometric transfer (Lemma \ref{prop:z[s]}).
    
    This phenomenon is inherent in the endoscopy for $\tilde{G}$ since the correspondence of conjugacy classes has a twist by $-1$. We have to take care of these subtle twists throughout this article. Note that $z[s]$ are essentially certain sign factors; their effect might be compared with the \emph{central signs} in Waldspurger's work \cite{Wa91} for $n=1$.
  \item The local-global argument is more delicate for coverings. One potential problem is that on an adélic covering $\tildering{G} \twoheadrightarrow \mathring{G}(\A)$, it is not so straightforward to extract the local component $\tilde{\gamma}_V$ of a rational element $\mathring{\gamma} \in \mathring{G}(\mathring{F}) \hookrightarrow \tildering{G}$, where $V$ is a sufficiently large set of places. Moreover, \textit{a priori} it is unclear if two conjugate elements in $\mathring{G}(\mathring{F})$ have conjugate local components, whenever the latter make sense. The question has been studied in \cite{Li14a} for general coverings; here we need some input from the Weil representation of $\tilde{G}$. Cf.\ the proof of Proposition \ref{prop:non-vanishing}.
\end{enumerate}

One obvious defect is that we have no formula for $\Delta(\phi, \tau)$. Hopefully this will be treated in a sequel to this article. New techniques will be needed.

\paragraph{The archimedean case}
The case $F=\R$ has been considered by Adams \cite{Ad98} for the principal endoscopic datum $(n,0)$, and Renard \cite{Re98,Re99} in a somewhat different framework. In this article, we will rephrase their results and establish the spectral transfer as identities
\begin{gather}
  f^\EndoE(\phi) = \sum_{\pi \in \Pi_{\text{temp},-}(\tilde{G})} \Delta(\phi, \pi) f_{\tilde{G}}(\pi), \quad \phi \in T^\EndoE(\tilde{G})
\end{gather}
for certain spectral transfer factors $\Delta(\phi, \pi)$, cf.\ \eqref{eqn:desideratum}. A striking result (Theorem \ref{prop:Li-lifting-real}) is that for discrete series $L$-parameters $\phi$, the factor $\Delta(\phi,\pi)$ are certain natural generalizations of the \textit{central signs} in Waldspurger's work \cite{Wa91} for the case $n=1$.

Following Shelstad's approach \cite{Sh08}, we will define the adjoint spectral transfer factors $\Delta(\pi, \phi)$ and establish the relevant inversion formulas (Theorem \ref{prop:spectral-inversion-R} and Corollary \ref{prop:spectral-inverted-transfer}) for $\phi \in T^\EndoE(\tilde{G})$ and $\pi \in \Pi_{2\uparrow,-}(\tilde{G})$, where $\Pi_{2\uparrow,-}(\tilde{G})$ is an explicitly defined set of genuine limits of discrete series; the awkward notation means ``lifted from $\Pi_2(G^\Endo)$'' for some elliptic endoscopic group $G^\Endo$. This is surely just the first step. A simple yet important observation is that inversion can be achieved without the $K$-group machinery in \cite{Sh08}.

\paragraph{Layout of this article}
In \S\ref{sec:review-metaplectic}, we summarize the basic properties of the eightfold local metaplectic coverings $\rev: \Mp(W) \to \Sp(W)$ as well as the coverings of metaplectic types. The unramified and adélic settings are also briefly reviewed. Endoscopy for $\Mp(W)$ is reviewed in \S\ref{sec:review-endoscopy}. These results are all contained in the prior works \cite{Li11,Li12a,Li12b,Li15}, albeit in French.

In \S\ref{sec:results-SO} we collect Arthur's basic results in \cite{Ar13} on the local Langlands correspondence for $\SO(2n+1)$, as well as a discussion on the archimedean cases.

An in-depth study of geometric transfer is undertaken in \S\ref{sec:geom-transfer}. We introduce the collective geometric transfer $\mathcal{T}^\EndoE$ and the key notion of adjoint transfer, due to Kottwitz. Then we establish the key Theorem \ref{prop:transfer-cusp-surj} that $\mathcal{T}^\EndoE$ restricts to an isometric isomorphism $\Iaspcusp(\tilde{G}) \rightiso \bigoplus_{G^\Endo} S\Icusp(G^\Endo)$ for non-archimedean $F$.

The non-archimedean spectral transfer is discussed in \S\ref{sec:spectral-transfer}; some of the formalisms are also used later in the archimedean case. The Main Theorem \ref{prop:character-relation} and its equivalent forms are stated. In \S\ref{sec:archimedean}, the results by Adams and Renard for $F=\R$ are rephrased in our formalism. The local character relations are explicitly written down. A short discussion for the case $F=\C$ is also included.

The proof of Theorem \ref{prop:character-relation} occupies \S\ref{sec:proof}, in which the necessary stable simple trace formula and reduction steps are set up.

\paragraph{Acknowledgements}
The author is greatly indebted to Jean-Loup Waldspurger for his valuable corrections and comments on an earlier draft of this article. He is also grateful to the referee for pointing out a mistake in \S\ref{sec:supplements-limits}.

\subsection*{Conventions}
\paragraph{Generalities}
Set $\Sph^1 := \{z \in \C^\times : |z|=1 \}$. For every $m \in \Z_{\geq 1}$ we set\index{$\bmu_m$}
$$ \bmu_m := \{ z \in \C^\times : z^m=1 \}. $$

The group of permutations on a set $I$ is denoted by $\mathfrak{S}(I)$.

The dual space of a vector space $E$ will be denoted by $E^\vee$ unless otherwise specified; $\Sym\; E$ stands for the symmetric algebra of $E$. The complexification of an $\R$-vector space $E$ is denoted by $E_\C$. The trace of an endomorphism $A: E \to E$ of trace class is denoted by $\Tr A$.


\paragraph{Fields}
Let $F$ be a local field. Fix a separable closure $\bar{F}$ of $F$ and define
\begin{compactitem}
  \item $\Gamma_F := \Gal{\bar{F}/F}$: the absolute Galois group;
  \item $\We_F$: the Weil group of $F$;
  \item the Weil-Deligne group is $\WD_F := \We_F \times \SU(2)$ when $F$ is non-archimedean, otherwise $\WD_F := W_F$;
  \item $|\cdot|$: the normalized absolute value on $F$.
\end{compactitem}
For non-archimedean $F$, we denote by $\mathfrak{o}_F$ the ring of integers, and $\mathfrak{p}_F \subset \mathfrak{o}_F$ the maximal ideal. Write $q_F := |\mathfrak{o}_F/\mathfrak{p}_F|$.

For a global field $F$ we still have $\Gamma_F$ and $\We_F$. The completion at a place $v$ is denoted by $F_v$, with ring of integers $\mathfrak{o}_v$ and maximal ideal $\mathfrak{p}_v$. We denote by $\A := \A_F = \Resprod_v F_v$ the ring of adèles of $F$. For any finite set $V$ of place of $F$, we write $F_V = \prod_{v \in V} F_v$ and $F^V = \Resprod_{v \notin V} F_v$.

In any case, the Galois cohomology over $F$ is denoted by $H^\bullet(F, \cdot)$; the relevant groups of cocycles and coboundaries are denoted as $Z^\bullet(\cdots)$ and $B^\bullet(\cdots)$, respectively.

\paragraph{Groups}
Let $k$ be a commutative ring with $1$. For any $k$-scheme $X$ and a $k$-algebra $A$, we denote by $X(A)$ the set of $A$-points of $X$. For a ring extension $k' \supset k$ we write $X_{k'} := X \times_k k'$ for the base change. Now take $k = F$ to be a field and assume $X$ is a group variety. The identity connected component of $X$ will be denoted by $X^0$. If $F$ is endowed with a topology, we topologize $X(F)$ accordingly.

Specialize now to the case of $F$-group varieties, or simply the $F$-groups. Let $G$ be an $F$-group. We write $\mathfrak{g} := \Lie(G)$. When $F$ is algebraically closed, $G$ will be systematically identified with $G(F)$. Centralizers (resp. normalizers) in $G$ are denoted by $Z_G(\cdot)$ (resp. $N_G(\cdot)$); denote by $Z_G$ the center of $G$. These notations also pertain to abstract groups.

The symbol $\Ad(\cdots)$ denotes the adjoint action of a group on itself, namely $\Ad(x): g \mapsto gxg^{-1}$. An $F$-group $G$ also acts on its Lie algebra by the adjoint action, written as $X \mapsto gXg^{-1}$; on the other hand, for every $X \in \mathfrak{g}$ we have $\ad(X): Y \mapsto [X, Y]$, $Y \in \mathfrak{g}$.

Hereafter $G$ is assumed to be connected and reductive. The derived subgroup (resp. adjoint group) of $G$ is written as $G_\text{der}$ (resp. $G_\text{AD} := G/Z_G$). For $\delta \in G(F)$, we set $G^\delta := Z_G(\delta)$ and $G_\delta := Z_G(\delta)^0$. A maximal $F$-torus $T$ in $G$ is called elliptic if $T/Z_G$ is anisotropic. The set of semisimple elements of $G(F)$ is denoted by $G(F)_\text{ss}$. An element $\delta \in G(F)_\text{ss}$ is called \emph{strongly regular} if $G^\delta = G_\delta$ is a torus; note that strong regularity is equivalent to the usual regularity when $G_\text{der}$ is simply connected. The Zariski open dense subset of strongly regular elements in $G$ is denoted by $G_\text{reg}$. More generally, for any subvariety $U \subset G$ we will denote
$$ U_\text{reg} := U \cap G_\text{reg}. $$

An element $\delta \in G_\text{reg}(F)$ is called elliptic if $G_\delta$ is an elliptic maximal $F$-torus.

The general notion of \emph{stable conjugacy} in $G(F)$ can be found in \cite[\S 3]{Ko82}; in this article the following special cases will suffice. Let $x, y \in G(F)$ be semisimple. If $G_\text{der}$ is simply connected or if $x, y$ are strongly regular, then $x$ and $y$ are stably conjugate if and only if they are conjugate in $G(\bar{F})$, where $\bar{F}$ is an algebraic closure of $F$.

Define the spaces
\begin{align*}
  \Gamma_\text{ss}(G) & := \{ \text{semisimple conjugacy classes in } G(F)\}, \\
  \Delta_\text{ss}(G) & := \{ \text{semisimple stable conjugacy classes in } G(F) \}.
\end{align*}
These notations have self-evident variants such as $\Gamma_\text{reg}(G)$, $\Gamma_\text{reg,ell}(G)$, $\Delta_\text{reg}(G)$, $\Delta_\text{reg,ell}(G)$, etc. If $M$ is a Levi subgroup of $G$, an element of $M(F)$ is called \emph{$G$-strongly regular} if it becomes strongly regular in $G$; define $\Gamma_{G-\text{reg}}(M)$ and $\Delta_{G-\text{reg}}(M)$, etc., accordingly.

Note that the ``stable'' notions will mainly be applied to quasisplit groups.

Call two maximal $F$-tori $T_1, T_2$ of $G$ \emph{stably conjugate} if there exists $g \in G(\bar{F})$ such that
\begin{compactitem}
  \item $g T_{1, \bar{F}} g^{-1} = T_{2, \bar{F}}$,
  \item $g^{-1} \tau(g) \in T_1(\bar{F})$ for every $\tau \in \Gamma_F$.
\end{compactitem}
In this case $\Ad(g)$ defines an isomorphism $T_1 \rightiso T_2$ of $F$-tori, and vice versa. If $\delta_1, \delta_2 \in G_\text{reg}(F)$, $T_i := G_{\delta_i}$ and $g \in G(\bar{F})$ realizes a stable conjugation $g \delta_1 g^{-1} = \delta_2$, then $g$ realizes a stable conjugation $\Ad(g): T_1 \rightiso T_2$ between maximal $F$-tori. 

Using the adjoint action of $G$ on $\mathfrak{g}$, we define $\Gamma_\text{reg,ell}(\mathfrak{g})$, $\Delta_\text{reg,ell}(\mathfrak{g})$, etc.; for $X \in \mathfrak{g}(F)$ we write $G^X := Z_G(X)$, $G_X :=Z_G(X)^0$.

The Langlands dual group of $G$ is defined over $\C$; the relevant definitions will be reviewed in \S\ref{sec:L-parameters}.

\paragraph{Classical groups}
The general linear group on a finite-dimensional $F$-vector space $V$ is written as $\GL(V)$, or simply as $\GL(n)$ if $\dim_F V = n$.

Assume the field $F$ to be of characteristic $\neq 2$. In this article, the notation $\SO(2n+1)$ always means a split special orthogonal group associated to a quadratic form on an $F$-vector space $V$ of dimension $2n+1$. We give a precise recipe below: take $V$ with basis $e_{-n}, \ldots, e_{-1}, e_0, e_1, \ldots, e_n$, with the quadratic form $q: V \times V \to F$ given by
\begin{gather*}
  q(e_i|e_{-j}) := \bdelta_{i,j},  \quad -n \leq i,j \leq n.
\end{gather*}
Here $\bdelta_{i,j}$ is Kronecker's delta. Similar conventions pertain to $\Spin(2n+1)$.

The symplectic group associated to a symplectic $F$-vector space $(W, \angles{\cdot|\cdot})$ is denoted by $\Sp(W)$, or sometimes by $\Sp(2n)$ if $\dim_F W = 2n$. It will be reviewed in detail in \S\ref{sec:the-central-extension}. 

Unless otherwise specified, we shall identify all these $F$-group schemes with their groups of $F$-points, to save clutter.

\paragraph{Combinatorics}
As usual, $G$ denotes a connected reductive $F$-group. For a Levi subgroup $M$ of $G$, define the following finite sets
\begin{itemize}
  \item $\mathcal{P}(M)$: the set of parabolic subgroups of $G$ with Levi component $M$;
  \item $\mathcal{L}(M)$: the set of Levi subgroups of $G$ containing $M$;
  \item $W(M) := N_G(M)(F)/M(F)$, a finite group.
\end{itemize}
The Levi decomposition is written as $P=MU$, where $U$ denotes the unipotent radical of $P$. When the rôle of $G$ is to be emphasized, we shall write $\mathcal{P}^G(M)$, $\mathcal{L}^G(M)$ and $W^G(M)$ instead. For a minimal Levi subgroup $M = M_0$, we use the shorthand $W^G_0 := W^G(M_0)$.

Assume $F$ to be of characteristic zero and fix a maximal $F$-torus $T \subset G$. The associated set of absolute roots is denoted by $\Sigma(G, T)_{\bar{F}}$ and the absolute coroots by $\Sigma(G, T)^\vee_{\bar{F}}$; the root-coroot correspondence is written as $\alpha \leftrightarrow \alpha^\vee$. There are at least three types of Weyl groups that we need.
\begin{enumerate}
  \item The relative Weyl group $W(G, T) := N_G(T)(F)/T(F)$.
  \item The group $W(G, T)(F) := \left( N_G(T)/Z_G(T) \right)(F)$; here we regard $N_G(T)/Z_G(T)$ as an $F$-group scheme. By Galois descent we have 
    \begin{gather}\label{eqn:W-intermediate}
      W(G, T)(F) = \left\{ g \in N_G(T)(\bar{F}) : \forall \tau \in \Gamma_F, \; g^{-1}\tau(g) \in T(\bar{F}) \right\} \bigg/ T(\bar{F}).
    \end{gather}
  \item The absolute Weyl group $W(G,T)(\bar{F}) := W(G_{\bar{F}}, T_{\bar{F}}) = \left( N_G(T)/Z_G(T) \right)(\bar{F})$, or equivalently $N_G(T)(\bar{F})/T(\bar{F})$.
\end{enumerate}
Note that $W(G,T)(\bar{F}) \supset W(G,T)(F) \supset W(G,T)$. Galois descent gives
$$ W(G,T)(F) = \left\{ w \in W(G_{\bar{F}}, T_{\bar{F}}) : \Ad(w): T_{\bar{F}} \rightiso T_{\bar{F}} \text{ is defined over } F \right\}. $$

For an $F$-torus $T$, we write $X^*(T) := \Hom_{F-\text{grp}}(T, \Gm)$ and $X_*(T) := \Hom_{F-\text{grp}}(\Gm, T)$; note that $\Gamma_F$ acts on $X^*(T_{\bar{F}})$ and $X_*(T_{\bar{F}})$.

In a similar manner, define $X^*(G) := \Hom_{F-\text{grp}}(G, \Gm)$ and set $\mathfrak{a}_G := \Hom(X^*(G), \R)$. For every $M$ as above, there is a canonically split short exact sequence of finite-dimensional $\R$-vector spaces
$$ 0 \to \mathfrak{a}_G \to \mathfrak{a}_M \leftrightarrows \mathfrak{a}^G_M \to 0. $$
Their dual spaces are denoted by $\mathfrak{a}^*_G$, etc. When $F$ is local, the Harish-Chandra homomorphism $H_G: G(F) \to \mathfrak{a}_G$ is the homomorphism characterized by
$$ \angles{\chi, H_G(x)} = \log|\chi(x)|_F, \quad \chi \in X^*(G). $$

\paragraph{Representations}
The representations under consideration are all over $\C$-vector spaces. Let $G$ be a connected reductive $F$-group where $F$ is a local field. The representations of $G(F)$ are supposed to be smooth, admissible etc., which will be clear according to the context. Define
\begin{itemize}
  \item $\Pi(G)$: the set of equivalence classes of representations of $G(F)$;
  \item $\Pi_{\text{unit}}(G)$: the subset of unitarizable representations;
  \item $\Pi_{\text{temp}}(G)$: the subset of tempered representations;
  \item $\Pi_{2,\text{temp}}(G)$: the subset of unitarizable representations which are square-integrable modulo the center, i.e.\ the discrete series representations.
\end{itemize}
These notions have obvious variants for finite coverings $\tilde{G}$ of $G(F)$. The appropriate object turns out to be the set of \emph{genuine} representations $\Pi_-(\tilde{G})$: see \S\ref{sec:the-central-extension}. For an abstract group $S$, the notation $\Pi(S)$ will also be used to denote its set of irreducible representations, taken up to equivalence.

For $\lambda \in \mathfrak{a}^*_{G,\C}$ and $\pi \in \Pi(G)$, we define $\pi_\lambda \in \Pi(G)$ by
\begin{gather}\label{eqn:pi-twist}
  \pi_\lambda := e^{\angles{\lambda, H_G(\cdot)}} \otimes \pi.
\end{gather}
Note that $\pi_{\lambda+\mu} = (\pi_\lambda)_\mu$. When restricted to $\lambda \in i\mathfrak{a}^*_G$, this operation preserves $\Pi_\text{temp}(G)$ and $\Pi_{2,\text{temp}}(G)$. 

Consider a parabolic subgroup $P=MU$ of $G$. The modulus character $\delta_P: P(F) \to \R_{> 0}$ is specified by
$$ (\text{left Haar measure}) = \delta_P \cdot (\text{right Haar measure}). $$
The \emph{normalized parabolic induction} functor from $P$ to $G$ is denoted by $I^G_P(\cdot) := \Ind^G_P(\delta^{\frac{1}{2}}_P \otimes \cdot)$, often abbreviated as $I_P(\cdot)$.

\section{Review of the metaplectic covering}\label{sec:review-metaplectic}
We shall review the materials in \cite{Li11} concerning the eightfold metaplectic covering $\rev: \Mp(W) \twoheadrightarrow \Sp(W)$ attached to a symplectic vector space $W$. The usual twofold version $\MMp{2}(W) \twoheadrightarrow \Sp(W)$ will be reviewed in Remark \ref{rem:eightfold}. \index{$\MMp{2}(W)$}

\subsection{The central extension}\label{sec:the-central-extension}
\paragraph{Weil's metaplectic covering}
Let $F$ be a field of characteristic $\neq 2$. By a \emph{symplectic $F$-vector space}, we mean a pair $(W, \angles{\cdot|\cdot})$ where $W$ is a finite dimensional $F$-vector space, and $\angles{\cdot|\cdot}: W \times W \to F$ is a non-degenerate alternating bilinear form. A maximal totally isotropic subspace (i.e.\ on which $\angles{\cdot|\cdot}$ is identically zero) of $W$ is called a \emph{Lagrangian}, usually denoted by $\ell$. Define
$$ \Sp(W) := \left\{g \in \GL(W) : \forall x,y \in W, \; \angles{gx|gy} = \angles{x|y} \right\}. $$
This actually defines a semisimple $F$-group.

Assume henceforth that $F$ is a local field of characteristic zero, so that $\Sp(W)$ becomes a locally compact group. Fix a non-trivial additive character $\psi: F \to \Sph^1$. In this article, Weil's \emph{metaplectic covering} is a central extension \index{covering!metaplectic}
\begin{gather}\label{eqn:central-extension-Mp}
  1 \to \bmu_8 \to \Mp(W) \xrightarrow{\rev} \Sp(W) \to 1
\end{gather}
of locally compact groups. This covering is non-linear, i.e.\ does not come from a central extension of $F$-groups, unless $F=\C$. We will identify $\bmu_8$ as a subgroup of $\Mp(W)$.

Since the symplectic $F$-vector spaces are classified up to isomorphism by their dimension, say $\dim_F W = 2n$, we will occasionally write $\Sp(2n)$, $\Mp(2n)$ instead. If $W = W_1 \oplus W_2$ as symplectic $F$-vector spaces, there is then a canonical homomorphism
$$ j: \Mp(W_1) \times \Mp(W_2) \to \Mp(W) $$
such that
\begin{compactitem}
  \item $\Ker(j) = \left\{ (\noyau, \noyau^{-1}) : \noyau \in \bmu_8 \right\}$,
  \item $j$ covers the natural embedding $\Sp(W_1) \times \Sp(W_2) \hookrightarrow \Sp(W)$.
\end{compactitem}

\paragraph{Covering groups in general}
Consider a connected reductive $F$-group $G$ together with a central extension of locally compact groups
\begin{gather}\label{eqn:central-extension-general}
  1 \to \bmu_m \to \tilde{G} \xrightarrow{\rev} G(F) \to 1
\end{gather}
where $m \in \Z_{\geq 1}$. These data form a \emph{covering group}, and there is an evident notion of isomorphisms between coverings. By parabolic (resp. Levi) subgroups of $\tilde{G}$ we mean the inverse images of parabolic (resp. Levi) subgroups of $G(F)$.

The objects living on $\tilde{G}$ will be systematically decorated with a $\sim$, such as $\tilde{x} \in \tilde{G}$; we also write $x := \rev(\tilde{x})$ in that case. If $E$ is a subset of $G(F)$, we will denote $\tilde{E} := \rev^{-1}(E)$. We say that an element of $\tilde{G}$ is semisimple, regular, etc., if its image in $G(F)$ is. Note that $G(F)$ acts on $\tilde{G}$ by conjugation: we will denote this action by $\tilde{x} \mapsto g\tilde{x}g^{-1}$, for all $g \in G(F)$.

\begin{itemize}
  \item The notions of $C^\infty_c$ functions, orbital integrals, smooth and admissible representations, etc., make sense on $\tilde{G}$; see \cite{Li12b}. Note that when $F$ is archimedean, $\tilde{G}$ is a group in \emph{Harish-Chandra class}; our notions of parabolic and Levi subgroups of $\tilde{G}$ are also consistent with the usual ones.
  \item A function $f: \tilde{G} \to \C$, or more generally $f: \tilde{E} \to \C$ where $E$ is a subset of $G(F)$, is called \emph{anti-genuine} if
    $$ \forall \noyau \in \bmu_m, \quad f(\noyau \cdot) = \noyau^{-1} f(\cdot). $$
    It is called genuine if we require $f(\noyau \cdot) = \noyau f(\cdot)$ instead.
  \item A representation $(\pi, V)$ of $\tilde{G}$ is called \emph{genuine} if
    $$ \forall \noyau \in \bmu_m, \quad \pi(\noyau) = \noyau \cdot \identity. $$
    It is called anti-genuine if we require $\pi(\noyau) = \noyau^{-1} \cdot \identity$ instead.
  \item A distribution $D$ (regarded as a linear functional $C^\infty_c(\tilde{G}) \to \C$) is called genuine if for all $f \in C^\infty_c(\tilde{G})$, we have
    $$ \forall \noyau \in \bmu_m, \quad D(f^{\noyau}) = \noyau \cdot D(f), $$
    where $f^{\noyau}(\cdot) := f(\noyau^{-1} \cdot)$. It is called anti-genuine if we require $D(f^{\noyau}) = \noyau^{-1} D(f)$ instead.
  \item We use the subscript $-$ (resp. $\asp$) to denote the genuine (resp. anti-genuine) objects. For example, $C^\infty_{c, \asp}(\tilde{G})$ denotes the space of anti-genuine $C^\infty_c$ functions on $\tilde{G}$, whereas $\Pi_-(\tilde{G})$ denotes the space of irreducible admissible representations of $\tilde{G}$, up to equivalence.
  \item Note that a genuine distribution $D$ is completely determined by its restriction on $C^\infty_{c, \asp}(\tilde{G})$. The character $\Theta_\pi: f \mapsto \Tr(\pi(f))$ of $\pi \in \Pi_-(\tilde{G})$ is a genuine distribution, as expected.
  \item The character $\Theta_\pi$ is actually locally $L^1$ and smooth on $\tilde{G}_\text{reg}$; for non-archimedean $F$, we have local character expansions around each semisimple element in terms of Fourier transforms of unipotent orbital integrals. This is Harish-Chandra's celebrated regularity theorem for characters, at least when $F$ is archimedean or when $\tilde{G}=G(F)$. For non-archimedean $F$, the corresponding result for coverings is proved in \cite[Théorème 4.3.2]{Li12b}.
\end{itemize}

In broad terms, the study of harmonic analysis on $\tilde{G}$ is the study of its genuine representations.

Denote by $\Gamma_\text{reg}(\tilde{G})$ (resp. $\Gamma_\text{ell, reg}(\tilde{G})$) the spaces of strongly regular (resp. strongly regular and elliptic) semisimple classes in $\tilde{G}$. They are equipped with natural maps $\Gamma_\text{reg}(\tilde{G}) \twoheadrightarrow \Gamma_\text{reg}(G)$, whose fibers are acted upon by $\bmu_m$, thus there is an evident notion of genuine/anti-genuine functions $\Gamma_\text{reg}(\tilde{G}) \to \C$.

Finally, suppose that a Haar measure on $G(F)$ is chosen. Define the Haar measure on $\tilde{G}$ by requiring that
\begin{gather}\label{eqn:measure}
  \mes(\tilde{E}) = \mes(E)
\end{gather}
for every measurable subset $E$ of $G(F)$. Consequently, for every function $\phi: \tilde{G} \to \C$ which factors through $G(F)$, we have $\int_{\tilde{G}} \phi = \int_{G(F)} \phi$ whenever it is integrable.

\subsection{The Weil representation and its character}\label{sec:Weil-rep}
For a chosen $\psi$, the covering group $\Mp(W)$ carries a special genuine representation $(\omega_\psi, S_\psi)$, called the \emph{Weil representation} or the \emph{oscillator representation}. It admits various realizations, such as the Schrödinger model, lattice model, mixed model, etc. We refer to \cite{MVW87} for details. We fix a Haar measure on $\Sp(W)$, hence a Haar measure on $\Mp(W)$ by the recipe \eqref{eqn:measure}.

The representation $\omega_\psi$ decomposes canonically into
$$ \omega_\psi = \omega_\psi^+ \oplus \omega_\psi^-, $$
called the even and odd parts of $\omega_\psi$, respectively. Each piece is an irreducible admissible unitarizable representation of $\Mp(W)$. It follows that the character
$$ \Theta^\pm_\psi : f \mapsto \Tr \left( \omega^\pm_\psi(f) \right) $$
as a genuine distribution on $\Mp(W)$, is locally $L^1$ and smooth on $\Mp(W)_\text{reg}$. So is $\Theta_\psi := \Theta_\psi^+ + \Theta_\psi^-$.

In fact, $\Theta_\psi$ is smooth on the larger dense open subset
$$ \Mp(W)^\dagger := \left\{ \tilde{x} \in \Mp(W) : \det(x-1|W) \neq 0 \right\}. $$
This result is originally due to Maktouf, who gave an elegant formula for $\Theta_\psi$ over $\Mp(W)^\dagger$. The reader may consult \cite[\S 4.1]{Li11} for a summary.

\begin{remark}\label{rem:eightfold}
  In the literature, the metaplectic covering is often a central extension of the form
  $$ 1 \to \bmu_2 \to \MMp{2}(W) \to \Sp(W) \to 1, $$
  and our extension can be described as
  $$ \Mp(W) = \left( \bmu_8 \times \MMp{2} \right) \Big/ (\pm \noyau, \tilde{x}) \sim (\noyau, \pm\tilde{x}), $$
  that is, the push-forward of central extensions via $\bmu_2 \hookrightarrow \bmu_8$.

  There is no difference between genuine and anti-genuine on $\MMp{2}(W)$. The genuine representations (resp. anti-genuine functions) on $\Mp(W)$ and $\MMp{2}(W)$ are in natural bijection, in view of the push-forward construction above.
\end{remark}

Working with $\bmu_8$ offers more flexibility. Below is a crucial instance.
\begin{definition}\label{def:-1}\index{$-1$}
  There exists a canonical element in $\rev^{-1}(-1)$, denoted by the same symbol $-1$, which satisfies
  $$ \omega_\psi^\pm(-1) = \pm \identity; $$
  consequently, if we write $-\tilde{x} = (-1)\tilde{x}$ then
  $$ (\Theta^+_\psi - \Theta^-_\psi)(-\tilde{x}) = (\Theta^+_\psi + \Theta^-_\psi)(\tilde{x}), \quad \tilde{x} \in \Mp(W)_\text{reg}. $$
\end{definition}
This property characterizes $-1 \in \Mp(W)$ as $\omega_\psi^\pm$ is genuine. It also follows that $-1$ is of order two. Such a choice is not always possible inside $\MMp{2}(W)$; see \cite[Définition 2.8]{Li11}.

In view of its genuineness, the character $\Theta_\psi$ will often be used to pin down elements in the fibers of $\rev: \Mp(W) \to \Sp(W)$, such as in the discussions on splittings (see \S\ref{sec:splittings}) or in the proof of Proposition \ref{prop:non-vanishing}.

\subsection{Splittings and coverings of metaplectic type}\label{sec:splittings}
\paragraph{Various splittings}
Let $\rev: \tilde{G} \to G(F)$ be a general covering as in \S\ref{sec:the-central-extension}. To do harmonic analysis on $\tilde{G}$, one has to specify splittings of the central extension over various subgroups.

\begin{enumerate}
  \item Let $P = MU$ be a parabolic subgroup of $G$. There is a canonical section
    $$ s: U(F) \to \tilde{U} $$
    of $\rev$ (see \cite[Appendice 1]{MVW87} or \cite[\S 2.2]{Li14a}). It is equivariant under $P(F)$-conjugation. Hence we may write $\tilde{P}=\tilde{M}U$ and define the parabolic induction functor $I^{\tilde{G}}_{\tilde{P}}(\cdot)$ for coverings.
  \item Specialize now to the metaplectic covering $\tilde{G} = \Mp(W) \xrightarrow{\rev} \Sp(W) = G(F)$ (fix $\psi$). For a Lagrangian $\ell \subset W$, its stabilizer $P = \Stab_{\Sp(W)}(\ell)$ is a parabolic subgroup of $G$; such subgroups are called Siegel parabolics. Thanks to our choice of $\bmu_8$, the Schrödinger model for the Weil representation furnishes a section
    $$ \sigma_\ell: P(F) \to \tilde{P} $$
    of $\rev$. It agrees with $s$ on the unipotent radical. Moreover $\sigma_\ell(-1)$ sends $-1 \in \GL(\ell)$ to the element $-1 \in \tilde{G}$ of Definition \ref{def:-1}. See \cite[Proposition 2.7]{Li11}.

    Let $\ell'$ be another Lagrangian such that $W = \ell \oplus \ell'$, which always exists. Denote the corresponding stabilizer by $P'$. Then
    $$ M := P \cap P' \simeq \GL(\ell) \qquad \text{(canonically)}, $$
    is a common Levi component of $P$ and $P'$. It turns out that $(\sigma_\ell)|_{M(F)}$ is independent of the choice of $\ell, \ell'$ \cite[Remarque 4.8]{Li11}. More generally, the Levi subgroups of $G$ arise from data
    $$ (\ell^i, \ell_i)_{i \in I}, \quad W^\flat $$
    where
    \begin{compactenum}[(i)]
      \item $I$ is a finite set,
      \item for each $i \in I$, $(\ell^i \oplus \ell_i, \angles{\cdot|\cdot})$ is a symplectic $F$-vector space for which $\ell^i$, $\ell_i$ are Lagrangians,
      \item $(W^\flat, \angles{\cdot|\cdot})$ is a symplectic $F$-vector space,
      \item we require that $W = \bigoplus_{i \in I} (\ell^i \oplus \ell_i) \oplus W^\flat$ as symplectic $F$-vector spaces (orthogonal direct sum).
    \end{compactenum}
    The Levi subgroup is then given by $M = \prod_{i \in I} \GL(\ell_i) \times \Sp(W^\flat) \hookrightarrow G$. We will often forget the Lagrangians and write
    $$ M = \prod_{i \in I} \GL(n_i) \times \Sp(W^\flat) $$
    whenever $I$ is a given set of indexes and
    $$ \dim_F W^\flat + 2\sum_{i \in I} n_i = 2n. $$
    The conjugacy classes of Levi subgroups of $\Sp(W)$ are in bijection with equivalence classes of data $(I, (n_i)_{i \in I}, W^\flat)$ subject to the conditions above.

    Now comes the covering. There exists a canonical isomorphism
    $$ \prod_{i \in I} \GL(\ell_i) \times \Mp(W^\flat) \rightiso \tilde{M} $$
    making the following diagram commutes
    $$ \begin{tikzcd}
      \prod_{i \in I} \GL(\ell_i) \times \Mp(W^\flat) \arrow{r}[above]{\sim} \arrow{rd}[left, inner sep=1em]{(\identity, \rev)} & \tilde{M} \arrow{d}[right]{\rev} \arrow[hookrightarrow]{r} & \tilde{G} \arrow{d}[right]{\rev} \\
      & \prod_{i \in I} \GL(\ell_i) \times \Sp(W^\flat) \arrow[hookrightarrow]{r} & G
    \end{tikzcd} $$
    see \cite[\S 5.4]{Li11}. These sections are all characterized using the character $\Theta_\psi$ of the Weil representation.
  \item Since we will invoke global arguments, the unramified setting is also needed. Assume $F$ is non-archimedean of residual characteristic $p>2$, $\psi|_{\mathfrak{o}_F} \equiv 1$ but $\psi|_{\mathfrak{p}_F^{-1}} \not\equiv 1$, and that $(W, \angles{\cdot|\cdot})$ admits an $\mathfrak{o}_F$-model with ``good reduction''; we refer to \cite[\S 2.3]{Li11} for precise definitions. Set $L := W(\mathfrak{o}_F)$, a lattice in $W = W(F)$, then
    $$ K := \Stab_{\Sp(W)}(L) $$
    is a \emph{hyperspecial subgroup} of $\Sp(W)$ -- it is exactly the group of $\mathfrak{o}_F$-points of $\Sp(W)$ with respect to the relevant integral model. Moreover, the lattice model for the Weil representation furnishes a section $s_L: K \to \tilde{K}$ of $\rev$. We record the following useful fact.

    \begin{lemma}
      The element $s_L(-1)$ equals the $-1 \in \Mp(W)$ in Definition \ref{def:-1}.
    \end{lemma}
    \begin{proof}
      This results either from \cite[Proposition 2.13]{Li11}, or from the comparison between the character formulas \cite[Corollaire 4.6]{Li11} and \cite[Proposition 4.21]{Li11} under the hypothesis $p>2$.
    \end{proof}

    In general, let $G$ be a reductive $F$-group. Then $G(F)$ possesses hyperspecial subgroups if and only if $G$ is quasisplit and splits over a unramified extension of $F$ (i.e.\ $G$ is \emph{unramified}); in this case, the hyperspecial subgroups are conjugate under $G_\text{AD}(F)$. Thus it makes sense to define the \emph{unramified Haar measure} on $G(F)$ by requiring that any hyperspecial subgroup has mass $1$. We shall use unramified measures in Theorem \ref{prop:FL}.
\end{enumerate}

In all cases, it is safe to omit the symbols $s$, $\sigma_\ell$, $s_L$. Unless otherwise specified, the subgroups $U(F)$, $K$ will thus regarded as subgroups of $\Mp(W)$; the inverse image $\tilde{M}$ of a Levi subgroup of $\Sp(W)$ will be identified with $\prod_{i \in I} \GL(n_i) \times \Mp(W^\flat)$.

\paragraph{Coverings of metaplectic type}
We need a very mild generalization of Weil's metaplectic coverings. Such a class of coverings should contain all $\Mp(W)$ and is stable under passage to Levi subgroups.\index{covering!metaplectic type}

\begin{definition}[cf. {\cite[Définition 3.1.1]{Li12a}}]\label{def:metaplectic-type}
  Coverings of the form
  $$ (\identity, \rev): \prod_{i \in I} \GL(n_i, F) \times \Mp(W) \to \prod_{i \in I} \GL(n_i, F) \times \Sp(W) $$
  are called the coverings of metaplectic type.
\end{definition}

Modulo the knowledge of $\GL(n_i)$, the harmonic analysis on coverings of metaplectic type reduces immediately to that of $\Mp(W)$. We record a crucial property for Weil's metaplectic coverings.

\begin{theorem}\label{prop:Mp-commute}
  Let $\rev: \tilde{M} \twoheadrightarrow M(F)$ be a covering of metaplectic type. Two elements $\tilde{x}, \tilde{y} \in \tilde{M}$ commute if and only if there images $x, y \in M(F)$ commute.
\end{theorem}
\begin{proof}
  As said before, it suffices to prove this for $\Mp(W) \twoheadrightarrow \Sp(W)$. This is well-known, see eg.\ \cite[Chapitre 2]{MVW87} for the non-archimedean case.
\end{proof}

\subsection{The global case}
Assume that $F$ is a number field. Fix a non-trivial additive character $\psi = \prod_v \psi_v: \A_F/F \to \Sph^1$.

Let $(W, \angles{\cdot|\cdot})$ be a symplectic $F$-vector space equipped with an $\mathfrak{o}_F$-model. Denote by $\Sp(W, \A_F) := \Resprod_v \Sp(W_v)$ the locally compact group of $\A_F$-points of $\Sp(W)$, where $W_v := W \otimes_F F_v$. One may still define Weil's metaplectic covering in the adélic setup: it is an eightfold central extension
$$ 1 \to \bmu_8 \to \Mp(W, \A_F) \xrightarrow{\rev} \Sp(W, \A_F) \to 1 $$
of locally compact groups. Moreover,
\begin{enumerate}[(i)]
  \item one may identify $\Mp(W, \A_F)$ with the quotient $\Resprod_v \Mp(W_v)/\mathbf{N}$, where
    \begin{compactitem}
      \item $\Resprod_v \Mp(W_v)$ is the restricted product of the local metaplectic groups with respect to the hyperspecial subgroups $K_v = \Sp(W(\mathfrak{o}_v)) \hookrightarrow \Mp(W_v)$ alluded to above, for almost all $v \nmid \infty$, and
      \item we take
	$$ \mathbf{N} := \Ker \left( \bigoplus_v \bmu_8 \xrightarrow{\text{product}} \bmu_8 \right); $$
    \end{compactitem}
  \item $\Mp(W, \A_F)$ still carries the Weil representation $\omega_\psi$, which is identified with the tensor product $\bigotimes_v \omega_{\psi_v}$ of its local avatars;
  \item there exists a unique section $i: \Sp(W, F) \hookrightarrow \Mp(W, \A_F)$ of $\rev$, by which we regard $\Sp(W, F)$ as a discrete subgroup of $\Mp(W, \A_F)$ of finite covolume.
\end{enumerate}
See \cite[\S 2]{Wa88} or \cite[\S 2.5]{Li11} for a detailed construction using the Stone-von Neumann theorem. The upshot is that
\begin{inparaenum}[(i)]
  \item one can develop the theory of automorphic forms and automorphic representations on the covering $\rev: \Mp(W, \A_F) \to \Sp(W, \A_F)$,
  \item and it satisfies all the requirements for the invariant trace formula à la Arthur: see \cite[\S 3.4, IV]{Li14b}. 
\end{inparaenum}

\section{Review of endoscopy}\label{sec:review-endoscopy}
In this section, $F$ always denotes a local field of characteristic zero.

\subsection{Orbital integrals}\label{sec:orbital-integrals}
Let $M$ be an arbitrary connected reductive $F$-group. We fix a Haar measure on $M(F)$. The notations here come from Arthur \cite{Ar96}.

Let $\delta \in M(F)_\text{ss}$ be semisimple. The \emph{Weyl discriminant} of $\delta$ is defined as
$$ D^M(\delta) := \det(1-\Ad(\delta)|\mathfrak{m}/\mathfrak{m}_\delta) \in F^\times. $$

Assume henceforth that $\delta \in M_\text{reg}(F)$, so that $M_\delta$ is an $F$-torus. Fix a Haar measure on $M_\delta(F)$. The normalized orbital integral of $f \in C^\infty_c(M(F))$ along the conjugacy class of $\delta$ is
\begin{gather*}
  f_M(\delta) := |D^M(\delta)|^{\frac{1}{2}} \int_{M_\delta(F) \backslash M(F)} f(x^{-1} \delta x) \dd x
\end{gather*}
where the quotient measure on $M_\delta(F) \backslash M(F)$ is used.

Throughout this article, we assume that these centralizers $M_\delta(F)$ carry prescribed Haar measures that respect conjugacy: for every $x \in M(F)$, the isomorphism $\Ad(x)$ transports the measure on $M_\delta(F)$ to $M_{x\delta x^{-1}}(F)$. Therefore for a given $f$, one may regard $f_M$ as a function $\Gamma_\text{reg}(M) \to \C$.

When $M$ is quasisplit, we have defined the set $\Delta_\text{reg}(M)$ of strongly regular semisimple stable conjugacy classes in $M(F)$. There is an obvious map from $\Gamma_\text{reg}(M) \twoheadrightarrow \Delta_\text{reg}(M)$; it will be written in the form $\delta \mapsto \sigma$. Furthermore, we assume that the Haar measures on the centralizers $M_\delta(F)$ respects stable conjugacy. Define the normalized stable orbital integral along the stable class $\sigma$ as \index{$f^M(\sigma)$}
\begin{align*}
  f^M(\sigma) := \sum_{\delta \mapsto \sigma} f_M(\delta).
\end{align*}
Again, one may regard $f^M$ as a function $\Delta_\text{reg}(M) \to \C$.

The same definitions can be easily adapted to the Lie algebra $\mathfrak{m}$, which will be used in our proofs later.  Define the Weyl discriminant of $X \in \mathfrak{m}_\text{reg}(F)$ as
$$ D^M(X) := \det(\ad(X)| \mathfrak{m}/\mathfrak{m}_X) \in F^\times. $$

Keep the same conventions on Haar measures. The normalized orbital integral of $f \in C^\infty_c(\mathfrak{m}(F))$ is now defined as
\begin{gather*}
  f_M(X) := |D^M(X)|^{\frac{1}{2}} \int_{M_\delta(F) \backslash M(F)} f(x^{-1} X x) \dd x.
\end{gather*}

When $M$ is quasisplit, the stable version $f^M(Y) = \sum_{X \mapsto Y} f_M(X)$ is defined analogously: simply use the notion of stable conjugacy on the level of Lie algebras.

Now consider a covering
$$ 1 \to \bmu_m \to \tilde{M} \xrightarrow{\rev} M(F) \to 1 $$
that is of metaplectic type, or more generally a covering satisfying the assertion in Theorem \ref{prop:Mp-commute}. For $\tilde{\delta} \in \tilde{M}_\text{reg}$ and $f \in C^\infty_{c, \asp}(\tilde{M})$, we may still define the normalized orbital integral \index{$f_{\tilde{M}}(\tilde{\delta})$}
\begin{gather*}
  f_{\tilde{M}}(\tilde{\delta}) := |D^M(\delta)|^{\frac{1}{2}} \int_{M_\delta(F) \backslash M(F)} f(x^{-1} \tilde{\delta} x) \dd x.
\end{gather*}
with the same convention of Haar measures (on $M$).  Note that $f \mapsto f_{\tilde{M}}(\tilde{\delta})$ can be regarded as a genuine distribution on $\tilde{M}$. Moreover, $f_{\tilde{M}}(\noyau\tilde{\delta}) = \noyau^{-1} f_{\tilde{M}}(\tilde{\delta})$ for every $\noyau \in \Ker(\rev) \subset \C^\times$.

A comprehensive treatment of the orbital integrals on coverings can be found in \cite[\S 4]{Li12b}.

\subsection{Spaces of orbital integrals}\label{sec:space-orbital-integral}
\paragraph{The spaces $\orbI(G)$}
Let $G$ be a connected reductive $F$-group. The Haar measures are prescribed in a coherent manner as in \S\ref{sec:orbital-integrals}. For $f \in C^\infty_c(G(F))$ and $\delta \in G_\text{reg}(F)$, we have defined the normalized orbital integral $f_G(\delta)$. Define the $\C$-vector space of orbital integrals on $G$ as
$$ \orbI(G) := \left\{f_G:  \Gamma_\text{reg}(G) \to \C : f \in C^\infty_c(G(F)) \right\}. $$

It is endowed with the linear surjection $C^\infty_c(G(F)) \twoheadrightarrow \orbI(G)$ given by $f \mapsto f_G$. The elements therein are characterized in terms of
\begin{compactenum}[(i)]
  \item local expansion in Shalika germs around each semisimple element, when $F$ is non-archimedean \cite{Vi81};
  \item Harish-Chandra's jump relations, when $F$ is archimedean \cite{Bo94b}.
\end{compactenum}

\begin{remark}\label{rem:completed-otimes}
  Note that $\orbI(G)$ is topologized in the archimedean case. It is actually a strict inductive limit of Fréchet spaces, also known as \emph{LF spaces}; see \cite[pp. 580--581]{Bo94b}. The map $C^\infty_c(G(F)) \twoheadrightarrow \orbI(G)$ is then an open surjection by the remarks in \cite[p.581]{Bo94b}. We also have a natural isomorphism
  $$ \orbI(G_1 \times G_2) \simeq
    \begin{cases}
      \orbI(G_1) \hat{\otimes} \orbI(G_2), & F \text{ archimedean}, \\
      \orbI(G_1) \otimes \orbI(G_2), & \text{otherwise}
    \end{cases} $$
  parallel to
  $$ C^\infty_c(G_1(F) \times G_2(F)) \simeq
    \begin{cases}
      C^\infty_c(G_1(F)) \hat{\otimes} C^\infty_c(G_2(F)), & F \text{ archimedean}, \\
      C^\infty_c(G_1(F)) \otimes C^\infty_c(G_2(F)), & \text{otherwise}
    \end{cases} $$
  The non-archimedean case is quite trivial, whereas in the archimedean case $\hat{\otimes}$ denotes the topological tensor product and the isomorphism is in the category of LF spaces. Any satisfactory description of $\hat{\otimes}$ would require the notion of \emph{nuclear spaces} \cite[\S 50]{Tr67} introduced by Grothendieck. The space $\orbI(G)$ is indeed nuclear since it is a quotient of $C^\infty_c(G(F))$, which is nuclear (cf. the Corollary to \cite[Theorem 51.5]{Tr67}). Regarding the isomorphism between $\orbI(G_1 \times G_2)$ and $\orbI(G_1) \hat{\otimes} \orbI(G_2)$, we refer to the arguments in \cite[Theorem 51.6]{Tr67}.
\end{remark}

The references \cite{Vi81,Bo94b} actually included the case of coverings as well. In particular, we may define the spaces $\Iasp(\tilde{G})$ of anti-genuine orbital integrals for a covering $\rev: \tilde{G} \to G(F)$. It is regarded as a space of anti-genuine functions $\Gamma_\text{reg}(\tilde{G}) \to \C$, equipped with the surjection $C^\infty_{c, \asp}(\tilde{G}) \twoheadrightarrow \Iasp(\tilde{G})$ and an appropriate topology in the archimedean case. \index{$\Iasp(\tilde{G})$}

When $G$ is quasisplit, we may also define the space $S\orbI(G)$ of stable orbital integrals, together with the linear surjection \index{$S\orbI(G)$}
\begin{align*}
  C^\infty_c(G(F)) & \longrightarrow S\orbI(G) \\
  f & \longmapsto \left[ \sigma \mapsto f^G(\sigma) = \sum_{\delta \mapsto \sigma} f_G(\delta) \right];
\end{align*}
in brief, the diagram
$$ \begin{tikzcd}[column sep=small]
  {} & C^\infty_c(G(F)) \arrow[twoheadrightarrow]{ld} \arrow[twoheadrightarrow]{rd} & \\
  \orbI(G) \arrow[twoheadrightarrow]{rr} & & S\orbI(G)
\end{tikzcd} \quad
\begin{tikzcd}[column sep=small]
  {} & f \arrow[mapsto]{ld} \arrow[mapsto]{rd} & \\
  f_G \arrow[mapsto]{rr} & & f^G
\end{tikzcd} $$
commutes. In the archimedean case, the characterization and topology of $S\orbI(G)$ are discussed in \cite[\S 6]{Bo94b}.

Elements in the dual spaces $\orbI(G)^\vee$ (resp. $\Iasp(\tilde{G})^\vee$, $S\orbI(G)^\vee$) are called \emph{invariant distributions} (resp. \emph{invariant genuine distributions}, \emph{stable distributions}) on $G$ or $\tilde{G}$. In each case, an element in $\orbI(G)$, $\orbI(\tilde{G})$ or $S\orbI(G)$ is determined by its restriction to any open dense subset of $G_\text{reg}(F)$ or $\tilde{G}_\text{reg}$, by \emph{the continuity of orbital integrals}.

All these constructions readily generalize to Lie algebras, the relevant definitions will be recalled in \S\ref{sec:image-transfer-cusp}.

\paragraph{Parabolic descent and induction}
Let $P = MU$ be a parabolic subgroup of $G$. Choose an appropriate maximal compact subgroup $K$ of $G(F)$ in good position relative to $M$ (see \cite[Définition 2.4.1]{Li14a}) so that the Iwasawa decomposition $G(F)=P(F)K$ holds. The Haar measures are normalized so that $\mes(K)=1$ and
$$ \int_{G(F)} f(x) \dd x = \iiint_{U(F) \times M(F) \times K} \delta_P(m)^{-1} f(umk) \dd u \dd m \dd k. $$

For $f \in C_c^{\infty}(G(F))$, define
$$ f_P: x \longmapsto \delta_P(x)^{\frac{1}{2}} \iint_{K \times U(F)} f(k^{-1} xu k) \dd u \dd k, \quad x \in M(F). $$
Then $f_P \in C^\infty_c(M(F))$. The linear map $f \mapsto f_P$ induces the \emph{parabolic descent} maps
\begin{align*}
  \orbI(G) & \longrightarrow \orbI(M)^{W^G(M)}, \\
  f_G & \longmapsto f_M
\end{align*}
and its stable version for quasisplit $G$
\begin{align*}
  S\orbI(G) & \longrightarrow S\orbI(M)^{W^G(M)} \\
  f^G & \longmapsto f^M.
\end{align*}
where $W^G(M)$ has a well-defined action on $\orbI(M)$ and $S\orbI(M)$ by conjugation.

For the next result, we recall briefly a well-known comparison between stable conjugacy and ordinary conjugacy. Let $H$ be any connected reductive $F$-group and $S$ be an $F$-subgroup of $H$. Define the pointed set \index{$\mathfrak{D}(S, H; F)$}
\begin{gather}\label{eqn:D}
  \mathfrak{D}(S, H; F) := \Ker\left[ H^1(F, S) \to H^1(F, H)  \right].
\end{gather}
Let $\sigma \in H_\text{reg}(F)$ and take $S=H_\sigma$. We claim that the conjugacy classes in the stable class of $\sigma$ is parametrized by $\mathfrak{D}(H_\sigma, H; F)$. Indeed, consider a $1$-cocycle $c: \Gamma_F \ni \tau \mapsto c(\tau)$ in $H_\sigma$ with trivial image in $H^1(F, H)$, we have $c(\tau) = x^{-1}\tau(x) \in H_\sigma(\bar{F})$ for some $x \in H(\bar{F})$ independent of $\tau$; the corresponding conjugacy class is then represented by $\sigma' := x \sigma x^{-1}$. Therefore we may define the ``relative position'' or the invariant as
\begin{gather}\label{eqn:inv}
  \inv(\sigma', \sigma) := [c] \in \mathfrak{D}(H_\sigma, H; F)
\end{gather}
for any $\sigma'$ stably conjugate to $\sigma$.

\begin{proposition}\label{prop:descent-orbint}
  For $\delta \in (M \cap G_\mathrm{reg})(F)$, we have
  $$ f_G(\delta) = f_M(\delta) $$
  for any $f \in C^\infty_c(G(F))$. When $G$ is quasisplit, for any $\sigma \in M_{G-\mathrm{reg}}(F)$ we have
  $$ f^G(\sigma) = f^M(\sigma). $$
\end{proposition}
Note that these equalities do not extend to $M_{\text{reg}}(F) \smallsetminus G_{\text{reg}}(F)$, in view of the germ expansions.
\begin{proof}
  The first equality is standard. For the second, observe that for a chosen $\sigma \in M_{G-\text{reg}}(F)$ we have $M_\sigma = G_\sigma$. It is well-known that $H^1(F, M) \hookrightarrow H^1(F, G)$, hence the natural map $\mathfrak{D}(M_\sigma, M; F) \to \mathfrak{D}(G_\sigma, G; F)$ is a bijection. Unwinding definitions, we see that $f^G(\sigma)$ equals
  $$ \sum_{\underbrace{\delta \mapsto \sigma}_{\text{in } G}} f_G(\delta) = \sum_{\underbrace{\delta \mapsto \sigma}_{\text{in } M}} f_G(\delta) = \sum_{\underbrace{\delta \mapsto \sigma}_{\text{in } M}} f_M(\delta) = f^M(\sigma) $$
  by the first assertion.
\end{proof}

\begin{corollary}\label{prop:parabolic-descent-transitivity}
  If $M \subset L$ are Levi subgroups of $G$, then
  $$ (f_L)_M = f_M, \quad f \in C^\infty_c(G(F)) $$
  and, when $G$ is quasisplit:
  $$ (f^L)^M = f^M. $$
\end{corollary}

The orbital integrals are dense in the space of invariant distributions relative to the topology induced by $\orbI(G)$. Hence $\orbI(G) \to \orbI(M)^{W^G(M)}$ depends only on $M$, not on the choice of $P$ and $K$. The same reasoning works for the stable orbital integrals.

Dualization gives \emph{parabolic induction} maps for distributions
\begin{align*}
  I_P: \orbI(M)^\vee / W^G(M) & \longrightarrow \orbI(G)^\vee, \\
  S\orbI(M)^\vee / W^G(M) & \longrightarrow S\orbI(G)^\vee, \quad G \text{ quasisplit }.
\end{align*}
The naming is justified since for every $\pi \in \Pi(M)$ we have the well-known identity
\begin{gather}\label{eqn:para-descent-character}
  \angles{I_P(\Theta_\pi), f} := \angles{\Theta_\pi, f_P} = \angles{\Theta_{I_P(\pi)}, f}, \quad f \in C^\infty_c(G(F)).
\end{gather}

In exactly the same manner, the parabolic descent $\Iasp(\tilde{G}) \to \Iasp(\tilde{M})^{W^G(M)}$ and its dual are defined for covering groups $\rev: \tilde{G} \to G(F)$.

\subsection{Endoscopy}\label{sec:endoscopy}
Fix a symplectic $F$-vector space $(W, \angles{\cdot|\cdot})$ of dimension $2n$. Write $G := \Sp(W)$ and $\rev: \tilde{G} := \Mp(W) \twoheadrightarrow G(F)$ be the metaplectic covering with $\Ker(\rev) = \bmu_8$.

\paragraph{Elliptic endoscopic data}
\begin{definition}\index{$\EndoE_\text{ell}(\tilde{G})$}
  An elliptic endoscopic datum for $\tilde{G}$ is a pair $(n',n'') \in \Z_{\geq 0}^2$ verifying $n'+n''=n$. The corresponding endoscopic group is
  $$ G^\Endo := \SO(2n'+1) \times \SO(2n''+1). $$
  The set of elliptic endoscopic data of $\tilde{G}$ is denoted by $\EndoE_\text{ell}(\tilde{G})$.
\end{definition}

\begin{remark}
  The definition is similar to that of $\SO(2n+1)$ in some aspects. However,
  \begin{compactitem}
    \item there is no symmetry $n' \leftrightarrow n''$,
    \item there is no ``outer automorphisms'' acting on $G^\Endo$.
  \end{compactitem}
  The latter point will be made clear in the proofs of our main theorems. By a standard abuse of notations, we will often write $G^\Endo$ to mean the endoscopic datum $(n', n'')$ together with the endoscopic group. However, we reiterate that due to the lack of symmetry, the endoscopic data \textsc{cannot} be identified with the corresponding $G^\Endo$. For example, the group $\SO(2n+1)$ appears twice in elliptic endoscopy for $\tilde{G}$, with different rôles.
\end{remark}

\begin{definition}\index{correspondence of classes}
  Fix $(n', n'') \in \EndoE_\text{ell}(\tilde{G})$. There is a correspondence between the semisimple stable conjugacy classes of $G(F)$ and that of $G^\Endo(F)$ in terms of eigenvalues. More precisely, we say that two semisimple elements $\delta \in G(F)$ and $\gamma=(\gamma', \gamma'') \in G^\Endo(F)$ correspond, if the eigenvalues of $\gamma'$, $\gamma''$ and $\delta$ (as elements in $\GL(2n'+1)$, $\GL(2n''+1)$ and $\GL(2n)$) can be arranged into the form
  \begin{gather*}
    a'_{n'}, \ldots, a'_1, 1, (a'_1)^{-1} \ldots (a'_{n'})^{-1}, \\
    a''_{n''}, \ldots, a''_1, 1, (a''_1)^{-1} \ldots (a''_{n''})^{-1}, \quad \text{and} \\
    a'_{n'}, \ldots, a'_1, (a'_1)^{-1} \ldots (a'_{n'})^{-1}, -a''_{n''}, \ldots, -a''_1, -(a''_1)^{-1} \ldots -(a''_{n''})^{-1}
  \end{gather*}
  respectively. In this case we write $\gamma \leftrightarrow \delta$.
\end{definition}
This is similar to the endoscopy for $\SO(2n+1)$, but there is a crucial difference: the eigenvalues of $\delta$ coming from $\SO(2n''+1)$ are twisted by $-1$.

We note that for a given $\gamma$, the conjugacy classes in $G(F)$ corresponding to $\gamma$ form a single stable conjugacy class. In fact, this correspondence induces a map between stable semisimple classes
\begin{gather}\label{eqn:mu}
  \mu: \Delta_\text{ss}(G^\Endo) \longrightarrow \Delta_\text{ss}(G)
\end{gather}
with finite fibers. An element of $\Delta_\text{ss}(G^\Endo)$ is called $G$-regular if its image under $\mu$ is regular. The notion of $G$-regularity is geometric: such elements form the $F$-points of a Zariski open dense subset $G^\Endo_{G-\text{reg}}$ of $G^\Endo$. Note that $G$-regular implies strongly regular in $G^\Endo$.

In \cite{Li11} it is shown that $G^\Endo_\gamma \simeq G_\delta$ for regular classes $\gamma \leftrightarrow \delta$, by describing their centralizers directly. Here we do it in another way which is closer to the spirit of endoscopy. As in \cite[\S 5.1]{Li11}, fix Borel pairs defined over $F$
$$ (B, T_s), \quad (B^\Endo, T^\Endo_s) $$
for $G$ and $G^\Endo$, respectively. Obviously we have
\begin{compactitem}
  \item an isomorphism $T^\Endo_s \rightiso T_s$ between split $F$-tori that preserves eigenvalues (we neglect the eigenvalues $1$ in $G^\Endo$, cf. the definition of the correspondence of classes),
  \item a homomorphism $W(G^\Endo, T^\Endo_s) \to W(G, T_s)$, making $T^\Endo_s \to T_s$ equivariant.
\end{compactitem}

Fix such homomorphisms. For $\gamma \leftrightarrow \delta$ as before, set $T := G_\delta$ and $T^\Endo := G^\Endo_\gamma$. There is a commutative diagram
\begin{equation}\label{eqn:diagram} \begin{tikzcd}
  T_{\bar{F}} \arrow{r}[above]{\Ad(x)} & (T_s)_{\bar{F}} \\
  T^\Endo_{\bar{F}} \arrow{r}[below]{\Ad(y)} \arrow[dashed]{u}[left]{\theta} & (T^\Endo_s)_{\bar{F}} \arrow{u}[right]{\text{defined over } \; F}
\end{tikzcd} \end{equation}
for some $x \in G(\bar{F})$ and $y \in G^\Endo(\bar{F})$; each arrow is invertible.

In view of the definition of $\gamma \leftrightarrow \delta$, we have $\theta(\gamma) \in W(G_{\bar{F}}, T_{\bar{F}}) \cdot \delta$. Upon modifying $x$, it can even be arranged that
\begin{equation}\label{eqn:theta-gamma-delta}
  \delta = \theta(\gamma).
\end{equation}

\begin{lemma}[cf. {\cite[Lemma 5.1]{Ad98}}]\label{prop:diagram}
  In the setup of \eqref{eqn:diagram}, one can choose $x, y$ so that $\theta$ is defined over $F$. The $F$-isomorphism $\theta$ so obtained is unique up $W(G, T)(F)$-action.

  In particular, $\gamma \leftrightarrow \delta$ implies $G^\Endo_\gamma \simeq G_\delta$.
\end{lemma} \index{standard isomorphism}
In \cite{Ad98}, Adams called such $\theta$ a \emph{standard isomorphism}. 

\begin{proof}
  Let us prove the uniqueness first. Suppose that the $F$-isomorphisms $\theta$, $\theta'$ arise from the pairs $(x,y)$ and $(x',y')$, respectively. We need to show that $\theta$ equals $\theta'$ up to $W(G, T)(F)$. Take $w \in W(G^\Endo_{\bar{F}}, T^\Endo_{\bar{F}})$ such that $\Ad(y') = \Ad(w)\Ad(y)$. Using the equivariance of $T^\Endo_s \to T_s$ with respect to $W(G^\Endo, T^\Endo_s) \to W(G, T_s)$, we choose $x''$ so that $\Ad(w)\Ad(x'') = \Ad(x')$. Then $\theta'$ equals the $F$-isomorphism arising from $(x'', y)$.
  
  Observe that $\Ad((x'')^{-1}x): T_{\bar{F}} \rightiso T_{\bar{F}}$ is defined over $F$. Hence $(x'')^{-1}x$ modulo $T(\bar{F})$ is in $W(G, T)(F)$, that is, $\theta$ and $\theta'$ differ by an element of $W(G, T)(F)$.
  
%
%
  To show the existence of a $\theta$ defined over $F$, we begin by choosing $x, y$ satisfying \eqref{eqn:theta-gamma-delta}. For any $\tau \in \Gamma_F$, the commutative diagrams
  $$ \begin{tikzcd}
    T_{\bar{F}} & \arrow{l}[above, inner sep=1em]{\Ad(x)^{-1}} (T_s)_{\bar{F}} \\
    T^\Endo_{\bar{F}} \arrow{r}[below, inner sep=1em]{\Ad(y)} \arrow{u}[left]{\theta} & (T^\Endo_s)_{\bar{F}} \arrow{u}
  \end{tikzcd} \quad
  \begin{tikzcd}
    T_{\bar{F}} & \arrow{l}[above, inner sep=1em]{\Ad(\tau(x))^{-1}} (T_s)_{\bar{F}} \\
    T^\Endo_{\bar{F}} \arrow{r}[below, inner sep=1em]{\Ad(\tau(y))} \arrow{u}[left]{{}^\tau\theta} & (T^\Endo_s)_{\bar{F}} \arrow{u}
  \end{tikzcd} $$
  both send $\gamma \in T^\Endo(F)$ to $\delta \in T(F)$. By the same reasoning as before, but without worry about rationality, we see that $\theta$ and ${}^\tau \theta$ differ by some element in $W(G_{\bar{F}}, T_{\bar{F}})$. The regularity of $\delta$ forces $\theta = {}^\tau\theta$, thereby proving the rationality of $\theta$.
\end{proof}

Next, consider a covering of metaplectic type
$$ (\identity, \rev): \tilde{M} = \prod_{i \in I} \GL(n_i, F) \times \Mp(W^\flat) \twoheadrightarrow \prod_{i \in I} \GL(n_i, F) \times \Sp(W^\flat) = M(F); $$
see Definition \ref{def:metaplectic-type}.

\begin{definition}\label{def:endoscopy-M}
  Let $\tilde{M} \twoheadrightarrow M(F)$ be a covering of metaplectic type described above. Define the set $\EndoE_\text{ell}(\tilde{M})$ of elliptic endoscopic data of $\tilde{M}$ to be $\EndoE_\text{ell}(\Mp(W^\flat))$. In other words, an elliptic endoscopic datum is a pair $(m', m'') \in \Z_{\geq 0}^2$ with $m'+m'' = \frac{1}{2} \dim_F W^\flat$; the corresponding endoscopic group is
  $$ M^\Endo := \prod_{i \in I} \GL(n_i) \times \SO(2m'+1) \times \SO(2m''+1). $$

  Let $\delta = (\delta_{\GL}, \delta^\flat)$ and $\gamma = (\gamma_{\GL}, \gamma^\flat)$ be semisimple elements in $M(F)$ and $M^\Endo(F)$, respectively, where the components in $\prod_i \GL(n_i)$ carry the subscript $\GL$. We say that $\delta$ corresponds to $\gamma$, denoted by $\delta \leftrightarrow \gamma$, if
  \begin{compactitem}
    \item $\delta_{\GL}$ and $\gamma_{\GL}$ are conjugate,
    \item $\delta^\flat \leftrightarrow \gamma^\flat$.
  \end{compactitem}

  As before, we obtain a map $\mu: \Delta_\text{ss}(M^\Endo) \to \Delta_\text{ss}(M)$ on semisimple conjugacy classes, of the form $(\identity, \mu^\flat)$. The notions of $M$-regularity, etc., are defined in the obvious way.
\end{definition}

Simply put, elliptic endoscopy for $\tilde{M}$ does not affect the components in $\prod_{i \in I} \GL(n_i)$.

\paragraph{Stable conjugacy}
Recall the notion of stable conjugacy in $\tilde{G}$ of regular semisimple elements in \cite[\S 5.2]{Li11}. Firstly, recall that $(\Theta^+_\psi - \Theta^-_\psi)(\cdot)$ can be evaluated on $\tilde{G}_\text{reg}$. \index{stable conjugacy}

\begin{definition}[Adams {\cite[\S 3]{Ad98}}]\label{def:stable-conj}
  Two regular semisimple elements $\tilde{\delta}, \tilde{\delta}' \in \tilde{G}$ are called \emph{stably conjugate} whenever
  \begin{compactenum}[(i)]
    \item their images $\delta, \delta'$ in $G(F)$ are stably conjugate in the ordinary sense, and
    \item $\Theta^+_\psi - \Theta^-_\psi$ takes the same value at $\tilde{\delta}$ and $\tilde{\delta}'$.
  \end{compactenum}
\end{definition}

\begin{lemma}\label{prop:stable-conj-lifting}
  Assume that $\delta_0, \delta_1 \in G_\mathrm{reg}(F)$ are stably conjugate, and let $\tilde{\delta}_0 \in \rev^{-1}(\delta_0)$. There exists a unique $\tilde{\delta}_1 \in \rev^{-1}(\delta_1)$ such that $\tilde{\delta}_1$ is stably conjugate to $\tilde{\delta}_0$.
\end{lemma}
\begin{proof}
  This is essentially \cite[Lemme 5.7]{Li11}.
\end{proof}
The notion of stable conjugacy is extended to all semisimple classes in \cite[Définition 4.1.1]{Li15}.

Naturally, stable conjugacy can also be defined for coverings of metaplectic type $\tilde{M} = \prod_{i \in I} \GL(n_i) \times \Mp(W^\flat)$. Recall that conjugacy and stably conjugacy are the same on the general linear groups $\GL(n_i)$.

\begin{definition}
  Two elements $((\delta_i)_{i \in I}, \tilde{\delta})$ and $((\delta_{i, 1})_{i \in I}, \tilde{\delta}_1)$ are called stable conjugate when
  \begin{compactenum}[(i)]
    \item $\delta_i$ and $\delta_{i, 1}$ are conjugate for each $i \in I$, and
    \item $\tilde{\delta}$ and $\tilde{\delta}_1$ are stably conjugate in $\Mp(W^\flat)$.
  \end{compactenum}
\end{definition}

\paragraph{Passage to Levi subgroups}
We are ready to define the general endoscopic data for $\tilde{G}$.
\begin{definition}
  The set of endoscopic data of $\tilde{G}$ is defined as
  $$ \EndoE(\tilde{G}) := \bigsqcup_{M/\text{conj}} \EndoE_{\text{ell}}(\tilde{M}) $$
  where $M$ runs over the Levi subgroups of $G$ modulo conjugation; the right hand side is well-defined since $W^G(M)$ leaves $\EndoE_{\text{ell}}(\tilde{M})$ intact.

  To each endoscopic datum in $\EndoE_{\text{ell}}(\tilde{M})$ we attach the endoscopic group $M^\Endo$. The correspondence of semisimple conjugacy classes is given by the composition of
  $$ \Delta_\text{ss}(M^\Endo) \longrightarrow \Delta_\text{ss}(M) \longrightarrow \Delta_\text{ss}(G) $$
  where the rightmost arrow is induced by the inclusion of Levi subgroup $M \hookrightarrow G$, well-defined up to $W^G(M)$.
\end{definition}

\begin{remark}
  We note that
  \begin{enumerate}[(a)]
    \item the same definition works when $F$ is a number field;
    \item our definition of endoscopic data is slightly different from that in \cite{Li12a}, cf. \cite[Proposition 3.1.8]{Li12a};
    \item there is an explication in terms of the dual group of $\tilde{G}$, defined as $\Sp(2n, \C)$ with trivial Galois action, see \cite[\S 3]{Li12a} for details;
    \item we will use diagrams of the form
    \begin{equation}\label{eqn:endo-incomplete} \begin{tikzcd}
      {} & \tilde{G} \\
      M^\Endo \arrow[-,dashed]{r}[above]{\text{ell.}}[below]{\text{endo.}} & \tilde{M} \arrow[hookrightarrow]{u}[right]{\text{Levi}}
    \end{tikzcd} \end{equation}
    to recapitulate the fact that $M^\Endo \in \EndoE_\text{ell}(\tilde{M}) \subset \EndoE(\tilde{G})$.
  \end{enumerate}
  Here the inclusions of Levi subgroup are always taken up to conjugacy.
\end{remark}

We wish to complete a given diagram \eqref{eqn:endo-incomplete} into \index{$G[s]$}
\begin{equation}\label{eqn:endo-complete} \begin{tikzcd}
  G[s] \arrow[-,dashed]{r}[above]{\text{ell.}}[below]{\text{endo.}} & \tilde{G} \\
  M^\Endo \arrow[-,dashed]{r}[above]{\text{ell.}}[below]{\text{endo.}} \arrow[hookrightarrow]{u}[left]{\text{Levi}} & \tilde{M} \arrow[hookrightarrow]{u}[right]{\text{Levi}}
\end{tikzcd} \end{equation}
in all possible ways, where $s$ is some parameter to be described. Such a recipe is given in \cite[\S 3.3]{Li12a} which we review below. Write
\begin{align*}
  M & = \prod_{i \in I} \GL(n_i) \times \Sp(W^\flat), \\
  M^\Endo & = \prod_{i \in I} \GL(n_i) \times \SO(2m' +1) \times \SO(2m'' + 1),
\end{align*}
the latter is attached to $(m', m'') \in \EndoE_\text{ell}(\tilde{M})$. For an endoscopic datum $(n', n'') \in \EndoE_\text{ell}(\tilde{G})$ giving rise to $G[s]$, the embedding $M^\Endo \hookrightarrow G[s]$ is determined by a decomposition $I = I' \sqcup I''$ up to conjugacy: set
\begin{align*}
  G[s] & := \SO(2n'+1) \times \SO(2n''+1), \\
  \prod_{i \in I'} \GL(n_i) \times \SO(2m'+1) & \hookrightarrow \SO(2n'+1), \\
  \prod_{i \in I''} \GL(n_i) \times \SO(2m''+1) & \hookrightarrow \SO(2n''+1).
\end{align*}
where $n' = m' + |I'|$ and $n'' = m'' + |I''|$ and the embeddings are the usual ones for Levi subgroups in odd $\SO$. It follows that the data in \eqref{eqn:endo-complete} are all encoded in the decomposition $I = I' \sqcup I''$, which we take to be the parameter $s$. We are led to the following. \index{$\EndoE_{M^\Endo}(\tilde{G})$}

\begin{definition}\label{def:s}
  Given a diagram \eqref{eqn:endo-incomplete}, we set $\EndoE_{M^\Endo}(\tilde{G})$ to be the set of all (ordered) pairs $(I', I'')$ such that $I = I' \sqcup I''$. To each $s = (I', I'') \in \EndoE_{M^\Endo}(\tilde{G})$ we attach
  $$ (n', n'') := (m' + |I'|, m'' + |I''|) \in \EndoE_\text{ell}(\tilde{G}), $$
  and obtain the completed diagram \eqref{eqn:endo-complete} by the recipe above.
\end{definition}

\begin{remark}\label{rem:EndoE-vs-embedding}
  In general, different elements in $\EndoE_{M^\Endo}(\tilde{G})$ can give rise to the same elliptic endoscopic datum of $\tilde{G}$. If $s, t \in \EndoE_{M^\Endo}(\tilde{G})$ give rise to the same endoscopic datum, the embeddings of $M^\Endo$ will differ by some element in $W^G(M) = \mathfrak{S}(I)$.
\end{remark}

In the diagram \eqref{eqn:endo-complete}, there are two ways to relate stable semisimple classes between $M^\Endo$ and $G$:
\begin{enumerate}[(i)]
  \item going in the direction $M^\Endo \leadsto G[s] \leadsto G$ furnishes a map
    \begin{gather}\label{eqn:mu_1}
      \mu_1: \Delta_\text{ss}(M^\Endo) \to \Delta_\text{ss}(G);
    \end{gather}
  \item going in the direction $M^\Endo \leadsto M \leadsto G$ furnishes a similar map
    \begin{gather}\label{eqn:mu_2}
      \mu_2: \Delta_\text{ss}(M^\Endo) \to \Delta_\text{ss}(G);
    \end{gather}
\end{enumerate}
In contrast with the endoscopy for linear reductive groups, the maps $\mu_1, \mu_2$ can be different.

\begin{lemma}[{\cite[Proposition 3.3.4]{Li12a}}]\label{prop:z[s]} \index{$z[s]$}
  Let $z[s] := ((z_i)_{i \in I}, 1) \in M^\Endo(F)$ be the element of order two defined by
  $$ z_i = \begin{cases} 1, & i \in I', \\ -1, & i \in I''. \end{cases} $$
  Then we have
  $$ \mu_1(t) = \mu_2(z[s] \cdot t), \quad t \in \Delta_\mathrm{ss}(M^\Endo) $$
\end{lemma}
Since $z[s]$ is central, its translation action on conjugacy classes in $M^\Endo(F)$ is well-defined.

On the other hand, we may also start from some $(n', n'') \in \EndoE_\text{ell}(\tilde{G})$ with endoscopic group $G^\Endo$, together with a Levi subgroup $M^\Endo$ of $G^\Endo$, and try to embed it into a diagram
\begin{equation}\label{eqn:M^Endo-to-s} \begin{tikzcd}
  G^\Endo \arrow[-,dashed]{r}[above]{\text{ell.}}[below]{\text{endo.}} & \tilde{G} \\
  M^\Endo \arrow[-,dashed]{r}[above]{\text{ell.}}[below]{\text{endo.}} \arrow[hookrightarrow]{u}[left]{\text{Levi}} & \tilde{M} \arrow[hookrightarrow]{u}[right]{\text{Levi}}
\end{tikzcd} \end{equation}
and identify $G^\Endo$ with some $G[s]$ as endoscopic data. This situation turns out to be much simpler: write
$$ M^\Endo = \left( \prod_{i \in I'} \GL(n_i) \times \SO(2m'+1) \right) \times \left( \prod_{i \in I''} \GL(n_i) \times \SO(2m''+1)\right) $$
embedded as a Levi subgroup into $G^\Endo = \SO(2n'+1) \times \SO(2n''+1)$, according to the parentheses. 

\begin{lemma}\label{prop:M^Endo-to-s}
  Given a Levi subgroup $M^\Endo$ of $G^\Endo$ as above, there exist
  \begin{compactitem}
    \item $M$: a Levi subgroup of $G$,
    \item $(m', m'') \in \EndoE_\mathrm{ell}(\tilde{M})$ such that the $M^\Endo$ can be realized as the endoscopic group,
    \item $s \in \EndoE_{M^\Endo}(\tilde{M})$
  \end{compactitem}
  such that $G^\Endo = G[s]$ (see \eqref{eqn:endo-complete}) as elliptic endoscopic data of $\tilde{G}$. Moreover, the triple
  $$ \left( M/\mathrm{conj}, (m', m''), s \right) $$
  is unique. 
\end{lemma}
\begin{proof}
  We have already written $(m',m'')$ down. The remaining choices are clear: set $I := I' \sqcup I''$, $M := \prod_{i \in I} \GL(n_i) \times \Sp(W^\flat)$ with $\dim_F W^\flat = 2(m'+m'')$, and $s=(I', I'')$.
\end{proof}

\subsection{Geometric transfer and the fundamental lemma}\label{sec:transfer-FL}
\paragraph{Geometric transfer factors}
The transfer factors are defined in \cite{Li11}. In this article it is more appropriate to call them \emph{geometric transfer factors}. We begin by reviewing the case of an elliptic endoscopic datum $(n', n'')$ for $\tilde{G} = \Mp(W)$. Let $G^\Endo = \SO(2n'+1) \times \SO(2n''+1)$ be the endoscopic group. For $\delta \in G_\text{reg}(F)$, $\tilde{\delta} \in \rev^{-1}(\delta)$ and $\gamma = (\gamma', \gamma'') \in G^\Endo_\text{reg}(F)$ such that $\delta \leftrightarrow \gamma$, we obtain an orthogonal decomposition
\begin{align*}
  W &= W' \oplus W'' \\
  \delta &= (\delta', \delta'')
\end{align*}
by separating the eigenvalues from $\gamma'$ and $\gamma''$; see \cite[\S 5.3]{Li11} for details. The factor $\Delta(\gamma, \tilde{\delta})$ is of the form
$$ \Delta(\gamma, \tilde{\delta}) = \Delta_0(\delta', \delta'') \Delta'(\tilde{\delta}') \Delta''(\tilde{\delta}'') \quad \in \bmu_8 , $$
where \index{$\Delta(\gamma, \tilde{\delta})$}
\begin{align*}
  \tilde{\delta} & = j(\tilde{\delta}', \tilde{\delta}''), \\
  \Delta'(\tilde{\delta}') & := \dfrac{ \Theta^+_\psi - \Theta^+_- }{ |\Theta^+_\psi - \Theta^+_-| } (\tilde{\delta}') \qquad \text{defined relative to } W', \\
  \Delta''(\tilde{\delta}'') & := \dfrac{ \Theta^+_\psi + \Theta^+_- }{ |\Theta^+_\psi + \Theta^+_-| } (\tilde{\delta}'') \qquad \text{defined relative to } W'',
\end{align*}
and $\Delta_0(\delta', \delta'')$ is a factor which has nothing to do with the coverings or $\psi$; it is unnecessary to recall the precise definition here. Note that $\Delta'(\tilde{\delta}') \Delta''(\tilde{\delta}'')$ does not depend on the choice of $(\tilde{\delta}', \tilde{\delta}'')$.

The notations $\Delta_{G^\Endo, \tilde{G}}$ or $\Delta_{(n',n'')}$ will also be used whenever appropriate. Some basic properties are recorded below.
\begin{enumerate}
  \item \emph{Invariance.} The factor $\Delta$ depends only on the conjugacy class of $\tilde{\delta}$ and the stable conjugacy class of $\gamma$.
  \item \emph{Genuineness.} For every $\noyau \in \bmu_8$ we have $\Delta(\gamma, \noyau\tilde{\delta}) = \noyau \Delta(\gamma, \tilde{\delta})$.
  \item \emph{Symmetry.} Write $\Delta_{n', n''}(\cdots)$ to indicate the endoscopic datum in question. Then
    $$ \Delta_{(n',n'')}((\gamma', \gamma''), \tilde{\delta}) = \Delta_{n'', n'}((\gamma'', \gamma'), -\tilde{\delta}) $$
    where $-\tilde{\delta} := (-1) \cdot \tilde{\delta}$ using the $-1 \in \tilde{G}$ of Definition \ref{def:-1}. It is implicit assumed that $(\gamma'', \gamma') \leftrightarrow -\delta$ relative to $(n'', n') \in \EndoE_\text{ell}(\tilde{G})$, which is clear from the definition of the correspondence of classes.
  \item \emph{Cocycle condition.} Let $\delta \leftrightarrow \gamma$ as before, so that $G_\delta$ is a maximal $F$-torus in $G$. If $\tilde{\delta}$ is stably conjugate to $\tilde{\delta}_1$, one can attach a cohomological invariant
    $$ \inv(\tilde{\delta}, \tilde{\delta}_1) = \inv(\delta, \delta_1) \in H^1(F, G_\delta) $$
    by the recipe \eqref{eqn:inv}. Here $H^1(F, G_\delta)$ equals the pointed set $\mathfrak{D}(G_\delta, G; F)$ of \eqref{eqn:D} since $H^1(F, G)$ is trivial, thus equals the abelianized version $\mathfrak{E}(G_\delta, G; F)$ used in the stabilization of trace formula \cite[\S 3.1]{Li15}. It is actually a finite commutative group of exponent $2$. \index{transfer factor!cocycle condition}

    On the other hand, for these data we have the \emph{endoscopic character} \index{endoscopic character}\index{$\kappa_T$}
    $$ \kappa: H^1(F, G_\delta) \to \bmu_2. $$
    Below is a sketch of its definition in \cite[\S 5.3]{Li11}. Decompose $T := G_\delta$ into $T' \times T''$ by separating the eigenvalues from $\gamma'$ and $\gamma''$, then $H^1(F, T'') \rightiso (\bmu_2)^{I''}$ for some finite set $I''$ and $\kappa$ is simply the composition of $H^1(F, T) \twoheadrightarrow H^1(F, T'')$ with the product map $(\bmu_2)^{I''} \to \bmu_2$, $(\lambda_i)_i \mapsto \prod_i \lambda_i$. All in all, the cocycle condition asserts that
    $$ \Delta(\gamma, \tilde{\delta}_1) = \angles{\kappa, \inv(\delta, \delta_1)} \cdot \Delta(\gamma, \tilde{\delta}). $$
  \item \emph{Parabolic descent.} Consider the diagram \eqref{eqn:endo-complete} and keep the notations thereof. An element $\gamma \in M^\Endo \cap G^\Endo_\text{reg}(F)$ is expressed in the form $\gamma = ((\gamma_i)_{i \in I}, \gamma^\flat)$ where $\gamma^\flat \in \SO(2m'+1) \times \SO(2m''+1)$; similarly we write $\tilde{\delta} = ((\delta_i)_{i \in I}, \tilde{\delta}^\flat) \in \tilde{M} \cap G_\text{reg}(F)$ where $\tilde{\delta}^\flat \in \Mp(W^\flat)$. If $\gamma \leftrightarrow \delta$, parabolic descent asserts\index{transfer factor!parabolic descent}
    \begin{gather}\label{eqn:Delta-paradescent}
      \Delta(\gamma, \tilde{\delta}) = \Delta(\gamma^\flat, \tilde{\delta}^\flat),
    \end{gather}
    the right-hand side being defined relative to $(m', m'') \in \EndoE_\text{ell}(\Mp(W^\flat))$. This is the content of \cite[Proposition 5.18]{Li11}.
\end{enumerate}

\paragraph{Transfer of orbital integrals}
Now we can state the geometric transfer and the fundamental lemma. Fix an elliptic endoscopic datum $(n', n'')$ for $\tilde{G}$. For $\gamma \leftrightarrow \delta$ as above, we define the notion of compatible Haar measures on $G^\Endo_\gamma(F)$ and $G_\delta(F)$ via the $G^\Endo_\gamma \simeq G_\delta$ furnished by Lemma \ref{prop:diagram}; the choice of the isomorphism is immaterial. In what follows, the Haar measures on $G(F)$ and $G^\Endo(F)$ are fixed, and we use compatible Haar measures on $G^\Endo_\gamma(F)$ and $G_\delta(F)$ whenever $\gamma \leftrightarrow \delta$.

\begin{theorem}[Geometric transfer {\cite[Théorème 5.20]{Li11}}]\label{prop:geometric-transfer}
  Let $f \in C^\infty_{c, \asp}(\tilde{G})$. There exists $f^\Endo \in C^\infty_c(G^\Endo(F))$ such that for any $\gamma \in \Delta_{G-\mathrm{reg}}(G^\Endo)$, we have
  $$ \sum_{\delta \leftrightarrow \gamma} \Delta(\gamma, \tilde{\delta}) f_{\tilde{G}}(\tilde{\delta}) = (f^\Endo)^{G^\Endo}(\gamma), $$
  where $\delta$ ranges over $\Gamma_\mathrm{reg}(G)$ and $\tilde{\delta} \in \rev^{-1}(\delta)$ is arbitrary.
\end{theorem}

Due to the genuineness of $\Delta$, the term $\Delta(\gamma, \tilde{\delta}) f_{\tilde{G}}(\tilde{\delta})$ is independent of the choice of $\tilde{\delta}$. Call $f^\Endo$ a transfer of $f$.

\begin{remark}\label{rem:ambiguity-transfer}\index{$\mathcal{T}_{(n',n'')}$}
  The transfer $f^\Endo$ as a function on $G^\Endo(F)$ is not unique, but it induces a canonical linear map
  $$ \mathcal{T}_{(n', n'')}: \Iasp(\tilde{G}) \to S\orbI(G^\Endo) $$
  between the spaces of normalized orbital integrals. When $F$ is archimedean, $\mathcal{T}_{(n', n'')}$ is even continuous: see \S\ref{sec:Renard}.
\end{remark}

\begin{theorem}[Fundamental lemma for the unit {\cite[Théorème 5.23]{Li11}}]\label{prop:FL}
  In the unramified setup summarized in \S\ref{sec:splittings}, take the hyperspecial subgroup $K := \Stab_{\Sp(W)}(L)$ of $G(F)$, regarded as a subgroup of $\tilde{G}$, and define $f_K \in C^\infty_{c, \asp}(\tilde{G})$ by \index{$f_K$}
  $$
    f_K(\tilde{x}) = \begin{cases}
      \noyau^{-1}, & \text{ if } \tilde{x} \in \noyau K, \quad \noyau \in \bmu_8, \\
      0, & \text{otherwise}.
    \end{cases}
  $$
  Choose the unramified Haar measures on $G(F)$ and $G^\Endo(F)$. Let $K^\Endo$ be any hyperspecial subgroup of $G^\Endo(F)$), then we may take
  $$ f^\Endo = \mathbf{1}_{K^\Endo} $$
  in Theorem \ref{prop:geometric-transfer} as a transfer of $f$.
\end{theorem}

More generally, consider a covering $\rev: \tilde{M} \twoheadrightarrow M(F)$ of metaplectic type and let $M^\Endo$ be the endoscopic group associated to some $(m', m'') \in \EndoE_\text{ell}(\tilde{M})$ (Definition \ref{def:endoscopy-M}). For regular semisimple elements $\gamma \leftrightarrow \delta$ and $\tilde{\delta} \in \rev^{-1}(\delta)$, one can still define the transfer factor $\Delta(\gamma, \tilde{\delta})$ as follows. Following \cite[Définition 3.2.3]{Li12a}, for $\gamma \in M^\Endo_\text{reg}(F)$ and $\tilde{\delta} \in \tilde{M}_\text{reg}$, write $\tilde{M} = \prod_{i \in I} \GL(n_i, F) \times \Mp(W^\flat)$ and
$$ \gamma = ((\gamma_i)_{i \in I}, \gamma^\flat), \quad \tilde{\delta} =  ((\delta_i)_{i \in I}, \tilde{\delta}^\flat) $$
as before. If $\gamma \leftrightarrow \delta$, we set
\begin{gather}\label{eqn:transfer-factor-Levi}
  \Delta_{M^\Endo, \tilde{M}}(\gamma, \tilde{\delta}) := \Delta(\gamma^\flat, \tilde{\delta}^\flat).
\end{gather}
The aforementioned properties of genuineness, symmetry, etc. of $\Delta$ still hold for $\Delta_{M^\Endo, \tilde{M}}$: the $\GL$-components do not interfere at all.

\begin{theorem}
  The assertions in Theorems \ref{prop:geometric-transfer} and \ref{prop:FL} continue to hold for an endoscopic datum $(m', m'') \in \EndoE_\mathrm{ell}(\tilde{M})$, with respect to the transfer factor \eqref{eqn:transfer-factor-Levi}.
\end{theorem}
\begin{proof}
  It suffices to consider the transfer of $f = f_I \otimes f^\flat \in C^\infty_{c, \asp}(\tilde{G})$, where $f_I \in C^\infty_c \left( \prod_{i \in I} \GL(n_i, F) \right)$ and $f^\flat \in C^\infty_{c, \asp}(\Mp(W^\flat))$. Then for $\gamma = ((\gamma_i)_{i \in I}, \gamma^\flat)$ as above,
  $$ \sum_{\delta \leftrightarrow \gamma} \Delta_{M^\Endo, \tilde{M}}(\gamma, \tilde{\delta}) f_{\tilde{M}}(\tilde{\delta}) = f_I((\gamma_i)_{i \in I}) \cdot \left(\; \sum_{\delta^\flat \leftrightarrow \gamma^\flat} \Delta(\gamma^\flat, \tilde{\delta}^\flat) f^\flat_{\Mp(W^\flat)}(\tilde{\delta}^\flat) \right) $$
  since the correspondence of classes is tautological on the $\GL$-components. Hence we are reduced to the corresponding assertions for $\Mp(W^\flat)$.
\end{proof}

For the next results, the maps $f \mapsto f_{\tilde{M}} \in \orbI(\tilde{M})$ and $f^\Endo \mapsto (f^\Endo)^{M^\Endo} \in S\orbI(M^\Endo)$ in \S\ref{sec:space-orbital-integral} will be used. In fact, we need an $s$-twisted version thereof.

\begin{definition}\label{def:s-descent}\index{$(f^\Endo)^{s, M^\Endo}$}
  Let $M$ be a Levi subgroup of $G = \Sp(W)$, and let $M^\Endo$ be an endoscopic group associated to $(m', m'') \in \EndoE_\mathrm{ell}(\tilde{M})$. Consider the diagram \eqref{eqn:endo-complete} determined by some $s \in \EndoE_{M^\Endo}(\tilde{G})$. Write $G^\Endo = G[s]$. For every $f^\Endo \in C^\infty_c(G^\Endo(F))$ we define
  $$ (f^\Endo)^{s, M^\Endo}: \gamma \longmapsto (f^\Endo)^{M^\Endo}(z[s]\gamma) $$
  where $z[s]$ is as in Lemma \ref{prop:z[s]}; this is an element in $S\orbI(M^\Endo)$ and depends only on the image of $(f^\Endo)$ in $S\orbI(G^\Endo)$. We shall call this the \emph{$s$-twisted parabolic descent}.
\end{definition}

\begin{theorem}\label{prop:transfer-parabolic}
  Let $G$, $M$, $M^\Endo$ and $s \in \EndoE_{M^\Endo}(\tilde{G})$ be as in the previous definition. For every $f \in C^\infty_{c, \asp}(\tilde{G})$ let $f^\Endo$ be a transfer of $f$ to $G^\Endo := G[s]$. Then
  $$ (f_{\tilde{M}})^\Endo = (f^\Endo)^{s, M^\Endo} $$
  as elements in $S\orbI(M^\Endo)$. In particular, $(f^\Endo)^{s, M^\Endo}$ is independent of $s$. Denote it by $f^{M^\Endo}$.
\end{theorem}
Here $(f_{\tilde{M}})^\Endo \in S\orbI(M^\Endo)$ stands for the transfer of $f_{\tilde{M}} \in \Iasp(\tilde{M})$.

\begin{proof}
  It suffices to check this for $\gamma \in M^\Endo_\text{reg}(F)$ such that $\gamma$ correspond to elements in $G_\text{reg}(F)$ via both the maps $\mu_1$ and $\mu_2$ in \eqref{eqn:mu_1}, \eqref{eqn:mu_2}. Denote this Zariski open subset by $M^\Endo_{G-\text{reg}}$. By Proposition \ref{prop:descent-orbint} for coverings,
  \begin{align*}
    (f_{\tilde{M}})^\Endo(\gamma) & = \sum_{\gamma \leftrightarrow \delta} \Delta_{M^\Endo, \tilde{M}}(\gamma, \tilde{\delta}) f_{\tilde{M}}(\tilde{\delta}) \\
    & = \sum_{\gamma \leftrightarrow \delta} \Delta_{M^\Endo, \tilde{M}}(\gamma, \tilde{\delta}) f_{\tilde{G}}(\tilde{\delta}), \quad \tilde{\delta} \in \rev^{-1}(\delta) \text{ is arbitrary}.
  \end{align*}

  The sum is taken over $\Xi^M := \{\delta \in \Gamma_\text{reg}(M) : \gamma \leftrightarrow \delta\}$. By the definition of the correspondence of classes, $\Xi^M$ is the set of conjugacy classes in some stable class in $M_{G-\text{reg}}(F)$. Pick $\sigma \in \Xi^M$. It is a standard fact that $M \hookrightarrow G$ identifies $\Xi^M$ with the set $\Xi^G$ of conjugacy classes in the stable class of $\sigma$ in $G_\text{reg}(F)$; see the proof of the second part of Proposition \ref{prop:descent-orbint}. From Lemma \ref{prop:z[s]}, we conclude $\Xi^G = \{\delta \in \Gamma_\text{reg}(G): \gamma[s] \leftrightarrow \delta \}$ with respect to the elliptic endoscopic datum attached to $s$.

  The parabolic descent for transfer factors \eqref{eqn:Delta-paradescent} implies
  $$ \Delta_{M^\Endo, \tilde{M}}(\gamma, \tilde{\delta}) = \Delta(\gamma^\flat, \tilde{\delta}^\flat) = \Delta_{G[s], \tilde{G}}(\gamma[s], \tilde{\delta}), $$
  for the $\gamma$, $\tilde{\delta}$ in the sum. Hence
  $$ (f_{\tilde{M}})^\Endo(\gamma) = \sum_{\gamma[s] \leftrightarrow \delta} \Delta_{G[s], \tilde{G}}(\gamma[s], \tilde{\delta}) f_{\tilde{G}}(\tilde{\delta}). $$
  The right-hand side is just $(f^\Endo)(\gamma[s])$, which equals $(f^\Endo)^{M^\Endo}(\gamma[s])$ by Proposition \ref{prop:descent-orbint}.
\end{proof}

All in all, given $(n', n'') \in \EndoE_\text{ell}(\tilde{G})$, our goal is to understand the dual
$$ \mathcal{T}_{(n', n'')}^\vee: S\orbI(G^\Endo)^\vee \to \Iasp(\tilde{G}) $$
of the geometric transfer $\mathcal{T}_{(n',n'')}$, in spectral terms. In other words, we aim to settle the \emph{spectral transfer}.

\section{Results for the odd orthogonal groups}\label{sec:results-SO}
In this section, $F$ always denotes a local field of characteristic zero. The materials below are largely based upon Arthur's monumental  work \cite{Ar13}.

\subsection{$L$-parameters}\label{sec:L-parameters}
\paragraph{Generalities}
Let $G$ be a connected reductive $F$-group; we assume $G$ quasisplit for simplicity. Its dual group $\hat{G}$ is endowed with a $\Gamma_F$-action which factors through a finite quotient. The precise construction of $\hat{G}$ involves choosing an $F$-pinning $(T, B, (E_\alpha)_{\alpha \in \Delta(B, T)})$ and taking the dual based root datum. Then $\hat{G}$ is endowed with a dual pinning $(\hat{T}, \hat{B}, \cdots)$ that is $\Gamma_F$-stable. We refer to \cite{Bo79} for details. The $L$-group of $G$ is $\Lgrp{G} = \hat{G} \rtimes \We_F$.

An $L$-parameter for $G$ is a homomorphism
$$ \phi: \WD_F \longrightarrow \Lgrp{G} $$
such that
\begin{itemize}
  \item the composition of $\phi$ with the projection $\Lgrp{G} \to \We_F$ equals $\WD_F \to W_F$;
  \item $\phi$ is continuous;
  \item the projection of $\Im(\phi|_{W_F})$ to $\hat{G}$ consists of semisimple elements.
\end{itemize}

Call two $L$-parameters $\phi_1$, $\phi_2$ \emph{equivalent}, written as $\phi_1 \sim \phi_2$, if they are conjugate by $\hat{G}$. We say that $\phi$ is \emph{bounded} if the projection of $\Im(\phi)$ to $\hat{G}$ is relatively compact. Given $\phi$, define the $S$-group as
$$ S_\phi := Z_{\hat{G}}(\Im(\phi)). $$
Its identity connected component $S^0_\phi$ is a connected reductive subgroup of $\hat{G}$. Also define
\begin{align*}
  S_{\phi, \text{ad}} & := S_\phi/Z_{\hat{G}}^{\Gamma_F}, \\
  \mathscr{S}_{\phi,\ad} & := \pi_0(S_{\phi, \text{ad}}), \quad \text{defined using the base point $1$}.
\end{align*}
The same symbol $\phi$ will be used to denote the $L$-parameter and its equivalence class, if there is no ambiguity to worry about.

One consequence is a correspondence between conjugacy classes of Levi subgroups $M \subset G$ and their dual avatars $\Lgrp{M} \hookrightarrow \Lgrp{G}$, the inclusion respects the projections onto $\We_F$.

By \cite[Proposition 3.6]{Bo79} and its proof, we obtain the following properties:
\begin{enumerate}[(i)]
  \item the Levi subgroups $\Lgrp{M} \subset \Lgrp{G}$ \cite[\S 3.4]{Bo79} which contain $\Im(\phi)$ minimally are conjugate by $S^0_\phi$;
  \item let $\Lgrp{M}$ be such a Levi subgroup, then $Z_{\hat{M}}^{\Gamma_F, 0}$ is a maximal torus of $S^0_\phi$.
\end{enumerate}
For every equivalence class of $L$-parameters $\phi$, we pick such a Levi subgroup $\Lgrp{M_\phi}$ and denote by $M_\phi$ the corresponding Levi subgroup of $G$. Define
\begin{align*}
  \Phi(G) & := \{\phi : \WD_F \to \Lgrp{G}, \; \phi \text{ $L$-parameter}\}/\sim, \\
  \Phi_\text{bdd}(G) & := \{\phi \in \Phi(G) : \phi \text{ is bounded} \}, \\
  \Phi_2(G) & := \{\phi \in \Phi(G): M_\phi = G \}, \\
  \Phi_{2,\text{bdd}}(G) & := \Phi_2(G) \cap \Phi_\text{bdd}(G).
\end{align*}

Let $M$ be a Levi subgroup of $G$. There is a natural map $\Phi(M) \to \Phi(G)$ induced by $\Lgrp{M} \hookrightarrow \Lgrp{G}$. It restricts to a map $\Phi_\text{bdd}(M) \to \Phi_\text{bdd}(G)$. There is an action of $\mathfrak{a}^*_{M,\C}$ on $\Phi(M)$, written as $\phi \mapsto \phi_\lambda$. In fact,
\begin{align*}
  \mathfrak{a}^*_{M,\C} & = X^*(M) \otimes_\Z \C, \\
  Z^{\Gamma_F, 0}_{\hat{M}} & = X^*(M) \otimes_\Z \C^\times,
\end{align*}
thus it makes sense to define $|w|^\lambda \in Z^{\Gamma_F, 0}_{\hat{M}}$ for all $w \in \We_F$, $\lambda \in \mathfrak{a}^*_{M,\C}$; the twist is just $\phi_\lambda = |\cdot|^\lambda \cdot \phi$. Then $\Phi_\text{bdd}(M)$ and $\Phi_{2,\text{bdd}}(M)$ are stable under $i\mathfrak{a}^*_M$. The $i\mathfrak{a}^*_M$-orbit decomposition makes $\Phi_{2,\text{bdd}}(M)$ into a disjoint union of compact tori. We have
\begin{gather}\label{eqn:Phi_bdd-decomp}
  \Phi_\text{bdd}(G) = \bigsqcup_{M \in \mathcal{L}(M_0)/W^G_0} \Phi_{2,\text{bdd}}(M)/W^G(M).
\end{gather}

\begin{remark}
  For general $G$, one must choose a quasisplit inner twist $G_{\bar{F}} \rightiso G^*_{\bar{F}}$ together with an $F$-splitting for $G^*$ to define $\Lgrp{G}$. Moreover, a \emph{relevance} condition has to be imposed on $\phi$. Roughly speaking, this means that the Levi subgroup $M^*_\phi \subset G^*$ (up to conjugacy) attached to $\phi$ should come from $G$. We will not encounter non-quasisplit groups in this article.
\end{remark}

In the next subsection, we will attach a tempered $L$-packet $\Pi_\phi = \Pi^G_\phi$ to every $\phi \in \Phi_\text{bdd}(G)$; the archimedean case is largely a paraphrase of Harish-Chandra's theory. Since some aspects will be needed in \S\ref{sec:archimedean}, we shall give a very sketchy review below. Details can be found in \cite[\S 11]{Bo79} and \cite[(4.3)]{Sh82}.

\paragraph{The case $F = \C$}
In this case $\We_\C = \C^\times$. It is customary to identify the complex groups with their $\C$-points. Consider the toric case $G=T$ first. By the Langlands correspondence for $T$, or local class field theory over $\C$, the continuous homomorphisms $T \to \C^\times$ are in natural bijection with $L$-parameters $\phi: \C^\times \to \hat{T} = X^*(T) \otimes \C^\times$. Write $z^\lambda := \lambda \otimes z$ for each $\lambda \in X^*(T)$ and $z \in \C^\times$. The parameter $\phi$ can be uniquely expressed as
\begin{align*}
  \phi: \C^\times & \longrightarrow \hat{T} \\ 
  z & \longmapsto z^\lambda \bar{z}^\mu = \left(\frac{z}{|z|}\right)^{\lambda - \mu} |z|^{\lambda + \mu}
\end{align*}
with $\lambda, \mu \in X^*(T) \otimes \C$ satisfying $\lambda - \mu \in X^*(T)$; it is bounded if and only if $\Re(\lambda+\mu)=0$.

In general, choose be a maximal torus $T \subset G$. The natural map $\Phi_\text{bdd}(T) \to \Phi_\text{bdd}(G)$ induces a bijection
$$ \Phi_\text{bdd}(G) \simeq \Phi_\text{bdd}(T)/W(G, T). $$

Note that the maximal tori of $G$ are all conjugate. The description of $\Phi_\text{bdd}(G)$ is parallel to the following representation-theoretical fact.

\begin{theorem}[See eg.\ \cite{Du75}]\label{prop:cplx-group}
  Let $G$ be a connected reductive $\C$-group.
  \begin{enumerate}
    \item Choose a Borel pair $(B, T)$ for $G$. For each unitary, continuous character $\chi$ of $T$, the induced representation $I_B(\chi)$ is irreducible; $I_B(\chi_1) \simeq I_B(\chi_2)$ if and only if $w\chi_1 =\chi_2$ for some $w \in W(G, T)$.
    \item The representations $I_B(\chi)$ so obtained exhaust the tempered spectrum of $G$.
  \end{enumerate}
\end{theorem}
The admissible dual of $G$ can be explicitly determined in terms of Langlands quotients, see \cite[I.4]{Du75}. For us the tempered dual suffices.

\paragraph{The case $F = \R$: discrete series and their limits}
In this case
$$ \We_\R = \C^\times \cdot \angles{\tau}, \quad \tau^2 = -1, \quad \forall z \in \C^\times, \; \tau z \tau^{-1} = \bar{z}, $$
and $\We_\R \twoheadrightarrow \angles{\tau}/\{\pm 1\} = \Gamma_\R$. Choose the $\Gamma_F$-stable Borel pair $(\hat{B}, \hat{T})$ for $\hat{G}$ which is part of the dual group datum. Let $\phi$ be an $L$-parameter for $G$. Upon conjugation, we may assume $\phi(\C^\times) \subset \hat{T}$. We begin with the case $\phi \in \Phi_{2,\text{bdd}}(G)$ which will yield the $L$-packets of discrete series.

Indeed, we may attach to $\phi|_{\C^\times}$ a pair $(\lambda, \mu)$ as in the complex toric case. It turns out that the subgroup of $\Lgrp{G}$ generated by $\hat{T}$ and $\phi(\We_\R)$ is isomorphic to $\Lgrp{T}$ in the category of $L$-groups, where $T$ is a maximal $\R$-torus of $G$ such that $T/Z_G$ is anisotropic --- do not confuse it with the earlier symbol $T$ in the Borel pair. Also choose a maximal compact subgroup $K$ containing $(T \cap G_\text{der})(\R)$.

View $\lambda$, $\mu$ as elements of $X^*(T_\C) \otimes_\Z \C$; there is an ambiguity by $W(G_\C, T_\C)$, which is harmless. By \cite[\S 10.5]{Bo79}, $\lambda$ is regular in the sense that
\begin{gather}\label{eqn:regular-HC-parameter}
  \forall \alpha \in \Sigma(G, T)_\C, \; \angles{\lambda, \alpha^\vee} \neq 0.
\end{gather}

In what follows, we make the usual identification
\begin{gather}\label{eqn:X-vs-Lie}
  X^*(T_\C) \otimes_\Z \C \rightiso \mathfrak{t}^*_\C.
\end{gather}
Choose a Borel subgroup $B_\C \supset T_\C$ for which $\lambda$ is dominant, and let $\rho = \rho_{B_\C}$ be the half-sum of positive roots. It turns out that $\lambda \in \rho_{B_\C} + X^*(T_\C)$. Define a \emph{Harish-Chandra parameter}\index{Harish-Chandra parameter} to be a pair $\vec{\lambda} = (\lambda, B_\C)$ where $\lambda \in \mathfrak{t}^*_\C$ is $B_\C$-dominant, regular and $\lambda \in \rho_{B_\C} + X^*(T_\C)$; note that $B_\C$ is uniquely determined by $\lambda$. To $\vec{\lambda}$ one may attach a representation $\pi(\vec{\lambda}) \in \Pi_{2,\text{temp}}(G)$ with infinitesimal character $\lambda$ modulo $W(G_\C, T_\C)$.

There is an obvious $W(G_\C, T_\C)$-action on Harish-Chandra parameters. Recall that for any $\vec{\lambda}_1$, $\vec{\lambda}_2$,
$$ \left[ \pi(\vec{\lambda}_1) \simeq \pi(\vec{\lambda}_2) \right] \iff \left[ \exists w \in W(G,T), \; w\vec{\lambda}_1 = \vec{\lambda}_2 \right]. $$
The $L$-packet in question is simply
$$ \Pi_\phi = \Pi_\lambda := \left\{ \pi(w\vec{\lambda}) : w \in W(G, T) \backslash W(G_\C, T_\C) \right\}. $$
Equivalently, it is the set of $\pi \in \Pi_{2,\text{temp}}(G)$ such that
\begin{compactitem}
  \item $\pi$ has infinitesimal character $\lambda \mod W(G_\C, T_\C)$,
  \item the central character of $\pi$ equals the restriction of $\lambda-\rho$ to $Z_G(\R)$.
\end{compactitem}
The last property no longer holds in the metaplectic setting, see \eqref{eqn:Adams-packet-ds}.

Keep the assumption that $T/Z_G$ is anisotropic. In the definition of Harish-Chandra parameters $\vec{\lambda} = (\lambda, B_\C)$, if we allow the $B_\C$-dominant weight $\lambda \in \rho_{B_\C} + X^*(T_\C)$ to be singular, i.e.\ without assuming \eqref{eqn:regular-HC-parameter}, one can still associate a representation $\pi(\vec{\lambda})$ of $G(\R)$: it is either\index{limit of discrete series}
\begin{compactitem}
  \item zero, or
  \item tempered irreducible.
\end{compactitem}
By \cite[Theorem 1.1b]{KZ82-1}, the first case happens if and only if $\angles{\lambda, \alpha^\vee}=0$ for some compact $B_\C$-simple root\index{compact root} $\alpha$ (see \cite[p.249]{KV95}). In the latter case $\pi(\vec{\lambda})$ has infinitesimal character $\lambda$ modulo $W(G_\C, T_\C)$ and central character $(\lambda-\rho)|_{Z_G(\R)}$, called the \emph{limit of discrete series} parametrized by $\vec{\lambda}$. 

As in the case of discrete series, the parameter $\vec{\lambda}$ for such a representation is unique up to $W(G,T)$. Given a possibly singular infinitesimal character, we may choose a representative $\lambda$ which is dominant relative to a fixed $B_\C$. Define the corresponding packet to be
$$ \Pi_\lambda = \left\{ \pi(w\vec{\lambda}): \text{nonzero}, \; w \in W(G, T) \backslash W(G_\C, T_\C) \right\}. $$

Thus being a limit of discrete series is a property of $L$-packets. An overview on limits of discrete series can be found in \cite{KZ82-1}. The standard construction of these representations is via coherent continuation or Zuckerman's translation functor $\psi_{\lambda_1}^\lambda$\index{translation functor} \cite[Chapter VII]{KV95}, whose effect is to ``shift the infinitesimal character'' from $\lambda_1$ to $\lambda$. As for their $L$-parameters, we refer to \cite[(4.3.4)]{Sh82}. These constructions also work for covering of metaplectic type over $\R$: see Lemma \ref{prop:calculation-rho}.

\subsection{Stable tempered characters}\label{sec:stable-character-SO}
Let $G = \SO(2n+1)$, so that $\hat{G} = \Sp(2n, \C)$ endowed with trivial $\Gamma_F$-action. We shall review the basic local results in \cite{Ar13} for $G$, in the tempered case at least. The first one is the tempered local Langlands correspondence for $G$.

Fix a maximal compact subgroup $K \subset G(F)$. Assume that $K$ corresponds to a special vertex in the Bruhat-Tits building if $F$ is non-archimedean. In the unramified case we assume $K$ hyperspecial.

\begin{theorem}[{\cite[Theorem 1.5.1]{Ar13}}]
  There is a decomposition
  $$ \Pi_\mathrm{temp}(G) = \bigsqcup_{\phi \in \Phi_\mathrm{bdd}(G)} \Pi_\phi, $$
  where
  \begin{itemize}
    \item each $\Pi_\phi$ is a finite set of tempered irreducible representations of $G$;
    \item there is a canonical injection
      \begin{align*}
        \Pi_\phi & \longrightarrow \Pi(\mathscr{S}_{\phi, \text{ad}}), \\
	\pi & \longmapsto \angles{\cdot, \pi};
      \end{align*}
    \item the map above is bijective in the non-archimedean case;
    \item in the unramified case, $\angles{\cdot, \pi} = 1$ whenever $\pi$ is unramified with respect to $K$.
  \end{itemize}

  Moreover, the decomposition restricts to
  $$ \Pi_{2,\mathrm{temp}}(G)  = \bigsqcup_{\phi \in \Phi_{2,\mathrm{bdd}}(G)} \Pi_\phi. $$
\end{theorem}
The finite subset $\Pi_\phi$ is called the tempered \emph{$L$-packet} associated to $\phi$.\index{$\Pi^G_\phi$}

The construction of the $L$-packets of discrete series or their limits in the archimedean case has been reviewed in \S\ref{sec:L-parameters}; further results, such as the stability of packets, will require deeper techniques from Shelstad et al. For general $F$, Arthur's approach of constructing $\Pi_\phi$ is based on realizing $G$ as a simple endoscopic group of the twisted group $\widetilde{\GL}(2n)$; we cannot delve into the details here. 

\begin{definition}\index{$S\Theta^G_\phi$}
  For every $\phi \in \Phi_\mathrm{bdd}(G)$, define the map
  \begin{align*}
    S\Theta_\phi: & C^\infty_c(G(F)) \longrightarrow \C \\ 
    f & \longmapsto f^G(\phi) :=\sum_{\pi \in \Pi_\phi} \Tr(\pi(f)).
  \end{align*}
\end{definition}
Call it the \emph{stable character} associated to $\phi$. We will write $\Pi^G_\phi$, $S\Theta^G_\phi$ to emphasize the ambient group if need be.

The second fundamental result is the stability of packets.
\begin{theorem}[{\cite[Theorem 2.2.1]{Ar13}}]\label{prop:SO-stability}
  The map $S\Theta_\phi$ factors through $\orbI(G) \twoheadrightarrow S\orbI(G)$. In other words, $f \mapsto f^G(\phi)$ is a stable distribution on $G$.
\end{theorem}

Let $M = \prod_{i \in I} \GL(n_i) \times G^\flat$ be a Levi subgroup of $G$, where $G^\flat$ is of the form $\SO(2n^\flat + 1)$. Note that $W^G(M)$ acts on $M$ by permuting $I$, up to inner automorphisms of $M(F)$. The aforementioned result extends to $M$. Naturally, they must be compatible with the local Langlands correspondence for the factors $\GL(n_i)$ (see \cite{He00}, for example), as well as the correspondence for the smaller orthogonal group $G^\flat$. All these are implicitly done in \cite{Ar13}.

\begin{theorem}\label{prop:LLC-vs-induction}
  Let $M = \prod_{i \in I} \GL(n_i) \times G^\flat$ be a Levi subgroup of $G$, and choose $P \in \mathcal{P}(M)$ arbitrarily. Suppose that $\phi_M \in \Phi_{2, \mathrm{bdd}}(M)$ has image $\phi$ in $\Phi_\mathrm{bdd}(G)$, then
  \begin{gather}\label{eqn:induced-packet}
    \Pi^G_\phi = \bigsqcup_{\sigma \in \Pi^M_{\phi_M}} \left\{ \text{the irreducible constituents of } I_P(\sigma) \right\}.
  \end{gather}

  Moreover, by decomposing $\phi_M$ as $\left(\boxtimes_{i \in I} \phi_i \right) \boxtimes \phi^\flat$, we have
  \begin{gather}\label{eqn:Levi-packet}
    \Pi^M_{\phi_M} = \left\{ (\boxtimes_{i \in I} \sigma_i ) \boxtimes \sigma^\flat : \sigma^\flat \in \Pi^{G^\flat}_{\phi^\flat} \right\}
  \end{gather}
  where $\sigma_i \in \Pi_2(\GL(n_i))$ is associated to $\phi_i \in \Phi_{2, \mathrm{bdd}}(\GL(n_i))$ by the local Langlands correspondence, for each $i \in I$.
\end{theorem}
\begin{proof}
  Discussed in the proof of \cite[Lemma 2.23]{Ar13}. Notice that the disjointness of the union \eqref{eqn:induced-packet} follows from Langlands' disjointness theorem: indeed, if $\sigma_1, \sigma_2 \in \Pi^M_{\phi_M}$ and $I_P(\sigma_1)$ intertwines with $I_P(\sigma_2)$, then there exists $w \in W^G(M)$ such that $w\sigma_1 \simeq \sigma_2$. In view of \eqref{eqn:Levi-packet} and the action of $W^G(M)$, this would imply $\sigma_1 \simeq \sigma_2$. 
\end{proof}

Finally, the representations $\sigma \in \Pi_\text{temp}(M)$ can be twisted by $\lambda \in i\mathfrak{a}^*_M$:
$$ \sigma \mapsto \sigma_\lambda := e^{\angles{\lambda, H_M(\cdot)}} \otimes \sigma. $$
It preserves $\Pi_{2,\text{temp}}(M)$ and is compatible with the similar twist on $L$-packets, namely $\Pi^M_{\phi_{M, \lambda}} = \{ \sigma_\lambda : \sigma \in \Pi^M_{\phi_M} \}$ where $\phi_M \in \Phi_\text{bdd}(M)$.

\section{Geometric transfer and its adjoint}\label{sec:geom-transfer}
Throughout this section, $F$ denotes a local field of characteristic zero; in \S\ref{sec:image-transfer-cusp} we will assume $F$ to be non-archimedean. Fix a non-trivial additive character $\psi: F \to \Sph^1$. 

\subsection{Collective geometric transfer}\label{sec:collective-geom-trans}
Fix a symplectic $F$-vector space $W$ of dimension $2n$, and define the metaplectic covering $\rev: \tilde{G} \to G(F)$ accordingly. Also fix a minimal Levi subgroup $M_0$ of $G$.

\paragraph{The unstable side}
In \S\ref{sec:space-orbital-integral} we have defined the space $\Iasp(\tilde{G})$ of normalized anti-genuine orbital integrals, realized as a space of anti-genuine functions $\Gamma_\text{reg}(\tilde{G}) \to \C$. Note that $\Gamma_\text{reg}(\tilde{G}) \to \Gamma_\text{reg}(G)$ is a $\bmu_8$-torsor by Theorem \ref{prop:Mp-commute}. Same for $\Gamma_\text{reg, ell}(\tilde{G})$, etc. For non-archimedean $F$, the following definitions are essentially from \cite{Ar96}; for discussions about the real case as well as some updates, see \cite[\S 1]{ArGerm}.

\begin{definition}\label{def:Iasp-filtration}
  For each $M \in \mathcal{L}(M_0)$, define $\mathcal{F}^M(\Iasp(\tilde{G}))$ to be the subspace of $f_{\tilde{G}} \in \Iasp(\tilde{G})$ such that $f_{\tilde{L}} = 0$ for every Levi subgroup $L$ that does not contain a conjugate of $M$. By Corollary \ref{prop:parabolic-descent-transitivity}, $M \subset L$ implies $\mathcal{F}^{L}(\Iasp(\tilde{G})) \subset \mathcal{F}^M(\Iasp(\tilde{G}))$. Therefore we get a filtration on $\Iasp(\tilde{G})$ indexed by the partially-ordered set $\mathcal{L}(M_0)$. It is thus natural to set \index{$\Iaspcusp(\tilde{G})$}\index{$\orbI_\gr(\tilde{G})$}
  \begin{align*}
    \Iaspcusp(\tilde{G}) & := \mathcal{F}^G(\Iasp(\tilde{G})), \\
    \gr^M \Iasp(\tilde{G}) & := \mathcal{F}^M \Iasp(\tilde{G}) \bigg/ \sum_{L\supsetneq M} \mathcal{F}^L \Iasp(\tilde{G}), \\
    \orbI_\gr(\tilde{G}) & := \bigoplus_{M \in \mathcal{L}(M_0)/W^G_0} \gr^M \Iasp(\tilde{G}).
  \end{align*}

  The same definitions also apply to the Levi subgroups of $\tilde{G}$, that is, coverings of metaplectic type. A functions whose image lies in $\Iaspcusp(\tilde{G})$ will be called \emph{cuspidal}.
\end{definition}

Note that for each $M \in \mathcal{L}(M_0)$, the group $W^G(M)$ operates on $\Iaspcusp(\tilde{M})$ by conjugation. By Corollary \ref{prop:parabolic-descent-transitivity}, the map $f_{\tilde{G}} \mapsto f_{\tilde{M}}$ induces an isomorphism
$$ \gr^M \Iasp(\tilde{G}) \rightiso \Iaspcusp(\tilde{M})^{W^G(M)}. $$
Hence
$$ \orbI_\gr(\tilde{G}) \rightiso \bigoplus_{M \in \mathcal{L}(M_0)/W^G_0} \Iaspcusp(\tilde{M})^{W^G(M)}. $$

In a parallel manner, the natural maps $\Gamma_{G-\text{reg}}(M) \to \Gamma_\text{reg}(G)$ induce
\begin{align*}
  \Gamma_\text{reg}(G) &= \bigsqcup_{M \in \mathcal{L}(M_0)/W^G_0} \Gamma_{G-\text{reg, ell}}(M) \bigg/ W^G(M), \\
  \Gamma_\text{reg}(\tilde{G}) &= \bigsqcup_{M \in \mathcal{L}(M_0)/W^G_0} \Gamma_{G-\text{reg, ell}}(\tilde{M}) \bigg/ W^G(M).
\end{align*}

We equip $\Gamma_\text{reg}(G)$ with a Radon measure as follows.
\begin{enumerate}
  \item For each $M \in \mathcal{L}(M_0)$, endow $\Gamma_\text{reg, ell}(M)$ with the Radon measure such that
    $$ \int_{\Gamma_\text{reg, ell}(M)} \alpha = \sum_{T: \text{ell.} / \text{conj.}}  |W(M, T)|^{-1} \int_{T(F)} \alpha $$
    for every $C_c$ test function $\alpha$. Here $T$ ranges over the elliptic maximal tori of $M$ modulo conjugacy.
  \item It is required that
    $$ \int_{\Gamma_\text{reg}(G)} \alpha = \sum_{M \in \mathcal{L}(M_0)/W^G_0} |W^G(M)|^{-1} \int_{\Gamma_{\text{reg, ell}}(M)} \alpha $$
    for every $C_c$ test function $\alpha$.
\end{enumerate}
This is exactly the definition in \cite[\S 1]{Ar96}; coverings do not intervene here. For $\alpha = f_G \in \orbI(G)$ it reduces to Weyl's integration formula for $f$.

\begin{definition}\label{def:pairing-geom}
  For $a_{\tilde{G}}, b_{\tilde{G}} \in \Iasp(\tilde{G})$, define the hermitian pairing
  $$ (a_{\tilde{G}} | b_{\tilde{G}}) := \int_{\Gamma_\text{reg}(G)} a_{\tilde{G}} \overline{b_{\tilde{G}}}. $$
  By the usual bounds for normalized orbital integrals \cite[Théorème 4.1.4]{Li12b}, the pairing is well-defined.
\end{definition}

\paragraph{The stable side}
Let $(n', n'') \in \EndoE_\text{ell}(\tilde{G})$ and denote the corresponding endoscopic group by $G^\Endo$ as usual. The preceding constructions have stable variants for $G^\Endo$ (see \cite[\S 1]{Ar96}), which we review below.

Fix a minimal Levi subgroup $M^\Endo_0$ of $G^\Endo$. The parabolic descent $f^{G^\Endo} \mapsto f^{M^\Endo}$ for various $M^\Endo \in \mathcal{L}(M^\Endo_0)$ allows us to define the filtration $\mathcal{F}^M(S\orbI(G^\Endo))$ as before. Similarly, using Corollary \ref{prop:parabolic-descent-transitivity} we define \index{$S\Icusp(G^\Endo)$} \index{$S\orbI_\gr(G^\Endo)$}
\begin{align*}
  S\Icusp(G^\Endo) & := \mathcal{F}^{G^\Endo}(S\orbI(G^\Endo)), \\
  \gr^{M^\Endo} S\orbI(G^\Endo) & := \mathcal{F}^{M^\Endo} S\orbI(G^\Endo) \bigg/ \sum_{L^\Endo \supsetneq M^\Endo} \mathcal{F}^{L^\Endo} S\orbI(G^\Endo) \\
  & \rightiso S\Icusp(M^\Endo)^{W^{G^\Endo}(M^\Endo)}, \\
  S\orbI_\gr(G^\Endo) & := \bigoplus_{M^\Endo \in \mathcal{L}(M^\Endo_0)/W^{G^\Endo}_0} \gr^{M^\Endo} S\orbI(G^\Endo) \\
  & \rightiso \bigoplus_{M^\Endo \in \mathcal{L}(M^\Endo_0)/W^{G^\Endo}_0} S\Icusp(M^\Endo)^{W^{G^\Endo}(M^\Endo)}.
\end{align*}
In this stable situation, a function whose image lies in $S\Icusp(G^\Endo)$ will be called \emph{cuspidal}.

The space of stable strongly regular semisimple classes decomposes as
$$ \Delta_\text{reg}(G^\Endo) = \bigsqcup_{M^\Endo \in \mathcal{L}(M^\Endo_0)/W^{G^\Endo}_0} \Delta_{G-\text{reg}}(M^\Endo)/W^{G^\Endo}(M^\Endo). $$

Prescribe a Radon measure on $\Delta_\text{reg}(G^\Endo)$ as follows.
\begin{enumerate}
  \item For each $M^\Endo \in \mathcal{L}(M^\Endo_0)$, endow $\Delta_\text{reg, ell}(M^\Endo)$ with the Radon measure such that
    $$ \int_{\Delta_\text{reg, ell}(M^\Endo)} \alpha = \sum_{T^\Endo: \text{ell.} / \text{st. conj.}}  |W(M^\Endo, T^\Endo)(F)|^{-1} \int_{T^\Endo(F)} \alpha $$
    for every $C_c$ test function $\alpha$. Here $T^\Endo$ ranges over the elliptic maximal $F$-tori of $M^\Endo$ modulo stable conjugacy.
  \item We require that
    $$ \int_{\Delta_\text{reg}(G^\Endo)} \alpha = \sum_{M^\Endo \in \mathcal{L}(M^\Endo_0)/W^{G^\Endo}_0} |W^{G^\Endo}(M^\Endo)|^{-1} \int_{\Delta_{\text{reg, ell}}(M^\Endo)} \alpha $$
    for every $C_c$ test function $\alpha$.
\end{enumerate}

To define the hermitian pairing of stable orbital integrals, we must incorporate the abelian group
\begin{gather}\label{eqn:E-group}
  \mathfrak{E}(G^\Endo_\sigma, G^\Endo; F) := \Ker\left[ H^1(F, G^\Endo_\sigma) \to H^1_\text{ab}(F, G^\Endo) \right]
\end{gather}
for every $\sigma \in \Delta_\text{reg}(G^\Endo)$; here $H^1_\text{ab}(F, -)$ stands for the functor of abelianized Galois cohomology, as recalled in \cite[\S 3.1]{Li15} (see also Borovoi \cite{Bor98} or Labesse \cite[I]{Lab99}). There is a functorial abelianization map $\text{ab}^1: H^1(F, -) \to H^1_\text{ab}(F, -)$ between pointed sets; it is bijective for tori or for non-archimedean $F$. In particular, $\mathfrak{D}(G^\Endo_\sigma, G^\Endo; F) \rightiso \mathfrak{E}(G^\Endo_\sigma, G^\Endo; F)$ in the non-archimedean setting.\index{$\mathfrak{E}(G^\Endo_\sigma, G^\Endo; F)$}

The hermitian pairing is defined as
$$
  (a^\Endo | b^\Endo) := \int_{\sigma \in \Delta_\text{reg}(G^\Endo)} |\mathfrak{E}(G^\Endo_\sigma, G^\Endo; F)|^{-1} a^\Endo(\sigma) \overline{b^\Endo(\sigma)}
$$
for every $a^\Endo, b^\Endo \in S\orbI(G^\Endo)$. As in the case for $\tilde{G}$, the convergence is guaranteed by the standard bounds for normalized stable orbital integrals. For archimedean $F$, the number $|\mathfrak{E}(G^\Endo_\sigma, G^\Endo; F)|$ is best explained by enlarging $G^\Endo$ to a $K$-group \cite[\S 4]{Ar02}: it then equals the number of conjugacy classes (in the sense of $K$-group) in the stable class of $\sigma$. See \cite[\S 4.17]{Wa14-1} for further discussions.

\begin{remark}\label{rem:measure-Sigma}
 The measure on $\Delta_\text{reg}(G^\Endo)$ satisfies
 $$ \int_{\sigma \in \Delta_\text{reg}(G^\Endo)} \sum_{\substack{\gamma \in \Gamma_{\text{reg}}(G^\Endo) \\ \gamma \mapsto \sigma}} \alpha(\gamma) = \int_{\gamma \in \Gamma_\text{reg}(G^\Endo)} \alpha $$
 for every $\alpha \in C_c(\Gamma_\text{reg}(G^\Endo))$. This is a direct consequence of our definition of measures: for non-archimedean $F$ it is actually  \cite[(1.3)]{Ar96}, whereas for $F=\R$, one may argue by Shelstad's description of stable conjugacy via $W(G,T) \backslash W(G,T)(F)$ reviewed in \S\ref{sec:st-Weyl}.
\end{remark}

\begin{remark}\label{rem:grading}
  Both $\orbI_\gr(\tilde{G})$ and $S\orbI(G^\Endo)$ carry natural $\mathcal{L}(M_0)$-filtrations coming from the gradings. In \S\ref{sec:spectral-formalism} we will exhibit filtration-preserving isomorphisms $\Iasp(\tilde{G}) \rightiso \orbI_\gr(\tilde{G})$ and $S\orbI(G^\Endo) \rightiso S\orbI_\gr(G^\Endo)$ by means of the trace Paley-Wiener theorems.
\end{remark}

\paragraph{Collective transfer}
The geometric transfers to various $G^\Endo$ can now be woven into a ``collective'' transfer. We shall follow \cite{Ar96} closely.

\begin{definition}
  Set
  $$ \Gamma_\text{reg,ell}^\EndoE(\tilde{G}) := \bigsqcup_{G^\Endo \in \EndoE_\text{ell}(\tilde{G})} \Delta_{G-\text{reg,ell}}(G^\Endo) $$
  where, by a standard abuse of notations, we use $G^\Endo$ to denote an endoscopic datum. An element in $\Gamma_\text{reg,ell}^\EndoE(\tilde{G})$ will be written in the form $(G^\Endo, \sigma)$ where $\sigma \in \Delta_{G-\text{reg,ell}}(G^\Endo)$, or simply as $\sigma$ according to the context. The same construction applies to any covering of metaplectic type $\tilde{M} = \prod_{i \in I} \GL(n_i, F) \times \Mp(2n^\flat)$ and its elliptic endoscopic data: the factors $\GL(n_i)$ will not interfere.

  Next, note that $W^G(M)$ has a well-defined action on $\Gamma_{G-\text{reg,ell}}^\EndoE(\tilde{M})$ for each $M \in \mathcal{L}(M_0)$, namely by permuting the indexing set $I$ of its $\GL$-components. Thus it makes sense to define
  $$ \Gamma^\EndoE_\text{reg}(\tilde{G}) := \bigsqcup_{M \in \mathcal{L}(M_0)/W^G_0} \Gamma_{G-\text{reg,ell}}^\EndoE(\tilde{M}) \bigg/ W^G(M). $$
\end{definition}

Equip each $\Gamma_\text{reg,ell}^\EndoE(\tilde{M}) \bigg/ W^G(M)$ with the quotient measure, and equip $\Gamma_\text{reg}^\EndoE(\tilde{G})$ with the measure of disjoint union. Elements in $\Gamma^\EndoE_\text{reg}(\tilde{G})$ are written as $(M^\Endo, \sigma)$ where $\sigma \in \Delta_{G-\text{reg,ell}}(M^\Endo)/W^G(M)$, or more succinctly as $\sigma$ whenever appropriate.

\begin{remark}
  In \cite{Ar96}, there are two ways to define the set $\Gamma_\text{reg}^\EndoE(\cdots)$ for reductive groups, say by looking at either
  \begin{inparaenum}[(i)]
    \item elliptic endoscopic data of Levi, or
    \item Levi of elliptic endoscopic data.
  \end{inparaenum}
  We have seen in Lemma \ref{prop:z[s]} that some subtleties arise in the metaplectic case. In this article we follow the previous viewpoint, as in \cite[(2.6) and (2.9)]{Ar96}.
\end{remark}

\begin{definition}\label{def:Endo-I} \index{$\orbI^\EndoE(\tilde{G})$} \index{$\Icusp^\EndoE(\tilde{G})$}
  Define a subspace $\orbI^\EndoE(\tilde{G})$ of $\bigoplus_{G^\Endo \in \EndoE_\text{ell}(\tilde{G})} S\orbI(G^\Endo)$ (always abusing notations...) as follows. Its elements are of the form $f^\EndoE = (f^{G^\Endo} \in S\orbI(G^\Endo))_{G^\Endo}$ such that
  \begin{itemize}
    \item for every $M \in \mathcal{L}(M_0)$ with $M^\Endo \in \EndoE_\text{ell}(\tilde{G})$, the function
      $$ \left( f^{G[s]} \right)^{s, M^\Endo} \in S\orbI(M^\Endo), \quad s \in \EndoE_{M^\Endo}(\tilde{G}) $$
      in Definition \ref{def:s-descent} is independent of $s$;
    \item denote the function above by $f^{M^\Endo}$, we require that $f^{M^\Endo} \in S\orbI(M^\Endo)^{W^G(M)}$, where $W^G(M)$ acts by permuting the indexing set $I$.
  \end{itemize}
  Also, define the cuspidal subspace
  \begin{gather*}
    \Icusp^\EndoE(\tilde{G}) := \bigoplus_{G^\Endo \in \EndoE_\text{ell}(\tilde{G})} S\orbI_\text{cusp}(G^\Endo) \quad \subset \orbI^\EndoE(\tilde{G}).
  \end{gather*}
\end{definition}

The definition of $\orbI^\EndoE(\tilde{G})$ is best explained by the proof of the following result.

\begin{proposition}\label{prop:image-in-EndoE}\index{$\mathcal{T}^\EndoE$}
  The transfer maps $f_{\tilde{G}} \mapsto f^{G^\Endo}$ assemble into the ``collective transfer''
  \begin{align*}
    \mathcal{T}^\EndoE : \Iasp(\tilde{G}) & \longrightarrow \orbI^\EndoE(\tilde{G}) \\
    f_{\tilde{G}} & \longmapsto f^\EndoE := \left( f^{G^\Endo} \right)_{G^\Endo \in \EndoE_\mathrm{ell}(\tilde{G})}.
  \end{align*}
  Its restriction to $\Iaspcusp(\tilde{G})$ gives $\Iaspcusp(\tilde{G}) \to \orbI^\EndoE_\mathrm{cusp}(\tilde{G})$.
\end{proposition}
\begin{proof}
  Let $f \in \Iasp(\tilde{G)}$. For every $(M^\Endo, \sigma) \in \Gamma^\EndoE_\text{reg}(\tilde{G})$ and $s_1, s_2 \in \EndoE_{M^\Endo}(\tilde{G})$, we have
  $$ \left(f^{G[s_1]}\right)^{s_1, M^\Endo} = (f_{\tilde{M}})^\Endo = \left(f^{G[s_2]}\right)^{s_2, M^\Endo} $$
  by Theorem \ref{prop:transfer-parabolic}. The required independence of $s$ and $W^G(M)$-invariance for the first assertion follow at once.
  
  Now assume $f \in \Iaspcusp(\tilde{G})$. Given $G^\Endo$ and its Levi subgroup $M^\Endo$, by Lemma \ref{prop:M^Endo-to-s} we produce a Levi subgroup $M$ of $G$ as well an $s \in \EndoE_{M^\Endo}(\tilde{G})$, such that $G^\Endo = G[s]$ as endoscopic data and we have a diagram as \eqref{eqn:endo-complete}. Elements in general position of $\Delta_\text{reg}(M^\Endo)$ can be expressed as $z[s]\sigma$ for some $\sigma \in \Delta_{G-\text{reg}}(M^\Endo)$. Theorem \ref{prop:transfer-parabolic} then implies
  $$ (f^{G^\Endo})^{M^\Endo}(z[s]\sigma) = f^{G^\Endo}(z[s]\sigma) = \left(f_{\tilde{M}}\right)^\Endo(\sigma) = 0. $$
  Hence $f^\EndoE \in \orbI^\EndoE_\text{cusp}(\tilde{G})$, as asserted.
\end{proof}

\begin{definition}\label{def:collective-factor}
  The correspondence of conjugacy classes and the transfer factors admit collective versions as follows. If $(M^\Endo, \sigma) \in \Gamma^\EndoE_\text{reg}(\tilde{G})$ and $\delta \in \Gamma_\text{reg}(G)$ corresponds to $\sigma$ via the diagram \eqref{eqn:endo-incomplete}, then $\delta \in \Gamma_{G-\text{reg,ell}}(M)/W^G(M)$; we write $\sigma \leftrightarrow \delta$ for such $(\sigma, \delta)$. For every $\tilde{\delta} \in \rev^{-1}(\delta)$ we set
  \begin{gather}\label{eqn:Delta-collective-nonell}
    \Delta(\sigma, \tilde{\delta}) := \sum_{w \in W^G(M)} \Delta_{M^\Endo, \tilde{M}}(w\sigma w^{-1}, \tilde{\delta}).
  \end{gather}
  Note that by $G$-regularity, there is at most one nonzero term in the sum. It is customary to set $\Delta(\sigma, \tilde{\delta})=0$ if $\sigma \not\leftrightarrow \delta$.
\end{definition}

In view of the previous definitions, the space $\orbI^\EndoE(\tilde{G})$ may be embedded into the space of functions $\Gamma_\text{reg}^\EndoE(\tilde{G}) \to \C$: an element $f^\EndoE = (f^{G^\Endo})_{G^\Endo \in \EndoE_{\text{ell}}(\tilde{G})}$ corresponds to the function
\begin{gather}
  (M^\Endo, \sigma) \longmapsto f^{M^\Endo}(\sigma) := \left( f^{G[s]} \right)^{s, M^\Endo}(\sigma)
\end{gather}
for any choice of $s \in \EndoE_{M^\Endo}(\tilde{G})$. Moreover, $\Icusp^\EndoE(\tilde{G})$ is precisely the subspace of functions supported in $\Gamma_\text{reg,ell}^\EndoE(\tilde{G})$. Now the geometric transfer can be rephrased in the collective terminology.

\begin{proposition}\index{$f^\EndoE$}
  For every $f_{\tilde{G}} \in \Iasp(\tilde{G})$, we have
  $$ f^\EndoE = \mathcal{T}^\EndoE(f_{\tilde{G}}): \sigma \mapsto \sum_{\sigma \leftrightarrow \delta} \Delta(\sigma, \tilde{\delta}) f_{\tilde{G}}(\tilde{\delta}) $$
  as a function on $\Gamma_\mathrm{reg}^\EndoE(\tilde{G})$.
\end{proposition}
\begin{proof}
  Let $(M^\Endo, \sigma) \in \Gamma^\EndoE_{\text{reg}}(\tilde{G})$. By Theorem \ref{prop:transfer-parabolic}, $f^\EndoE(\sigma)$ equals
  $$ \left( f^{G[s]} \right)^{s, M^\Endo}(\sigma) = (f_{\tilde{M}})^\Endo(\sigma). $$
  On the other hand, the $G$-regularity of $\sigma$ entails
  \begin{align*}
    (f_{\tilde{M}})^\Endo(\sigma) & = \sum_{\substack{\delta \in \Gamma_{G-\text{reg,ell}}(M) \\ \sigma \underset{M}{\leftrightarrow} \delta}} \Delta_{M^\Endo, \tilde{M}}(\sigma, \tilde{\delta}) f_{\tilde{M}}(\tilde{\delta}) \\
    & = \sum_{\substack{\delta \in \Gamma_{G-\text{reg,ell}}(M)/W^G(M) \\ (M^\Endo, \sigma) \leftrightarrow \delta}} \; \sum_{w \in W^G(M)} \Delta_{M^\Endo, \tilde{M}}(w\sigma w^{-1}, \tilde{\delta}) f_{\tilde{M}}(\tilde{\delta}) \\
    & = \sum_{\substack{\delta \in \Gamma_\text{reg}(G) \\ \sigma \leftrightarrow \delta}} \Delta(\sigma, \tilde{\delta}) f_{\tilde{G}}(\tilde{\delta}),
  \end{align*}
  in which the last $\Delta$ is the collective geometric transfer factor.
\end{proof}

The next result will serve as a change of variables in certain integrations over $\Gamma_\text{reg,ell}(\tilde{G})$.
\begin{lemma}[Cf.\ {\cite[Lemma 2.3]{Ar96}}]\label{prop:change-variables}
  For all $\alpha \in C_c(\Gamma_{\mathrm{reg}}(\tilde{G}))$ and $\beta \in C_c(\Gamma^\EndoE_{\mathrm{reg}}(\tilde{G}))$ such that $\alpha$ is genuine, we have
  \begin{multline*}
    \int_{\delta \in \Gamma_\mathrm{reg,ell}(G)} \quad \sum_{\substack{\sigma \in \Gamma^\EndoE_\mathrm{reg,ell}(\tilde{G}) \\ \sigma \leftrightarrow \delta}} \beta(\sigma) \Delta(\sigma, \tilde{\delta}) \alpha(\tilde{\delta})\; \dd\delta \\
    = \int_{\sigma \in \Gamma^\EndoE_\mathrm{reg,ell}(\tilde{G})} \quad \sum_{\substack{\delta \in \Gamma_\mathrm{reg,ell}(G) \\ \sigma \leftrightarrow \delta}} \beta(\sigma) \Delta(\sigma, \tilde{\delta}) \alpha(\tilde{\delta}) \;\dd\sigma
  \end{multline*}
  where $\tilde{\delta} \in \rev^{-1}(\delta)$ is arbitrarily chosen.
\end{lemma}
\begin{proof}
  In view of Remark \ref{rem:measure-Sigma}, the left-hand side of the assertion equals
  \begin{gather}\label{eqn:change-var-1}
    \int_{\theta \in \Delta_\text{reg,ell}(G)}\; \sum_{\substack{\sigma \in \Gamma^\EndoE_\text{reg,ell}(\tilde{G}) \\ \sigma \leftrightarrow \theta}} \; \sum_{\delta \mapsto \theta} \beta(\sigma) \Delta(\sigma, \tilde{\delta}) \alpha(\tilde{\delta}).
  \end{gather}
  Writing $\sigma \leftrightarrow \theta$ is legitimate, since the correspondence of classes depends only on the stable conjugacy classes.
  
  On the other hand, the second projection
  $$ \text{pr}_2: \left\{ (\theta, \sigma) \in \Delta_\text{reg,ell}(G) \times \Gamma^\EndoE_\text{reg,ell}(\tilde{G}) : \sigma \leftrightarrow \theta \right\} \longrightarrow \Gamma^\EndoE_\text{reg,ell}(\tilde{G}) $$
  is a bijective: its inverse is given by the maps \eqref{eqn:mu}. To do integration, note that $\Gamma^\EndoE_\text{reg,ell}(\tilde{G})$ can be described as the set of pairs $(T^\Endo, \sigma)$ where $T^\Endo \subset G^\Endo$ is an elliptic maximal $F$-torus (taken up to stable conjugacy), and $\sigma \in T^\Endo_\text{reg}(F)/W(G^\Endo, T^\Endo)(F)$. When $(T^\Endo, \sigma)$ varies, the inverse image by $\text{pr}_2$ runs over the triples $(T, \theta, \sigma)$ where
  \begin{compactitem}
    \item $T$: elliptic maximal torus of $G$, taken up to stable conjugacy;
    \item $\theta \in T_\text{reg}(F)/W(G,T)(F)$;
    \item $\sigma \in \Gamma^\EndoE_\text{ell,reg}(\tilde{G})$ satisfying $\sigma \leftrightarrow \theta$.
  \end{compactitem}
  Moreover, to such data we pick a standard isomorphism $T^\Endo \rightiso T$ sending $\sigma$ to $\theta$. Notice that each elliptic maximal torus $T \subset G$ occurs in some $(T, \theta, \sigma)$: this is contained in \cite[Lemme 5.2.1]{Li15}.

  In view of the definition of the measure on $\Gamma^\EndoE_\text{reg,ell}(\tilde{G})$, the right-hand side of the assertion can be transformed into
  $$ \sum_{\substack{T \subset G \\ \text{elliptic} \\ /\text{st. conj.}}} \quad \int_{\theta \in T_\text{reg}(F)/W(G, T)(F)} \; \sum_{\substack{\sigma \in \Gamma^\EndoE_\text{reg,ell}(\tilde{G}) \\ \sigma \leftrightarrow \theta }} \sum_{\delta \mapsto \theta} \beta(\sigma) \Delta(\sigma, \tilde{\delta}) \alpha(\tilde{\delta}). $$

  By the definition of the measure on $\Delta_\text{reg,ell}(G)$, this equals \eqref{eqn:change-var-1}.
\end{proof}

\subsection{Adjoint transfer}\label{sec:adjoint-transfer}
Hereafter we make systematic use of collective transfer factors and correspondences.

\begin{definition}\label{def:adjoint-factor}\index{$\Delta(\tilde{\delta}, \sigma)$}
  For $\sigma \in \Gamma^\EndoE_\text{reg,ell}(\tilde{G})$ and $\delta \in \Gamma_\text{reg,ell}(G)$ such that $\sigma \leftrightarrow \delta$, define the \emph{adjoint transfer factor} as
  $$ \Delta(\tilde{\delta}, \sigma) := |\mathfrak{D}(G_\delta, G; F)|^{-1} \overline{\Delta(\sigma, \tilde{\delta})}, $$
  where $\tilde{\delta} \in \rev^{-1}(\delta)$ as usual. The same definition carries over to coverings of metaplectic type and the elliptic conjugacy classes therein.
  
  More generally, let $\sigma \in \Gamma^\EndoE_\text{reg}(\tilde{G})$ and $\delta \in \Gamma_\text{reg}(G)$ such that $\sigma \leftrightarrow \delta$. By Definition \ref{def:collective-factor}, there exists a unique $M \in \mathcal{L}(M_0)/W^G_0$ such that $\sigma \in \Gamma^\EndoE_\text{$G$-reg,ell}(\tilde{M})/W^G(M)$, $\delta \in \Gamma_\text{$G$-reg,ell}(M)/W^G(M)$. Set
  $$ \Delta(\tilde{\delta}, \sigma) := \sum_{w \in W^G(M)} \Delta_{\tilde{M}}(\tilde{\delta}, w\sigma w^{-1}) $$
  where $\Delta_{\tilde{M}}(\cdots)$ denotes the adjoint transfer factor for $\tilde{M}$; it also equals $\sum_w \Delta_{\tilde{M}}(w\tilde{\delta}w^{-1}, \sigma)$. It is customary to set $\Delta(\tilde{\delta}, \sigma)=0$ if $\delta \not\leftrightarrow \sigma$.
\end{definition}

Recall that when $\sigma \leftrightarrow \delta$, the usual transfer factor satisfies a similar equation (see \eqref{eqn:Delta-collective-nonell})
$$ \Delta(\sigma, \tilde{\delta}) = \sum_{w \in W^G(M)} \Delta_{\tilde{M}}(w\sigma w^{-1}, \tilde{\delta}) = \sum_{w \in W^G(M)} \Delta_{\tilde{M}}(\sigma, w\tilde{\delta} w^{-1}). $$
In all cases, the sum over $W^G(M)$ contains at most one nonzero term. Also note that $\Delta(\tilde{\delta}, \sigma)$ is anti-genuine in $\tilde{\delta}$. Hence products of the form $\Delta(\sigma, \tilde{\delta}) \Delta(\tilde{\delta}, \sigma_1)$ only depend on $(\sigma, \delta, \sigma_1)$.

For $\tilde{\delta}, \tilde{\delta}_1 \in \Gamma_\text{reg}(\tilde{G})$, put
$$ \bdelta_{\tilde{\delta}, \tilde{\delta}_1} := \begin{cases} 
  \noyau, & \text{if } \tilde{\delta}_1 = \noyau\tilde{\delta}, \; \noyau \in \bmu_8, \\
  0, & \text{otherwise},
\end{cases} $$
whereas for $\sigma, \sigma_1 \in \Gamma^\EndoE_\text{reg}(\tilde{G})$, we denote by $\bdelta_{\sigma, \sigma_1} \in \{0, 1\}$ the usual Kronecker's delta. The following lemma for reductive linear groups can be found in \cite[Lemma 2.2, (2.10) and (2.11)]{Ar96}.

\begin{lemma}\label{prop:geometric-inversion}
  For all $\tilde{\delta}, \tilde{\delta}_1 \in \Gamma_\mathrm{reg}(\tilde{G})$,
  $$ \sum_{\sigma \in \Gamma^\EndoE_\mathrm{reg}(\tilde{G})} \Delta(\tilde{\delta}, \sigma) \Delta(\sigma, \tilde{\delta}_1) = \bdelta_{\tilde{\delta}, \tilde{\delta}_1}. $$
  For all $\sigma, \sigma_1 \in \Gamma^\EndoE_\mathrm{reg}(\tilde{G})$,
  $$ \sum_{\delta \in \Gamma_\mathrm{reg}(G)} \Delta(\sigma, \tilde{\delta}) \Delta(\tilde{\delta}, \sigma_1) = \bdelta_{\sigma, \sigma_1} $$
  where $\tilde{\delta} \in \rev^{-1}(\delta)$ is arbitrary.
\end{lemma}
\begin{proof}
  We begin with the case $\tilde{\delta}, \tilde{\delta}_1 \in \Gamma_\text{reg,ell}(\tilde{G})$ for the first assertion. The elements $\sigma$ with $\Delta(\tilde{\delta}, \sigma) \Delta(\sigma, \tilde{\delta}_1) \neq 0$ are necessarily elliptic. We may assume that $\delta := \rev(\tilde{\delta})$ and $\delta_1 := \rev(\tilde{\delta}_1)$ are stably conjugate, otherwise both sides of the first assertion are zero. Then
  $$ \Delta(\tilde{\delta}, \sigma) \Delta(\sigma, \tilde{\delta}_1) = |\mathfrak{D}(G_\delta, G; F)|^{-1} \overline{\Delta(\sigma, \tilde{\delta})} \Delta(\sigma, \tilde{\delta}_1). $$
  Furthermore, by the genuineness of $\Delta(\sigma, \cdot)$ and Lemma \ref{prop:stable-conj-lifting}, we may assume that $\tilde{\delta}$ and $\tilde{\delta}_1$ are stably conjugate, so that $\bdelta_{\tilde{\delta}, \tilde{\delta}_1} = \bdelta_{\delta, \delta_1}$.

  By a local variant of \cite[Lemme 5.2.1]{Li15} (same proof), there is a natural bijection
  $$ (\delta, \kappa) \stackrel{1:1}{\longleftrightarrow} \sigma \in \Gamma^\EndoE_\text{ell,reg}(\tilde{G}), $$
  where $\delta \in \Delta_\text{reg,ell}(G)$ and $\kappa$ belongs to the Pontryagin dual $\mathfrak{R}(G_\delta, G; F)$ \index{$\mathfrak{R}(G_\delta, G; F)$} of $H^1(F, G_\delta)$; it is characterized by
  \begin{inparaenum}[(i)]
    \item $\sigma \leftrightarrow \delta$, and
    \item $\kappa$ is the endoscopic character associated to $\delta$ and $\sigma$ (recall \S\ref{sec:endoscopy}).
  \end{inparaenum}
  This applies to those $\sigma$ and $\delta \stackrel{\text{st.}}{\sim} \delta_1$ satisfying $\overline{\Delta(\sigma, \tilde{\delta})} \Delta(\sigma, \tilde{\delta}_1) \neq 0$: it follows from the cocycle property for $\Delta(\sigma, \cdot)$ that
  $$ \overline{\Delta(\sigma, \tilde{\delta})} \Delta(\sigma, \tilde{\delta}_1) = \angles{\kappa, \inv(\delta, \delta_1)} $$
  for the $\kappa$ so obtained. Note that $\mathfrak{D}(G_\delta, G; F) = H^1(F, G_\delta)$ since $H^1(F, G)=\{1\}$. Hence
  $$ \sum_{\sigma \in \Gamma^\EndoE_\mathrm{reg}(\tilde{G})} \overline{\Delta(\sigma, \tilde{\delta})} \Delta(\sigma, \tilde{\delta}_1) = \sum_\kappa \angles{\kappa, \inv(\delta, \delta_1)} = |\mathfrak{D}(G_\delta, G; F)| \bdelta_{\delta, \delta_1} $$
  by Fourier inversion on $H^1(F, G_\delta)$.
  
  As for the second assertion in the elliptic case, assume $\sigma, \sigma_1 \in \Gamma^\EndoE_\text{reg,ell}(\tilde{G})$. As before, the classes $\delta$ having nonzero contribution must lie in a single elliptic stable conjugacy class. Enumerate these classes as $\delta^1, \ldots, \delta^m$. Choose $\tilde{\delta}^i \in \rev^{-1}(\delta^i)$ for $i=1, \ldots, m$ so that all the $\tilde{\delta}^i$ are stably conjugate; this can always be done in view of Lemma \ref{prop:stable-conj-lifting}. On the other hand, the set $\left\{ \sigma \in \Gamma^\EndoE_\text{ell,reg}(\tilde{G}) : \sigma \leftrightarrow \delta \right\}$ also has cardinality $m$, and we may enumerate its elements as $\sigma^1, \ldots, \sigma^m$. Indeed, this follows from the aforementioned bijection $(\delta, \kappa) \stackrel{1:1}{\longleftrightarrow} \sigma$, and we have $m = |H^1(F, G_\delta)|$. Form the matrices
  $$ A := \left( \Delta(\tilde{\delta}^i, \sigma^j) \right)_{1 \leq i, j \leq m}, \quad B := \left( \Delta(\sigma^i, \tilde{\delta}^j) \right)_{1 \leq i, j \leq m}. $$
  The first assertion amounts to $AB=1$ whilst the second amounts to $BA=1$. This concludes the elliptic case. It is routine to extend this to coverings of metaplectic type.

  Finally, the general, non-elliptic case follow from Definition \ref{def:adjoint-factor}. Take the first assertion for example. As before, we reduce to the case in which $\delta, \delta_1 \in \Gamma_{G-\text{reg,ell}}(M)/W^G(M)$, and the sum may be taken over $\sigma \in \Gamma^\EndoE_{G-\text{reg,ell}}(\tilde{M})/W^G(M)$, where $M \in \mathcal{L}(M_0)$. Upon a $W^G(M)$-action we may assume $\delta \stackrel{\text{st}}{\sim} \delta_1$ in $M$. Then
  \begin{multline*}
    \sum_{\sigma \in \Gamma^\EndoE_{G-\text{reg,ell}}(\tilde{M})/W^G(M)} \Delta(\tilde{\delta}, \sigma) \Delta(\sigma, \tilde{\delta}_1) = \\
    \sum_{\substack{\sigma \in \Gamma^\EndoE_{G-\text{reg,ell}}(\tilde{M})/W^G(M) \\ u,v \in W^G(M)}} \Delta_{\tilde{M}}(u\tilde{\delta}u^{-1}, \sigma) \Delta_{\tilde{M}}(v\sigma v^{-1}, \tilde{\delta}_1).
  \end{multline*}
  And this reduces immediately to the elliptic case for $\tilde{M}$.
\end{proof}

\begin{definition}
  The adjoint transfer factor yields an adjoint transfer map
  \begin{align*}
    \mathcal{T}_\EndoE: \{\text{functions } \Gamma^\EndoE_\text{reg}(\tilde{G}) \to \C \} & \longrightarrow \{\text{anti-genuine functions } \Gamma_\text{reg}(\tilde{G}) \to \C \} \\
    b & \longmapsto \left[ \tilde{\delta} \mapsto \sum_{\sigma \in \Gamma^\EndoE_\text{reg}(\tilde{G})} \Delta(\tilde{\delta}, \sigma) b(\sigma) \right].
  \end{align*}
\end{definition}

Likewise, one may also regard $\mathcal{T}^\EndoE$ as a linear map
\begin{align*}
  \mathcal{T}^\EndoE: \{\text{anti-genuine functions } \Gamma_\text{reg}(\tilde{G}) \to \C \} & \longrightarrow \{\text{functions } \Gamma^\EndoE_\text{reg}(\tilde{G}) \to \C \} \\
  a & \longmapsto \left[ \sigma \mapsto \sum_{\sigma \leftrightarrow \delta} \Delta(\sigma, \tilde{\delta}) a(\tilde{\delta}) \right].
\end{align*}

\begin{proposition}\label{prop:transfer-injective}
  Extend the definition of $\mathcal{T}^\EndoE$ as above. Then $\mathcal{T}_\EndoE$ and $\mathcal{T}^\EndoE$ are mutually inverse. In particular, $\mathcal{T}_\EndoE \circ \mathcal{T}^\EndoE = \identity: \Iasp(\tilde{G}) \to \Iasp(\tilde{G})$; consequently the transfer map $\mathcal{T}^\EndoE: \Iasp(\tilde{G}) \to \orbI^\EndoE(\tilde{G})$ is injective.
\end{proposition}
\begin{proof}
  Let $a \in \Iasp(\tilde{G})$, then its image under $\mathcal{T}_\EndoE$ sends $\tilde{\delta} \in \Gamma_\text{reg}(\tilde{G})$ to
  \begin{align*}
    \sum_\sigma \Delta(\tilde{\delta}, \sigma) \left(\sum_{\delta_1} \Delta(\sigma, \tilde{\delta}_1) a(\tilde{\delta}_1) \right) 
    &= \sum_{\delta_1} \left( \sum_\sigma \Delta(\tilde{\delta}, \sigma) \Delta(\sigma, \tilde{\delta}_1) \right) a(\tilde{\delta}_1) \\
    &= \sum_{\delta_1} \bdelta_{\tilde{\delta}, \tilde{\delta}_1} a(\tilde{\delta}_1) = a(\tilde{\delta})
  \end{align*}
  by Lemma \ref{prop:geometric-inversion}; note that each sum is finite, and the last equality follows from the definition of $\bdelta_{\tilde{\delta}, \tilde{\delta}_1}$ together with the anti-genuineness of $a$. This shows that $\mathcal{T}_\EndoE \mathcal{T}^\EndoE = \identity$. The proof for the other side is similar.
\end{proof}

Define an hermitian pairing on $\Icusp^\EndoE(\tilde{G})$ as follows
\begin{gather}\label{eqn:collective-ell-ip}
  (a^\EndoE \big| b^\EndoE) := \sum_{G^\Endo \in \EndoE_\text{ell}(\tilde{G})} \iota(\tilde{G}, G^\Endo) (a^\Endo|b^\Endo)
\end{gather}
where
\begin{itemize}
  \item $a^\EndoE = (a^\Endo)_{G^\Endo \in \EndoE_\text{ell}(\tilde{G})}$ and $b^\EndoE = (b^\Endo)_{G^\Endo \in \EndoE_\text{ell}(\tilde{G})}$ are in $\Icusp^\EndoE(\tilde{G})$;
  \item the hermitian pairing $(a^\Endo|b^\Endo)$ is defined relative to $G^\Endo$;
  \item for $G^\Endo$ arising from $(n',n'') \in \EndoE_\text{ell}(\tilde{G})$, we define
    $$ \iota(\tilde{G}, G^\Endo) = \left|Z_{\widehat{G^\Endo}}\right|^{-1} = \begin{cases}
    \frac{1}{4}, & n', n'' \geq 1, \\
    \frac{1}{2}, & n > 0, \; n'=0 \text{ or } n''=0, \\
    1, & n=0;
  \end{cases} $$
  this is the same as the global coefficient used in the stabilization of trace formula \cite[Définition 5.2.5]{Li15}.
\end{itemize}

\begin{corollary}\label{prop:isometry}
  The map $\mathcal{T}^\EndoE$ is an isometry from $\Iaspcusp(\tilde{G})$ onto its image inside $\Icusp^\EndoE(\tilde{G})$.
\end{corollary}
\begin{proof}
  In view of Proposition \ref{prop:transfer-injective}, it suffices to show that
  $$ (\mathcal{T}^\EndoE(a) \big| b^\EndoE) = (a \big| \mathcal{T}_\EndoE(b^\EndoE)), \quad a \in \Iaspcusp(\tilde{G}), \quad b^\EndoE \in \Icusp^\EndoE(\tilde{G}). $$
  From the definition, $(\mathcal{T}^\EndoE(a) | b^\EndoE)$ equals
  \begin{gather*}
   \sum_{G^\Endo \in \EndoE_\text{ell}(\tilde{G})} \iota(\tilde{G}, G^\Endo) \int_{\sigma \in \Delta_{G-\text{reg,ell}}(G^\Endo)} \sum_{\substack{\delta \in \Gamma_\text{reg,ell}(G) \\ \sigma \leftrightarrow \delta}} |\mathfrak{E}(G^\Endo_\sigma, G^\Endo; F)|^{-1} \overline{b^\Endo(\sigma)} \Delta(\sigma, \tilde{\delta}) a(\tilde{\delta}).
  \end{gather*}

  Recall the definition \eqref{eqn:E-group} for $\mathfrak{E}(G^\Endo_\sigma, G^\Endo; F)$. We contend that for $\sigma \leftrightarrow \delta$ as above,
  \begin{gather}\label{eqn:role-of-iota}
    \iota(\tilde{G}, G^\Endo) |\mathfrak{E}(G_\delta, G; F)|^{-1} = |\mathfrak{D}(G^\Endo_\sigma, G^\Endo; F)|^{-1}.
  \end{gather}
  Indeed, write $T = G_\delta$ and identify $T$ with $T^\Endo := G^\Endo_\sigma$. We have $H^1(F, G)=\{1\}$, thus $\mathfrak{D}(T, G; F) = H^1(F, T)$ and
  \begin{align*}
    \mathfrak{D}(T, G; F)/\mathfrak{E}(T, G^\Endo ; F) & \simeq \Image[ H^1(F, T) \to H^1_\text{ab}(F, G^\Endo) ] \\
    & = \Image[ H^1(F, G^\Endo) \to H^1_\text{ab}(F, G^\Endo) ] \quad \text{(apply \cite[10.2 Lemma]{Ko86})} \\
    & = H^1_\text{ab}(F, G^\Endo) \quad \text{(apply \cite[Proposition 1.6.7]{Lab99})}.
  \end{align*}
  The last term is canonically isomorphic to the Pontryagin dual of $Z_{\widehat{G^\Endo}}$: this follows from the Propositions 1.10, 2.8 and \S 4 of \cite{Bor98}. Since $\iota(\tilde{G}, G^\Endo) = \left| Z_{\widehat{G^\Endo}} \right|^{-1}$, this establishes \eqref{eqn:role-of-iota}.
  
  Now we may write $(\mathcal{T}^\EndoE(a) | b^\EndoE)$ as
  $$ \int_{\sigma \in \Gamma^\EndoE_\text{reg,ell}(\tilde{G})} \sum_{\substack{\delta \in \Gamma_\text{reg,ell}(G) \\ \sigma \leftrightarrow \delta}} \overline{b(\sigma)} |\mathfrak{D}(G_\delta, G; F)|^{-1} \Delta(\sigma, \tilde{\delta}) a(\tilde{\delta}). $$

  Apply Lemma \ref{prop:change-variables} to transform it into
  $$ \int_{\delta \in \Gamma_\text{reg,ell}(G)} \sum_{\substack{\sigma \in \Gamma^\EndoE_\text{reg,ell}(\tilde{G}) \\ \sigma \leftrightarrow \delta}} \overline{b(\sigma)} |\mathfrak{D}(G_\delta, G; F)|^{-1} \Delta(\sigma, \tilde{\delta}) a(\tilde{\delta}). $$

  By Definition \ref{def:adjoint-factor}, this equals $(a \big| \mathcal{T}_\EndoE(b^\EndoE))$.
\end{proof}

\subsection{Image of non-archimedean cuspidal transfer}\label{sec:image-transfer-cusp}
Assume $F$ non-archimedean in this subsection.

\paragraph{Desiderata}
We have defined the collective transfer map
\begin{align*}
  \mathcal{T}^\EndoE: \Iasp(\tilde{G}) & \longrightarrow \Iasp^\EndoE(\tilde{G}) \\
  \Iaspcusp(\tilde{G}) & \longrightarrow \Icusp^\EndoE(\tilde{G}) = \bigoplus_{G^\Endo} S\Icusp(G^\Endo).
\end{align*}
Our aim is to show that $\mathcal{T}^\EndoE$ is an isomorphism. The general case will be done in Corollary \ref{prop:transfer-surj}. Here we consider the cuspidal part only; the injectivity follows from Proposition \ref{prop:transfer-injective}, and the isomorphy is implied by the following result.

\begin{theorem}\label{prop:transfer-cusp-surj}
  We have $\mathcal{T}^\EndoE(\Iaspcusp(\tilde{G})) = \Icusp^\EndoE(\tilde{G})$.
\end{theorem}
\begin{corollary}
  The inverse of $\mathcal{T}^\EndoE: \Iaspcusp(\tilde{G}) \to \Icusp^\EndoE(\tilde{G})$ is the adjoint transfer $\mathcal{T}_\EndoE$.
\end{corollary}
\begin{proof}
  Apply Proposition \ref{prop:transfer-injective}.
\end{proof}

The proof of Theorem \ref{prop:transfer-cusp-surj} occupies the rest of this section. Firstly, we have to review the geometric transfer on the level of Lie algebras.

\paragraph{Standard endoscopy on Lie algebras}
A succinct introduction to the endoscopy on Lie algebras can be found in \cite{C11}; for a comprehensive treatment, see \cite{Wa08}.

In the following discussions, $G$ denotes an arbitrary connected reductive $F$-group. Let $(G^\Endo, \mathcal{G}^\Endo, s, \hat{\xi})$ be an elliptic endoscopic datum for $G$, where $G^\Endo$ is the endoscopic group (see \cite[\S 1.3]{Wa08}); in particular $G^\Endo$ is quasisplit.

Denote by $\Delta_{G^\Endo, G}(\cdot, \cdot)$ a geometric transfer factor on \emph{Lie algebras}, deprived of the factor $\Delta_\mathrm{IV}$ of Langlands-Shelstad; it is canonical only up to a multiplicative constant of absolute value $1$. In this setting we still have:
\begin{compactitem}
  \item the normalized orbital integrals $X \mapsto f^\flat_G(X)$ for $X \in \Gamma_\text{reg}(\mathfrak{g})$, $f^\flat \in C^\infty_c(\mathfrak{g}(F))$, as well as the stable version $Y \mapsto (f^\flat)^G(Y)$ when $G$ is quasisplit;
  \item the spaces $\orbI(\mathfrak{g})$, its stable variant $S\orbI(\mathfrak{g}^\Endo)$, as well as their cuspidal subspaces $\Icusp(\mathfrak{g})$, $S\Icusp(\mathfrak{g}^\Endo)$ since the parabolic descent (see Proposition \ref{prop:descent-orbint}) can also be defined on Lie algebras;
  \item the correspondence of conjugacy classes on Lie algebras, written as $Y \leftrightarrow X$;
  \item the geometric transfer with respect to $\Delta_{G^\Endo, G}(\cdots)$: for each $a \in \orbI(\mathfrak{g})$ we set
    $$ \mathcal{T}(a): Y \longmapsto \sum_{\substack{X \in \Gamma_\text{reg}(\mathfrak{g}) \\ Y \leftrightarrow X}} \Delta_{G^\Endo, G}(Y, X) a(X), \quad Y \in \Delta_{G-\text{reg}}(\mathfrak{g}^\Endo). $$
    The Main Theorem of the endoscopy for Lie algebras, due primarily to B. C. Ngô and Waldspurger, asserts that $\mathcal{T}$ induces a linear map $\orbI(\mathfrak{g}) \to S\orbI(\mathfrak{g}^\Endo)$. It restricts to $\Icusp(\mathfrak{g}) \to S\Icusp(\mathfrak{g}^\Endo)$.
\end{compactitem}

We remark that for the Lie algebras, there is no need to introduce the $z$-pairs as in \cite[\S 2.2]{KS99}.

There is a finite group $\text{Out}_G(G^\Endo)$ of outer $F$-automorphisms of $G^\Endo$; see \cite{Ar06}. Thus it acts on $\Delta_\text{reg}(\mathfrak{g}^\Endo)$, the space of stable regular semisimple classes in $\mathfrak{g}^\Endo(F)$. Note that if $\Delta_{G^\Endo, G}$ is a transfer factor, then so is $\Delta_{G^\Endo, G}(\tau^{-1}(\cdot), \cdot)$. For non-quasisplit $G$, this notion of \emph{outer automorphisms} of an endoscopic datum can be quite subtle, as illustrated by the following result.

\begin{theorem}[Arthur, Hiraga-Saito]\label{prop:Delta-std-equivariance}
  There exists a canonical homomorphism $\chi: \mathrm{Out}_G(G^\Endo) \to \C^\times$, which is trivial for quasisplit $G$, such that for every choice of $\Delta_{G^\Endo, G}$,
  $$ \Delta_{G^\Endo, G}(\tau^{-1}(Y), X) = \chi(\tau)\Delta_{G^\Endo, G}(Y, X), \quad \tau \in \mathrm{Out}_G(G^\Endo), $$
  for all conjugacy classes $X$ (resp. stable conjugacy classes $Y$) in $\mathfrak{g}_\mathrm{reg}(F)$ (resp. in $\mathfrak{g}^\Endo_\mathrm{reg}(F)$) in correspondence.
\end{theorem}
\begin{proof}
  See \cite[(3.1)]{Ar06} or \cite[\S 6]{HS12}.
\end{proof}
We will not need the precise description of $\chi$. Implicit in the assertion above is the fact $(Y \leftrightarrow X) \iff (\tau^{-1}(Y) \leftrightarrow X)$. As an immediate consequence, the image of transfer lies in the subspace
$$ S\orbI(\mathfrak{g}^\Endo, G) := \left\{ b \in S\orbI(\mathfrak{g}^\Endo) : \forall \tau \in \text{Out}_G(G^\Endo), \; b(\tau^{-1}(\cdot)) = \chi(\tau) b(\cdot) \right\}. $$

Define $S\Icusp(\mathfrak{g}^\Endo, G) := S\orbI(\mathfrak{g}^\Endo, G) \cap S\Icusp(\mathfrak{g}^\Endo)$. We are actually interested in the \emph{germs} around $0$ of cuspidal normalized orbital integrals. To be precise, we shall consider the spaces
\begin{align*}
  \mathcal{G}_\text{cusp}(\mathfrak{g}) & := \varinjlim_{\mathcal{U}} \Icusp(\mathcal{U}), \\
  S\mathcal{G}_\text{cusp}(\mathfrak{g}^\Endo, G) & := \varinjlim_{\mathcal{V}} S\Icusp(\mathcal{V}, G),
\end{align*}
where $\mathcal{U}$ (resp. $\mathcal{V}$) ranges over the invariant open subsets of $\mathfrak{g}(F)$ (resp. stably invariant and $\text{Out}_G(G')$-invariant open subsets of $\mathfrak{g}^\Endo(F)$) containing $0$, and $\Icusp(\mathcal{U})$ (resp. $S\Icusp(\mathcal{V}, G)$) is defined as above except that we consider solely the orbital integrals along the classes inside $\mathcal{U}$ (resp. $\mathcal{V}$). All in all, the transfer map induces a transfer of cuspidal germs:
$$ \mathcal{GT}: \mathcal{G}_\text{cusp}(\mathfrak{g}) \to S\mathcal{G}_\text{cusp}(\mathfrak{g}^\Endo, G) $$
for our chosen $\Delta_{G^\Endo, G}$.

As in \S\ref{sec:adjoint-transfer}, the adjoint transfer factor can also be defined on Lie algebras, namely
\begin{gather}\label{eqn:adjoint-Delta-std}
  \Delta_{G, G^\Endo}(X, Y) := |\mathfrak{D}(G_X, G; F)|^{-1} \overline{\Delta_{G^\Endo, G}(Y, X)}
\end{gather}
for $Y \leftrightarrow X$. We use it to invert the transfer.

\begin{theorem}[Arthur]\label{prop:GT-surj}
  The map $\mathcal{GT}$ is surjective. In fact, if $b \in S\mathcal{G}_\mathrm{cusp}(\mathfrak{g}^\Endo, G)$ then the function
  $$ a : X \longmapsto \sum_{\substack{Y \in \Delta_{\mathrm{reg,ell}}(\mathfrak{g}^\Endo) \\ Y \leftrightarrow X}} \Delta_{G^\Endo, G}(X, Y) b(Y) $$
  on $\Gamma_\mathrm{reg,ell}(\mathfrak{g})$ belongs to $\mathcal{G}_\mathrm{cusp}(\mathfrak{g})$, and satisfies $\mathcal{GT}(a)=b$.
\end{theorem}
\begin{proof}
  See the first part of the proof of \cite[Lemma 3.4]{Ar96}. Note that Arthur implicitly assumed in his proof that $\chi \equiv 1$, so he worked with the space $S\orbI(\mathfrak{g}^\Endo)^{\text{Out}_G(G')}$ instead of $S\orbI(\mathfrak{g}^\Endo, G)$. It is straightforward to adapt the cited arguments to the general case.
\end{proof}

\paragraph{Non-standard endoscopy on Lie algebras}
The non-standard endoscopy for Lie algebras is developed in \cite[\S 1.8]{Wa08}. We give a sketch here. A non-standard endoscopic datum is a triple $(G_1, G_2, j_*)$ where $G_1$, $G_2$ are quasisplit simply connected semisimple groups; choose a Borel pair $(B_i, T_i)$ for $G_i$ ($i=1,2$), the crucial datum $j_*$ is an isomorphism $j_*: X_*(T_{1, \bar{F}})_\Q \rightiso X_*(T_{2, \bar{F}})_\Q$ of $\Q$-vector spaces, such that there exist bijections
\begin{align*}
  \check{\tau}: \Sigma(G_1, T_1)^\vee_{\bar{F}} & \to \Sigma(G_2, T_2)^\vee_{\bar{F}}, \\
  \tau: \Sigma(G_2, T_2)_{\bar{F}} & \to \Sigma(G_1, T_1)_{\bar{F}},
\end{align*}
satisfying
\begin{compactenum}[(i)]
  \item $\tau$ and $\check{\tau}$ are mutually inverse up to $\alpha \leftrightarrow \alpha^\vee$;
  \item for roots $\alpha_i$ for $T_{i, \bar{F}}$ ($i=1, 2$), the elements $j_*(\alpha_1^\vee)$ and $\check{\tau}(\alpha_1^\vee)$ (resp. $j^*(\alpha_2)$ and $\tau(\alpha_2)$) are proportional by a factor in $\Q_{> 0}$, where $j^*$ denotes the dual of $j_*$;
  \item the maps $j_*, j^*$ are $\Gamma_F$-equivariant.
\end{compactenum}
In short, the root systems of $G_1$, $G_2$ become proportional under $j_*$. As $\mathfrak{t}_i(\bar{F}) = X_*(T_i)_\Q \otimes_\Q \bar{F}$, from $j_*$ we may construct an isomorphism $\mathfrak{t}_1/W_1 \rightiso \mathfrak{t}_2/W_2$ between $F$-varieties, where $W_i := N_{G_i}(T_i)/T_i$. This gives a bijection between $\Delta_\text{reg}(\mathfrak{g}_1)$ and $\Delta_\text{reg}(\mathfrak{g}_2)$ which preserves ellipticity. Write the bijection in the familiar way $X_1 \leftrightarrow X_2$. The non-standard transfer can now be enunciated.

\begin{theorem}[B. C. Ngô, Waldspurger]\label{prop:nonstd-endoscopy}
  There is a linear map $\mathcal{T}_{1,2}: S\orbI(\mathfrak{g}_1) \to S\orbI(\mathfrak{g}_2)$, written as $f^{G_1} \mapsto f^{G_2}$, which is characterized by $f^{G_1}(X_1) = f^{G_2}(X_2)$ whenever $X_1 \leftrightarrow X_2$.
\end{theorem}

\begin{remark}\label{rem:nonstd-endoscopy}
  Note the following properties.
  \begin{enumerate}
    \item As in the standard case, one readily shows that $\mathcal{T}_{1,2}$ preserves cuspidality.
    \item Since $(G_1, G_2, j_*)$ is a non-standard endoscopic datum if and only if $(G_2, G_1, j_*^{-1})$ is, and the characterization of non-standard transfer is clearly symmetric, we have
      $$ \mathcal{T}_{2,1} \mathcal{T}_{1,2} = \identity, \quad \mathcal{T}_{1,2} \mathcal{T}_{2,1} = \identity. $$
      In particular, non-standard transfer is an isomorphism. On the level of germs at $0$, it induces 
      \begin{gather*}
        \mathcal{GT}_{1,2}: S\mathcal{G}_\text{cusp}(\mathfrak{g}_1) \rightiso S\mathcal{G}_\text{cusp}(\mathfrak{g}_2)
      \end{gather*}
    \item The simply-connectedness condition can be dropped. More precisely, let $G_i \to \underline{G_i}$ ($i=1,2$) be isogenies. They do not affect the correspondence of stable conjugacy classes on Lie algebras and the non-standard transfer $f^{\underline{G_1}}(X_1) = f^{\underline{G_2}}(X_2)$ (for $X_1 \leftrightarrow X_2$) still holds in this context. This mechanism is explicated in detail in \cite[Lemme 8.4]{Li11}.
  \end{enumerate}
\end{remark}

What we will need is just the non-standard triplet $(\Sp(2n), \Spin(2n+1), j_*)$, where $j_*$ is described in \cite[p.15]{Wa08} (see also \cite[\S 8.2]{Li11}). If we replace $\Spin(2n+1)$ by its quotient $\SO(2n+1)$, as justified in the preceding remark, the correspondence of classes is given by matching eigenvalues. Namely, $X_1 \in \mathfrak{so}(2n)_\text{reg}$ corresponds to $X_2 \in \mathfrak{so}(2n+1)_\text{reg}$ if and only if
$$ X_1 \quad \text{has eigenvalues} \quad a_1, \ldots, a_n, -a_n, \ldots, -a_1 $$
and
$$ X_2 \quad \text{has eigenvalues} \quad a_1, \ldots, a_n, 0, -a_n, \ldots, -a_1 $$
for suitable $a_1, \ldots, a_n \in \bar{F}$. This may be compared with the correspondence for the endoscopic datum $(n, 0)$ of $\Mp(2n)$.

\paragraph{Semisimple Descent}
Return to the metaplectic setup. Let $\rev: \tilde{G} = \Mp(W) \to \Sp(W)$ be the metaplectic covering with $\dim_F W = 2n$. Fix an endoscopic datum $(n', n'') \in \EndoE_\text{ell}(\tilde{G})$ with endoscopic group $G^\Endo$. By Proposition \ref{prop:image-in-EndoE}, we have the transfer map $\mathcal{T}_{(n',n'')}: \Iaspcusp(\tilde{G}) \to S\Icusp(G^\Endo)$.

We will need the \emph{semisimple descent} for normalized orbital integrals. This is elementary as it is based solely on harmonic analysis. See \cite[\S 4.4.1]{Li15} or \cite[\S 4.2]{Li12b}; some notions below are borrowed from \cite{Bo94b}.

Let $\delta \in G(F)_\text{ss}$ and choose $\tilde{\delta} \in \rev^{-1}(\delta)$. We consider certain invariant open subsets of $G$ (called \emph{completely invariant} in \cite[\S 2.1]{Bo94b}) of the form
$$ \tilde{\mathcal{U}} = \left\{ x^{-1} \exp(X)\tilde{\delta} x : X \in \mathcal{U}^\flat, x \in G(F) \right\} \cdot \bmu_8 $$ 
where $\mathcal{U}^\flat \ni 0$ is a $G^\delta(F)$-invariant open subset of $\mathfrak{g}_\delta(F)$; we assume $\mathcal{U}^\flat$ to be \emph{$G$-admissible} in the sense of \textit{loc.\ cit.}, sufficiently small so that $\exp$ is a well-defined homeomorphism onto its image in $\tilde{G}$. Once the Haar measures on $G(F)$ and $G_\delta(F)$ are chosen, the semisimple descent around $\tilde{\delta}$ is given by a map
\begin{align*}
  \Iasp(\tilde{\mathcal{U}}) & \longrightarrow \orbI(\mathcal{U}^\flat)^{G^\delta(F)} \\
  f_{\tilde{G}} & \longmapsto f^\flat_{G_\delta},
\end{align*}
where $\Iasp(\tilde{\mathcal{U}})$ is the space of anti-genuine normalized orbital integrals along classes inside $\tilde{\mathcal{U}} \subset \tilde{G}$; similarly for $\orbI(\mathcal{U}^\flat)$ on the Lie algebra $\mathfrak{g}_\delta$, upon which $G^\delta(F)$ acts. The semisimple descent satisfies
$$ f_{\tilde{G}}(\exp(X)\tilde{\delta}) = f^\flat_{G_\delta}(X), \quad X \in \mathcal{U}^\flat_{\mathrm{reg}}. $$
Moreover, when $\delta$ is elliptic, the descent induces $\Iaspcusp(\tilde{\mathcal{U}}) \to \Icusp(\mathcal{U}^\flat)^{G^\delta(F)}$. Also note that for $X \in \mathcal{U}^\flat_\text{reg}$, we have $G_{\exp(X)\delta} = (G_\delta)_X$.

We will also need the semisimple descent for stable orbital integrals for $G^\Endo$ in \cite[\S 4.8]{Wa14-1}. It is much deeper as it requires the transfer between inner forms. Let $\gamma \in G^\Endo(F)_\text{ss}$ such that $G^\Endo_\gamma$ is quasisplit. Choose Haar measures as before. This time we fix a \emph{completely stably invariant} (see \cite[\S 6.1]{Bo94b}) open subset $\mathcal{V} \ni \gamma$, which arises from an open subset $\mathcal{V}^\flat \ni 0$ of $\mathfrak{g}^\Endo_\gamma(F)$ via the exponential map. Fix a system of representatives $\gamma_1, \ldots, \gamma_k$ of conjugacy classes in $G^\Endo(F)$ stably conjugate to $\gamma$. Define the finite $F$-group scheme $\Xi_\gamma := (G^\Endo)^\gamma/G^\Endo_\gamma$. The semisimple descent around $\gamma$ is a map
\begin{align*}
  S\orbI(\mathcal{V}) & \longrightarrow S\orbI(\mathcal{V}^\flat)^{\Xi_\gamma(F)} \\
  f^{G^\Endo} & \longmapsto f^{G^\Endo_\delta, \flat}
\end{align*}
with the same conventions as before. It preserves the subspaces $S\Icusp(\cdots)$ and satisfies
$$f^{G^\Endo}(\exp(Y)\gamma) = f^{G^\Endo_\delta, \flat}(Y), \quad Y \in \mathcal{V}^\flat_\text{reg}. $$

\begin{proof}[Proof of Theorem \ref{prop:transfer-cusp-surj}]
  In view of the definition of $\Icusp^\EndoE(\tilde{G})$, it suffices to show that when $a^\EndoE \in \Icusp^\EndoE(\tilde{G})$ lies in the image of
  $$ S\Icusp(G^\Endo) \hookrightarrow \Icusp^\EndoE(\tilde{G}), $$
  i.e.\ when the components of $a^\EndoE$ vanish except for the component $a^\Endo \in S\Icusp(G^\Endo)$ indexed by $(n', n'')$, there exists $a \in \Iaspcusp(\tilde{G})$ such that $\mathcal{T}^\EndoE(a) = a^\EndoE$.
  
  Take the adjoint transfer $a := \mathcal{T}_\EndoE(a^\EndoE)$. It is an anti-genuine function on $\Gamma_\text{reg}(\tilde{G})$. In fact, $a$ is supported on $\Gamma_\text{reg,ell}(\tilde{G})$, since the correspondence $\sigma \leftrightarrow \delta$ preserves ellipticity. Claim: there exists $a_\text{gen} \in \Iaspcusp(\tilde{G})$ such that $a=a_\text{gen}$ for elements in general position. Indeed, if this holds then $\mathcal{T}^\EndoE(a_\text{gen}) = \mathcal{T}^\EndoE(a) = a^\EndoE$ for elements in general position (recall Proposition \ref{prop:transfer-injective}). Hence $\mathcal{T}^\EndoE(a_\text{gen}) = a^\EndoE$ everywhere by the continuity of stable orbital integrals, thereby establishing our theorem.
  
  Choose a complete set of representatives $\gamma_1, \ldots, \gamma_k \in G^\Endo(F)_\text{ss}$ of the stable semisimple conjugacy classes in $G^\Endo(F)$ corresponding to $\delta$. By \cite[Lemma 3.3]{Ko82}, we may arrange that $G^\Endo_i := G^\Endo_{\gamma_i}$ is quasisplit for $i=1, \ldots, k$. Define $\overline{G^\Endo_i}$ as follows: $G^\Endo_i$ admits a canonical decomposition $R \times \prod_a \SO(2a+1)$, where $R$ contains no factor of the type of odd split $\SO$; we put $\overline{G^\Endo_i} := R \times \prod_a \Sp(2a)$. By \cite[Théorèmes 7.10, 7.23]{Li11}, the situation can be summarized below
  \begin{equation*} \begin{tikzcd}[column sep=6em, row sep=small]
    \text{endoscopic data}: & G^\Endo_i \arrow[-,dashed]{r}[above]{\text{non-standard}}[below]{\text{endoscopy}} & \overline{G^\Endo_i} \arrow[-,dashed]{r}[above]{\text{elliptic}}[below]{\text{endoscopy}} & G_\delta \\
    \text{Lie algebra level}: & Y_i \arrow[leftrightarrow]{r}[above]{\text{via eigenvalues}} & \overline{Y_i} \arrow[leftrightarrow]{r} & X \\
    \text{group level}: & \exp(Y_i)\gamma_i \arrow[leftrightarrow]{rr}[above]{\text{via } (n',n'') \in \EndoE_\text{ell}(\tilde{G})} & & \exp(X)\delta,
  \end{tikzcd} \end{equation*}
  \begin{equation}\label{eqn:descent-Delta}
    \Delta(\exp(Y_i)\gamma_i, \exp(X)\tilde{\delta}) = \Delta_{\overline{G^\Endo_i}, G_\delta}(\overline{Y_i}, X), \quad X, Y_i \text{ close to } 0,
  \end{equation}
  where $\Delta_{\overline{G^\Endo_i}, G_\delta}$ is some transfer factor for the endoscopic datum $(\overline{G^\Endo_i}, \ldots)$ for $G_\delta$. The non-standard endoscopy here is understood in an extended sense so that the pairs $(\Sp(2a), \SO(2a+1))$ are allowed; see Remark \ref{rem:nonstd-endoscopy}.
  
  Now pass to the adjoint transfer factors: recall that $\Delta_{G_\delta, \overline{G^\Endo_i}}$ is defined by \eqref{eqn:adjoint-Delta-std}. Then \eqref{eqn:descent-Delta} implies
  \begin{align*}
    \Delta(\exp(X)\tilde{\delta}, \exp(Y_i)\gamma_i) &= \Delta_{G_\delta, \overline{G^\Endo_i}}(X, \overline{Y_i}) \cdot \frac{|\mathfrak{D}((G_\delta)_X, G_\delta; F)|} {|\mathfrak{D}(G_{\exp(X)\delta}, G; F)|} \\
    & = \Delta_{G_\delta, \overline{G^\Endo_i}}(X, \overline{Y_i}) \cdot |H^1(F, G_\delta)|^{-1}
  \end{align*}
  for $X \leftrightarrow Y_i$ close to $0$. The first equality is just definition whereas the second stems from the following fact: for every connected reductive $F$-group $H$ and any elliptic maximal $F$-torus $S \subset H$, we have $|\mathfrak{D}(S, H; F)| = |H^1(F, S)| \cdot |H^1(F, H)|^{-1}$. Indeed, these $H^1$ are abelian groups and $H^1(F, S) \to H^1(F, H)$ is surjective by \cite[10.2 Lemma]{Ko86}.
  
  Set $\Xi_i := \Xi_{\gamma_i}$. Observe that for $X \in \mathfrak{g}_{\delta, \text{reg}}(F)$ sufficiently close to $0$,
  $$ \left\{ \sigma \in \Delta_{G-\text{reg}}(G^\Endo) : \sigma \leftrightarrow  \exp(X)\delta \right\} = \bigsqcup_{i=1}^k \left\{ \exp(Y_i) \gamma_i : Y_i \in \Delta_\text{reg}(\mathfrak{g}^\Endo_i), \; \overline{Y_i} \leftrightarrow X \right\} \big/ \Xi_i(F) $$
  where each $Y_i$ is also assumed to be close to $0$. To show this, just copy the arguments of \cite[p.583]{Li11}; note that the rôles of $\sigma$ and $\exp(X)\delta$ are reversed there.
  
  Summing up, for $X \in \mathfrak{g}_{\delta, \text{reg}}(F)$ sufficiently close to $0$ and in general position, we have
  \begin{align*}
    a(\exp(X)\tilde{\delta}) & = \sum_{\sigma \leftrightarrow \delta} \Delta(\exp(X)\tilde{\delta}, \sigma) a^\Endo(\sigma) \\
    & = \sum_{i=1}^k \sum_{Y_i \leftrightarrow \overline{Y_i} \leftrightarrow X} \Delta_{G_\delta, \overline{G^\Endo_i}}(X, \overline{Y_i}) a^\Endo_i(Y_i) \\
    & = \sum_{i=1}^k \sum_{\overline{Y_i} \leftrightarrow X} \Delta_{G_\delta, \overline{G^\Endo_i}}(X, \overline{Y_i}) b^\Endo_i(\overline{Y_i})
  \end{align*}
  for suitable $a^\Endo_i \in S\Icusp(\mathfrak{g}^\Endo_i)^{\Xi_i(F)}$. Here $b^\Endo_i \in S\Icusp(\overline{\mathfrak{g}^\Endo_i})$ is the non-standard transfer of $a^\Endo_i$ (Theorem \ref{prop:nonstd-endoscopy}) characterized by
  $$ b^\Endo_i(\overline{Y_i}) = a^\Endo_i(Y_i). $$
  By ``in general position'' we require that each stable class $Y_i$ (resp. $\overline{Y_i}$) corresponding to $X$ as above has trivial stabilizer under $\Xi_i(F)$ (resp. $\text{Out}_{G_\delta}(\overline{G^\Endo_i})$). Moreover, by the $\text{Out}_{G_\delta}(\overline{G^\Endo_i})$-equivariance property of $\Delta_{\overline{G^\Endo_i}, G_\delta}$ in Theorem \ref{prop:Delta-std-equivariance} (whence the opposite equivariance of the adjoint transfer factor $\Delta_{G_\delta, \overline{G^\Endo_i}}$), we may modify $a^\Endo_i$ so that $b^\Endo_i \in S\Icusp(\mathfrak{g}^\Endo_i, G_\delta)$.
  
  To conclude the proof, apply Theorem \ref{prop:GT-surj} to the germs
  $$ a^i_{\text{gen}, \tilde{\delta}}(X) := \sum_{\overline{Y_i} \leftrightarrow X} \Delta_{G_\delta, \overline{G^\Endo_i}}(X, \overline{Y_i}) b^\Endo_i(\overline{Y_i}), \quad i=1, \ldots, k $$
  to obtain $a_{\text{gen}, \tilde{\delta}} := \sum_{i=1}^k a^i_{\text{gen}, \tilde{\delta}} \in \mathcal{G}_\text{cusp}(\mathfrak{g}_\delta)$. We have $a(\exp(\cdot)\tilde{\delta}) = a_{\text{gen}, \tilde{\delta}}$ for elements in general position, hence $a_{\text{gen}, \tilde{\delta}}$ is $G^\delta(F)$-invariant by the continuity of orbital integrals. The local characterization \cite{Vi81} of $\Iasp(\tilde{G})$ asserts that these $a_{\text{gen}, \tilde{\delta}}$ patch together into $a_\text{gen} \in \Iaspcusp(\tilde{G})$ such that $a=a_\text{gen}$ for elements in general position: indeed, this results from the corresponding property for $a(\exp(\cdot)\tilde{\delta})$ (for various $\tilde{\delta} \in \tilde{G}_\text{ss}$) by the continuity of orbital integrals. This establishes our earlier claim.
\end{proof}

\section{Spectral transfer in the non-archimedean case}\label{sec:spectral-transfer}
Keep the assumptions in \S\ref{sec:geom-transfer}. In particular, we consider a fixed local metaplectic covering $\rev: \tilde{G} \to G(F)$, where $\tilde{G}=\Mp(W)$, $\dim_F W = 2n$. We also fix a minimal Levi subgroup $M_0$ of $G$. Although the main concern of this section is the non-archimedean case, some definitions are also useful for archimedean $F$; this will be mentioned explicitly in what follows.

\subsection{Paley-Wiener spaces}\label{sec:PW-spaces}
The exposition below for $\tilde{G}$ will be sketchy; details can be found in \cite[\S 5.4]{Li12b} or \cite{Ar93,Ar96}. A large portion of our discussions also makes sense for archimedean $F$, but one has to be careful about the choice of test functions (eg.\ $\tilde{K}$-finite functions, Schwartz-Harish-Chandra functions, etc.) We will try to indicate the necessary adaptations.

\paragraph{The unstable side}
For each $L \in \mathcal{L}(M_0)$ and $\pi \in \Pi_{2,-}(\tilde{L})$, the Knapp-Stein theory of intertwining operators furnishes a central extension
\begin{gather}\label{eqn:R-extension}
  1 \to \Sph^1 \to \tilde{R}_\pi \to R_\pi \to 1.
\end{gather} \index{$\tilde{R}_\pi$}
The finite group $R_\pi$ is the $R$-group attached to $\pi$, here viewed as a quotient of $W_\pi := \Stab_{W^G(M)}(\pi)$ by some normal subgroup $W^\circ_\pi$ described in terms of Harish-Chandra's $\mu$-functions. For $P \in \mathcal{P}(L)$, there is a homomorphism from $\tilde{R}_\pi$ to $\Aut_{G(F)}(I_P(\pi))$:
$$ r \longmapsto R_{\tilde{P}}(r, \pi): \quad \text{unitary operator}, $$
under which $\Sph^1 \ni z \mapsto z^{-1} \cdot \identity$. Here $R_{\tilde{P}}(r, \pi)$ is the \emph{normalized intertwining operator}, it depends on the choice of \emph{normalizing factors} for the standard intertwining operators, whose existence for covering groups is established in \cite[\S 3]{Li12b}. Note that in \textit{loc.\ cit.}, \eqref{eqn:R-extension} is reduced to an extension by some finite subgroup of $\Sph^1$.

Define $\tilde{T}_-(\tilde{G})$ to be the set of triples $(L, \pi, r)$ as above such that $zr$ is conjugate to $r$ if and only if $z=1$; such triples are called \emph{essential} in \textit{loc.\ cit.} Set $\tilde{T}_{\text{ell},-}(\tilde{G})$ to be the subset defined by requiring $\det(1-r|\mathfrak{a}^G_L) \neq 0$. For $\tau = (L, \pi, r) \in \tilde{T}_{\text{ell},-}(\tilde{G})$ we set
\begin{gather}\label{eqn:d(tau)}
  d(\tau) := \det(1 - r| \mathfrak{a}^G_L) \quad \in F.
\end{gather}

These definitions generalized to the Levi subgroups of $\tilde{G}$ as well, which are coverings of metaplectic type. There is a canonical map $\tilde{T}_-(\tilde{M}) \to \tilde{T}_-(\tilde{G})$ giving rise to the decomposition
$$ \tilde{T}_-(\tilde{G}) = \bigsqcup_{M \in \mathcal{L}(M_0)} \tilde{T}_{\text{ell},-}(\tilde{M}). $$

Note that $i\mathfrak{a}^*_M$ acts on $\tilde{T}_{\text{ell}, -}(\tilde{M})$ by sending $\tau = (L, \pi, r)$ to $\tau_\lambda := (L, \pi_\lambda, r)$. By decomposition into $i\mathfrak{a}^*_M$-orbits, each $\tilde{T}_{\text{ell},-}(\tilde{M})/\Sph^1$ becomes a disjoint union of compact tori (resp. Euclidean spaces in the archimedean case), thus so is $\tilde{T}_-(\tilde{G})/\Sph^1$. On the other hand, we have a $W^G_0$-action on $\tilde{T}_-(\tilde{G})$ by transport of structure, written as $(L, \pi, r) \mapsto (wLw^{-1}, w\pi, wrw^{-1})$, that commutes with $\Sph^1$. Set \index{$T_-(\tilde{G})$}
\begin{align*}
  T_{\text{ell},-}(\tilde{G}) & := \tilde{T}_{\text{ell},-}(\tilde{G})/W^G_0, \\
  T_-(\tilde{G}) & := \tilde{T}_-(\tilde{G})/W^G_0 \\
  & = \bigsqcup_{M \in \mathcal{L}(M_0)/\text{conj}} T_{\text{ell},-}(\tilde{M})/W^G(M),
\end{align*}
and so on; they still have natural structures of analytic $\R$-varieties. Also note that $\tilde{T}_-(\tilde{G}) \twoheadrightarrow \tilde{T}_-(\tilde{G})/\Sph^1$ and each $\tilde{T}_{\text{ell},-}(\tilde{M}) \twoheadrightarrow \tilde{T}_{\text{ell},-}(\tilde{M})/\Sph^1$ are $\Sph^1$-torsors.

\begin{remark}\label{rem:T-splitting}
  The $\Sph^1$-torsor $\tilde{T}_{\text{ell},-}(\tilde{M})$ over $\tilde{T}_{\text{ell},-}(\tilde{M})/\Sph^1$ is actually split, albeit non-canonically. Indeed, consider a connected component $\mathcal{C}$ of $\tilde{T}_{\text{ell},-}(\tilde{M})$ and fix some $(L, \pi, r) \in \mathcal{C}$. Write $r \mapsto \underline{r} \in R_\pi$. Then a trivialization of $\mathcal{C} \twoheadrightarrow \mathcal{C}/\Sph^1$ is given by
  $$ (L, \pi_\lambda, \underline{r}) \longmapsto (L, \pi_\lambda, r), \quad \lambda \in i\mathfrak{a}^*_M. $$
  Such trivializations of $\mathcal{C}$ are in bijection with $\{ r \in \tilde{R}_\pi: r \mapsto \underline{r}\}$.
\end{remark}

Let us complexify this construction. By allowing $\pi \in \Pi_-(\tilde{L})$ to be only essentially square-integrable modulo centre (i.e.\ $\pi_\lambda \in \Pi_{2,-}(\tilde{L})$ for some $\lambda \in \mathfrak{a}^*_{L,\C}$), and letting $\mathfrak{a}^*_{M, \C}$ act on the triples $(L, \pi, r) \in \tilde{T}_{\text{ell},-}(\tilde{M})_\C$ so obtained, we obtain $\tilde{T}_-(\tilde{G})_\C$ and $T_-(\tilde{G})_\C := \tilde{T}_-(\tilde{G})/W^G_0$. They are complex varieties. Moreover, the intertwining operators $R_{\tilde{P}}(r, \pi)$ extend meromorphically to the complexified setup.

For $\tau = (L, \pi, r) \in \tilde{T}_-(\tilde{G})$, we have the representation $\mathcal{R}_{\tilde{P}} := R_{\tilde{P}}(\cdot, \pi) I_{\tilde{P}}(\pi, \cdot)$ of $\tilde{R}_\pi \times \tilde{G}$ on the underlying vector space of $I_{\tilde{P}}(\pi)$. Let $\Pi_-(\tilde{R}_\pi)$ be the subset of $\Pi(\tilde{R}_\pi)$ consisting of representations on which $\Sph^1$ acts by $z \mapsto z\cdot\identity$. The upshot is the existence of a bijection $\rho \mapsto \xi_\rho \in \Pi_{\text{temp},-}(\tilde{G})$ between $\Pi_-(\tilde{R}_\pi)$ and the set of irreducible constituents of $I_{\tilde{P}}(\pi)$, characterized by
$$ \mathcal{R}_{\tilde{P}} \simeq \bigoplus_{\rho \in \Pi_-(\tilde{R}_\pi)} \rho^\vee \boxtimes \xi_\rho. $$
Define the genuine invariant distributions
\begin{align*}
  \Theta_\tau & = \Theta^{\tilde{G}}_\tau := \Tr \mathcal{R}_{\tilde{P}}(r, \cdot) = \Tr\left( R_{\tilde{P}}(r, \pi) I_{\tilde{P}}(\pi, \cdot) \right) \\
  & = \sum_{\rho \in \Pi_-(\tilde{R}_\pi)} \Tr\left(\rho^\vee(r)\right) \Theta_{\xi_\rho}, \\
  I_{\tilde{G}}(\tau, \cdot) & := |D^G(\cdot)|^{\frac{1}{2}} \Theta_\tau.
\end{align*}
They can be shown to be independent of the choice of $P \in \mathcal{P}(L)$. Moreover, $\Theta_\tau$ depends only on the image of $\tau$ in $T_-(\tilde{G})$. In this manner, $T_-(\tilde{G})$ furnishes a basis of the space of virtual genuine tempered characters of $\tilde{G}$, up to dilation by $\Sph^1$. The distributions $\Theta_\tau$ and $I_{\tilde{G}}(\tau, \cdot)$ also admit meromorphic continuation to $T_-(\tilde{G})_\C$. They satisfy
\begin{gather}\label{eqn:T-induction}
  I_{\tilde{P}} \left( \Theta^{\tilde{M}}_\tau \right) = \Theta^{\tilde{G}}_\tau
\end{gather}
for any $M \in \mathcal{L}(M_0)$, $P \in \mathcal{P}(M)$ and $\tau \in \tilde{T}_{\text{ell},-}(\tilde{M})_\C \hookrightarrow \tilde{T}_-(\tilde{G})_\C$; see \cite[Définition 3.1.1 (R5)]{Li12b}.

For $f_{\tilde{G}} \in \Iasp(\tilde{G})$, define
\begin{equation}\label{eqn:f-pi}\begin{aligned}
  f_{\tilde{G}}(\tau) & := \int_{\Gamma_\text{reg}(G)} I_{\tilde{G}}(\tau, \tilde{\delta}) f_{\tilde{G}}(\tilde{\delta}) \dd\delta \\
  & = \Theta_\tau(f_{\tilde{G}})
\end{aligned}\end{equation} \index{$f_{\tilde{G}}(\tau)$}
for every $\tau \in T_-(\tilde{G})$; the second equality follows from Weyl's integration formula. Denote the resulting function on $T_-(\tilde{G})$ as $f_{\tilde{G}, \gr}$.

\begin{lemma}\label{prop:gr-injectivity}
  The map $f_{\tilde{G}} \mapsto f_{\tilde{G}, \gr}$ is a linear injection into the space of functions $T_-(\tilde{G}) \to \C$.
\end{lemma}
\begin{proof}
  The injectivity stems from \cite[Théorème 5.8.10]{Li12b}. The archimedean case is contained in \cite[Corollaire 3.3.2]{Bo94b}.
\end{proof}

Denote the image of $\Iasp(\tilde{G})$ as $\text{PW}_{\asp}(\tilde{G})$. For non-archimedean $F$, it consists of functions $\alpha: \tilde{T}_-(\tilde{G}) \to \C$ satisfying
\begin{compactenum}[(i)]
  \item $\alpha$ factors through $T_-(\tilde{G})$;
  \item $\alpha(z \cdot) = z^{-1} \alpha(\cdot)$ for every $z \in \Sph^1$;
  \item $\alpha$ is supported on finitely many connected components of $\tilde{T}_-(\tilde{G})$;
  \item the restriction of $\alpha$ to each connected component $\mathcal{C}$, viewed as a function on $\mathcal{C}/\Sph^1$ by choosing any trivialization as in Remark \ref{rem:T-splitting}, is a Paley-Wiener function.
\end{compactenum}
This characterization of $\text{PW}_{\asp}(\tilde{G})$ is nothing but the trace Paley-Wiener theorem; its justification for coverings is given in \cite[\S 2]{Li12b} and \cite[\S 3.4]{Li14b}. For archimedean versions, see \cite{CD84,Bo94b}. Consequently there is a notion of \emph{pseudo-coefficients} for every $\tau \in T_{\text{ell},-}(\tilde{G})$. The same spaces may be defined for any $\tilde{M}$, where $M \in \mathcal{L}(M_0)$. There is then a natural action of $W^G(M)$ on $\text{PW}_{\asp}(\tilde{M})$.

\begin{proposition}\label{prop:PW-spectral}
  For each $M \in \mathcal{L}(M_0)$, the inverse image of
  \begin{align*}
    \mathrm{PW}_{\asp}(\tilde{G})_M & := \left\{ \alpha \in \mathrm{PW}_{\asp}(\tilde{G}) : \Supp(\alpha) \subset W^G_0 \cdot \tilde{T}_{\mathrm{ell},-}(\tilde{M}) \right\} \\
    & \simeq \mathrm{PW}_{\asp}(\tilde{M})_M^{W^G(M)} \quad \text{(by restriction to $\tilde{T}_-(\tilde{M})$)}
  \end{align*}
  is contained in $\mathcal{F}^M \Iasp(\tilde{G})$. Composition with
  $$ \begin{tikzcd}[row sep=tiny]
    \mathcal{F}^M \Iasp(\tilde{G}) \arrow[twoheadrightarrow]{r} & \gr^M \Iasp(\tilde{G}) \arrow{r}[above]{\sim} & \Iaspcusp(\tilde{M})^{W^G(M)} \\
    f_{\tilde{G}} \arrow[mapsto]{rr} & & f_{\tilde{M}}
  \end{tikzcd} $$
  yields $\mathrm{PW}_{\asp}(\tilde{G})_M \rightiso \Iaspcusp(\tilde{M})^{W^G(M)}$.
\end{proposition}
\begin{proof}
  It suffices to note that for every $L \in \mathcal{L}(M_0)$, the diagram
  $$ \begin{tikzcd}
    \Iasp(\tilde{G}) \arrow{d}[left]{f_{\tilde{G}} \mapsto f_{\tilde{L}}} \arrow{r}[above]{\sim} & \text{PW}_{\asp}(\tilde{G}) \arrow{d}[right]{\text{restriction to } \tilde{T}_-(\tilde{L})} \\
    \Iasp(\tilde{L})^{W^G(L)} \arrow{r}[below]{\sim} & \text{PW}_{\asp}(\tilde{L})^{W^G(L)}
  \end{tikzcd} $$
  commutes by \eqref{eqn:T-induction}.
\end{proof}

In view of the definition of $\text{PW}_{\asp}(\tilde{G})$, there is an evident decomposition
\begin{equation}\label{eqn:PW-decomp}\begin{aligned}
  \text{PW}_{\asp}(\tilde{G}) &= \bigoplus_{M \in \mathcal{L}(M_0)/W^G_0} \text{PW}_{\asp}(\tilde{G})_M \\
  & \rightiso \bigoplus_{M \in \mathcal{L}(M_0)/W^G_0} \Iaspcusp(\tilde{M})^{W^G(M)}.
\end{aligned}\end{equation}

\begin{corollary}\label{prop:Iasp-grading}
  Composing $f_{\tilde{G}} \mapsto f_{\tilde{G},\gr}$ with the decomposition \eqref{eqn:PW-decomp} induces a filtration-preserving isomorphism
  $$ \Iasp(\tilde{G}) \rightiso \orbI_\gr(\tilde{G}) $$
  as claimed in Remark \ref{rem:grading}.
\end{corollary}

Summing up, elements of $\Iasp(\tilde{G})$ may be viewed either
\begin{inparaenum}[(i)]
  \item as functions on $\Gamma_\text{reg}(\tilde{G})$ (geometrically), or
  \item as functions on $\tilde{T}_-(\tilde{G})$, via $f_{\tilde{G}} \mapsto f_{\tilde{G}, \gr}$ (spectrally).
\end{inparaenum}
Parallel to \S\ref{sec:collective-geom-trans}, we set out to define a Radon measure on $T_-(\tilde{G})/\Sph^1$. For $M \in \mathcal{L}(M_0)$, $\tau \in \tilde{T}_{\text{ell},-}(\tilde{M})$, set
\begin{align*}
  \mathfrak{a}_{M,\tau}^\vee & := \Stab_{i\mathfrak{a}^*_M}(\tau), \\
  i\mathfrak{a}^*_{M, \tau} & := i\mathfrak{a}^*_M/\mathfrak{a}_{M,\tau}^\vee.
\end{align*}
The measures are defined by
\begin{align*}
  \int_{T_{\text{ell},-}(\tilde{M})/\Sph^1} \alpha & = \sum_{\substack{\tau \in \Sph^1 \backslash T_{\text{ell},-}(\tilde{M})/i\mathfrak{a}^*_M \\ \tau=(L, \pi, r)}} \; |Z_{R_\pi}(r)|^{-1} \int_{i\mathfrak{a}^*_{M, \tau}} \alpha(\tau_\lambda) \dd\lambda ,\\
  \int_{T_-(\tilde{G})/\Sph^1} \alpha & = \sum_{M \in \mathcal{L}(M_0)/W^G_0} |W^G(M)|^{-1} \int_{T_{\text{ell},-}(\tilde{M})/\Sph^1} \alpha
\end{align*}
for suitable test functions $\alpha$, where $Z_{R_\pi}(r) := \Stab_{R_\pi}(r)$ for $r \in \tilde{R}_\pi$; also observe that $Z_{\tilde{R}_\pi}(r)/\Sph^1 = Z_{R_\pi}(r)$ by our definition of essential triples.

Denote the quotient map $T_-(\tilde{G}) \to T_-(\tilde{G})/\Sph^1$ by $\bm{\pi}$, we may deduce a Radon measure on $T_-(\tilde{G})$ by requiring
$$ \mes(\bm{\pi}^{-1}(E)) = \mes(E) $$
for every measurable $E \subset T_-(\tilde{G})/\Sph^1$. This is redundant somehow, since we will only integrate over $T_-(\tilde{G})/\Sph^1$ in this article. Nevertheless, the same recipe defines an analogous measure in \cite[p.834]{Li12b} on a bundle over $T_-(\tilde{G})/\Sph^1$, the only difference being that $\Ker(\tilde{R}_\pi \twoheadrightarrow R_\pi)$ is taken to be a finite cyclic group in \textit{loc.\ cit.} Thus these two formalisms can be reconciled.

Define an hermitian pairing
$$ (a_1|a_2)_\text{ell} := \int_{T_{\text{ell},-}(\tilde{G})/\Sph^1} |d(\tau)|^{-1} a_{1,\gr}(\tau) \overline{a_{2,\gr}(\tau)} \dd\tau, \quad a_1, a_2 \in \Iaspcusp(\tilde{G}), $$
with the $d(\tau)$ in \eqref{eqn:d(tau)}.

\begin{lemma}\label{prop:ell-ip}
  For all $a_1, a_2 \in \Iaspcusp(\tilde{G})$, the hermitian pairing $(a_1|a_2)_\mathrm{ell}$ is convergent and equals the $(a_1|a_2)$ in Definition \ref{def:pairing-geom}.
\end{lemma}
\begin{proof}
  In view of the compatibility of measures alluded to above, this is a special case of \cite[Théorème 5.8.7]{Li12b}; cf.\ \cite[Corollary 3.2]{Ar93}.
\end{proof}

\paragraph{The stable side}
All the results below are contained in \cite[\S 5]{Ar96}, in view of the construction of local $L$-packets in \cite[Chapter 6]{Ar13}. We shall be brief here.

Let $G^\Endo$ be an elliptic endoscopic group of $\tilde{G}$. We fix a minimal Levi subgroup $M^\Endo_0 \subset G^\Endo$. In \S\ref{sec:stable-character-SO}, we have defined the stable tempered character $S\Theta_\phi$ for $\phi \in \Phi_\text{bdd}(G^\Endo)$. They are stable distributions by Theorem \ref{prop:SO-stability} and have normalized version
$$ S^{G^\Endo}(\phi, \cdot) := |D^{G^\Endo}(\cdot)|^{\frac{1}{2}} S\Theta_\phi, $$
viewed as a smooth function on $G^\Endo_\text{reg}(F)$. The same definition works for any $M^\Endo \in \mathcal{L}(M^\Endo_0)$ in place of $G^\Endo$. For every $\lambda \in i\mathfrak{a}^*_{M^\Endo}$ we have $S^{M^\Endo}(\phi_\lambda, \cdot) = e^{\angles{\lambda, H_M(\cdot)}} S^{M^\Endo}(\phi, \cdot)$.


As in \eqref{eqn:f-pi}, given $\phi \in \Phi_\text{bdd}(G^\Endo)$ and $a \in S\orbI(G^\Endo)$ we set
\begin{equation}\label{eqn:a-phi}\begin{aligned}
  a(\phi) & := \int_{\Delta_\text{reg}(G^\Endo)} S^{G^\Endo}(\phi, \sigma) a(\sigma) \dd\sigma \\
  & = S\Theta_\phi(a).
\end{aligned}\end{equation}

\begin{remark}\label{rem:PW-st}
  As on the unstable side,
  \begin{itemize}
    \item the map $a \mapsto [\phi \mapsto a(\phi)]$ identifies $S\orbI(G^\Endo)$ as a space of functions $\Phi_\text{bdd}(G^\Endo) \to \C$;
    \item the image of $S\orbI(G^\Endo)$ under the identification above has a characterization à la Paley-Wiener, cf.\ Proposition \ref{prop:PW-spectral} and \cite{Wa13-4};
    \item consequently, we deduce a filtration-preserving isomorphism $S\orbI(G^\Endo) \rightiso S\orbI_\gr(G^\Endo)$ as claimed in Remark \ref{rem:grading}.
  \end{itemize}
  The precise formulation is completely analogous to the unstable side. 
\end{remark}

Recall the decomposition \eqref{eqn:Phi_bdd-decomp} by which $\Phi_\text{bdd}(G^\Endo)$ and the various $\Phi_{2,\text{bdd}}(M^\Endo)$ acquire $\R$-analytic structures. It makes sense to define a Radon measure on $\Phi_\text{bdd}(G^\Endo)$ by stipulating
\begin{align*}
  \int_{\Phi_{2,\text{bdd}}(M)} \alpha & = \sum_{\phi \in \Phi_{2,\text{bdd}}(M^\Endo)/i\mathfrak{a}^*_{M^\Endo} } \; \int_{i\mathfrak{a}^*_{M^\Endo, \phi}} \alpha(\phi_\lambda) \dd\lambda, \\
  \int_{\Phi_{\text{bdd}}(G)} \alpha & = \sum_{M^\Endo \in \mathcal{L}(M^\Endo_0)/W^{G^\Endo}_0} |W^{G^\Endo}(M^\Endo)|^{-1} \int_{\Phi_{2,\text{bdd}}(M)} \alpha
\end{align*}
for suitable test functions $\alpha$, where
\begin{align*}
  i\mathfrak{a}^*_{M^\Endo, \phi} & := i\mathfrak{a}^*_{M^\Endo}/\mathfrak{a}^\vee_{M^\Endo, \phi}, \\
  \mathfrak{a}^\vee_{M^\Endo, \phi} & := \Stab_{i\mathfrak{a}^*_{M^\Endo}}(\phi).
\end{align*}

\begin{lemma}\label{prop:st-ell-ip}
  Assume $F$ non-archimedean. For every $a_1, a_2 \in S\Icusp(G^\Endo)$, we have
  $$ (a_1|a_2) = \int_{\Phi_{\mathrm{bdd},2}(G^\Endo)} |\mathscr{S}_{\phi, \ad}|^{-1} a_1(\phi) \overline{a_2(\phi)} \dd\phi $$
  where the $\mathscr{S}$-groups $\mathscr{S}_{\phi,\ad} := \pi_0(S_{\phi,\mathrm{ad}})$ are defined in \S\ref{sec:L-parameters}.
\end{lemma}
\begin{proof}
  See \cite[p.542]{Ar96}.
\end{proof}

Parallel to \eqref{eqn:a-phi}, we have the following inversion formula due to Arthur. The proof is based on the inversion formulas in \cite{Ar94} which are valid for archimedean $F$ as well.
\begin{lemma}[{\cite[Lemma 6.3]{Ar96}}]\label{prop:stable-orbint-Fourier}
  There is a smooth function $S^{G^\Endo}(\sigma, \phi)$ of $(\sigma, \phi) \in \Delta_\mathrm{reg}(G^\Endo) \times \Phi_\mathrm{bdd}(G^\Endo)$, such that
  $$ a(\sigma) = \int_{\Phi_\mathrm{bdd}(G^\Endo)} S^{G^\Endo}(\sigma, \phi) a(\phi) \dd\phi $$
  for any $a \in S\orbI(G^\Endo)$.
\end{lemma}

\subsection{Spectral transfer factors}\label{sec:spectral-formalism}
In this subsection, $F$ can be any local field of characteristic zero except in the second part where we introduce the spectral transfer factors. We set
$$ T^\EndoE_\text{ell}(\tilde{G}) := \bigsqcup_{G^\Endo \in \EndoE_\text{ell}(\tilde{G})} \Phi_{2,\text{bdd}}(G^\Endo). $$
By the foregoing constructions, it has the structure of an analytic $\R$-variety and comes equipped with a Radon measure. As usual, this can be extended to Levi subgroups and we set \index{$T^\EndoE(\tilde{G})$}
$$ T^\EndoE(\tilde{G}) := \bigsqcup_{M \in \mathcal{L}(M_0)/W^G_0} T^\EndoE_\text{ell}(\tilde{M})/W^G(M). $$
Here, by writing $M = \prod_{i \in I} \GL(n_i) \times \Sp(W^\flat)$, the action of $W^G(M) = \mathfrak{S}(I)$ is permutation of the indexing set $I$ of the $\GL$-factors of each $M^\Endo \in \EndoE_\text{ell}(\tilde{M})$. As in \S\ref{sec:PW-spaces}, one also can define their complexified versions $T^\EndoE(\tilde{G})_\C$, etc.

Next, define \index{$\orbI^\EndoE_\gr(\tilde{G})$}
$$ \orbI^\EndoE_\gr(\tilde{G}) := \bigoplus_{M \in \mathcal{L}(M_0)/W^G_0} \Icusp^\EndoE(\tilde{M})^{W^G(M)}. $$
We shall regard each $\Icusp^\EndoE(\tilde{M})^{W^G(M)}$ on the right-hand side as a space of functions $T^\EndoE(\tilde{M})/W^G(M) \to \C$, using Remark \ref{rem:PW-st}.

\begin{lemma}\label{prop:IEndo-grading}
  The linear map
  \begin{align*}
    \orbI^\EndoE(\tilde{G}) & \longrightarrow \orbI^\EndoE_\gr(\tilde{G}) \\
    f^\EndoE = \left( f^{G^\Endo} \right)_{G^\Endo \in \EndoE_\mathrm{ell}(\tilde{G})} & \longmapsto \left( f^{M^\Endo}_\mathrm{ell} \right)_{\substack{M \in \mathcal{L}(M_0)/W^G_0 \\ M^\Endo \in \EndoE_\mathrm{ell}(\tilde{M})}}
  \end{align*}
  is an isomorphism. Here $f^{M^\Endo}_\mathrm{ell}$ denotes the restriction of $f^{M^\Endo} := \left(f^{G[s]}\right)^{s, M^\Endo} \in S\orbI(M^\Endo)$ to $\Phi_{2, \mathrm{bdd}}(M^\Endo)$, with arbitrary $s \in \EndoE_{M^\Endo}(\tilde{G})$ (recall Theorem \ref{prop:transfer-parabolic}); each family $\left( f^{M^\Endo}_\mathrm{ell} \right)_{M^\Endo \in \EndoE_\mathrm{ell}(\tilde{M})}$ is regarded as a function $T^\EndoE_{\mathrm{ell}}(\tilde{M})/W^G(M) \to \C$.
\end{lemma}

\begin{proof}
  The map is well-defined: $f^{M^\Endo}_\text{ell}$ is independent of $s$ and invariant under $W^G(M)$ by Definition \ref{def:Endo-I}. By interpreting $\orbI^\EndoE(\tilde{G})$ and $\orbI^\EndoE_\gr(\tilde{G})$ in terms of stable Paley-Wiener spaces for various $M^\Endo$, the isomorphy follows immediately from the definition of $\orbI^\EndoE(\tilde{G})$.  
\end{proof}

\begin{definition}\label{def:f^EndoE} \index{$f^\EndoE(\phi)$}
  Given $\phi \in T^\EndoE(\tilde{G})$ coming from $\phi_{M^\Endo} \in \Phi_{2,\text{bdd}}(M^\Endo)$ (up to $W^G(M)$), we set
  \begin{gather*}
    f^\EndoE(\phi) := f^{M^\Endo}_\mathrm{ell}(\phi_{M^\Endo}), \quad f^\EndoE \in \orbI^\EndoE(\tilde{G}),
  \end{gather*}
  in the previous notations; see \eqref{eqn:a-phi} for the meaning of evaluation at $\phi_{M^\Endo}$.
\end{definition}

\begin{remark}\label{rem:f^EndoE}
  To decipher $f^\EndoE(\phi)$, we choose $s \in \EndoE_{M^\Endo}(\tilde{G})$ to obtain a diagram as \eqref{eqn:endo-complete}. For $s=(I',I'')$, the corresponding endoscopic group $G[s]$ of $\tilde{G}$ is
  $$ \begin{tikzcd}
    G^\Endo = G[s] \arrow[-,double equal sign distance]{r} & \SO(2n'+1) & \times & \SO(2n''+1) \\
    M^\Endo \arrow[-, double equal sign distance]{r} \arrow[hookrightarrow]{u} & \prod_{i \in I'} \GL(n_i) \times \SO(2m'+1) \arrow[hookrightarrow]{u} & \times & \prod_{i \in I''} \GL(n_i) \times \SO(2m''+1) \arrow[hookrightarrow]{u}
  \end{tikzcd} $$
  and we let $\phi^\Endo \in \Phi_\text{bdd}(G^\Endo)$ be the image of $\phi_{M^\Endo}$. The representations in the $L$-packet $\Pi^{M^\Endo}_{\phi_{M^\Endo}}$ share a common $\prod_{i \in I''} \GL(n_i)$-component; denote its central character by $\bomega''$. Then we deduce that
  \begin{align*}
    f^\EndoE(\phi) & = \left( f^{G^\Endo} \right)^{s, M^\Endo}(\phi_{M^\Endo}) \\
    & = \bomega''(-1) \left( f^{G^\Endo} \right)^{M^\Endo}(\phi_{M^\Endo}) \\
    & = \bomega''(-1) (f^{G^\Endo})(\phi^\Endo)
  \end{align*}
  from Theorem \ref{prop:LLC-vs-induction} and the definition of $z[s]$.
\end{remark}

Assume hereafter that $F$ is non-archimedean.
\begin{definition}[Spectral transfer factors and their adjoint] \index{$\Delta(\phi,\tau)$}\index{$\Delta(\tau,\phi)$}
  For every $(\phi, \tau) \in T^\EndoE_\mathrm{ell}(\tilde{G}) \times T_{\mathrm{ell},-}(\tilde{G})$, define $\Delta(\phi, \tau)$ by requiring that
  \begin{align*}
    \mathcal{T}^\EndoE(f_{\tilde{G}})(\phi) & = \sum_{\tau \in T_{\mathrm{ell},-}(\tilde{G})/\Sph^1} \Delta(\phi, \tau) f_{\tilde{G}}(\tau), \\
    \Delta(\phi, z\tau) & = z\Delta(\phi, \tau), \quad z \in \Sph^1;
  \end{align*}
  where $f_{\tilde{G}} \in \Iasp(\tilde{G})$ and $\mathcal{T}^\EndoE(f_{\tilde{G}}) \in \Icusp^\EndoE(\tilde{G})$ is its collective transfer (Proposition \ref{prop:image-in-EndoE}). Its adjoint is defined by
  $$ \Delta(\tau, \phi) := \iota(\tilde{G}, G^\Endo) |\mathscr{S}_{\phi,\ad}|^{-1} |d(\tau)| \overline{\Delta(\phi, \tau)} $$
  whenever $\phi \in \Phi_{2,\text{bdd}}(G^\Endo)$. These definitions generalizes to any $M \in \mathcal{L}(M_0)$ in place of $G$, the corresponding factors are denoted by $\Delta_{\tilde{M}}(\cdots)$. Keep in mind that $\Delta_{\tilde{M}}(\phi_M, \tau_M)$ and $\Delta_{\tilde{M}}(\tau_M, \phi_M)$ vanish unless the $\GL$-components of $\phi_M$ and $\tau_M$ match under local Langlands correspondence.
  
  More generally, let $(\phi, \tau) \in T^\EndoE(\tilde{G}) \times T_-(\tilde{G})$. There exists a unique $M \in \mathcal{L}(M_0)/W^G_0$ such that $\phi$ (resp. $\tau$) comes from $\phi_M \in T^\EndoE_\text{ell}(\tilde{M})$ (resp. $\tau_M \in T_{\text{ell},-}(\tilde{M})$); both are unique up to $W^G(M)$. If $\phi$ and $\tau$ come from the same $M$ modulo $W^G_0$, set
  \begin{align*}
    \Delta(\phi, \tau) & := \sum_{\tau^\dagger_M \in W^G(M)\tau_M} \Delta_{\tilde{M}}(\phi_M, \tau^\dagger_M), \\
    \Delta(\tau, \phi) & := \sum_{\phi^\dagger_M \in W^G(M)\phi_M} \Delta_{\tilde{M}}(\tau_M, \phi^\dagger_M);
  \end{align*}
  otherwise set $\Delta(\phi,\tau)=0$. Note that there is at most one nonzero term in each sum above, namely the $\tau^\dagger_M$ (resp. $\phi^\dagger_M$) whose $\GL$-components match those of $\phi_M$ (resp. $\tau_M$) under the local Langlands correspondence.
\end{definition}

\begin{lemma}\label{prop:spectral-Delta-supp}
  For $\tau$ fixed, $\Delta(\cdot,\tau)$ and $\Delta(\tau, \cdot)$ are functions of finite support. So are $\Delta(\cdot, \phi)$ and $\Delta(\phi, \cdot)$ for $\phi$ fixed.
\end{lemma}
\begin{proof}
  One reduces immediately to the elliptic case. In the case of fixed $\tau$, take $f_{\tilde{G}} \in \Iaspcusp(\tilde{G})$ to be the pseudo-coefficient of $\tau$ in the definition of $\Delta(\cdot, \tau)$ and apply the Paley-Wiener theorems to $\Icusp^\EndoE(\tilde{G}) = \bigoplus_{G^\Endo} S\Icusp(G^\Endo)$. As to the case of fixed $\phi$, we may reverse the rôles of $\Iaspcusp(\tilde{G})$ and $\Icusp^\EndoE(\tilde{G})$ by using Theorem \ref{prop:transfer-cusp-surj}.
\end{proof}

For $\phi, \phi_1 \in T^\EndoE(\tilde{G})$, denote by $\bdelta_{\phi, \phi_1}$ the usual Kronecker's delta. For $\tau, \tau_1 \in T_-(\tilde{G})$, we set
$$
  \bdelta_{\tau, \tau_1} = \begin{cases}
    z, & \text{if } \tau_1 = z\tau, \; z \in \Sph^1, \\
    0, & \text{otherwise}.
  \end{cases}
$$

\begin{lemma}\label{prop:spectral-inversion}
  We have
  \begin{align*}
    \sum_{\phi \in T^\EndoE(\tilde{G})} \Delta(\tau, \phi) \Delta(\phi, \tau_1) & = \bdelta_{\tau, \tau_1}, \\
    \sum_{\tau \in T_-(\tilde{G})/\Sph^1} \Delta(\phi, \tau) \Delta(\tau, \phi_1) & = \bdelta_{\phi, \phi_1}.
  \end{align*}
\end{lemma}
\begin{proof}
  To begin with, suppose that $(\phi, \tau) \in T^\EndoE_\mathrm{ell}(\tilde{G}) \times T_{\mathrm{ell},-}(\tilde{G})$. Recall that $\mathcal{T}^\EndoE: \Iaspcusp(\tilde{G}) \rightiso \Icusp^\EndoE(\tilde{G})$ is an isometry by Corollary \ref{prop:isometry}. The relevant hermitian pairings can be interpreted via the Lemmas \ref{prop:ell-ip} and \ref{prop:st-ell-ip}. Choose for each $\tau \in T_{\mathrm{ell},-}(\tilde{G})/\Sph^1$ a representative in $T_{\mathrm{ell},-}(\tilde{G})$. By the description of Paley-Wiener spaces, $T_{\mathrm{ell},-}(\tilde{G})/\Sph^1$ and $T^\EndoE_\mathrm{ell}(\tilde{G})$ provide orthogonal bases for the two sides, and $(\Delta(\phi, \tau))_{\phi, \tau}$ is the matrix of $\mathcal{T}^\EndoE$. The elliptic case then follows from linear algebra. We derive the non-elliptic cases in the standard fashion, cf.\ the proof of Lemma \ref{prop:geometric-inversion}.
\end{proof}

\begin{lemma}\label{prop:change-variables-spectral}
  For all $\alpha \in C_c(T_{\mathrm{ell},-}(\tilde{G}))$ and $\beta \in C_c(T^\EndoE_{\mathrm{ell}}(\tilde{G}))$ such that $\alpha(z\cdot)=z^{-1}\alpha(\cdot)$ for all $z \in \Sph^1$, we have
  \begin{gather*}
    \int_{T_{\mathrm{ell},-}(\tilde{G})/\Sph^1} \;\sum_{\phi \in T^\EndoE_{\mathrm{ell}}(\tilde{G})} \beta(\phi) \Delta(\phi, \tau) \alpha(\tau) \dd\tau
    = \int_{T^\EndoE_{\mathrm{ell}}(\tilde{G})} \;\sum_{\tau \in T_{\mathrm{ell},-}(\tilde{G})/\Sph^1} \beta(\phi) \Delta(\phi, \tau) \alpha(\tau) \dd\phi,
  \end{gather*}
  where we pick an arbitrary representative in $T_{\mathrm{ell},-}(\tilde{G})$ for every $\tau \in T_{\mathrm{ell},-}(\tilde{G})/\Sph^1$. Note that the inner sums are finite by Lemma \ref{prop:spectral-Delta-supp}.
\end{lemma}
\begin{proof}
  The argument is similar to that of Lemma \ref{prop:change-variables} but much easier, since the integrals are actually sums in ours case. It suffices to re-index those sums.
\end{proof}

Finally, we note that the spectral transfer factor $\Delta$ has a natural extension to $T^\EndoE(\tilde{G})_\C \times T_-(\tilde{G})_\C$: extend each $\Delta_{\tilde{M}}$ by requiring that
\begin{inparaenum}[(i)]
  \item $\Delta_{\tilde{M}}(\phi_M, \tau_M)=0$ unless the $\GL$-components of $\phi_M$ and $\tau_M$ match, and
  \item $\Delta_{\tilde{M}}(\phi_{M,\lambda}, \tau_{M,\lambda}) = \Delta_{\tilde{M}}(\phi_M, \tau_M)$ for all $\lambda \in \mathfrak{a}^*_{M,\C}$.
\end{inparaenum}
Same for the adjoint transfer factor.

\subsection{Statement of the character relations}\label{sec:statement}
Retain the previous notations and assume $F$ to be non-archimedean. The main result of this section may be stated in an abstract form as follows.

\begin{theorem}\label{prop:character-relation}
  Define the isomorphism $\orbI_\gr(\tilde{G}) \rightiso \orbI^\EndoE_\gr(\tilde{G})$ as the direct sum over $M \in \mathcal{L}(M_0)/W^G_0$ of the cuspidal transfer maps (see Theorem \ref{prop:transfer-cusp-surj})
  $$ \mathcal{T}^\EndoE_{\tilde{M}}: \Iaspcusp(\tilde{M})^{W^G(M)} \rightiso \Icusp^\EndoE(\tilde{M})^{W^G(M)}. $$
  Then the diagram
  \begin{equation}\label{eqn:character-relation-diagram} \begin{tikzcd}
    \Iasp(\tilde{G}) \arrow{d}[right]{\simeq}[left]{\text{Corollary } \ref{prop:Iasp-grading}} \arrow{r}[above]{\mathcal{T}^\EndoE} & \orbI^\EndoE(\tilde{G}) \arrow{d}[left]{\simeq}[right]{\text{Lemma } \ref{prop:IEndo-grading}} \\
    \orbI_\gr(\tilde{G}) \arrow{r}[above]{\simeq} & \orbI^\EndoE_\gr(\tilde{G})
  \end{tikzcd} \end{equation}
  commutes.
\end{theorem}
The vertical isomorphisms interpret a test function as an element in the relevant Paley-Wiener space, i.e.\ as a function of spectral parameters. Thus the commutative diagram may be seen as an identification between the geometric and spectral transfers.

\begin{remark}\label{rem:character-relation-equiv}
  For $f_{\tilde{G}} \in \Iasp(\tilde{G}))$ we define \index{$f^\EndoE_\gr$}
  $$ f^\EndoE_\gr: \phi \longmapsto \sum_{\tau \in T_-(\tilde{G})/\Sph^1} \Delta(\phi, \tau) f_{\tilde{G}}(\tau) $$
  as a function on $T^\EndoE(\tilde{G})$. Here we pick an arbitrary representative in $T_-(\tilde{G})$ for each $\tau \in T_-(\tilde{G})/\Sph^1$. Take $f^\EndoE := \mathcal{T}^\EndoE(f_{\tilde{G}})$. We contend that Theorem \ref{prop:character-relation} is equivalent to the assertion that
  \begin{equation}\label{eqn:character-relation-equiv}
    f^\EndoE(\phi) = f^\EndoE_\gr(\phi)
  \end{equation}
  for all $f_{\tilde{G}} \in \Iasp(\tilde{G})$ and $\phi \in T^\EndoE(\tilde{G})$, where $f^\EndoE(\phi)$ is as in Definition \ref{def:f^EndoE}. In view of Remark \ref{rem:f^EndoE}, this is exactly \cite[Theorem 6.2]{Ar96} except for the twist by $\bomega''(-1)$, which is a metaplectic feature.
  
  Let us show the equivalence. In the commutative diagram of Theorem \ref{prop:character-relation}, going in the direction \begin{tikzpicture}[scale=0.5] \draw[->] (0,0.7) -- (0,0) -- (1,0); \end{tikzpicture} maps $f_{\tilde{G}}$ first to the function $\tau \mapsto f_{\tilde{G}}(\tau)$ on $T_-(\tilde{G})$ (say via $f_{\tilde{G}} \mapsto f_{\tilde{G}, \gr}$), then to the function $\phi \mapsto \sum_\tau \Delta(\phi, \tau) f_{\tilde{G}}(\tau)$ on $T^\EndoE(\tilde{G})$. We arrive at $f^\EndoE_\gr(\phi)$.

  On the other hand, by Lemma \ref{prop:IEndo-grading} and the notations therein, going in the direction \begin{tikzpicture}[scale=0.5] \draw[->] (0,0.7) -- (1,0.7) -- (1,0); \end{tikzpicture} maps $f_{\tilde{G}}$ to the function on $T^\EndoE(\tilde{G})$ that sends $\phi_{M^\Endo} \in \Phi_{2,\text{bdd}}(M^\Endo)$ to $f^{M^\Endo}(\phi_{M^\Endo})$, which is exactly $f^\EndoE(\phi)$. This concludes the equivalence since every $\phi \in T^\EndoE(\tilde{G})$ arises from some $\phi_{M^\Endo}$.
\end{remark}
 
The proof of Theorem \ref{prop:character-relation}, or its equivalent form \eqref{eqn:character-relation-equiv} will occupy \S\ref{sec:proof}. The upshot will be proving \eqref{eqn:character-relation-equiv} for $\phi$ elliptic and $f_{\tilde{G}}$ non-cuspidal. We record several consequences thereof.

\begin{corollary}\label{prop:transfer-surj}
  The collective transfer map $\mathcal{T}^\EndoE: \Iasp(\tilde{G}) \to \orbI^\EndoE(\tilde{G})$ is an isomorphism.
\end{corollary}

We will also need the space $\Iasp^1(\tilde{G})$ defined as follows. Its elements are functions on $\Gamma_\text{reg}(\tilde{G})$ of the form $f^{\tilde{G}}: \tilde{\delta} \mapsto \sum_{\tilde{\delta}_1} f_{\tilde{G}}(\tilde{\delta}_1)$, where $f_{\tilde{G}} \in \Iasp(\tilde{G})$ and $\tilde{\delta}_1 \in \Gamma_\text{reg}(\tilde{G})$ ranges over the classes stably conjugate to $\tilde{\delta}$ (see Definition \ref{def:stable-conj}). Thus we have a surjection $\Iasp(\tilde{G}) \twoheadrightarrow \Iasp^1(\tilde{G})$.\index{$\Iasp^1(\tilde{G})$}

For archimedean $F$, the space $\Iasp^1(\tilde{G})$ can be made into a nuclear LF space: see \cite[\S 5]{Re98}.

\section{The archimedean case}\label{sec:archimedean}
Throughout this section, we assume $F=\R$ except in \S\ref{sec:complex} where $F=\C$. Fix a non-trivial additive character $\psi: F \to \Sph^1$.

\subsection{Renard's formalism}\label{sec:Renard}
Let $(W, \angles{\cdot|\cdot})$ be a symplectic $\R$-vector space of dimension $2n$. Let $G := \Sp(W)$. The metaplectic covering $\rev: \tilde{G} = \Mp(W) \twoheadrightarrow G(\R)$ with $\Ker(\rev) = \bmu_8$ is defined with respect to $\psi$.

Let $(n', n'') \in \EndoE_\text{ell}(\tilde{G})$. Following Renard \cite{Re99}, we introduce the groups
\begin{align*}
  G^\diamond & := \Sp(W') \times \Sp(W''), \\
 \tilde{G}^\diamond & := \Mp(W') \times \Mp(W''),
\end{align*}
where $W'$, $W''$ are symplectic $F$-vector spaces such that
$$ \dim_F W' = 2n', \quad \dim_F W'' = 2n'', $$
and $W = W' \oplus W''$. Thus $\tilde{G}^\diamond$ comes equipped with a homomorphism $j: \tilde{G}^\diamond \to \tilde{G}$ with $\Ker(j) = \{(z, z^{-1}) : z \in \bmu_8 \}$. Note that $\tilde{G}^\diamond$ is not a covering group in our sense (the kernel of $\tilde{G}^\diamond \to G^\diamond(\R)$ is not cyclic), but the relevant properties carry over by working with each component separately. Renard actually considered the group $\tilde{G}^\diamond/\Ker(j)$ instead, cf. \cite[\S 4.6]{Li11}.

In Definition \ref{def:-1}, for any $m \in \Z_{\geq 0}$ we have defined a canonical element $-1 \in \Mp(2m)$ above $-1 \in \Sp(2m)$ satisfying $(-1)^2 = 1$. Write $-\tilde{x} = (-1) \cdot \tilde{x}$ for any $\tilde{x} \in \Mp(2m)$. Define the involution
\begin{align*}
  \tau: \tilde{G}^\diamond & \longrightarrow \tilde{G}^\diamond, \\
  (\tilde{x}', \tilde{x}'') & \longmapsto (\tilde{x}', -\tilde{x}'').
\end{align*}

\begin{definition}
  A distribution or function on $\tilde{G}^\diamond$ is called genuine if it is genuine for both components $\Mp(2n')$ and $\Mp(2n'')$. The same for \emph{anti-genuine} distributions or functions. Thus we shall continue to use the usual notations $C^\infty_{c, \asp}(\tilde{G}^\diamond)$, $\Pi_-(\tilde{G}^\diamond)$, etc.
\end{definition}

The normalized stable orbital integral of $f \in C^\infty_{c, \asp}(\tilde{G})$ along $\tilde{\delta} \in \tilde{G}_\text{reg}$ is defined at the end of \S\ref{sec:statement}. For the group $\tilde{G}^\diamond$, the normalized orbital integral $f_{\tilde{G}^\diamond}(\tilde{\delta})$ and its stable version $f^{\tilde{G}^\diamond}(\tilde{\delta})$ are still defined: it suffices to work component-wise. Following Harish-Chandra, Shelstad and Bouaziz, Renard defined in \cite[\S 3]{Re99} the space $\Iasp^1(\tilde{G}^\diamond)$ of stable anti-genuine orbital integrals on $\tilde{G}^\diamond$. It is an LF space, viewed as a space of anti-genuine functions $\tilde{G}^\diamond_\text{reg} \to \C$. As in \S\ref{sec:space-orbital-integral}, there is a continuous linear surjection
\begin{align*}
  C^\infty_{c, \asp}(\tilde{G}^\diamond) & \longrightarrow \mathcal{I}^1_{\asp}(\tilde{G}^\diamond) \\
  f & \longmapsto \left[ \tilde{\delta} \mapsto f^{\tilde{G}^\diamond}(\tilde{\delta}) \right].
\end{align*}

Define now
$$ \mathcal{I}^\kappa_{\asp}(\tilde{G}^\diamond) := \left\{ \phi^\diamond: \tilde{G}^\diamond_\text{reg} \to \C, \; \text{ such that } \tau^* \phi^\diamond := \phi^\diamond \circ \tau \in \mathcal{I}^1_{\asp}(\tilde{G}^\diamond) \right\}. $$
The superscript $\kappa$ might suggest the endoscopic character of \S\ref{sec:transfer-FL}, although the latter object only makes sense after fixing a maximal torus; see the discussion in \cite[p.1220]{Re99}.

The geometric transfer over $\R$ is decomposed into several stages.

\begin{enumerate}
  \item The factor $\Delta_0(\delta', \delta'')$ in \S\ref{sec:transfer-FL} defines the transfer map à la Renard \index{$\mathcal{T}_R$}
    \begin{align*}
      \mathcal{T}_R: \Iasp(\tilde{G}) & \longrightarrow \mathcal{I}^\kappa_{\asp}(\tilde{G}^\diamond) \\
      \phi & \longmapsto \left[ \phi^\diamond: (\tilde{\delta}', \tilde{\delta}'') \mapsto \Delta_0(\delta', \delta'') \sum_{\tilde{\delta}_1} \angles{\kappa, \inv(\delta, \delta_1)} \phi(\tilde{\delta}) \right]
    \end{align*}
    where
    \begin{compactitem}
      \item $\tilde{\delta} := j(\tilde{\delta}', \tilde{\delta}'')$ is assumed to lie in $\tilde{G}_\text{reg}$,
      \item $\tilde{\delta}_1$ ranges over the elements in $\Gamma_\text{reg}(\tilde{G})$ stably conjugate to $\tilde{\delta}$, which have representatives in $\tilde{G}^\diamond$ by the previous assumption,
      \item $\kappa = \kappa_{G_\delta}: H^1(\R, G_\delta) \to \bmu_2$ is the endoscopic character attached to $(n', n'')$ and $G_\delta$.
    \end{compactitem}
    This transfer is established in \cite[Theorem 4.7]{Re99} and rephrased in \cite[Théorème 6.3]{Li11}. When $n'=n$ it is merely the surjection $\Iasp(\tilde{G}) \twoheadrightarrow \mathcal{I}^1_{\asp}(\tilde{G})$.
  \item There is a transfer map of Adams-Renard
    \begin{align*}
      \mathcal{T}_{(n, 0)}: \Iasp(\tilde{G}) & \longrightarrow S\orbI(\SO(2n+1)) \\
      \phi & \longmapsto \left[ \gamma \mapsto \Delta_{(n,0)}(\gamma, \tilde{\delta}) \phi(\tilde{\delta}) \right]
    \end{align*}
    where $\delta \leftrightarrow \gamma$ and $\tilde{\delta} \in \rev^{-1}(\delta)$ is arbitrary; the $\Delta_{(n,0)}$ means the transfer factor for the endoscopic datum $(n,0)$ for $\tilde{G}$. It factors through $\Iasp(\tilde{G}) \twoheadrightarrow \mathcal{I}^1_{\asp}(\tilde{G})$. This transfer $\mathcal{T}_{(n,0)}$ is originally conjectured in \cite{Ad98} and proved in \cite{Re98}.
  \item The transfer defined in \S\ref{sec:transfer-FL} is
    $$ \mathcal{T}_{(n', n'')}: \Iasp(\tilde{G}) \longrightarrow S\orbI(G^\Endo). $$
    The ambiguity of $f^\Endo$ (Remark \ref{rem:ambiguity-transfer}) disappeared since we work with the spaces $\Iasp(\tilde{G})$, $\mathcal{I}^1(G^\Endo)$ of orbital integrals. When $n'=0$ or $n''=0$, we are reduced to the previous two cases.
\end{enumerate}

\begin{lemma}
  The following commutative diagram commutes.
  $$ \begin{tikzcd}
    \Iasp(\tilde{G}) \arrow{r}[above]{\mathcal{T}_R} \arrow{rdd}[left, inner sep=1em]{\mathcal{T}_{(n', n'')}} & \mathcal{I}^\kappa_{\asp}(\tilde{G}^\diamond) \arrow{d}[right]{\tau^*}[left]{\simeq} & \\
    & \mathcal{I}^1_{\asp}(\tilde{G}^\diamond) \arrow{d}[right]{ \mathcal{T}_{(n', 0)} \hat{\otimes} \mathcal{T}_{(n'', 0)} } \arrow[-,double equal sign distance]{r} &  \mathcal{I}^1_{\asp}(\Mp(2n')) \hat{\otimes} \mathcal{I}^1_{\asp}(\Mp(2n'')) \\
    & S\orbI(G^\Endo) \arrow[-,double equal sign distance]{r} & S\orbI(\SO(2n'+1)) \hat{\otimes} S\orbI(\SO(2n''+1))
  \end{tikzcd} $$
\end{lemma}
\begin{proof}
  This is done in the proof of \cite[Théorème 6.8]{Li11}. For the horizontal equalities and the meaning of $\hat{\otimes}$, see Remark \ref{rem:completed-otimes}.
\end{proof}

Dualization yields the commutative diagram below.
\begin{equation}\label{eqn:spectral-trans-R} \begin{tikzcd}
  \Iasp(\tilde{G})^\vee & \mathcal{I}^\kappa_{\asp}(\tilde{G}^\diamond)^\vee \arrow{l}[above]{\mathcal{T}_R^\vee} \\
  & \mathcal{I}^1_{\asp}(\tilde{G}^\diamond)^\vee \arrow{u}[right]{\tau_*}[left]{\simeq} \\
  & S\orbI(G^\Endo)^\vee \arrow{u}[right]{ \mathcal{T}_{(n', 0)}^\vee \hat{\otimes} \mathcal{T}_{(n'', 0)}^\vee } \arrow{luu}[left, inner sep=1em]{\mathcal{T}_{(n', n'')}^\vee }
\end{tikzcd} \end{equation}

The same constructions also apply to any covering of metaplectic type $\rev: \tilde{M} \to M(\R)$ and its elliptic endoscopic data: the common $\GL$ factors of $\tilde{M}$ and $M^\Endo$ are unaffected.

\subsection{Stable Weyl groups}\label{sec:st-Weyl}
Our basic references for stable conjugacy are \cite{Sh79,Ad98}. As before, consider $\tilde{G} = \Mp(W)$ and fix a maximal $\R$-torus $T$ of $G$. Let $A$ be the split component of $T$ and set $M := Z_G(A)$. The \emph{stable Weyl group} of $T$ is defined to be
\begin{equation}\label{eqn:W_st}
  W_\text{st}(G, T) := W(G, T) W(M_\C, T_\C)
\end{equation}
as a subgroup of the absolute Weyl group $W(G_\C, T_\C)$. It fits into a natural short exact sequence
$$ 1 \to W(M, T) \to W(M_\C, T_\C) \rtimes W(G, T) \to W_\text{st}(G, T) \to 1. $$

\begin{lemma}\label{prop:W-st-vs-W-F}
  The subgroups $W_\mathrm{st}(G, T)$ and $W(G, T)(\R)$ of $W(G_\C, T_\C)$ coincide.
\end{lemma}
\begin{proof}
  It follows from \cite[Theorem 2.1]{Sh79} that
  \begin{gather}\label{eqn:W-st-alt}
    W_\text{st}(G, T) = \left\{w \in W(G_\C, T_\C) : \Ad(w)|_{T_\C} \text{ is defined over } \R \right\}.
  \end{gather}
  This is exactly the characterization of $W(G, T)(\R)$ in \eqref{eqn:W-intermediate}.
\end{proof}

As a matter of notations, we prefer $W_\text{st}(G,T)$ over $W(G,T)(\R)$ in order to adhere to \cite{Ad98,Re98}. The group $W_\text{st}(G, T)$ acts on $T(\R)$ by the adjoint action. The elements in $T_\text{reg}(\R)/W(G, T)$ can be viewed as conjugacy classes in $G_\text{reg}(\R)$ intersecting $T(\R)$. Moreover, two elements $\delta$, $\delta_1 \in T_\text{reg}(\R)$ are stably conjugate in $G$ if and only if $\delta_1 = w\delta w^{-1}$ for some $w \in W_\text{st}(G, T)$. This is how Shelstad defined stable conjugacy for real groups. The same construction applies to every connected reductive $\R$-group; for $\GL(n)$ we have $W_\text{st}(\GL(n), T_\GL) = W(\GL(n), T_\GL)$ for any maximal torus $T_\GL$.

Note that any $T \subset G$ takes the form
$$ T \simeq (\C^\times)^m \times (\Sph^1)^r \times (\R^\times)^s $$
as $\R$-tori, for a unique triple $(m, r, s)$ satisfying $2m + r + s = n$.

\begin{lemma}\label{prop:Wst-identification}
  We have canonical identifications
  \begin{gather*}
    W(G, T) \backslash W_\mathrm{st}(G, T) = \bmu_2^r = H^1(\R, T).
  \end{gather*}
  Under this bijection, an element $(t_1, \ldots, t_r) \in \bmu_2^r$ corresponds to the coset containing the $w \in W_\mathrm{st}(G, T)$ which acts on $\Sph^r$ as $(z_i)_{i=1}^r \mapsto (z_i^{t_i})_{i=1}^r$, and $w$ acts trivially on $(\C^\times)^m$ and $(\R^\times)^s$ 
\end{lemma}
\begin{proof}
  This is well-known and we reproduce the arguments below. Identify a given $w \in W_\text{st}(G, T)$ with some representative in $G(\C)$, and let $\bar{w}$ be its complex conjugate. By Lemma \ref{prop:W-st-vs-W-F} we have $w^{-1}\bar{w} \in T(\C)$. This defines an element in $Z^1(\R, T)$. Compare the descriptions of stable vs. ordinary conjugacy in $T_\text{reg}(\R)$ in two ways:
  \begin{compactenum}[(i)]
    \item via the action of stable Weyl groups as discussed above, and
    \item via the pointed set $\mathfrak{D}(T, G; F)$ discussed in \eqref{eqn:D}.
  \end{compactenum}
  We deduce a bijection of pointed sets
  $$ W(G, T) \backslash W_\text{st}(G, T) \xrightarrow{1:1} H^1(\R, T). $$
  A standard argument (see \cite[\S 5.3]{Li11}) gives $H^1(\R, T) = \bmu_2^r$; this can also be done on the Weyl group side as done in \cite[p.1218]{Re99}, together with the precise description as asserted.
\end{proof}

For stably conjugate classes $\delta, \delta_1 \in T_\text{reg}(F)/W(G,T)$, their ``relative position'' can be defined as the element in $W(G, T) \backslash W_\text{st}(G, T)$ with a representative $w$ satisfying $\delta_1 = w\delta w^{-1}$. This is compatible with our general recipe \eqref{eqn:inv} using the cohomological invariant $\inv(\delta_1, \delta) \in H^1(\R, T)$: modulo the identification of Lemma \ref{prop:Wst-identification},
$$ \inv(w\delta w^{-1}, \delta) = w \in H^1(\R, T). $$
Unwinding matters, this equality reduces to the definition of $\inv(w\delta, \delta)$.

\begin{theorem}[Adams {\cite[Lemma 3.3]{Ad98}}]
  The action of $W_\mathrm{st}(G, T)$ on $\mathfrak{t}(\R)$ lifts to an action on $\tilde{T}$ extending the adjoint action of $W(G,T)$. Two elements $\tilde{\delta}, \tilde{\delta}_1 \in \tilde{T}_\mathrm{reg}$ are stably conjugate if and only if $\tilde{\delta}_1 = w\tilde{\delta} w^{-1}$ for some $w \in W_\mathrm{st}(G, T)$.
\end{theorem}

Note that Adams considered the twofold covering $\MMp{2}(W)$ instead.

As the element $-1$ in Definition \ref{def:-1} is central of order two, one obtains the map
\begin{gather}\label{eqn:kappa_0}
 \kappa_{T,0}: W(G, T) \backslash W_\text{st}(G, T) = H^1(\R, T) \longrightarrow \bmu_2
\end{gather}
characterized by $\kappa_{T,0}(w)(-1) = w(-1)w^{-1} \in \tilde{T}$.

\begin{lemma}\label{prop:kappa_0}
  The map $\kappa_{T,0}$ is given by
  \begin{align*}
    H^1(\R, T) = \bmu_2^r & \longrightarrow \bmu_2, \\
    (t_i)_{i=1}^r & \longmapsto \prod_{i=1}^r t_i
  \end{align*}
  modulo the identifications in Lemma \ref{prop:Wst-identification}.
\end{lemma}
\begin{proof}
  Write
  $$ T = \underbrace{(\C^\times)^m \times (\R^\times)^s}_{:= T_0} \times \underbrace{(\Sph^1)^r}_{:= T_1} $$
  and set $n_0 := s + 2m$, $n_1 := r$. We may choose a compatible orthogonal decomposition $W = W_0 \oplus W_1$ (cf. the parametrization of maximal tori in \cite[\S 3]{Li11}) so that the homomorphism
  $$ \tilde{T}_0 \times \tilde{T}_1 \hookrightarrow \Mp(W_0) \times \Mp(W_1) \xrightarrow{j} \tilde{G} $$
  reviewed in \S\ref{sec:the-central-extension} sends $(-1, -1)$ to $-1 \in \tilde{G}$ (sorry for overloading the symbol $-1$...). One way to see this is to combine Proposition 4.25 and Corollaire 4.6 in \cite{Li11}.

  In view of Lemma \ref{prop:Wst-identification}, the inclusion maps induce bijections
  $$ \begin{tikzcd}
    W(G, T) \backslash W_\text{st}(G, T) \arrow[-,double equal sign distance]{d} & W(\Sp(W_1), T_1) \big\backslash W_\text{st}(\Sp(W_1), T_1) \arrow{l}[above]{\sim} \arrow[-,double equal sign distance]{d} \\
    H^1(\R, T) & H^1(\R, T_1) \arrow{l}[below]{\sim}
  \end{tikzcd} $$
  Under this identification we have
  $$ \underbrace{w(-1)w^{-1}}_{\in \tilde{G}} = j(-1, \underbrace{w(-1)w^{-1}}_{\in \Mp(W_1)}),  \quad w \in W_\text{st}(\Sp(W_1), T_1). $$

  By construction, $j$ satisfies $j(\noyau_0 \tilde{x}_0, \noyau_1 \tilde{x}_1) = \noyau_0 \noyau_1 \cdot j(\tilde{x}_0, \tilde{x}_1)$ for any $\noyau_0, \noyau_1 \in \bmu_8$. The problem reduces immediately to the elliptic case $T = T_1$.

  Suppose $w$ corresponds to $(t_i)_{i=1}^r \in \bmu_2^r$, then the transfer factor $\Delta_{(0, n)}$ satisfies
  $$ \Delta_{(0, n)}(\gamma, w\tilde{\delta}w^{-1}) = \prod_{i=1}^r t_i \cdot \Delta_{(0, n)}(\gamma, \tilde{\delta}) $$
  whenever $\gamma \leftrightarrow \delta$ are regular semisimple: indeed, this is just the cocycle condition in \S\ref{sec:transfer-FL}. On the other hand, for sufficiently regular $\tilde{\delta}$, there exists a unique corresponding stable class $\gamma$ in $\SO(2n+1)$; the descent of transfer factors \cite[\S 7.10]{Li11} implies that $\Delta_{(0,n)}(\gamma, \tilde{\delta})$ is constant for $\tilde{\delta} = (-1)\exp(X)$, where $X \in \mathfrak{g}_\text{reg}(\R)$ is close to $0$. The assertion follows from the genuineness of $\Delta_{(0,n)}$ by taking $X \in \mathfrak{t}_\text{reg}(\R)$, $X \to 0$.
\end{proof}

Fix an endoscopic datum $(n', n'') \in \EndoE_\text{ell}(\tilde{G})$, construct the corresponding objects $\tilde{G}^\diamond$, $G^\Endo$ and set \index{$\tilde{y}$}
\begin{gather}\label{eqn:y}
  \tilde{y} := j(1, -1) \in \tilde{G}.
\end{gather}
This is slightly different from the choice of $\tilde{y}$ in \cite{Re99}.

Let $T \subset G$ be a maximal torus containing $y := \rev(\tilde{y})$, so that $T \subset Z_G(y) = G^\diamond$. Therefore there is a canonical decomposition $T = T' \times T''$ with $T' \subset \Sp(W')$, $T'' \subset \Sp(W'')$. In \cite[p.1220]{Re99}, Renard defines the map
$$ \kappa_T: W(G, T) \backslash W_\text{st}(G, T) \longrightarrow \bmu_2 $$
characterized by
$$ w \tilde{y} w^{-1} = \kappa_T(w) \tilde{y}, \quad w \in W_\text{st}(G, T). $$

When $n''=n$, we revert to the map $\kappa_{T,0}$ defined in \eqref{eqn:kappa_0} since $\tilde{y} = -1$. The next task is to reconcile Renard's character with our endoscopic character. Denote by $\text{pr}'': T \to T''$ the natural projection.

\begin{lemma}\label{prop:kappa_T-identification}
  The map $\kappa_T$ equals the composition
  $$ W(G, T) \backslash W_\mathrm{st}(G, T) \rightiso H^1(\R, T) \xrightarrow{\mathrm{pr}''_*} H^1(\R, T'') \xrightarrow{\kappa_{T, 0}} \bmu_2, $$
  via Lemma \ref{prop:Wst-identification}. In particular, $\kappa_T$ equals the endoscopic character in \S\ref{sec:transfer-FL} denoted by the same symbol.
\end{lemma}
\begin{proof}
  Write $G' := \Sp(W')$, $G'' := \Sp(W'')$. Under the identifications in Lemma \ref{prop:Wst-identification} we have
  \begin{align*}
    W(G, T) \backslash W_\text{st}(G, T) & = H^1(\R, T) = H^1(\R, T') \times H^1(\R, T'') \\
    & = \left( W(G', T') \backslash W_\text{st}(G', T') \right) \times \left( W(G'', T'') \backslash W_\text{st}(G'', T'') \right).
  \end{align*}

  The homomorphism $j: \tilde{G}' \times \tilde{G}'' = \tilde{G}^\diamond \to \tilde{G}$ in \S\ref{sec:the-central-extension} is clearly equivariant with respect to the decomposition above. The first component of $\tilde{y}$ does not matter at all (it is $1$), hence we are reduced to the case $n''=n$, $\tilde{y}=-1$ treated in Lemma \ref{prop:kappa_0}.
\end{proof}

\subsection{Spectral transfer for $G$-regular parameters}\label{sec:spectral-transfer-R-reg}

\paragraph{Preparations}
Fix a maximal compact subgroup $K$ of $G(\R)$ and an anisotropic maximal torus $T \subset G$ such that $T(\R) \subset K$. We have
\begin{align*}
  W(G, T) = W(K, T) & = W(K_\C, T_\C)  \qquad (\text{standard fact, see \cite[Lemma 4.43]{KV95}}),  \\
  W(G_\C, T_\C) & = W_\text{st}(G, T) \qquad (\text{because } M =G \text{ in } \eqref{eqn:W_st}).
\end{align*}

By choosing some $\delta \in T_\text{reg}(\R)$, $\gamma \in G^\Endo_\text{reg}(\R)$ with $\gamma \leftrightarrow \delta$, we obtain a standard isomorphism
$$ \theta: T^\Endo \rightiso T $$
by Lemma \ref{prop:diagram}, where $T^\Endo := G^\Endo_\gamma$; it is unique up to $W_\text{st}(G, T) = W(G,T)(\R)$. The pull-back of $\theta^{-1}$ induces a canonical surjection
\begin{gather}\label{eqn:map-inf-char}
  (\mathfrak{t}^\Endo)^*_\C/W(G^\Endo_\C, T^\Endo_\C) \twoheadrightarrow \mathfrak{t}^*_\C/W(G_\C, T_\C)
\end{gather}
with finite fibers, which is an isomorphism when $n'=n$ or $n''=n$. Hence we deduce a correspondence between infinitesimal characters of representations of $G$ and $G^\Endo$. Write \index{infinitesimal character!correspondence}
$$ \lambda^\Endo \leftrightarrow \lambda $$
if $\lambda^\Endo \mapsto \lambda$ under \eqref{eqn:map-inf-char}. A precise description of $\theta$ is worked out in \cite[\S 7]{Ad98}.

\begin{proposition}\label{prop:preservation-infchar}\index{infinitesimal character!preservation}
  Consider the canonical map
  $$ \begin{tikzcd}[row sep=small]
    \mathfrak{Z}(\mathfrak{g}) \simeq (\Sym\;\mathfrak{t}_\C)^{W(G_\C, T_\C)} \arrow{r}[below, inner sep=1em]{\text{dual to } \eqref{eqn:map-inf-char}} & (\Sym\;\mathfrak{t}^\Endo_\C)^{W(G^\Endo_\C, T^\Endo_\C)} \simeq \mathfrak{Z}(\mathfrak{g}^\Endo) \\
    z \arrow[mapsto]{r} & z^\Endo
  \end{tikzcd} $$
  between centres of universal enveloping algebras. We have $\mathcal{T}_{(n',n'')}(zf_{\tilde{G}}) = z^\Endo \mathcal{T}_{(n',n'')}(f_{\tilde{G}})$ for all $f_{\tilde{G}} \in \Iasp(\tilde{G})$. Consequently,  if $\Lambda^\Endo$ is an invariant eigendistribution on $G^\Endo(\R)$ with eigencharacter $\lambda^\Endo \in \mathfrak{t}^*_\C/W(G^\Endo, T^\Endo)$, $\lambda^\Endo \mapsto \lambda$, then so is its transfer $\Lambda$ with eigencharacter $\lambda$.
\end{proposition}
\begin{proof}
  This is based on the differential equations satisfied by orbital integrals, due to Harish-Chandra. A proof can be found in \cite[2.8 Corollaire]{Wa14-1} in a much more complicated setting.
\end{proof}

\paragraph{The case $(n, 0)$}
In what follows we assume $n'=n$. Recall that we have chosen a Borel pair $(B, T_s)$ for $G$ in \S\ref{sec:endoscopy}.
\begin{definition}
  For every $a_1, \ldots, a_n \in \C$, set
  $$ [a_1, \ldots, a_n] \in X^*(T_\C) \otimes \C = \mathfrak{t}^*_\C $$
  to be the image of $(a_1, \ldots, a_n) \in X^*(T_s) \otimes \C$ under $T_\C \leftiso (T_s)_\C$, where \eqref{eqn:X-vs-Lie} is used. The same notation pertains to $T^\Endo$.
\end{definition}

The notations are tailored so that $\theta$ sends $[a_1, \ldots, a_n]$ to $[a_1, \ldots, a_n]$. Since $T(\R)$ is the maximal compact subgroup of $T(\C) \rightiso T_s(\C) = (\C^\times)^n$, we may identify $X^*(T_\C)$ with $\Hom_\text{cont}(T(\R), \Sph^1)$ by Weyl's unitarian trick. The same holds for $T^\Endo$ as well.

Once the Borel subgroups $B \supset T$, $B^\Endo \supset T^\Endo$ are chosen, one defines the half sums of positive roots $\rho$, $\rho^\Endo$ respectively, and transport them to $\mathfrak{t}^*_\C, (\mathfrak{t}^\Endo)^*_\C$ via the recipe above. The discussion in \S\ref{sec:L-parameters} for the discrete series and their limits also applies to $\tilde{G}$.

\begin{lemma}\label{prop:calculation-rho}
  For appropriately chosen Borel pairs $(B, T_s)$ and $(B^\Endo, T^\Endo_s)$, we have
  \begin{compactenum}[(i)]
    \item $\rho = [n, n-1, \ldots, 1]$;
    \item $\rho^\Endo = \left[ n-\frac{1}{2}, \ldots, \frac{1}{2} \right]$;
    \item the Weil representation $\omega_\psi$ admits an infinitesimal character, whose representative in $\mathfrak{t}^*_\C$ can be taken to be $\lambda_0 := \left[ n-\frac{1}{2}, \ldots, \frac{1}{2} \right]$;
    \item an irreducible discrete series representation or their limits $\pi$ of $\tilde{G}$ with infinitesimal character $\lambda$ is genuine only if $\lambda \in \lambda_0 + X^*(T_\C)$, it factors through $\rev: \tilde{G} \to G(\R)$ only if $\lambda \in \rho + X^*(T_\C)$; the direction ``if'' holds when one works with the twofold covering $\MMp{2}(2n)$.
  \end{compactenum}
\end{lemma}
\begin{proof}
  The first two assertions are standard calculations in $T_s, T^\Endo_s$. The infinitesimal character of $\lambda_0$ is recorded in \cite[pp.151--152]{Ad98}, hence the third assertion. The final assertion follows since $\MMp{2}(2n)$ is connected as a Lie group.
\end{proof}

Let $\rho$, $\rho^\Endo$ and $\lambda_0 \in \mathfrak{t}^*_\C$ be as in Lemma \ref{prop:calculation-rho}. Following Adams, we define the $L$-packet $\Pi^{\tilde{G}}_\lambda$ of genuine discrete series of $\tilde{G}$ with infinitesimal character $\lambda \in \lambda_0 + X^*(T_\C)$ modulo $W(G_\C, T_\C)$ as
\begin{gather}\label{eqn:Adams-packet-ds}
  \Pi^{\tilde{G}}_\lambda := \left\{ \pi(w\vec{\lambda}) : w \in W(G, T) \backslash W(G_\C, T_\C) \right\}.
\end{gather}
Here the Harish-Chandra parameter $\vec{\lambda} = (\lambda, B_\C)$ \index{$\vec{\lambda}$} is defined by choosing the $B_\C$ for which $\lambda$ is dominant. Similarly, by allowing certain singularities of $\lambda$, the $L$-packet $\Pi^{\tilde{G}}_\lambda$ of genuine limits of discrete series of $\tilde{G}$ is defined by the same recipe. The stable character attached to those $\Pi^{\tilde{G}}_\lambda$ is \index{$\Pi^{\tilde{G}}_\lambda$}
\begin{gather}\label{eqn:Adams-st-character}
  S\Theta^{\tilde{G}}_\lambda := \sum_{\pi \in \Pi^{\tilde{G}}_\lambda} \Theta^{\tilde{G}}_\pi.
\end{gather}

The counterparts $\Pi^{G^\Endo}_{\lambda^\Endo}$ and $S\Theta^{G^\Endo}_{\lambda^\Endo}$ for the endoscopic side have been reviewed in \S\ref{sec:L-parameters} and \S\ref{sec:stable-character-SO}.

\begin{theorem}[Adams {\cite[Proposition 11.3]{Ad98}}]\label{prop:Adams-lifting}
  If $\lambda^\Endo \leftrightarrow \lambda$, then the map $\mathcal{T}_{(n, 0)}^\vee$ sends $S\Theta^{G^\Endo}_{\lambda^\Endo}$ to $S\Theta^{\tilde{G}}_\lambda$. In particular, $S\Theta^{\tilde{G}}_\lambda \in \mathcal{I}^1_{\asp}(\tilde{G})^\vee$.
\end{theorem}

\begin{remark}
  The convention in \cite{Ad98, Re99} is slightly different. Adams multiplied $S\Theta^{\tilde{G}}_\lambda$ (resp. $S\Theta^{G^\Endo}_{\lambda^\Endo}$) by $(-1)^{q(G)}$ (resp. $(-1)^{q(G^\Endo)}$). Here, for every semisimple $\R$-group $H$ admitting an anisotropic maximal torus, we set
  $$ q(H) := \frac{1}{2} \dim H(\R)/K_H \quad \in \Z, $$
  where $K_H$ is any maximal compact subgroup. This is harmless since
  $$ q(\Sp(2n)) = q(\SO(2n+1)) = {n+1 \choose 2}.$$
  At present we do not need to consider all the inner forms of $\SO(2n+1)$ at once as in \cite[(11.1)]{Ad98}. 
\end{remark}

\paragraph{The case $(n', n'')$}
Consider an arbitrary elliptic endoscopic datum $(n', n'') \in \EndoE_\text{ell}(\tilde{G})$, and introduce the groups $\tilde{G}^\diamond$, $G^\Endo$, etc. We have
\begin{compactitem}
  \item a chosen maximal compact subgroup $K^\diamond = K' \times K''$ of $G^\diamond(\R)$,
  \item a chosen Borel pair $(B^\diamond, T^\diamond) := (B'_s \times B''_s, T'_s \times T''_s)$ of $G^\diamond$,
  \item the half-sum of positive roots $\rho^\diamond = (\rho', \rho'')$.
  \item an anisotropic maximal torus $T^\diamond = T' \times T''$ of $G^\diamond$ contained in $K^\diamond$.
\end{compactitem}

Let $T := j(T^\diamond) \subset G$ be the isomorphic image of $T^\diamond$, which is a maximal torus of $G$; we may also assume $B^\diamond \subset B$. To save notations we shall identify $\mathfrak{t}$ and $\mathfrak{t}^\diamond$. In view of Lemma \ref{prop:Wst-identification}, we get the corresponding endoscopic character
$$ \kappa_T: W(G, T) \backslash W(G_\C, T_\C) = H^1(\R, T) \to \bmu_2. $$

Introduce an anisotropic maximal torus $T^\Endo \subset G^\Endo$, etc., so that there is some standard isomorphism $\theta: T^\Endo \to T$ furnished by Lemma \ref{prop:diagram}. Recall that $\theta$ induces a map between infinitesimal characters of $\tilde{G}$ and $G^\Endo(\R)$.

\begin{definition}\index{infinitesimal character!$G$-regular}
  Let $\lambda^\Endo$ be the infinitesimal character of a discrete series representation of $G^\Endo$ (i.e.\ $\lambda^\Endo$ is regular). We say $\lambda^\Endo$ is $G$-regular if $\lambda^\Endo \mapsto \lambda$ for some regular infinitesimal character $\lambda$ of $\tilde{G}$.
\end{definition}

For $\lambda$ as above, $\lambda - \rho$ lifts to a genuine character $e^{\lambda - \rho}$ of $\tilde{T}$; note that $\tilde{T}$ is not connected, so $\lambda-\rho$ is actually exponentiated to the twofold connected covering, then extended to $\tilde{T}$ by genuineness. Same for $\lambda - \rho^\diamond$. Consider $\pi \in \Pi^{\tilde{G}}_\lambda$ with infinitesimal character $\lambda$, $\lambda^\Endo \mapsto \lambda$. We may choose a Harish-Chandra parameter for $\pi$ (the Borel subgroup datum omitted) of the form
\begin{equation}\label{eqn:choice-HC-diamond}\begin{gathered}
	\left[ \lambda'_1, \ldots, \lambda'_{n'}, \lambda''_1, \ldots. \lambda''_{n''} \right] \in \mathfrak{t}^*_\C, \\
	\lambda'_1 \geq \cdots \geq \lambda'_{n'}, \quad \lambda''_1 \geq \cdots \geq \lambda''_{n''}, \\
	\forall i,j, \; \lambda'_i, \lambda''_j \in \Z + \frac{1}{2}
\end{gathered}\end{equation}
such that the orbits of $[\lambda'_1, \ldots, \lambda'_{n'}]$ and $[\lambda''_1, \ldots, \lambda''_{n''}]$ under $W_\text{st}(G', T')$ and $W_\text{st}(G'', T'')$ correspond to the two components of $\lambda^\Endo$, respectively. This is done by modifying a given parameter for $\pi$ by using $W(G, T) \simeq \mathfrak{S}_n$, and such a choice is unique. In this manner we deduce from $\pi$ the genuine representation of $\tilde{G}^\diamond$ with Harish-Chandra parameter $([\lambda'_1, \ldots], [\lambda''_1, \ldots])$, which we shall denote by $\pi^\diamond$. Its central character is denoted by $\bomega_{\pi^\diamond}$\index{$\pi^\diamond$}, which can be evaluated at $\tilde{y}$. Note that the same construction generalizes to the case where $\pi$ is a limit of genuine discrete series.

\begin{theorem}[Renard \cite{Re99}]\label{prop:Li-lifting-real-reg}
  Let $\lambda^\Endo$ be a $G$-regular infinitesimal character of a discrete series representation of $G^\Endo$ and assume $\lambda^\Endo \mapsto \lambda$. Then
  $$ \mathcal{T}_{(n', n'')}^\vee \left( S\Theta^{G^\Endo}_{\lambda^\Endo} \right) = \sum_{\pi \in \Pi^{\tilde{G}}_\lambda} \bomega_{\pi^\diamond}(\tilde{y}) \Theta^{\tilde{G}}_\pi $$
  where $\tilde{y}$ is as in \eqref{eqn:y}. Furthermore, $\bomega_{\pi(w\vec{\lambda})^\diamond}(\tilde{y}) = \bomega_{\pi(\vec{\lambda})^\diamond}(\tilde{y}) \kappa_T(w)$ for any $w \in W(G_\C, T_\C)$ and $\pi(\vec{\lambda}) \in \Pi^{\tilde{G}}_\lambda$.
\end{theorem}
\begin{proof}
  Let $P_G$ denote the system of $B_\C$-positive roots for $(G_\C, T_\C)$. One defines $P_{G^\diamond}$ similarly. Choose a Harish-Chandra parameter $\vec{\lambda}$ as in \eqref{eqn:choice-HC-diamond} so that $\pi(\vec{\lambda}) \in \Pi^{\tilde{G}}_\lambda$. In \cite[(6.5)]{Re99}, Renard proved that $\mathcal{T}_R^\vee \circ \tau^* \left( S\Theta^{\tilde{G}^\diamond}_\lambda  \right)$ equals
  $$ (-1)^{|P_G| - |P_{G^\diamond}|} e^{\lambda - \rho^\diamond}(\tilde{y}) \sum_{w \in W(G, T) \backslash W(G_\C, T_\C)} \kappa_T(w) \Theta^{\tilde{G}}_{\pi(w\vec{\lambda})}. $$
  Here we can let $w$ range over $W(G^\diamond, T^\diamond) \backslash W(G^\diamond_\C, T^\diamond_\C)$, cf. the proof of Lemma \ref{prop:Wst-identification}.
 
  By Theorem \ref{prop:Adams-lifting} plus the commutative diagram \eqref{eqn:spectral-trans-R}, the assertion on $\mathcal{T}_{(n',n'')}^\vee$ amounts to the equality between the above sum and $\sum_w \bomega_{\pi(w\vec{\lambda})^\diamond}(\tilde{y}) \Theta^{\tilde{G}}_{\pi(w\vec{\lambda})}$.

  Since $|P_G| = n^2$ and $|P_{G^\diamond}| = (n')^2 + (n'')^2$,
  $$ |P_G| \equiv n \equiv n' + n'' \equiv |P_{G^\diamond}| \mod 2. $$
  Hence $(-1)^{|P_G| - |P_{G^\diamond}|} = 1$. On the other hand, the character $e^{\lambda - \rho^\diamond}$ restricts to $\bomega_{\pi(\vec{\lambda})^\diamond}$ on $Z_{\tilde{G}^\diamond}$ (see \S\ref{sec:L-parameters}). Thus the terms indexed by $w=1$ match. To deal with the other terms, observe that $\bomega_{\pi(w\vec{\lambda})^\diamond}$ equals the restriction of $e^{w(\lambda - \rho^\diamond)}$ to $Z_{\tilde{G}^\diamond}$, thus $\bomega_{\pi(w\vec{\lambda})^\diamond}(\tilde{y}) = \bomega_{\pi(\vec{\lambda})^\diamond}(\tilde{y}) \kappa_T(w)$ for all $w \in W(G^\diamond, T^\diamond) \backslash W(G^\diamond_\C, T^\diamond_\C)$.
\end{proof}

Note that when $n''=n$, the factor $\bomega_{\pi^\diamond}(\tilde{y}) = \bomega_\pi(-1)$ generalizes Waldspurger's central signs in the case $n=1$.

\subsection{Spectral transfer in general}\label{sec:spectral-transfer-R}
We will reformulate and generalize the prior results in the language of \S\ref{sec:statement}. To begin with, fix $(n',n'') \in \EndoE_\text{ell}(\tilde{G})$ and let $\phi$ be the $L$-parameter for the discrete series $L$-packet of $G^\Endo$ with infinitesimal character $\lambda^\Endo$; we do not assume $\lambda^\Endo$ to be $G$-regular. Note that $\lambda^\Endo$ may be put in the normal form $([b_1, \cdots, b_{n'}], [c_1, \ldots, c_{n''}]) \in (\mathfrak{t}^\Endo)^*_\C$ where $b_i$, $c_j$ are strictly decreasing sequences of positive half-integers. Cf.\ Lemma \ref{prop:calculation-rho}.

Any infinitesimal character $\lambda$ of a genuine limit of discrete series has a similar normal form $[a_1, \ldots, a_n] \in \mathfrak{t}^*_\C$ with $a_1 \geq \cdots \geq a_n > 0$: just take the $B_\C$-dominant representative. Each equality sign in this list entails a singularity. Those $\lambda$ obtained from $\lambda^\Endo$ by the recipe above, for various $(G^\Endo, \phi) \in T^\EndoE_\text{ell}(\tilde{G})$, are characterized by
\begin{equation}\label{eqn:nondegenerate-lambda}\begin{gathered}
  \lambda = [a_1, \ldots, a_n], \quad a_1 \geq \ldots \geq a_n > 0, \quad a_i \in \Z + \frac{1}{2}, \\
  \text{with allowed singularities of the form} \quad  \ldots > a_i = a_{i+1} > \ldots .
\end{gathered}\end{equation}
Indeed, $\lambda$ is obtained by merging two lists $b_1 > \cdots > b_{n'} > 0$, $c_1 > \cdots > c_{n''} > 0$ of half-integers; no $a_\bullet$ in the merged list can occur more then twice. We shall show in Remark \ref{rem:nondegenerate} that $\lambda$ is the infinitesimal character of some limit of discrete series.

\begin{definition}\label{def:2uparrow}\index{$\Pi_{2\uparrow,-}(\tilde{G})$}
  Denote by $\Pi_{2\uparrow,-}(\tilde{G})$ the subset of $\Pi_{\text{temp},-}(\tilde{G})$ consisting of limits of discrete series whose infinitesimal characters satisfy \eqref{eqn:nondegenerate-lambda}. 
\end{definition}
Since $\Pi_{2\uparrow,-}(\tilde{G})$ is defined in terms of infinitesimal characters, it is a union of $L$-packets of genuine limits of discrete series. For $\pi \in \Pi_{2\uparrow,-}(\tilde{G})$, the recipe $\pi \mapsto \pi^\diamond \in \Pi_{2,-}(\tilde{G}^\diamond)$ in Theorem \ref{prop:Li-lifting-real-reg} is still applicable, and we define the sign factor $\bomega_{\pi^\diamond}(\tilde{y})$ as before.

The following result is stated in \cite[p.1241]{Re99}.

\begin{theorem}\label{prop:Li-lifting-real}
  Fix $\phi$ and $\lambda^\Endo \mapsto \lambda$ as above. For $\pi \in \Pi_{\mathrm{temp},-}(\tilde{G})$, set $\Delta(\phi, \pi) = \bomega_{\pi^\diamond}(\tilde{y})$ if $\pi \in \Pi_{2\uparrow,-}(\tilde{G})$ has infinitesimal character $\lambda$, otherwise set $\Delta(\phi, \pi)=0$. We have
  $$ f^{G^\Endo}(\phi) = \sum_{\pi \in \Pi_{\mathrm{temp},-}(\tilde{G})} \Delta(\phi, \pi) f_{\tilde{G}}(\pi) $$
  for any $f_{\tilde{G}} \in \Iasp(\tilde{G})$ with $f^{G^\Endo} := \mathcal{T}_{(n',n'')}(f_{\tilde{G}})$. Furthermore, $\Delta(\phi, \pi(w\vec{\lambda})) = \kappa_T(w) \Delta(\phi, \pi(\vec{\lambda}))$ for any $w \in W(G_\C, T_\C)$.
\end{theorem}
\begin{proof}[Sketch of proof]
  For $G$-regular $\lambda^\Endo$ this is just Theorem \ref{prop:Li-lifting-real-reg}. To extend it to the present case, we resort to coherent continuation as in \cite{Sh82,Sh10}: fix a representative of $\lambda^\Endo$ (resp. $\lambda$) dominant in some Weyl chamber $\mathcal{C}^\Endo$ (resp. $\mathcal{C}$), then take $\mu^\Endo \in X^*(T^\diamond)$ deep enough in $\mathcal{C}^\Endo$ so that $\lambda^\Endo + \mu^\Endo$ is $G$-regular, say $\lambda^\Endo + \mu^\Endo \leftrightarrow \lambda + \mu \in \mathcal{C}$. The idea is to apply the translation functor $\psi_{\lambda+\mu}^\lambda$ to both sides of the assertion with $G$-regular input $\phi_{\lambda^\Endo+\mu^\Endo}$, with $\lambda^\Endo + \mu^\Endo \leftrightarrow \lambda + \mu$. Recall that $\psi_{\lambda+\mu}^\lambda$ produces limits $\pi(\vec{\lambda})$ from the genuine discrete series $\pi(\overrightarrow{\lambda+\mu})$; similarly for the $G^\Endo$ side and for the Weyl translates. By the diagram \eqref{eqn:spectral-trans-R}, the spectral transfer breaks into three stages.
  \begin{enumerate}[(i)]
    \item Adams' transfer $\mathcal{T}_{(n,0)}^\vee$ commutes with translation functors by \cite[Corollary 14.7]{Ad98}. Recall that the coefficients appearing in $\mathcal{T}_{(n,0)}^\vee$ are all $1$.
    \item Renard's transfer $\mathcal{T}_R^\vee$ also commutes with translation functors: this is briefly mentioned in \cite[(6.5)]{Re99} which refers to \cite[Lemma 4.4.8]{Sh82}, cf.\ \cite[p.38]{Sh10}. Loosely speaking, it means that a character relation for $\phi_{\lambda^\Endo + \mu^\Endo}$ still holds after shifting to $\phi = \phi_{\lambda^\Endo}$, with the same coefficients.
    \item The operation $\tau_*$ does not commute with $\psi_{\lambda+\mu}^\lambda$: we get an extra factor $e^\mu(\tilde{y})$.
  \end{enumerate}
  Accordingly, the coefficients in $f^{G^\Endo}(\phi)$ are obtained from $f^{G^\Endo}(\phi_{\lambda^\Endo + \mu^\Endo})$ by shifting from $\bomega_{\pi(\overrightarrow{\lambda+\mu})^\diamond}(\tilde{y})$ to $\bomega_{\pi(\vec{\lambda})^\diamond}(\tilde{y})$ as required (cf. the proof of Theorem \ref{prop:Li-lifting-real-reg}). The assertion on $\Delta(\phi, \pi(w\vec{\lambda}))$ can be proved in the same manner as in Theorem \ref{prop:Li-lifting-real-reg}.
\end{proof}

Generalization to coverings of metaplectic types is straightforward. Now we introduce the collective geometric transfer $f_{\tilde{G}} \mapsto f^\EndoE$ and regard $f^\EndoE$ as a function $T^\EndoE(\tilde{G}) \to \C$ by the recipe of Lemma \ref{prop:IEndo-grading}. Let $\phi \in T^\EndoE(\tilde{G})$, say coming from $\phi_{M^\Endo} \in \Phi_{2,\text{bdd}}(M^\Endo)$, where $M^\Endo$, $\tilde{M} = \prod_{i \in I} \GL(n_i) \times \Mp(W^\flat)$ and $\tilde{G}$ sit in a diagram \eqref{eqn:endo-incomplete}. Pick any $s \in \EndoE_{M^\Endo}(\tilde{G})$ and $P \in \mathcal{P}(M)$. By Theorems \ref{prop:Li-lifting-real}, \ref{prop:transfer-parabolic} and the very definition of $f^\EndoE(\phi)$, we get
\begin{equation}\label{eqn:spectral-factor-R-descent}\begin{aligned}
  f^\EndoE(\phi) & = \left( f^{G[s]} \right)^{s, M^\Endo}(\phi_{M^\Endo}) = \sum_{\pi_{\tilde{M}} \in \Pi_{2\uparrow,-}(\tilde{M})} \Delta_{\tilde{M}}(\phi_{M^\Endo}, \pi_{\tilde{M}}) f_{\tilde{M}}(\pi_{\tilde{M}}) \\
  & = \sum_{\pi_{\tilde{M}} \in \Pi_{2\uparrow,-}(\tilde{M})} \Delta_{\tilde{M}}(\phi_{M^\Endo}, \pi_{\tilde{M}}) f_{\tilde{G}}\left( I_{\tilde{P}}(\pi_{\tilde{M}}) \right)
\end{aligned}\end{equation}
for any $f_{\tilde{G}} \in \Iasp(\tilde{G})$, where $\Delta_{\tilde{M}}$ denotes the spectral transfer factor for $\tilde{M}$. Note that $\Delta_{\tilde{M}}(\phi_{M^\Endo}, \pi_{\tilde{M}}) \neq 0$ only if for each $i \in I$, the $\GL(n_i)$-components of $\phi_{M^\Endo}$ and $\pi_{\tilde{M}}$ match by local Langlands correspondence. From this we deduce the general statement of spectral transfer as follows.

\begin{theorem}\label{prop:character-relation-R}\index{$\Delta(\phi,\pi)$}
  There exists a function $\Delta: T^\EndoE(\tilde{G}) \times \Pi_{\mathrm{temp},-}(\tilde{G}) \to \bmu_2$ such that $\Delta(\phi, \cdot)$ (resp. $\Delta(\cdot, \pi)$) is of finite support for any given $\phi$ (resp. $\pi$), and is characterized by
  $$ f^\EndoE(\phi) = \sum_{\pi \in \Pi_{\mathrm{temp},-}(\tilde{G})} \Delta(\phi, \pi) f_{\tilde{G}}(\pi) $$
  for all $\phi \in T^\EndoE(\tilde{G})$. Suppose that $\phi$ comes from $\phi_{M^\Endo} \in \Phi_{2,\mathrm{bdd}}(G^\Endo)$ and $P \in \mathcal{P}(M)$ is arbitrary, then
  $$ \Delta(\phi, \pi) = \sum_{\pi_{\tilde{M}} \in \Pi_{2\uparrow,-}(\tilde{M})} \Delta_{\tilde{M}}(\phi_{M^\Endo}, \pi_{\tilde{M}}) \mult(I_{\tilde{P}}(\pi_{\tilde{M}}) : \pi). $$
\end{theorem}
\begin{remark}\label{rem:character-relation-R}
  In the formula for $\Delta(\phi, \pi)$, the sum is actually taken over a packet $\Pi^{\tilde{M}}_\lambda \subset \Pi_{2\uparrow,-}(\tilde{M})$ determined by $\phi_{M^\Endo}$. We contend that $\mult(I_{\tilde{P}}(\pi_{\tilde{M}}) : \pi) \leq 1$, with equality for at most one $\pi_{\tilde{M}}$. Indeed, as in the reductive case \cite[p.408]{Sh82}, $\Pi^{\tilde{M}}_\lambda$ turns out to be the set of irreducible constituents in the normalized parabolic induction of some genuine discrete series $L$-packet $\Pi^{^{\tilde{L}}}_\nu$ for some $L \subset M$. By Langlands' disjointness theorem and the theory minimal $\tilde{K}$-types \cite[Theorem 1.1]{Vo79}, which works for coverings of metaplectic type by \cite[\S 3.4]{Li14b}, the parabolic induction to $\tilde{G}$ of $\bigoplus_{\sigma \in \Pi^{\tilde{L}}_\nu} \sigma$ is multiplicity-free. Our claim follows immediately.
\end{remark}

\begin{theorem}[$K$-finite transfer, cf.\ {\cite[Appendice]{CD84}}]\label{prop:preservation-K-finiteness}
  Let $\tilde{K}$ (resp. $K^\Endo$) be a maximal compact subgroup of $\tilde{G}$ (resp. of $G^\Endo(\R)$). For any $f \in C^\infty_{c,\asp}(\tilde{G})$ which is $\tilde{K} \times \tilde{K}$-finite under bilateral translation, its transfer $f^\Endo \in C^\infty_c(G^\Endo(\R))$ can be taken to be $K^\Endo \times K^\Endo$-finite as well.
\end{theorem}
\begin{proof}
  The elements of $S\orbI(G^\Endo)$ arising from $K^\Endo \times K^\Endo$-finite $C^\infty_c$ functions have a characterization à la Paley-Wiener: see \cite[Théorème A.1]{CD84} or \cite[\S 2.9]{Wa13-4}. Upon some contemplation on Definition \ref{def:f^EndoE}, it suffices to show that the function $\phi \mapsto f^\EndoE(\phi)$ is supported on finitely many connected components of $T^\EndoE(\tilde{G})$.
  
 As remarked in \cite[\S 3.4]{Li14b}, the $\tilde{K}$-finite trace Paley-Wiener theorem of \cite{CD84} holds for $\tilde{G}$. Also, by \cite[Lemme A.5]{CD84} there are only finitely many genuine limits of discrete series of $\tilde{G}$ containing a given $\tilde{K}$-type. The required finiteness condition follows readily from the Theorems \ref{prop:Li-lifting-real} and \ref{prop:character-relation-R}.
\end{proof}

\subsection{Adjoint spectral transfer factors}\label{sec:supplements-limits}
To offer a partial justification of our theory, we shall establish inversion formulas à la \cite[Corollary 7.7]{Sh08} for $\pi \in \Pi_{2\uparrow,-}(\tilde{G})$ which are parallel to Lemma \ref{prop:spectral-inversion}. We adopt the previous conventions for Borel subgroups, etc., unless otherwise stated. Let $\lambda$ be an infinitesimal character on the $\tilde{G}$ side in its normal form \eqref{eqn:nondegenerate-lambda}. Divide the entries of $\lambda$ into
\begin{compactitem}
	\item the ``pairs'' of the form $(a_i, a_{i+1} = a_i)$, and
	\item the remaining ``singletons''. 
\end{compactitem}
The compact roots in question are of the form $e_j - e_k$, where $e_j$ stands for the $j$-th coordinate; the $B_\C$-simple ones are $\{ e_k - e_{k+1}: 1 \leq k < n \}$. To avoid singularities with respect to compact roots, we flip $a_{i+1}$ into $-a_{i+1}$ in each repetition, by applying some $v \in W(G_\C, T_\C)$.

\begin{remark}\label{rem:nondegenerate}
	The representation $\pi(v\vec{\lambda})$ is a genuine limit of discrete series of $\tilde{G}$. Furthermore, it is non-degenerate\index{limit of discrete series!non-degenerate} in the sense of Knapp-Zuckerman \cite[\S 12]{KZ82-2}: $\vec{\lambda}$ is non-singular with respect to compact roots. See \cite[\S 14]{Sh10} for discussions in the case of reductive groups as well as a description for the corresponding $L$-parameters.
\end{remark}

Let $\mathcal{W}_\lambda := \{w \in W(G_\C, T_\C) : \pi(wv\vec{\lambda}) \neq 0\}$. We have to identify
$$ \mathfrak{E} := W(G, T) \backslash \mathcal{W}_\lambda / \Stab(v\vec{\lambda}).$$
View $\bmu_2^n = W(G,T) \backslash W(G_\C, T_\C)$ as a subgroup of $W(G_\C, T_\C)$ so that $(t_i)_{i=1}^n \in \bmu_2^n$ acts via $[x_1, \ldots, x_n] \mapsto [t_1 x_1, \ldots, t_n x_n]$.

Let $t = (t_i)_{i=1}^n \in \bmu_2^n$. When do we have $t \in \mathcal{W}_\lambda$? Singularities can only occur within the pairs. Consider a pair in $v\vec{y}$, say $(a, -a)$ together with its accompanying Weyl chamber. Its orbit under $(\bmu_2)^2$ consists of $(a,-a)$ itself and
\begin{compactitem}
	\item $(a,a)$, $(-a,-a)$: singular with respect to a simple compact root (see below);
	\item $(-a,a)$: parametrizes a non-degenerate limit of discrete series. It is obtained by applying $(-1,-1) \in (\bmu_2)^2$.
\end{compactitem}
We must include the Weyl chambers in these parameters on which $W(G_\C, T_\C)$ acts; thus $(-a, a)$ with its accompanying chamber cannot be obtained from that of $(a,-a)$ via $W(G,T)$. The reader is invited to visualize the case $n=2$: we use the usual simple roots $e_1 - e_2, 2e_2$ for $B_\C$; the thick line below depicts the singular locus with respect to the compact roots $\pm(e_1 - e_2)$.
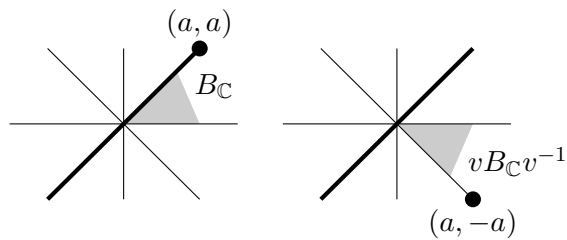
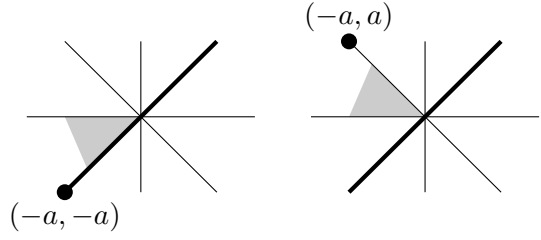
\begin{center}
	\begin{tikzpicture}[baseline]
	\fill[gray!40!white] (0,0) -- (1,0) -- (0.7, 0.7) -- (0,0);
	\node[right] at (0.8, 0.5) {$B_\C$};
	\fill (1,1) circle (3pt) node [above] {$(a,a)$};
	
	\draw (-1.5, 0) -- (1.5, 0);
	\draw[ultra thick] (-1, -1) -- (1,1);
	\draw (1, -1) -- (-1, 1);
	\draw (0, -1) -- (0, 1);
	\end{tikzpicture} \quad
	\begin{tikzpicture}[baseline]
	\fill[gray!40!white] (0,0) -- (1,0) -- (0.7, -0.7) -- (0,0);
	\node[right] at (0.8, -0.5) {$v B_\C v^{-1}$};
	\fill (1,-1) circle (3pt) node [below] {$(a,-a)$};
	
	\draw (-1.5, 0) -- (1.5, 0);
	\draw[ultra thick] (-1, -1) -- (1,1);
	\draw (1, -1) -- (-1, 1);
	\draw (0, -1) -- (0, 1);
	\end{tikzpicture} \\
	\begin{tikzpicture}[baseline]
	\fill[gray!40!white] (0,0) -- (-1, 0) -- (-0.7, -0.7) -- (0,0);
	\fill (-1,-1) circle (3pt) node [below] {$(-a,-a)$};
	
	\draw (-1.5, 0) -- (1.5, 0);
	\draw[ultra thick] (-1, -1) -- (1,1);
	\draw (1, -1) -- (-1, 1);
	\draw (0, -1) -- (0, 1);
	\end{tikzpicture} \quad
	\begin{tikzpicture}[baseline]
	\fill[gray!40!white] (0,0) -- (-1,0) -- (-0.7, 0.7) -- (0,0);
	\fill (-1,1) circle (3pt) node [above] {$(-a,a)$};
	
	\draw (-1.5, 0) -- (1.5, 0);
	\draw[ultra thick] (-1, -1) -- (1,1);
	\draw (1, -1) -- (-1, 1);
	\draw (0, -1) -- (0, 1);
	\end{tikzpicture}
\end{center}
Hence $\mathfrak{E} = \bmu_2^{|\text{singletons}|+|\text{pairs}|}$. Therefore we obtain an embedding
$$ \mathfrak{E} \xrightarrow{\identity \times \text{diag}} \bmu_2^{|\text{singletons}|} \times (\bmu_2 \times \bmu_2)^{|\text{pairs}|} = H^1(\R, T). $$
It also follows that $|\Pi^{\tilde{G}}_\lambda| = 2^{|\text{singletons}|+|\text{pairs}|}$ as expected. Denote by $\mathfrak{R} = \mathfrak{R}_\lambda$ the Pontryagin dual of $\mathfrak{E}$. We deduce a surjection $\mathfrak{R}(T, G; \R) \twoheadrightarrow \mathfrak{R}$.

\begin{lemma}\label{prop:spectral-bijection}
	Let $\lambda \in \mathfrak{t}^*_\C/W(G_\C, T_\C)$ be an infinitesimal character with a representative of the form \eqref{eqn:nondegenerate-lambda}.	There is a canonical bijection $\phi \mapsto \kappa$ from
	$$ \left\{ \phi \in T^\EndoE(\tilde{G}) : \text{transfers to a parameter with inf.\! char.}\; \lambda \right\} $$
	onto $\mathfrak{R}_\lambda$. It satisfies
	$$ \Delta(\phi, \pi_1) = \Delta(\phi, \pi) \kappa(\bar{w}) $$
	for all $\pi=\pi(v\vec{\lambda})$, $\pi_1 = \pi(wv\vec{\lambda})$ as discussed above, where $\bar{w} \in \mathfrak{E}$.
\end{lemma}
\begin{proof}
	Decompose $\phi$ into $(\phi_{\GL}, \phi', \phi'')$ and note that $\phi_{\GL}$ contributes only to the pairs. To define $\kappa = (s_i)_{i: \text{singletons}} \times (s_j)_{j: \text{pairs}}$, we set $s_i=0$ (resp. $s_i = 1$) if the singleton comes from $\phi'$ (resp. from $\phi''$). If a pair $j$ comes by ``merging'' singletons from $\phi'$ and $\phi''$, we have $s_j=1$. When the endoscopic datum coming with $\phi$ is elliptic, i.e.\! without $\GL$-component, all pairs arise in this way; it is readily seen that $\kappa$ equals the restriction of $\kappa_T$. In this case the relation between spectral transfer factors follows from Theorem \ref{prop:Li-lifting-real}.

	If a pair $j$ does not arise from merging singletons, it must come from $\phi_{\GL}$ since $\phi'$, $\phi''$ are both discrete series parameters. We set $s_j=0$ in this case. Note that the Levi subgroup $\tilde{M}$ coming with $\phi \in T^\EndoE(\tilde{G})$ takes the form $\prod_{j: s_j=0} \GL(2) \times \Mp(2n^\flat)$. The relation between spectral transfer factors follows by their description in Remark \ref{rem:character-relation-R}.

	Define the inverse $\kappa \mapsto \phi = (\phi_{\GL}, \phi', \phi'')$ by breaking the singletons into two different piles according to the values of $s_i$. If a pair $j$ in $\lambda$ satisfies $s_j = 1$, we divide it equally into $\phi'$ and $\phi''$; if $s_j=0$, it falls into the $\phi_{\GL}$. It is routine to check that they are mutually inverse.
\end{proof}

\begin{theorem}[Cf.\ {\cite[\S 7]{Sh08}}]\label{prop:spectral-inversion-R} \index{$\Delta(\pi,\phi)$}
	For all $(\phi, \pi) \in T^\EndoE(\tilde{G}) \times \Pi_{2\uparrow,-}(\tilde{G})$, define $\Delta(\pi, \phi) := |\Pi^{\tilde{G}}_\lambda|^{-1} \overline{\Delta(\phi, \pi)}$ when $\Delta(\phi, \pi) \neq 0$, otherwise $\Delta(\pi, \phi) := 0$. Then we have
	\begin{align*}
		\sum_{\phi \in T^\EndoE(\tilde{G})} \Delta(\pi, \phi) \Delta(\phi, \pi_1) & = \bdelta_{\pi, \pi_1}, \\
		\sum_{\pi \in \Pi_{\mathrm{temp},-}(\tilde{G})} \Delta(\phi, \pi) \Delta(\pi, \phi_1) &= \bdelta_{\phi, \phi_1}
	\end{align*}
	for $\phi, \phi_1 \in T^\EndoE(\tilde{G})$ and $\pi, \pi_1 \in \Pi_{2\uparrow,-}(\tilde{G})$, respectively.
\end{theorem}
\begin{proof}
	Consider the first assertion. We readily reduce to the case that $\pi$, $\pi_1$ belong to the same packet $\Pi^{\tilde{G}}_\lambda$ with $\lambda$ as in \eqref{eqn:nondegenerate-lambda}. Suppose $\pi=\pi(v\vec{\lambda})$ and $\pi_1 = \pi(wv\vec{\lambda})$ for a unique $\bar{w} \in \mathfrak{E}$. For each $\phi$, the recipe in Lemma \ref{prop:spectral-bijection} leads to
	$$ \Delta(\pi, \phi) \Delta(\phi, \pi_1) = |\Pi^{\tilde{G}}_\lambda|^{-1} \kappa(\bar{w}). $$
	Furthermore, summing over $\phi$ amounts to summing over $\kappa \in \mathfrak{R} = \mathfrak{R}_\lambda$. We conclude by Fourier inversion on $\mathfrak{E}$. As for the second assertion, the sum is taken over some packet $\Pi^{\tilde{G}}_\lambda$. In view of the bijection from Lemma \ref{prop:spectral-bijection} together with $|\Pi_\lambda(\tilde{G})| = |\mathfrak{E}| = |\mathfrak{R}|$, we are reduced to the previous case by linear algebra, as in the proof of Lemma \ref{prop:geometric-inversion}.
\end{proof}

\begin{corollary}[Cf.\ {\cite[Corollary 7.7]{Sh08}}]\label{prop:spectral-inverted-transfer}
	Let $f_{\tilde{G}} \in \Iasp(\tilde{G})$ with $f^\EndoE = \mathcal{T}^\EndoE(f_{\tilde{G}})$. Then
	$$ f_{\tilde{G}}(\pi) = \sum_{\phi \in T^\EndoE(\tilde{G})} \Delta(\pi, \phi) f^\EndoE(\phi) $$
	for every $\pi \in \Pi_{2\uparrow,-}(\tilde{G})$.
\end{corollary}

\subsection{The case $F=\C$}\label{sec:complex}
Let $F = \C$. Hereafter, the $\C$-groups are identified with their groups of $\C$-points. Note the following facts.
\begin{itemize}
  \item The covering $\rev: \tilde{G} \to G$ splits canonically: $\tilde{G} = \bmu_8 \times G$. Henceforth we view $G$ as $\{1\} \times G \subset \tilde{G}$. 
  \item The element $-1 \in \tilde{G}$ in Definition \ref{def:-1} equals $-1 \in G$. Indeed, this can be seen by combining \cite[Remarque 4.3]{Li11} with \cite[Corollaire 4.6]{Li11}.
\end{itemize}
In particular, the genuine representation theory of $\tilde{G}$ is no different from $G$. Theorem \ref{prop:cplx-group} will be applied to $G$ and $G^\Endo$. Choose Borel pairs $(B, T)$ and $(B^\Endo, T^\Endo)$ as above. Choose any standard isomorphism $\theta: T^\Endo \rightiso T$ constructed in Lemma \ref{prop:diagram}. In view of the recollections in \S\ref{sec:L-parameters}, the set $T^\EndoE(\tilde{G})$ may be identified with the set of $W(G, T)$-orbits of continuous unitary characters of $T$, whose elements we represent as $[\chi]$. The bijection
\begin{equation}\begin{aligned}\label{eqn:cplx-bijection}
  T^\EndoE(\tilde{G}) & \longrightarrow \Pi_\text{temp}(G) = \Pi_{\text{temp},-}(\tilde{G}) \\
  \phi=[\chi] & \longmapsto I_B(\chi).
\end{aligned}\end{equation}
is independent of the choice of $\theta$.

\begin{theorem}
  For every $f_{\tilde{G}} \in \Iasp(\tilde{G})$ with $f^\EndoE = \mathcal{T}^\EndoE(f_{\tilde{G}})$ and $\phi \in T^\EndoE(\tilde{G})$, we have
  $$ f^\EndoE(\phi) = \sum_{\pi \in \Pi_{\mathrm{temp},-}(\tilde{G})} \Delta(\phi, \pi) f_{\tilde{G}}(\pi), $$
  where $\Delta(\phi, \pi)=1$ when $\phi \mapsto \pi$ via \eqref{eqn:cplx-bijection}, otherwise $\Delta(\phi,\pi)=0$. Moreover, by setting $\Delta(\pi, \phi) := \Delta(\phi, \pi)$ we have the inversion formula
  $$ f_{\tilde{G}}(\pi) = \sum_{\phi \in T^\EndoE(\tilde{G})} \Delta(\pi, \phi) f^\EndoE(\phi). $$
\end{theorem}
\begin{proof}
  To prove the first assertion, we argue as in Theorem \ref{prop:character-relation-R}. In the complex setting we reduce to the case $M = T$. Obviously $\Delta_{\tilde{T}}(\cdot, \cdot)$ reduces to Kronecker's delta. It remains to recall from Theorem \ref{prop:cplx-group} the irreducibility of $I_B(\chi)$. The inversion formula follows immediately.
\end{proof}

Obvious analogues of Proposition \ref{prop:preservation-infchar} and Theorem \ref{prop:preservation-K-finiteness} hold in the complex case. We omit the details.

\section{Proof of the non-archimedean character relations}\label{sec:proof}
The arguments are largely based on \cite{Ar96}. Some non-trivial fine-tunings are needed, however.

\subsection{A stable simple trace formula}\label{sec:stable-simple-trace-formula}
In this subsection, $\mathring{F}$ will denote a number field. Write $\A = \A_{\mathring{F}}$ for its ring of adèles and fix a non-trivial additive character $\mathring{\psi} = \prod_v \psi_v: \A/\mathring{F} \to \Sph^1$. We consider an adélic metaplectic covering
$$ \rev: \tildering{G} \twoheadrightarrow \mathring{G}(\A), \quad \mathring{G} = \Sp(\mathring{W}) $$
attached to a symplectic $\mathring{F}$-vector space $(\mathring{W}, \angles{\cdot|\cdot})$ of dimension $2n$ and $\mathring{\psi}$.

\paragraph{Simple trace formula}
We shall formulate the Arthur-Selberg trace formula for $\tildering{G}$; the basic reference is \cite{Li14b}. Following the prescription of \textit{loc.\ cit.}, we use the Tamagawa measures on the adélic groups and their quotients.

Fix a large set $V$ of places of $\mathring{F}$ such that
$$ V \supset V_\text{ram} \supsetneq V_\infty := \{v : v | \infty \}, $$
where $V_\text{ram}$ is the set of places over which $\rev: \tildering{G} \twoheadrightarrow \mathring{G}(\A)$ ``ramifies'' in the sense of \cite[\S 3]{Li14a}. Here we may simply choose $V$ and endow $(\mathring{W}, \angles{\cdot|\cdot})$ with a model over $\mathfrak{o}_V$, the ring of $V$-integers in $\mathring{F}$, so that every $v \notin V$ satisfies
\begin{itemize}
  \item $v$ is non-archimedean of residual characteristic $\neq 2$;
  \item the $\mathfrak{o}_V$-model of $(\mathring{W}, \angles{\cdot|\cdot})$ has good reduction at $v$; in particular, $W_v := \mathring{W} \otimes_{\mathring{F}} \mathring{F}_v$ admits the self-dual lattice $\mathring{W}(\mathfrak{o}_v)$ with respect to $\angles{\cdot|\cdot}$;
  \item $\psi_v|_{\mathfrak{o}_v} \equiv 1$ but $\psi_v|_{\mathfrak{p}_v^{-1}} \not\equiv 1$;
  \item the lattice model for the Weil representation $\omega_{\psi_v}$ furnishes a splitting of $\rev$ over $K_v := \Sp(W_v, \mathfrak{o}_v)$.
\end{itemize}

As in \cite{Li14a}, we write $\tilde{G}_V := \rev^{-1}(G(\mathring{F}_V))$; similarly for $\tilde{G}^V$. Set $K^V := \prod_{v \notin V} K_v$ so that $K^V \hookrightarrow \tilde{G}^V$ is a continuous splitting. Hence we may define the \emph{spherical Hecke algebra} $\mathcal{H}(\tilde{G}^V \sslash K^V)$ outside $V$: its unit is the genuine function $f_{K^V} := \prod_{v \notin V} f_{K_v}$ on $\tilde{G}^V$. Also fix a maximal compact subgroup $K_V = \prod_{v \in V} K_v$ of $G(F_V)$. In the foregoing construction, we may even arrange that each $K_v$ is in good position relative to some chosen minimal Levi subgroup $\mathring{M}_0$ of $\mathring{G}$. Put $K := \prod_v K_v$, $K_\infty := \prod_{v \in V_\infty} K_v$.

For any finite subset $\Gamma$ of $\Pi(\tilde{K}_\infty)$, define $\mathcal{H}_{\asp}(\tilde{G}_V)_\Gamma$ to be the subspace functions in $C^\infty_{c,\asp}(\tilde{G}_V)$ which generate a space isomorphic to a sum of representations from $\Gamma \times \Gamma$ under bilateral translation by $\tilde{K}_\infty$. Define $\mathcal{H}_{\asp}(\tilde{G}_V) := \bigcup_{\Gamma} \mathcal{H}_{\asp}(\tilde{G}_V)_\Gamma$.

\begin{definition}\index{$\mathcal{H}_{\text{simp}}(\tilde{G}_V)$}
  Define the space of simple test functions $\mathcal{H}_{\text{simp}}(\tilde{G}_V)$ for $\tildering{G}$ to be the subspace of $\mathcal{H}_{\asp}(\tilde{G}_V)$ generated by $f_V = \prod_{v \in V} f_v$ satisfying the local conditions
  \begin{enumerate}[(i)]
    \item there exists $v_1, v_2 \in V \smallsetminus V_\infty$, $v_1 \neq v_2$, at which $f$ is cuspidal (Definition \ref{def:Iasp-filtration});
    \item there exists $w \in V \smallsetminus V_\infty$ such that $f_w$ is supported on the semisimple, strongly regular elliptic locus of $\tilde{G}_w$.
  \end{enumerate}
  Define the subspace $\mathcal{H}_{\text{simp,adm}}(\tilde{G}_V)$ by imposing the following \emph{admissibility}\index{admissibility} condition (cf.\ \cite[(20)]{Li14b} and \cite[\S 5.6]{Li14a}):
  $$ f_V \in \mathcal{H}_{\text{adm},\asp}(\tilde{G}_V). $$

  For $\Gamma$ as above, define $\mathcal{H}_{\text{simp}}(\tilde{G}_V)_\Gamma$ to be $\mathcal{H}_{\text{simp}}(\tilde{G}_V) \cap \mathcal{H}_{\asp}(\tilde{G}_V)_\Gamma$.
\end{definition}

\begin{theorem}[Simple trace formula {\cite[Théorème 6.7]{Li14b}}]\label{prop:simple-trace-formula}
  The distribution
  $$ I: \mathcal{H}_{\asp}(\tilde{G}_V) \to \C $$
  in the invariant trace formula takes the following form when applied to $f_V \in \mathcal{H}_{\mathrm{simp}}(\tilde{G}_V)$. Put $\mathring{f} := f_V f^{K^V}$, then
  \begin{align*}
    I(f_V) & = \sum_{\gamma \in G(\mathring{F})_\mathrm{ell,ss}/\mathrm{conj}} a_{\tildering{G}}(\gamma) I^{\tildering{G}}(\gamma, \mathring{f}), \\
    & = \sum_{t \geq 0} I_t(f),
  \end{align*}
  called the geometric and spectral expansions of $I$, respectively, with
  \begin{gather*}
    I_t(f_V) := \sum_{\mathring{\pi} \in \Pi_{\mathrm{disc},t,-}(\tildering{G})} a_{\tildering{G}}(\mathring{\pi}) I^{\tildering{G}}(\mathring{\pi}, \mathring{f}).
  \end{gather*}
  Here
  \begin{itemize}
    \item $I^{\tildering{G}}(\gamma, \cdot)$ is the orbital integral along the orbit of $\gamma$;
    \item $I^{\tildering{G}}(\mathring{\pi}, \cdot)$ is the character of the representation $\mathring{\pi}$;
    \item $a_{\tildering{G}}(\gamma) = \mes(G_\gamma(\mathring{F}) \backslash G_\gamma(\A))$, the Tamagawa number of $G_\gamma$;
    \item $\Pi_{\mathrm{disc},t,-}(\tildering{G})$ is a set of irreducible genuine representations of $\tildering{G}$ whose archimedean infinitesimal character $\nu$ has height $\|\Im(\nu)\| = t$; it contains the representations in the genuine discrete spectrum of $G(\mathring{F}) \backslash \tildering{G}$, together with certain ``phantoms''.
  \end{itemize}
  We refer to \cite[\S 7]{Li13} for the precise definition for the spectral objects.
\end{theorem}

Note that by Theorem \ref{prop:Mp-commute}, all elements in $\tildering{G}$ are \emph{good} in the sense of \cite[Définition 2.6.1]{Li14a}, thereby simplifying the trace formula for coverings.

\begin{remark}\label{rem:harmless}
  In \cite{Li14b} it is required that $f_V \in \mathcal{H}_{\text{simp,adm}}(\tilde{G}_V)$ in deducing the simplified geometric expansion of $I(f_V)$. The admissibility of $f_V$ is a global property depending on $\Supp(f_V)$, but also a harmless one. In fact, given $f_V \in \mathcal{H}_{\text{simp}}(\tilde{G}_V)$, we may always enlarge $V$ to some $S \supset V$, replacing $f_V$ by $f_S := f_V f_{K^V_S}$ (here $K^V_S := \prod_{v \in S \smallsetminus V} K_v$) simultaneously, so that $f_S \in \mathcal{H}_{\text{simp,adm}}(\tilde{G}_S)$. Since $\mathring{f}$ remains unaltered, none of the expansions above are affected by this procedure.
\end{remark}

Fix a finite set $\Gamma$ of $\tilde{K}_\infty$-types and assume $f_V \in \mathcal{H}_\text{simp}(\tilde{G}_V)_\Gamma$. In Arthur's original works, the spectral expansion $I(f_V) = \sum_{t \geq 0} \sum_{\mathring{\pi} \in \Pi_{\mathrm{disc},t,-}(\tildering{G})} \cdots$ is considered as a convergent \emph{iterated sum}. The individual sums $I_t(f_V)$ are actually finite sums by \cite[Lemme 7.2]{Li13}. It may be further refined as another convergent iterated sum
$$ I(f_V) = \sum_\nu I_\nu(f_V), $$
where \index{$I_\nu$}
$$ I_\nu(f_V) := \sum_{\substack{\mathring{\pi} \in \Pi_{\text{disc},-}(\tildering{G}) \\ \text{inf. char.}=\nu}} a_{\tildering{G}}(\mathring{\pi}) I_{\tildering{G}}(\mathring{\pi}, \mathring{f}). $$
Moreover, the sum $I = \sum_\nu I_\nu$ satisfies the \emph{multiplier convergence estimate} in \cite[(3.3)]{Ar02}\index{multiplier convergence estimate}.

The results by Finis-Lapid-Müller \cite{FLM11}, once generalized to the metaplectic covering $\tildering{G}$, will solve all these convergence issues in our simple trace formula; cf.\ \cite[Remarque 7.5]{Li13}. As in \cite{Ar96}, we opt to use the version $I = \sum_\nu I_\nu$ in this article. 

\paragraph{Stabilization}
In \cite{Li11, Li14b} we have defined the set $\EndoE_\text{ell}(\tildering{G})$ of elliptic endoscopic data for $\tildering{G}$: its members are always given by pairs $(n',n'') \in \Z^2_{\geq 0}$ satisfying $n'+n''=n$; consequently we can pass to local elliptic endoscopic data of $\tildering{G}_v$ at each place $v$ of $F$. Theorems \ref{prop:geometric-transfer} and \ref{prop:FL} together give adélic transfer of test functions
\begin{align*}
  f_V & \longmapsto f^\Endo_V \quad \in C^\infty_c(G^\Endo(\mathring{F}_V)), \\
  \mathring{f} = f_V f_{K^V} & \longmapsto \mathring{f}^\Endo = f^\Endo_V \mathbf{1}_{K^{V,\Endo}} \quad \in C^\infty_c(G^\Endo(\A)),
\end{align*}
where $G^\Endo := \SO(2n'+1) \times \SO(2n''+1)$ is the endoscopic group corresponding to $(n',n'')$, and $K^{V,\Endo} \subset G^\Endo(F^V)$ is any product of hyperspecial subgroups off $V$. In what follows we take $f_V \in \mathcal{H}_\text{simp}(\tildering{G})$. Its transfer to $G^\Endo$ may be taken to be $K^\Endo_\infty \times K^\Endo_\infty$-finite by Theorem \ref{prop:preservation-K-finiteness} and its complex analogue, where $K^\Endo_\infty$ is any maximal compact subgroup of $\prod_{v | \infty} G^\Endo(F_v)$.

To the quasisplit $\mathring{F}$-group $G^\Endo$, Arthur defined the stable distribution $S^\Endo$ in his stable trace formula \cite[\S 10]{Ar02}. The transfer $f^\Endo_V$ turns out to be a \emph{simple test function} in the sense of \cite[p.556]{Ar96}. Indeed:
\begin{compactenum}[(i)]
  \item $f^\Endo_V$ is cuspidal at two distinct places $v_1, v_2 \in V \smallsetminus V_\infty$ since the transfer is compatible with parabolic descent by Theorem \ref{prop:transfer-parabolic};
  \item $f^\Endo_w$ is supported in the elliptic $G$-regular semisimple locus of $G^\Endo$ at the place $w \in V \smallsetminus V_\infty$, by the very definition of geometric transfer (Theorem \ref{prop:geometric-transfer});
\end{compactenum}

Under this circumstance, $S^\Endo(f^\Endo_V) = S^\Endo(\mathring{f}^\Endo)$ is a sum of stable orbital integrals. More precisely, when applied to simple test functions,  $S^\Endo$ coincides with the regular part of the distribution $ST^{G^\Endo}_{\text{équi,ell}}$ in \cite[Définition 5.2.3]{Li15}, which equals
$$ S^\Endo_{G-\text{reg}, \text{ell}}(\mathring{f}^\Endo) := \tau(G^\Endo) \sum_{\sigma \in G^\Endo(\mathring{F})_\text{reg,ell}/\text{st.conj.}} S^{G^\Endo}(\sigma, \mathring{f}^\Endo) $$
where
\begin{compactitem}
  \item $\tau(G^\Endo)$ is the Tamagawa number of $G^\Endo$,
  \item $S^{G^\Endo}(\sigma, \mathring{f}^\Endo)$ is the stable orbital integral of $\mathring{f}^\Endo$ along $\sigma$, defined relative to Tamagawa measures.
\end{compactitem}

On the other hand, $S^\Endo(\mathring{f}^\Endo)$ also admits a spectral description $\sum_t S^\Endo_t(\mathring{f}^\Endo)$, each $S^\Endo_t$ being an infinite sum of adélic stable characters. Its precise form is contained in Arthur's \emph{stable multiplicity formula} \cite[Theorem 4.1.2]{Ar13}, applied to the discrete parts $S^{\SO(2n'+1)}_\text{disc}$ and $S^{\SO(2n''+1)}_\text{disc}$ separately.  

For the $G^\Endo \in \EndoE_\mathrm{ell}(\tildering{G})$ (abusing notations...) above, we set
$$ \iota(\tildering{G}, G^\Endo) := \begin{cases}
  \frac{1}{4}, & n', n'' \geq 1, \\
  \frac{1}{2}, & n \geq 1, \; n'=0 \text{ or } n''=0, \\
  1, & n=0.
\end{cases} $$

\begin{theorem}[{\cite[Théorème 5.2.6]{Li15}}]\label{prop:stable-trace-formula}
  For $f \in \mathcal{H}_\mathrm{simp}(\tilde{G}_V)$, we have
  $$ I(f_V) = I^\EndoE(f_V) :=\sum_{G^\Endo \in \EndoE_\mathrm{ell}(\tildering{G})} \iota(\tildering{G}, G^\Endo) S^\Endo(\mathring{f}^\Endo). $$
\end{theorem}

Parallel to the unstable side, each distribution $S^\Endo$ has an expansion
$$ S^\Endo = \sum_{t \geq 0} S^\Endo_t = \sum_{\nu^\Endo: \text{inf. char.}} S^\Endo_{\nu^\Endo} $$
which also satisfies the multiplier convergence estimate by \cite[Proposition 10.5 (b)]{Ar02}. In \S\ref{sec:spectral-transfer-R} we have defined a canonical finite-to-one map $\nu^\Endo \mapsto \nu$ of infinitesimal characters, for every $G^\Endo$. Therefore we can consider $S^\Endo = \sum_\nu S^\Endo_\nu$ with $S^\Endo_\nu := \sum_{\nu^\Endo \mapsto \nu} S^\Endo_{\nu^\Endo}$ as well. With the help of multiplier convergence estimates, in \cite[\S 7]{Ar96} Arthur derived a version of stable trace formula that would yield the result below in the metaplectic setup.

\begin{corollary}\label{prop:stable-trace-formula-nu}\index{simple trace formula!stabilization}
  For any chosen archimedean infinitesimal character $\nu$ of $\tildering{G}$, we have \index{$I^\EndoE_\nu$}
  $$ I_\nu(f_V) = I^\EndoE_\nu(f_V) := \sum_{G^\Endo} \iota(\tildering{G}, G^\Endo) S^\Endo_\nu(\mathring{f}^\Endo). $$
\end{corollary}
\begin{proof}[Sketch of the proof]
  Arthur's arguments in \textit{op.\ cit.}\ can be easily adapted to our case as follows.
  \begin{compactenum}[(i)]
    \item Modify the test function $f_V$ by suitable \emph{multipliers} $\hat{\alpha}$ (certain functions of $\nu$) in order to obtain $f_{V,\alpha} \in \mathcal{H}_\text{simp}(\tilde{G}_V)$ with
      $$ \forall \nu, \; I_\nu(\mathring{f}_\alpha) = \hat{\alpha}(\nu) I_\nu(\mathring{f}). $$
      The theory of multipliers for coverings is recapitulated in \cite[Théorème 4.4]{Li13}.
    \item Similar constructions apply to the stable side: $\forall \nu^\Endo, \; S^\Endo_{\nu'}(\mathring{f}^\Endo_\beta) = \hat{\beta}(\nu') S^\Endo_{\nu'}(\mathring{f}^\Endo)$.
    \item Plug $f_{V,\alpha}$ into Theorem \ref{prop:stable-trace-formula}, for various $\alpha$.
    \item Set $\hat{\alpha}(\nu') = \hat{\alpha}(\nu)$ if $\nu' \mapsto \nu$. Use the preservation of infinitesimal characters under archimedean transfer (Proposition \ref{prop:preservation-infchar} and its complex analogue) to see $(f_{V,\alpha})^\Endo = (f^\Endo_V)_\alpha$.
  \end{compactenum}
  By varying the multipliers, this trick will isolate the $\nu$-parts in the equality of Theorem \ref{prop:stable-trace-formula}; the analytic subtleties are taken care of by the multiplier convergence estimates.
\end{proof}

\subsection{Compression of coefficients}\label{sec:compression}
Keep the previous notations. Fix a non-archimedean place $u \in V$ and write $\tilde{G} := \tilde{G}_u$, $F := \mathring{F}_u$, $\psi := \psi_u$. Fix a finite set of $\tilde{K}_\infty$-types $\Gamma$. As before, the test functions are of the form $f_V = \prod_{v \in V} f_v \in \mathcal{H}_{\asp}(\tilde{G}_V)_\Gamma$. Write $f := f_u \in C^\infty_{c,\asp}(\tilde{G})$, so that $f_V = f_u f^u_V$.

\textbf{Assumption}: $f^u_V$ satisfies all the conditions defining of $\mathcal{H}_\text{simp}(\tilde{G}_V)_\Gamma$. Therefore $f_V \in \mathcal{H}_\text{simp}(\tilde{G}_V)_\Gamma$ for any choice of $f$.


We set out to isolate the $u$-components in the simple stable trace formula (Theorem \ref{prop:stable-trace-formula}) and encapsulate the contribution from the remaining places. This compression procedure is similar to \cite[\S 5.2, \S 5.5]{Li14b}, in principle.

\paragraph{Unstable geometric side}
For the geometric side only, we impose the extra condition that
\begin{gather}\label{eqn:adm-condition}
  f_V \in \mathcal{H}_\text{simp,adm}(\tilde{G}_V).
\end{gather}
It is largely a harmless assumption according to Remark \ref{rem:harmless}. The coefficients $a^{\tildering{G}}_\text{ell}(\cdot)$ in \cite[(23)]{Li14b} will be used.

\begin{definition}
  Assume \eqref{eqn:adm-condition}. For any $\gamma \in \Gamma_\mathrm{reg}(G)$ and $\tilde{\gamma} \in \rev^{-1}(\gamma)$, set
  $$ I(f^u_V, \tilde{\gamma}) := 8 \sum_{\tilde{\gamma}^u} a^{\tildering{G}}_\mathrm{ell}(\tilde{\gamma}^u \tilde{\gamma}) (f^u_V)_{\tilde{G}^u_V}(\tilde{\gamma}^u), $$
  where $(f^u_V)_{\tilde{G}^u_V}(\tilde{\gamma}^u)$ is a product of normalized orbital integrals over $V \smallsetminus \{u\}$, and $\tilde{\gamma}^u$ ranges over the regular semisimple classes in $\tilde{G}^u_V$. The sum is finite for given $\Supp(f_V)$.
\end{definition}
By \cite[Lemme 5.4]{Li14b}, $I(f^u_V, \noyau\tilde{\gamma}) = \noyau I(f^u_V, \tilde{\gamma})$ for every $\noyau \in \bmu_8$.

\begin{lemma}\label{prop:compressed-geom}
  Assume \eqref{eqn:adm-condition}. We have
  $$ I(f_V) = \sum_{\gamma \in \Gamma_\mathrm{ell,reg}(G)} I(f^u_V, \tilde{\gamma}) f_{\tilde{G}}(\tilde{\gamma}) $$
  where $\tilde{\gamma} \in \rev^{-1}(\gamma)$ is arbitrary.
\end{lemma}
\begin{proof}
  Recall that in the proof of part 4 of \cite[Théorème 6.5]{Li14b}, it is shown that the geometric expansion of $I(f_V)$ reduces to the $I_\text{ell}(f_V)$ defined by \cite[(22)]{Li14b}; the latter has an expansion
  $$ I_\text{ell}(f_V) = \sum_{\tilde{\gamma}_V} a^{\tildering{G}}_\text{ell}(\tilde{\gamma}_V) (f_V)_{\tilde{G}_V}(\tilde{\gamma}_V) $$
  where $\tilde{\gamma}_V = \tilde{\gamma}_u \tilde{\gamma}^u$ ranges over the regular semisimple classes in $\tilde{G}_V$. Collecting terms according to $\tilde{\gamma}_u$ yields the result.
\end{proof}


\paragraph{Unstable spectral side}
Fix $\nu$. First, recall the objects
\begin{compactitem}
  \item $\Pi_{\text{disc},-,\nu}(\tildering{G}, V)$: a set of unitary genuine irreducible representations of $\tilde{G}_V$, whose archimedean infinitesimal character is $\nu$;
  \item $\mathcal{C}^V_{\text{disc},-}(\tildering{G})$: a set of characters of $\mathcal{H}(\tilde{G}^V \sslash K^V)$, i.e.\ Satake parameters outside $V$;
  \item $a^{\tildering{G}}_\text{disc}(\mathring{\pi})$: the discrete spectral coefficient of a genuine representation $\mathring{\pi}$ of $\tildering{G}$ in the Arthur-Selberg trace formula,
\end{compactitem}
which are defined in \cite[\S 5.5]{Li14b} and \cite[\S 7]{Li13}. The only difference is that we pin down the infinitesimal character $\nu$ here.

\begin{definition}
  For any $\pi \in \Pi_-(\tilde{G})$, set
  $$ I_\nu(f^u_V, \pi) := \sum_{\substack{\pi^u_V \in \Pi_-(\tilde{G}^u) \\ \pi_V := \pi^u_V \boxtimes \pi \in \Pi_{\text{disc}, -, \nu}(\tildering{G}, V) }} \; \sum_{c \in \mathcal{C}^V_{\text{disc},-}(\tildering{G}) } a^{\tildering{G}}_\text{disc}(\pi_V \boxtimes c) \cdot (f^u_V)_{\tilde{G}^u_V}(\pi^u_V), $$
  where $(f^u_V)_{\tilde{G}^u_V}(\pi^u_V)$ is a product of characters.
\end{definition}

\begin{lemma}
  We have
  $$ I_\nu(f_V) = \sum_\pi I_\nu(f^u_V, \pi) f_{\tilde{G}}(\pi). $$
\end{lemma}
\begin{proof}
  In the proof of part 1 of \cite[Théorème 6.5]{Li14b}, we derived an expansion
  $$ I(f_V) = \sum_{\mathring{\pi}} a^{\tildering{G}}_\text{disc}(\mathring{\pi}) I^{\tildering{G}}(\mathring{\pi}, \mathring{f}). $$
  It remains to
  \begin{inparaenum}[(i)]
    \item isolate the $\nu$-parts,
    \item collect terms according to $\mathring{\pi}_u$, and
    \item unfold the definitions of $\Pi_{\text{disc},-}(\tildering{G}, V)$ and $\mathcal{C}^V_{\text{disc},-}(\tildering{G})$.
  \end{inparaenum}
\end{proof}

By the Langlands classification for $\tilde{G}$, the character of $\pi$ can be expressed in terms of the character of genuine standard modules. Therefore we may rewrite the sum over $\pi$ as a sum over $T_-(\tilde{G})_\C/\Sph^1$: each element $\tau$ therein (take a representative in $T_-(\tilde{G})_\C$) defines a character $f \mapsto f_{\tilde{G}}(\tau)$ satisfying $f_{\tilde{G}}(z\tau) = z^{-1} f_{\tilde{G}}(\tau)$, for $z \in \Sph^1$. We recapitulate the discussion as follows.

\begin{lemma}\label{prop:compression-spectral}
  We may define distributions $f^u_V \mapsto I_\nu(f^u_V, \tau)$ for $\tau \in T_-(\tilde{G})_\C$, such that
  \begin{itemize}
    \item each $I_\nu(\cdot, \tau)$ is a finite linear combination of the $I_\nu(\cdot, \pi)$, and vice versa;
    \item $I_\nu(\cdot, z\tau) = z I_\nu(\cdot, \tau)$ for each $z \in \Sph^1$;
    \item we have
      $$ I_\nu(f_V) = \sum_{\tau \in T_-(\tilde{G})_\C/\Sph^1} I_\nu(f^u_V, \tau) f_{\tilde{G}}(\tau) $$
      for all $f_V$ satisfying our assumptions.
  \end{itemize}
\end{lemma}

\paragraph{Stable spectral side}
Always fix the archimedean infinitesimal character $\nu$ for $\tildering{G}$. Let $G^\Endo$ be an elliptic endoscopic group of $\tildering{G}$. Write $f^\Endo_V = \prod_v f^\Endo_v \in C^\infty_c(G^\Endo(F_V))$, we have the corresponding global test function $\mathring{f}^\Endo := f^\Endo_V \mathbf{1}_{K^{V,\Endo}}$ as usual. Set $f^\Endo := f^\Endo_u \in C^\infty_c(G^\Endo(F))$ so that
$$ f^\Endo_V = (f^\Endo)^u_V f^\Endo_u. $$
Moreover, $(f^\Endo)^u_V$ is assumed to be a simple test function \cite[p.556]{Ar96}.

Denote the distribution in the stable trace formula for $G^\Endo$ as $S^\Endo = S^{G^\Endo}$. We will actually work with $S^\Endo_\nu = \sum_{\nu^\Endo \mapsto \nu} S^\Endo_{\nu^\Endo}$. In this case, the compression of coefficients has been done in \cite[(8.6)]{Ar96}; we summarize as follows.

\begin{lemma}\label{prop:compression-stable}
  For every place $v$, denote by $\Phi_\mathrm{bdd}(G^\Endo_v)_\C$ the space of bounded $L$-parameters of $G^\Endo \times_{\mathring{F}} \mathring{F}_v$. Abbreviate
  $$ \Phi_\mathrm{bdd}(G^\Endo)_\C := \Phi_\mathrm{bdd}(G^\Endo_u)_\C .$$
  One can define distributions $(f^\Endo)^u_V \mapsto S^\Endo_\nu((f^\Endo)^u_V, \phi^\Endo)$ for every $\phi^\Endo \in \Phi_\mathrm{bdd}(G^\Endo)_\C$, such that
  \begin{compactitem}
    \item each $S^\Endo_\nu(\cdot, \phi^\Endo)$ is a linear combination of stable characters coming from $\prod_{\substack{v \in V \\ v \neq u}} \Phi_\mathrm{bdd}(G^\Endo_v)_\C$;
    \item for every $f^\Endo_V$ as above, we have
      $$ S^\Endo_\nu(f^\Endo_V) = \sum_{\phi^\Endo \in \Phi_\mathrm{bdd}(G^\Endo)_\C} S^\Endo_\nu((f^\Endo)^u_V, \phi^\Endo) (f^\Endo)^{G^\Endo}(\phi^\Endo). $$
  \end{compactitem}
\end{lemma}

Plug this into the expression $I^\EndoE_\nu(f_V) = \sum_{G^\Endo \in \EndoE_\text{ell}(\tildering{G})} \iota(\tildering{G}, G^\Endo) S^\Endo_\nu(\mathring{f}^\Endo)$ of Corollary \ref{prop:stable-trace-formula-nu}. Take $f_V \in \mathcal{H}_\text{simp}(\tilde{G}_V)_\Gamma$ as in the beginning of this subsection. We shall use the map $\phi \mapsto f^\EndoE(\phi)$ of Definition \ref{def:f^EndoE} in what follows.

\begin{lemma}\label{prop:compression-EndoE}
  One can define distributions $f^u_V \mapsto S^\Endo_\nu(f^u_V, \phi)$ for every $\phi \in T^\EndoE(\tilde{G})_\C$ such that for every $f_V$ as above, we have
  $$ I^\EndoE_\nu(f_V) = \sum_{\phi \in T^\EndoE(\tilde{G})_\C} I^\EndoE_\nu(f^u_V, \phi) f^\EndoE(\phi). $$
\end{lemma}
\begin{proof}
  By Lemma \ref{prop:compression-stable},
  \begin{align*}
    I^\EndoE_\nu(f_V) & = \sum_{G^\Endo} \iota(\tildering{G}, G^\Endo) S^\Endo_\nu(\mathring{f}^\Endo) \\
    & = \sum_{(G^\Endo, \phi^\Endo)} \iota(\tildering{G}, G^\Endo) S^\Endo_\nu((f^\Endo)^u_V, \phi^\Endo) (f^\Endo)^{G^\Endo}(\phi^\Endo),
  \end{align*}
  where $f^\Endo$ is a transfer of $f$, for each $G^\Endo \in \EndoE_\text{ell}(\tildering{G})$.

  For every pair $(G^\Endo, \phi^\Endo)$ above, there exists a unique $M^\Endo \in \mathcal{L}(M^\Endo_0)/W^{G^\Endo}(M^\Endo_0)$ and $\phi_{M^\Endo} \in \Phi_{2,\text{bdd}}(M^\Endo)_\C /W^{G^\Endo}(M^\Endo)$ such that $\phi_{M^\Endo} \mapsto \phi^\Endo$. By Lemma \ref{prop:M^Endo-to-s}, there exists a unique pair $(M, s)$ such that we may complete $(G^\Endo, M^\Endo)$ into a diagram
  $$ \begin{tikzcd}
  \overbrace{G^\Endo = G[s]}^{\text{as endoscopic data}} \arrow[-,dashed]{r}[above]{\text{ell.}}[below]{\text{endo.}} & \tilde{G} \\
  M^\Endo \arrow[-,dashed]{r}[above]{\text{ell.}}[below]{\text{endo.}} \arrow[hookrightarrow]{u}[left]{\text{Levi}} & \tilde{M} \arrow[hookrightarrow]{u}[right]{\text{Levi}}
  \end{tikzcd} $$
  Therefore $\phi_{M^\Endo}$ determines an element $\phi_M \in T^\EndoE_\text{ell}(\tilde{M})_\C$, which in turn maps to $\phi \in T^\EndoE(\tilde{G})_\C$. On the other hand, to $(G^\Endo, \phi^\Endo)$ is associated the sign factor $\bomega''(-1)$ by Remark \ref{rem:f^EndoE}, satisfying
  $$ (f^\Endo)(\phi^\Endo) = \bomega''(-1) \cdot \underbrace{f^\EndoE(\phi)}_{\text{independent of } G^\Endo}. $$
  This is originally established for $\phi \in T^\EndoE(\tilde{G})$, but extends easily to the present case by meromorphic continuation. Collecting terms according to $\phi \in T^\EndoE(\tilde{G})_\C$, we see that $I^\EndoE(f_V)$ equals
  $$ \sum_\phi \left( \sum_{\substack{s \\ G^\Endo := G[s]}} \iota(\tildering{G}, G^\Endo) S^\Endo_\nu((f^\Endo)^u_V, \phi^\Endo) \bomega''(-1) \right) f^\EndoE(\phi), $$
  the $\phi^\Endo$ and $\bomega''$ in the inner sum are both determined by $(\phi, s)$ by the foregoing discussion. Now set $I^\EndoE_\nu(f^u_V, \phi)$ to be the inner sum.
\end{proof}

\subsection{Proof of the main theorem: preparations}\label{sec:proof-preparation}
Now we can undertake the proof of Theorem \ref{prop:character-relation}. Let us begin with the local setup in \S\ref{sec:spectral-transfer}.
\begin{itemize}
  \item Fix $f_{\tilde{G}}$. By the meromorphic extension of stable tempered characters of each $G^\Endo$, the function $\phi \mapsto f^\EndoE(\phi)$ has obvious meromorphic extension to $T^\EndoE(\tilde{G})_\C$. In a similar vein, $\phi \mapsto f^\EndoE_\gr(\phi)$ extends to $T^\EndoE(\tilde{G})_\C$ as well: this time we use the meromorphic extensions of $\tau \mapsto f_{\tilde{G}}(\tau)$ and of the spectral transfer factors (cf.\ the discussion at the end of \S\ref{sec:spectral-formalism}).
  \item All the definitions in \S\ref{sec:statement} extend to coverings of metaplectic type, say $\tilde{L} = \prod_{i \in I} \GL(n_i) \times \Mp(W^\flat)$, by treating the $\GL$-factors and $\Mp(W^\flat)$ separately; the former case is covered by \cite{Ar96}. For instance, $\Iasp(\tilde{L}) \to \orbI^\EndoE(\tilde{L})$ is simply the $\otimes$-product of the identity map on the $\GL$-factors with $\Iasp(\Mp(W^\flat)) \to \orbI^\EndoE(\Mp(W^\flat))$.
  
  This allows an inductive structure in our proof of Theorem \ref{prop:character-relation}. The following hypothesis will be in force throughout this section.
\end{itemize}

\begin{hypothesis}
  We assume the validity of Theorem \ref{prop:character-relation} for all $\Mp(W^\flat)$ with $\dim_F W^\flat < \dim_F W = 2n$.
\end{hypothesis}

Thus Theorem \ref{prop:character-relation} also holds for coverings of metaplectic type of the form $\tilde{L} = \prod_{i \in I} \GL(n_i, F) \times \Mp(W^\flat)$. Note that the case $W^\flat = \{0\}$ is trivially true. Our global arguments will be based on the two simple facts below.

\begin{lemma}\label{prop:reduction-cuspidal-f}
  If $f_{\tilde{G}} \in \Iaspcusp(\tilde{G})$, then \eqref{eqn:character-relation-equiv} holds for all $\phi$. Same for $\tilde{L}$ in place of $\tilde{G}$, where $L \in \mathcal{L}(M_0)$.
\end{lemma}
\begin{proof}
  It amounts to show the commutativity of \eqref{eqn:character-relation-diagram}, in which all $\orbI_{\cdots}(\cdots)$ are replaced by their cuspidal avatars so that the subscripts $\gr$ become superfluous. We are reduced to tautology. 
\end{proof}

\begin{lemma}\label{prop:reduction-para-phi}
  If $\phi \in T^\EndoE(\tilde{G})_\C \smallsetminus T^\EndoE_\mathrm{ell}(\tilde{G})_\C$, then \eqref{eqn:character-relation-equiv} holds for $\phi$ and all $f_{\tilde{G}}$.
\end{lemma}
\begin{proof}
  Let $L \in \mathcal{L}(M_0)$, $L \neq G$ and assume $\phi$ comes from $\phi_L \in T^\EndoE_\text{ell}(\tilde{L})_\C$. As $L$ is of the form $\prod_{i \in I} \GL(n_i) \times \Mp(W^\flat)$, the diagram
  $$ \begin{tikzcd}
    \Iasp(\tilde{L}) \arrow{d} \arrow{r} & \orbI^\EndoE(\tilde{L}) \arrow{d} \\
    \orbI_\gr(\tilde{L}) \arrow{r} & \orbI^\EndoE_\gr(\tilde{L})
  \end{tikzcd} $$
  commutes by assumption; moreover, every arrow is a $W^G(L)$-equivariant isomorphism. Now consider the diagram
  $$ \begin{tikzcd}[xscale=0.6]
    \Iasp(\tilde{G}) \arrow{rr} \arrow{rd} \arrow{dd} & & \orbI^\EndoE(\tilde{G}) \arrow{rd} \arrow{dd} & \\
    & \orbI_\gr(\tilde{G}) \arrow[crossing over]{rr} & & \orbI^\EndoE_\gr(\tilde{G} )\arrow{dd} \\
    \Iasp(\tilde{L})^{W^G(L)} \arrow{rd} \arrow{rr} & & \orbI^\EndoE(\tilde{L})^{W^G(L)} \arrow{rd} & \\
    & \orbI_\gr(\tilde{L})^{W^G(L)} \arrow[crossing over, leftarrow]{uu} \arrow{rr} & & \orbI^\EndoE_\gr(\tilde{L})^{W^G(L)}
  \end{tikzcd} $$
  in which:
  \begin{compactitem}
    \item $\orbI_\gr(\tilde{G}) \to \orbI_\gr(\tilde{L})^{W^G(L)}$ and $\orbI^\EndoE_\gr(\tilde{G}) \to \orbI^\EndoE_\gr(\tilde{L})^{W^G(L)}$: restriction maps,
    \item $\Iasp(\tilde{G}) \to \Iasp(\tilde{L})^{W^G(L)}$: parabolic descent of test functions,
    \item $\orbI^\EndoE(\tilde{G}) \to \orbI^\EndoE(\tilde{L})^{W^G(L)}$: for any $L^\Endo \in \EndoE_\text{ell}(\tilde{L})$, the $L^\Endo$-component of $f^\EndoE$ is $f^{L^\Endo} \in S\orbI(L^\Endo)^{W^G(L)}$ in the notation of Definition \ref{def:Endo-I}.
  \end{compactitem}
  
  The bottom layer of the diagram is commutative. We claim that all the four walls are commutative. Indeed:
  \begin{compactenum}[(a)]
    \item for the two walls of shape\;
      \begin{tikzpicture}
	\draw (0,0) rectangle (0.5,0.3);
      \end{tikzpicture}, use the compatibility between parabolic descent and geometric transfer (Theorem \ref{prop:transfer-parabolic});
    \item for the leftmost wall\;
      \begin{tikzpicture}[shape border rotate=90]
	\tikzstyle{every node}=[trapezium, draw]
	\node[trapezium left angle=60, trapezium right angle=120] at (0,0) {};
      \end{tikzpicture}, use the compatibility between parabolic descent and induction (see \eqref{eqn:para-descent-character});
    \item the rightmost wall\;
      \begin{tikzpicture}[shape border rotate=90]
	\tikzstyle{every node}=[trapezium, draw]
	\node[trapezium left angle=60, trapezium right angle=120] at (0,0) {};
      \end{tikzpicture}\; may be treated in a similar fashion, by carefully unwinding the $z[s]$-twists in $\orbI^\EndoE(\tilde{G}) \to \orbI^\EndoE(\tilde{L})^{W^G(L)} \to \orbI^\EndoE_\gr(\tilde{L})^{W^G(L)}$ and $\orbI^\EndoE(\tilde{G}) \to \orbI^\EndoE_\gr(\tilde{G})$.
  \end{compactenum}

  Now recall the arguments in Remark \ref{rem:character-relation-equiv}. Establishing \eqref{eqn:character-relation-equiv} for all $T^\EndoE(\tilde{L}) \ni \phi_L \mapsto \phi$ and all $f_{\tilde{G}}$ is equivalent to showing the commutativity of the top layer upon composition with $\orbI^\EndoE_\gr(\tilde{G}) \to \orbI^\EndoE_\gr(\tilde{L})^{W^G(L)}$, or more concretely with the restriction map to $T^\EndoE_\text{ell}(\tilde{L})/W^G(L)$. A straightforward diagram chasing suffices.
\end{proof}

Therefore we are reduced to showing \eqref{eqn:character-relation-equiv} for $\phi \in T^\EndoE_\text{ell}(\tilde{G})$. For $f_{\tilde{G}} \in \Iasp(\tilde{G})$ and $f^\EndoE = \mathcal{T}^\EndoE(f_{\tilde{G}})$, define
\begin{gather}
  f^\gr_{\tilde{G}}(\tau) := \sum_{\phi \in T^\EndoE(\tilde{G})} \Delta(\tau, \phi) f^\EndoE(\phi), \quad \tau \in T_-(\tilde{G}).
\end{gather}
It admits meromorphic continuation to $T_-(\tilde{G})_\C$. By Lemma \ref{prop:spectral-inversion} we have $f^\EndoE(\phi) = \sum_\tau \Delta(\phi, \tau) f^\gr_{\tilde{G}}(\tau)$, and \eqref{eqn:character-relation-equiv} is equivalent to the assertion
\begin{gather}\label{eqn:character-relation-final}
  f^\gr_{\tilde{G}}(\tau) = f_{\tilde{G}}(\tau).
\end{gather}
By the definition of $\Delta$, showing \eqref{eqn:character-relation-equiv} for elliptic $\phi$ amounts to showing \eqref{eqn:character-relation-final} for $\tau \in T_{\text{ell},-}(\tilde{G})$. This is what we will actually prove.

Write
$$ T_{\text{para},-}(\tilde{G}) := T_-(\tilde{G}) \smallsetminus T_{\text{ell},-}(\tilde{G}), $$
a union of connected components. Define $T_{\text{para},-}(\tilde{G})_\C$, $
\tilde{T}_{\text{para},-}(\tilde{G})$ and $\tilde{T}_{\text{para},-}(\tilde{G})_\C$ in a similar fashion. By the foregoing discussion, we already have $f^\gr_{\tilde{G}}(\tau) = f_{\tilde{G}}(\tau)$ for $\tau \in T_{\text{para},-}(\tilde{G})_\C$. Now switch to the global setup.

\textbf{Assumption}. Henceforth we suppose that $\rev: \tilde{G} \to G(F)$ is isomorphic to the localization at $u$ of the adélic covering $\rev: \tildering{G} \to \mathring{G}(\A)$ (as coverings), with the notations in \S\S \ref{sec:stable-simple-trace-formula}-\ref{sec:compression}: here $u \in V$ is a distinguished place of the number field $\mathring{F}$. We also keep the conventions on the test function $f_V = f f^u_V \in \mathcal{H}_\text{simp}(\tilde{G}_V)$ in \S\ref{sec:compression}. The choice of $f \in C^\infty_{c,\asp}(\tilde{G})$ is free.

Fix an archimedean infinitesimal character $\nu$ for $\tildering{G}$. From Corollary \ref{prop:stable-trace-formula-nu} we have $I_\nu(f_V) = I^\EndoE_\nu(f_V)$. By Lemma \ref{prop:compression-EndoE} and the inversion formula, $I^\EndoE_\nu(f_V)$ equals
\begin{align*}
  \sum_{\phi \in T^\EndoE(\tilde{G})_\C} I^\EndoE_\nu(f^u_V, \phi) f^\EndoE(\phi) & = \sum_{\phi, \tau} I^\EndoE_\nu(f^u_V, \phi) \Delta(\phi, \tau) f^\gr_{\tilde{G}}(\tau) \\
  & = \sum_{\tau \in T_-(\tilde{G})_\C/\Sph^1} I^\EndoE_\nu(f^u_V, \tau) f^\gr_{\tilde{G}}(\tau),
\end{align*}
where we put
$$ I^\EndoE_\nu (f^u_V, \tau) := \sum_{\phi \in T^\EndoE(\tilde{G})_\C} I^\EndoE_\nu(f^u_V, \phi) \Delta(\phi, \tau). $$
for every $\tau \in T_-(\tilde{G})_\C$.

\begin{lemma}\label{prop:ell-vs-para}
  For any $f_V$ as above, we have
  $$ I^\EndoE_\nu(f^u_V, \tau) = I_\nu(f^u_V, \tau), \quad \tau \in T_{\mathrm{ell},-}(\tilde{G}), $$
  and
  $$ \sum_{\tau \in T_{\mathrm{ell},-}(\tilde{G})/\Sph^1} I_\nu(f^u_V, \tau) \left( f^\gr_{\tilde{G}}(\tau) - f_{\tilde{G}}(\tau) \right) = \sum_{\tau \in T_{\mathrm{para},-}(\tilde{G})_\C /\Sph^1} \left( I_\nu(f^u_V, \tau) - I^\EndoE_\nu(f^u_V, \tau) \right) f_{\tilde{G}}(\tau) $$
  for every $\nu$.
\end{lemma}
\begin{proof}
  From Lemma \ref{prop:compression-spectral} and $I_\nu(f_V) = I^\EndoE_\nu(f_V)$, we obtain
  $$ \sum_{\tau \in T_-(\tilde{G})_\C /\Sph^1} \left( I^\EndoE_\nu(f^u_V, \tau) f^\gr_{\tilde{G}}(\tau) - I_\nu(f^u_V, \tau) f_{\tilde{G}}(\tau) \right) = 0. $$
  No need to worry about $T_{\text{ell},-}(\tilde{G})_\C \smallsetminus T_{\text{ell},-}(\tilde{G})$ in the sum. Given $\tau \in T_{\text{ell},-}(\tilde{G})$, choose $f_{\tilde{G}} \in \Iaspcusp(\tilde{G})$ to be a pseudo-coefficient at $\tau$. Note that \eqref{eqn:character-relation-final} holds for this $f_{\tilde{G}}$ by Lemma \ref{prop:reduction-cuspidal-f}. The first assertion follows. Now resume the setting of an arbitrary $f_{\tilde{G}}$, we have
  $$ \sum_{\tau \in T_{\text{ell},-}(\tilde{G}) /\Sph^1} I_\nu(f^u_V, \tau) \left(f^\gr_{\tilde{G}}(\tau) - f_{\tilde{G}}(\tau) \right) = \sum_{\tau \in T_{\mathrm{para},-}(\tilde{G})_\C /\Sph^1} \left( I_\nu(f^u_V, \tau) f_{\tilde{G}}(\tau) - I^\EndoE_\nu(f^u_V, \tau) f^\gr_{\tilde{G}}(\tau) \right). $$
  
  We have verified \eqref{eqn:character-relation-final} for $\tau \in T_{\text{para},-}(\tilde{G})_\C$. Plugging this into the previous displayed equation yields the second assertion.
\end{proof}

Observe that the function
\begin{gather}\label{eqn:h_G}
  h_{\tilde{G}}: \tau \longmapsto f^\gr_{\tilde{G}}(\tau) - f_{\tilde{G}}(\tau), \quad \tau \in T_-(\tilde{G})
\end{gather}
has been shown to be supported on $T_{\text{ell},-}(\tilde{G})$. It satisfies the other conditions ($\Sph^1$-equivariance and finite support, cf.\ Lemma \ref{prop:spectral-Delta-supp}) characterizing the Paley-Wiener space $\text{PW}_{\asp}(\tilde{G})$, hence comes from an element of $\Iaspcusp(\tilde{G})$ which we still denote by $h_{\tilde{G}}$.

\begin{proposition}[{\cite[Lemma 9.3]{Ar96}}]
  We have $I^\EndoE_\nu(f^u_V, \tau) = I_\nu(f^u_V, \tau)$ for every $\tau \in T_{\mathrm{para},-}(\tilde{G})_\C$ and every infinitesimal character $\nu$.
\end{proposition}
\begin{proof}
  Let $\Omega$ be a connected component of $T_{\text{para},-}(\tilde{G})$. Set
  $$ \mathrm{PW}_{\asp}(\Omega) := \left\{ a \in \mathrm{PW}_{\asp}(\tilde{G}) : \Supp(a) \subset \Omega \right\}. $$
  For $\omega \in \mathrm{PW}_{\asp}(\Omega)$, let $f_\omega \in C^\infty_{c,\asp}(\tilde{G})$ be such that $\forall \tau, \; f_{\omega,\tilde{G}}(\tau) = \omega(\tau)$. As the first step, let us prove that for every linear functional $I: \Iaspcusp(\tilde{G}) \to \C$, there exists $F \in C^\infty(\Omega)$ such that $F(z\tau)=zF(\tau)$ for all $z \in \Sph^1$, $\tau \in \Omega$, and
  \begin{gather}\label{eqn:F}
    I\left( f^\gr_{\omega,\tilde{G}} - f_{\omega,\tilde{G}} \right) = \int_{\Omega/\Sph^1} F\omega, \quad \forall \; \omega \in \mathrm{PW}_{\asp}(\Omega).
  \end{gather}

  Via $\mathcal{T}^\EndoE: \Iaspcusp(\tilde{G}) \rightiso \Icusp^\EndoE(\tilde{G})$, we transport $I$ to a linear functional $J: \Icusp^\EndoE(\tilde{G}) \to \C$. Let $f \in C^\infty_{c,\asp}(\tilde{G})$. Lemma \ref{prop:reduction-cuspidal-f} applied to \eqref{eqn:h_G} implies that $\mathcal{T}^\EndoE(h_{\tilde{G}})$ sends any $\phi \in T^\EndoE_\text{ell}(\tilde{G})$ to
  \begin{gather}\label{eqn:h-vs-h^EndoE}
    \sum_\tau \Delta(\phi, \tau) h_{\tilde{G}}(\tau) = \sum_\tau \Delta(\phi, \tau) \left( f^\gr_{\tilde{G}}(\tau) - f_{\tilde{G}}(\tau) \right) = f^\EndoE(\phi) - f^\EndoE_\gr(\phi).
  \end{gather}
  As a by product, we conclude that $f^\EndoE - f^\EndoE_\gr$ comes from some element $h^\EndoE \in \Icusp^\EndoE(\tilde{G})$.

  Specialize to the case $f = f_\omega$. Claim: the function
  $$ \omega \mapsto I\underbrace{ \left( f^\gr_{\omega,\tilde{G}} - f_{\omega,\tilde{G}} \right) }_{=: h_{\omega,\tilde{G}}} = J\underbrace{ \left( f^\EndoE_\omega - f^\EndoE_{\omega,\gr} \right) }_{=: h^\EndoE_\omega} $$
  is a finite linear combination of functions of the form
  $$ \omega \mapsto h^\EndoE_\omega(\sigma), \quad \sigma \in \Gamma^\EndoE_\text{reg,ell}(\tilde{G}). $$
  Indeed, the normalized stable orbital integrals $a \mapsto a(\sigma)$ are (weakly) dense in $S\orbI(G^\Endo)^\vee$ for each $G^\Endo$; since $h^\EndoE_\omega \in \Icusp^\EndoE(\tilde{G}) = \bigoplus_{G^\Endo \in \EndoE_\text{ell}(\tilde{G})} S\orbI(G^\Endo)$, Howe's conjecture on finiteness \cite{BM00} applied to each endoscopic group $G^\Endo$ implies our claim.

  It remains to fix $\sigma$ and show that $h^\EndoE_\omega(\sigma) = \int_{\Omega/\Sph^1} F(\sigma, \tau)\omega(\tau) \dd\tau$ for some smooth function $F$ that is independent of $\omega$ and verifies $\forall z \in \Sph^1, \;F(\sigma, z\tau) = zF(\sigma,\tau)$. We apply Lemma \ref{prop:stable-orbint-Fourier} together with \eqref{eqn:h-vs-h^EndoE} to obtain
  \begin{align*}
    h^\EndoE_\omega(\sigma) & = \int_{\phi \in T^\EndoE_\text{ell}(\tilde{G})} S(\sigma, \phi) h^\EndoE_\omega(\phi) \dd\phi \\
    & = \int_{\phi \in T^\EndoE_\text{ell}(\tilde{G})} \sum_{\tau \in T_{\text{ell},-}(\tilde{G})/\Sph^1} S(\sigma, \phi) \Delta(\phi, \tau) h_{\omega,\tilde{G}}(\tau) \dd\phi
  \end{align*}
  where $S(\cdot, \cdot)$ is the ``collective'' version of the smooth functions $S^{G^\Endo}(\cdot, \cdot)$ restricted to the elliptic parameters $(\sigma, \phi)$. Changing variables using Lemma \ref{prop:change-variables-spectral}, we arrive at
  $$ \int_{T_{\text{ell},-}(\tilde{G})/\Sph^1} \sum_{\phi \in T^\EndoE_\text{ell}(\tilde{G})} S(\sigma, \phi) \Delta(\phi, \tau) h_{\omega,\tilde{G}}(\tau) \dd\tau. $$
  The function $F(\sigma, \tau) := \sum_\phi S(\sigma, \phi) \Delta(\phi, \tau)$ is smooth in $\sigma$ and $\tau$ (use the smoothness of $\Delta$), it also has the right $\Sph^1$-equivariance so that the integral makes sense. Hence \eqref{eqn:F} is established.
  
  We are now ready prove the Proposition for $\tau \in \Omega$; the case $\tau \in \Omega_\C$ will follow by analytic continuation. Consider the genuine invariant distribution $I: k \mapsto I_\nu(f^u_V k)$, where $f^u_V$ is fixed and $k \in C^\infty_{c,\asp}(\tilde{G})$. Recall that
  $$ I_\nu \left( f^\gr_{\tilde{G}} - f_{\tilde{G}} \right) = \sum_{\tau \in T_{\text{ell},-}(\tilde{G})/\Sph^1} I_\nu(f^u_V, \tau)\left( f^\gr_{\tilde{G}}(\tau) - f_{\tilde{G}}(\tau) \right). $$
  Take $f = f_\omega$ for $\omega \in \text{PW}_{\asp}(\Omega)$ as before. Using Lemma \ref{prop:ell-vs-para},
  $$ I_\nu(h_{\omega,\tilde{G}}) = \sum_{\tau \in \Omega_\C/\Sph^1} \left( I_\nu(f^u_V, \tau) - I^\EndoE_\nu(f^u_V, \tau) \right) \omega(\tau). $$
  By fixing $f^u_V$ and $\Omega$, the possible $\tilde{K}$-types of the automorphic representations in the spectral expansion of $I_\nu(f_V)$ are pinned down. By of \cite[Proposition 7.4]{Li13}, only finitely many automorphic representations contribute to $I_\nu(f_V)$, and that is why we opt to fix $\nu$; in particular the sum over $\tau$ is actually finite. On the other hand, $I_\nu(h_{\omega,\tilde{G}})$ also equals $\int_{\Omega/\Sph^1} F(\tau)\omega(\tau) \dd\tau$ for some $F \in C^\infty(\Omega)$ with $\forall z \in \Sph^1, \;F(z\tau)=zF(\tau)$, by \eqref{eqn:F}.

  Write $\Omega = \Omega_M/W(\Omega)$ where $M \in \mathcal{L}(M_0)$, $\Omega_M$ is a connected component of $T_{\text{ell},-}(\tilde{M})$ and $W(\Omega) := \Stab_{W^G(M)}(\Omega_M)$; let $\tilde{\Omega}_M$ be the inverse image of $\Omega_M$ in $\tilde{T}_{\text{ell},-}(\tilde{M})$. Then $F$ may be viewed as a $W(\Omega)$-invariant smooth function on $\tilde{\Omega}_M$, so do $(I^\EndoE_\nu(f^u_V, \cdot) - I^\EndoE_\nu(f^u_V, \cdot))|_\Omega$. The symmetry constraints on $\omega$ can thus be removed so that
  $$ \int_{\tilde{\Omega}_M /\Sph^1} F(\tau)\omega(\tau) \dd\tau = \sum_{\substack{\tau \in \tilde{\Omega}_{M,\C}/\Sph^1 \\ \text{finite sum}}} \left( I_\nu(f^u_V, \tau) - I^\EndoE_\nu(f^u_V, \tau) \right) \omega(\tau) $$
  holds for every Paley-Wiener function $\omega$ on the compact torus $\tilde{\Omega}_M/\Sph^1$; here we got rid of $\Sph^1$-equivariance by trivializing the torsors as in Remark \ref{rem:T-splitting}.
  
  Via Fourier transform we obtain an equality between distributions on the Pontryagin dual of $\tilde{\Omega}_M/\Sph^1$, which is a lattice. The left-hand side gives a rapidly decreasing function, whereas the right-hand side gives a finite sum of $\C^\times$-valued characters. The only possibility is that $F = I_\nu(f^u_V,\cdot) - I^\EndoE_\nu(f^u_V, \cdot) = 0$; this is elementary, see \cite[\S 2.7, Lemme]{Wa13-4}.
\end{proof}

\begin{corollary}\label{prop:compressed-h}
  Let $h_{\tilde{G}} \in \Iaspcusp(\tilde{G})$ be attached to $f^\gr_{\tilde{G}} - f_{\tilde{G}}$ as in \eqref{eqn:h_G}. Pick any $h \in C^\infty_{c,\asp}(\tilde{G})$ mapping to $h_{\tilde{G}}$, we have
  $$ I_\nu(f^u_V h) = \sum_{\tau \in T_{\mathrm{ell},-}(\tilde{G})/\Sph^1} I_\nu(f^u_V, \tau) h_{\tilde{G}}(\tau) = 0 $$
  for every $f^u_V$ and every infinitesimal character $\nu$. Consequently, $I(f^u_V h)=0$.
\end{corollary}
\begin{proof}
  Apply Lemma \ref{prop:compression-spectral} plus the second equality of Lemma \ref{prop:ell-vs-para}, and recall that $I = \sum_\nu I_\nu$.
\end{proof}

\subsection{Proof of the main theorem: local-global argument}
The local-global loop is now to be closed. We revert to the given local covering $\rev: \tilde{G} \to G(F)$ attached to $(W, \angles{\cdot|\cdot})$ and try to embed it into an adélic one, by suitably choosing $\mathring{F}$, etc.

We will globalize the data of linear algebra by which $\tilde{G}$ and $\Delta$ are defined; the additive character $\psi$ will be globalized only up to $F^{\times 2}$.

\begin{proposition}\label{prop:globalization}
  Let $T$ be an elliptic maximal $F$-torus of $G$. For every $r \geq 1$, there exist
  \begin{itemize}
    \item a number field $\mathring{F}$,
    \item a non-trivial additive character $\mathring{\psi}=\prod_v \psi_v: \A/\mathring{F} \to \Sph^1$,
    \item a symplectic $\mathring{F}$-vector space $(\mathring{W}, \angles{\cdot|\cdot}_{\mathring{F}})$,
    \item $u_0, \ldots, u_r$: distinct non-archimedean places of $\mathring{F}$,
    \item a maximal $\mathring{F}$-torus $\mathring{T}$ of $\mathring{G}$,
  \end{itemize}
  such that the following properties are satisfied for the adélic metaplectic covering $\rev: \tildering{G} \to \mathring{G}(\A)$ attached to $(\mathring{W}, \mathring{\psi} \circ \angles{\cdot|\cdot}_{\mathring{F}})$.
  \begin{enumerate}
    \item For $i=0, \ldots, r$, we have $\mathring{F}_{u_i} \rightiso F$, under which $(\mathring{W}, \angles{\cdot|\cdot}_{\mathring{F}}) \otimes_{\mathring{F}} \mathring{F}_{u_i} \rightiso (W, \angles{\cdot|\cdot})$; fix such identifications.
    \item The natural homomorphism $H^1(\mathring{F}_v, \mathring{T}_v) \to H^1(\A/\mathring{F}, \mathring{T})$ (see \cite[\S 3.1.2]{Li15}) is an isomorphism at $v=u_0, \ldots, u_r$.
    \item There exists $a_i \in F^\times$ such that
      $$ \forall t \in F, \; \psi_{u_i}(t) = \psi(a_i^2 t). $$
    \item As coverings, $\tilde{G}_{u_i} \to \mathring{G}(\mathring{F}_{u_i})$ is isomorphic to $\tilde{G} \to G(F)$. Moreover, the geometric transfer factors $\Delta_{(n',n'')}$ on $\tilde{G}_{u_i}$ and $\tilde{G}$ agree under this isomorphism.
    \item Under the identifications above, the localization of $\mathring{T}$ at each $u_i$ is conjugate to $T$; in particular, $\mathring{T}$ is $\mathring{F}$-elliptic.
    \item The localization map $\mathring{T}(\mathring{F}) \to T(F)$ at any $u_i$ has dense image.
  \end{enumerate}
\end{proposition}
\begin{proof}
  The parametrization of conjugacy classes of $T$ can be deduced from that of regular semisimple elements \cite[\S 3.1]{Li11}. They are in bijection with equivalence classes of data $(L, \tau, c)$, where
  \begin{itemize}
    \item $L$ is a finite-dimensional étale $F$-algebra,
    \item $\tau: L \to L$ is an $F$-involution, whose fixed subalgebra we denote by $L^\sharp$,
    \item $c \in L^\times/N_{L/L^\sharp}(L^\times)$ satisfies $\tau(c)=-c$, equivalently $\Tr_{L/L^\sharp}(c)=0$.
  \end{itemize}
  Here $N_{L/L^\sharp}(x)=x\tau(x)$, $\Tr_{L/L^\sharp}(x)=x+\tau(x)$, and there is an evident notion of equivalence between these triples. The condition is that the $F$-vector space $L$ together with the bilinear form $(a,b) \mapsto \Tr_{L/F}(a\tau(b)c)$ is isomorphic to $(W, \angles{\cdot|\cdot})$ as symplectic $F$-vector spaces. The torus $T$ so obtained is isomorphic to $\Ker N_{L/L^\sharp}$. Call $(L,\tau,c)$ or $(L, \tau)$ split if $L \simeq (F \times F)^n$, where $\tau: (x,y) \mapsto (y,x)$ on each factor $F \times F$; split triples correspond to split maximal $F$-tori.
  
  The same parametrization works for any field of characteristic $\neq 2$; the base change relative to any field extension $E/F$ is straightforward --- simply apply $- \otimes_F E$ to étale $F$-algebras with involution. We proceed to globalize $(L,\tau,c)$.
  
  The data $(L,\tau,c)$ parametrizing $T$ can be described via Galois descent, say by comparison with the split one. We shall treat the pair $(L,\tau)$ first. Let $E/F$ be a Galois extension splitting $(L,\tau)$. Take a number field $\mathring{E}'$ with a place $w'$ such that there is an isomorphism $\mathring{E}'_{w'} \rightiso E$, which we fix. Therefore $\Gal{E/F}$ acts on $\mathring{E}'$; denote its fixed field by $\mathring{F}'$. Then $\Gal{E/F} = \Gal{\mathring{E}'/\mathring{F}'}$.
  
  Let $u'$ be the place of $\mathring{F}'$ such that $w'|u'$, then $\mathring{E}'_{w'} \rightiso E$ restricts to $\mathring{F}'_{u'} \rightiso F$. In fact, $w'$ is the unique place of $\mathring{F}'$ above $u'$. Fix an algebraic closure of $\mathring{E}'$ and pick a Galois extension $\mathring{F}/\mathring{F}'$ therein, of degree $> r$ over which $u'$ splits completely. Therefore we obtain distinct places $u_0, \ldots, u_r$ with $\mathring{F}_{u_i} \rightiso F$ for all $i$. Consider
  $$ \begin{tikzcd}[row sep=tiny]
    {} & & \mathring{E} := \mathring{E}' \mathring{F} \arrow[-]{dd}  \arrow[-]{ld} & \\
    w' & \mathring{E}' \arrow[-]{dd} & & \\
    & & \mathring{F} \arrow[-]{ld} & u_i \\
    u' & \mathring{F}' & &
  \end{tikzcd} $$
  For each $i$, take place $w_i$ in $\mathring{E}$ such that $w_i|u_i$. By the foregoing discussion, it satisfies the identification
  \begin{gather}\label{eqn:Gal-identification}
    \Gal{\mathring{E}_{w_i}/\mathring{F}_{u_i}} = \Gal{\mathring{E}'_{w'}/\mathring{F}_{u'}} = \Gal{E/F},
  \end{gather}
  and $\mathring{E}_{u_i} \simeq \mathring{E}_{w_i}$ is a field.
  
  Plugging these into the machine of Galois descent, we produce a global pair $(\mathring{L}, \mathring{\tau})$ over $\mathring{F}$ that splits over $\mathring{E}$ and localizes to $(L, \tau)$ at each $u_i$. Let us globalize the remaining piece $c$. We claim that the simultaneous localization map
  $$ \left\{ \mathring{c} \in \mathring{L}^\times : \Tr(\mathring{c})= 0 \right\} \big/ N(\mathring{L}^\times) \longrightarrow \prod_{i=0}^r \left\{ c \in L^\times = \mathring{L}_{u_i}^\times : \Tr(c)=0 \right\} \big/ N(L^\times), $$
  is surjective, where $N$ (resp. $\Tr$) denotes the relevant norm (resp. trace) map. Since $N(L^\times)$ is open in $L^{\sharp\times}$, it suffices to show the density of $\{\mathring{c} : \mathring{\tau}(\mathring{c})=-\mathring{c}\}$ in $\prod_{i=1}^r \{c : \tau(c)=-c\}$. Indeed, weak approximation holds for the rational variety defined by $\Tr=0$ for any finite set of places $S$; here we take $S=\{u_0, \ldots, u_r\}$. 
  
  Therefore $c$ can also be globalized to $\mathring{c}$. Consequently we get $(\mathring{W}, \angles{\cdot|\cdot}_{\mathring{F}})$ which localizes to $(W, \angles{\cdot|\cdot})$ at each $u_i$, together with the maximal torus $\mathring{T}$ of $\mathring{G}$ parametrized by $(\mathring{L}, \mathring{\tau}, \mathring{c})$ that localizes to $T$ modulo conjugacy. Moreover, $\mathring{T}$ splits over $\mathring{E}$.
  
  The proof of $H^1(\mathring{F}_{u_i}, \mathring{T}_{u_i}) \rightiso H^1(\A/\mathring{F}, \mathring{T})$ is based on \eqref{eqn:Gal-identification}. One can either
  \begin{inparaenum}[(a)]
    \item invoke Tate-Nakayama duality as in \cite[p.528]{Ar88-TF2} to describe these $H^1$, or
    \item use the explicit description of these groups in \cite[\S 3.1.3]{Li15} together with the Galois descent construction.
  \end{inparaenum}

  Next, let us globalize $\psi: F \to \Sph^1$. Fix $\mathring{\psi}_0: \A/\mathring{F} \to \Sph^1$. For each $i$, there exists $b_i \in F^\times$ such that $\forall t, \; \psi(t) = \mathring{\psi}_{0,u_i}(b_i t)$. Since $F^{\times 2} \subset F^\times$ is open, weak approximation for the rational $\mathring{F}$-variety $\Ga$ (take $S=\{u_0, \ldots, u_r \}$) yields $\mathring{b} \in \mathring{F}^\times$ and $a_0, \ldots, a_r \in F^\times$ such that
  $$ \mathring{b}/b_i = a_i^2, \quad i=0, \ldots, r. $$
  The additive character $t \mapsto \mathring{\psi}(t) := \mathring{\psi}_0(\mathring{b}t)$ satisfies our requirements. From the description of metaplectic coverings in terms of Maslov cocycle \cite[\S 2.4]{Li11}, the covering $\tilde{G} \to G(F)$ is determined by $\psi \circ \left( \angles{\cdot|\cdot} \mod F^{\times 2} \right)$; indeed, this just reflects the properties of \emph{Weil index} $\gamma_\psi(\cdots)$. Hence the requirements on the localization of coverings are also satisfied. The coincidence of transfer factors follows by the same reason: see the formulas in \cite[\S 4 and \S 5.3]{Li11}.

  It remains to show that the localization map $\mathring{T}(\mathring{F}) \to T(F)$, say at the place $u_0$, has dense image. To this end, we apply the weak approximation in \cite[Lemma 1(b)]{KR00} to $\mathring{T}$ with $K=\mathring{E}$, $w = u_1$ and $S = \{u_0\}$, noting that $\mathring{E}_{u_i}$ is a field.
\end{proof}

Fix the local data $(T, \cdots)$ and their globalization obtained thus far, with $r=2$. Set $u := u_0$. The next step is to apply simple trace formula with
\begin{compactitem}
  \item a large finite set $V \supset V_\text{ram} \sqcup \{u_1, u_2, u\}$ of places,
  \item the distinguished place $u$,
  \item a suitable test function $f_V = f^u_V h$
\end{compactitem}
under the formalism of \S\ref{sec:proof-preparation}; here $h \in C^\infty_{c,\asp}(\tilde{G})$.

In what follows, we fix $\mathring{\gamma} \in \mathring{T}_\text{reg}(\mathring{F})$ and denote by $\gamma \in T_\text{reg}(F)$ its localization at $u$. Recall that the data $(\mathring{W}, \cdots)$ carry an $\mathfrak{o}_V$-model for large enough $V$.

\begin{proposition}\label{prop:non-vanishing}
  Given $h \in C^\infty_{c,\asp}(\tilde{G})$. We may choose a sufficiently large finite set of places $V \supset V_\mathrm{ram} \sqcup \{u_1, u_2, u\}$ and choose $f^u_V = \prod_{\substack{v \in V \\ v \neq u}} f_v$, satisfying
  \begin{compactenum}[(i)]
    \item $\mathring{\gamma}$ has regular reduction outside $V$ relative to the $\mathfrak{o}_V$-model which is a part of our adélic covering.;
    \item for every $v \in V$, $v \neq u$, $f_{v, \tildering{G}_v}$ is sufficiently close to an anti-genuine Dirac measure concentrated at the image of $\mathring{\gamma}$ in $\mathring{G}(\mathring{F}_v)$;
    \item for $v = u_1, u_2$, we assume in addition that $f_{v, \tilde{G}_v} \in \Iaspcusp(\tildering{G}_v)$
  \end{compactenum}
  such that
  \begin{gather*}
    f^u_V h \in \mathcal{H}_{\mathrm{simp,adm}}(\tilde{G}_V), \\
    I(f^u_V h) = I(f^u_V, \tilde{\gamma}) h_{\tilde{G}}(\tilde{\gamma}), \\
    I(f^u_V, \tilde{\gamma}) \neq 0.
  \end{gather*}
\end{proposition}
\begin{proof}
  The admissibility of $f^u_V h$ depends only on its support. Thus upon enlarging $V$ (cf.\ Remark \ref{rem:harmless}), one may assume that $f_V = f^u_V h \in \mathcal{H}_{\mathrm{simp,adm}}(\tilde{G}_V)$ is chosen so that $\Supp(\mathring{f}) \ni \mathring{\gamma}$, and that (i) is satisfied. Lemma \ref{prop:compressed-geom} gives
  \begin{align}
    \label{eqn:I-geom-1} I(f_V) & = \sum_{\delta \in \Gamma_\mathrm{ell,reg}(G)} \overbrace{ I(f^u_V, \tilde{\delta}) h_{\tilde{G}}(\tilde{\delta}) }^{\text{depending only on } \delta} \\
    \label{eqn:I-geom-2} & = \sum_{\mathring{\delta} \in \mathring{G}(\mathring{F})_\mathrm{ell,ss}/\mathrm{conj}} a^{\tildering{G}}(\mathring{\delta}) I^{\tildering{G}}(\mathring{\delta}, \mathring{f}).
  \end{align}
  
  Recall the recipe in \cite[\S 5.2]{Li14b} for passing from \eqref{eqn:I-geom-2} to \eqref{eqn:I-geom-1}. Any regular semisimple class $\mathring{\delta}$ with $I^{\tildering{G}}(\mathring{\delta}, \mathring{f}) \neq 0$ must admit a representative (called an \emph{admissible representative} in \textit{loc.\ cit.}) with image in $\mathring{G}(\A)$ of the form $\delta_V \delta^V$, such that $\delta_v \in K_v$ has regular reduction for all $v \notin V$. Indeed, this is just a paraphrase of the notion of $V$-admissibility in \cite[Définition 5.6.1]{Li14a}. We extract the components in $V$ of $\mathring{\delta}$ in the following manner: using the splittings $\mathring{G}(\mathring{F}) \hookrightarrow \tildering{G}$ and $K^V \hookrightarrow \tildering{G}^V$, one can pick any appropriate representative of $\mathring{\delta}$ as above, and write
  \begin{gather*}
    \mathring{\delta} = \tilde{\delta}_V \delta^V \in \tildering{G}, \\
    \tilde{\delta}_V = \tilde{\delta}^u_V \tilde{\delta}_u \in \tilde{G}^V, \quad \delta^V \in K^V;
  \end{gather*}
  Denote this procedure as
  $$ \mathring{\delta} \leadsto \tilde{\delta}_V. $$
  Attention: this is not necessarily a map from global conjugacy classes into $\Gamma_\text{reg}(\tilde{G}_V)$, but only a correspondence.

  Let $(\tilde{\delta}_v)_{v \in V} \in \prod_{v \in V} \tilde{G}_v$ be such that $(\tilde{\delta}_v)_{v \in V} \mapsto \tilde{\delta}_V$. Our assumptions outside $V$ imply $I^{\tildering{G}}(\mathring{\delta}, \mathring{f}) = I^{\tilde{G}_V}(\tilde{\delta}_V, f_V) = \prod_{v \in V} f_{v,\tilde{G}_v}(\tilde{\delta}_v)$. Set $\tilde{\delta} := \tilde{\delta}_u$. By collecting the contributions from all such $\mathring{\delta}$ and averaging over $\bmu_8$, we get the summand in \eqref{eqn:I-geom-1} indexed by $\delta$.

  \textit{Claim 1}: we can choose $f^u_V$ in the foregoing construction so that every $\mathring{\delta}$ with $I^{\tildering{G}}(\mathring{\delta}, \mathring{f}) \neq 0$ must be conjugate to $\mathring{\gamma}$ in $\mathring{G}(\mathring{F}_v)$ at every place $v \neq u$. Indeed, the adjoint quotient of $\mathring{G}$ is an affine $\mathring{F}$-variety. Thus by taking $f_v$ sufficiently close to an anti-genuine Dirac measure concentrated at the image of $\gamma$ in $G(F_v)$, for each $v \in V \smallsetminus \{u\}$, the condition $I^{\tildering{G}}(\mathring{\delta}, \mathring{f}) \neq 0$ will force $\mathring{\delta}$ to be stably conjugate to $\mathring{\gamma}$ at every place. Since both classes intersect $K_v$ for $v \notin V$, this implies ordinary conjugacy outside $V$ by a result of Kottwitz \cite[Proposition 5.6.2]{Li14a}. As to the places $v \in V \smallsetminus \{u\}$, we take $f_v$ so close to an anti-genuine Dirac measure to force ordinary conjugacy.
  
  Property (ii) is thus inherent in our construction. Moreover, (iii) is also satisfied since $T$ is elliptic. Shrinking the support does not destroy admissibility, therefore our choice of $f^u_V$ is accomplished.

  \textit{Claim 2}: $\mathring{\delta}$ and $\mathring{\gamma}$ are also conjugate in $\mathring{G}(\mathring{F}_u)$. By \cite[\S 3.1.2]{Li15}, there is a exact sequence
  $$ \begin{tikzcd}[row sep=tiny]
    H^1(\mathring{F}, \mathring{T}) \arrow{r} & \bigoplus_v H^1(\mathring{F}_v, \mathring{T}_v) \arrow{r} & H^1(\A/\mathring{F}, \mathring{T}) \\
    & (\lambda_v)_v \arrow[mapsto]{r} & \sum_v \left( \text{image of } \lambda_v \right) 
  \end{tikzcd} $$
  in which the first term measures global conjugacy classes in a stable class intersecting $\mathring{T}$, and the second measures the local situation. Apply this to $\mathring{\gamma}$ and $\mathring{\delta}$ and notice that $H^1(\mathring{F}_u, \mathring{T}_u) \hookrightarrow H^1(\A/\mathring{F}, \mathring{T})$ to get Claim 2.
  Hence $I(f_V) = I(f^u_V, \tilde{\gamma}) h_{\tilde{G}}(\tilde{\gamma})$.
  
  It remains to prove $I(f^u_V, \tilde{\gamma}) \neq 0$. At this stage we may vary $h$ and it suffices to show $I(f_V) \neq 0$. The rational classes $\mathring{\delta}$ contributing to $I(f_V)$ become conjugate in $\mathring{G}(\A)$. If they are also conjugate in $\tildering{G}$, the non-vanishing will follow at once by taking $h_{\tilde{G}}(\tilde{\gamma}) \neq 0$, since $a^{\tildering{G}}(\mathring{\delta}) > 0$. To show this, suppose that $\mathring{\delta}_i$ ($i=1,2$) are elliptic, semisimple regular elements such that
  \begin{compactitem}
    \item $I^{\tildering{G}}(\mathring{\delta}_i, \mathring{f}) \neq 0$ for $i=1,2$,
    \item $\mathring{\delta}_i \leadsto \tilde{\delta}_{V,i}$ for $i=1,2$,
    \item $\tilde{\delta}_{V,1} = \noyau\tilde{\delta}_{V,2}$ as conjugacy classes in $\tilde{G}_V$, where $\noyau \in \bmu_8$.
  \end{compactitem}
  It remains to show $\noyau=1$. For each place $v$, set
  $$ \Delta_{(n,0),v} := \dfrac{\Theta^+_{\psi_v} - \Theta^-_{\psi_v}}{| \Theta^+_{\psi_v} - \Theta^-_{\psi_v} |} $$
  with the notations in \S\ref{sec:Weil-rep}. This is a locally constant function on $\tilde{G}_{v, \text{reg}}$; it is essentially the transfer factor for $(n,0) \in \EndoE_\text{ell}(\tildering{G}_v)$. Let $\tilde{x} \in \tildering{G}$ with an inverse image $(\tilde{x}_v)_v \in \Resprod_v \tilde{G}_v$. Suppose that $\rev(\tilde{x}) \in \mathring{G}(\A)$ is locally stably conjugate to an element of $\mathring{G}_\text{reg}(\mathring{F})$, we define
  $$ \Delta_{(n,0)}(\tilde{x}) := \prod_v \Delta_{(n,0),v}(\tilde{x}_v) $$
  with the following properties.
  \begin{compactenum}[(a)]
    \item The infinite product $\prod_v \Delta_{(n,0), v}(\tilde{x}_v)$ is well-defined: almost all terms are $1$. It depends only on the conjugacy class of $\tilde{x}$.
    \item If $\tilde{x} \in \mathring{G}(\mathring{F})$, then $\Delta_{(n,0)}(\tilde{x})=1$.
    \item If $v \notin V$ and $\delta_v \in K_v$ has regular reduction, then $\Delta_{(n,0),v}(\delta_v)=1$.
    \item $\Delta_{(n,0),v}(\noyau\tilde{x}_v) = \noyau\Delta_{(n,0),v}(\tilde{x}_v)$ for all $\noyau \in \bmu_8$ and $\tilde{x}_v \in \tilde{G}_v$.
  \end{compactenum}
  Since $(\Theta^+_{\psi,v} - \Theta^-_{\psi,v})(\tilde{x}) = \Theta_{\psi_v}(-\tilde{x})$ for all place $v$ and regular semisimple $\tilde{x} \in \tilde{G}_v$ (Definition \ref{def:-1}), these properties are consequences of \cite[Théorème 4.28, Proposition 4.21]{Li11} and the genuineness of the Weil representations, in that order.
  
  Apply this to $\mathring{\delta}_1$, $\mathring{\delta}_2$, we see
  $$ 1 = \Delta_{(n,0)}(\mathring{\delta}_i) = \left(\prod_{v \in V} \Delta_{(n,0),v}\right)(\tilde{\delta}_{V,i}), \quad i=1,2. $$
  A comparison using (d) gives $\noyau=1$, as asserted
\end{proof}

\begin{proof}[Proof of Theorem \ref{prop:character-relation}]
  Let $f \in C^\infty_{c,\asp}(\tilde{G})$. From $f$ we deduce the function $h_{\tilde{G}} \in \Iaspcusp(\tilde{G})$ by \eqref{eqn:h_G}. In view of the reduction steps in \S\ref{sec:proof-preparation}, it remains to show that $h_{\tilde{G}}=0$. This amounts to $h_{\tilde{G}}(\tilde{\gamma})=0$ for all $\gamma \in T_\text{reg}(F)$ and any $\tilde{\gamma} \in \rev^{-1}(\gamma)$, where $T$ is any given elliptic maximal $F$-torus of $G$. Globalize the data $(T, \tilde{G} \to G(F), \cdots)$ by Proposition \ref{prop:globalization} with $r=2$. Choose the $f^u_V$ from Proposition \ref{prop:non-vanishing} for any given $\mathring{\gamma} \in \mathring{T}_\text{reg}(\mathring{F})$. Choose $h \in C^\infty_{c,\asp}(\tilde{G})$ mapping to $h_{\tilde{G}}$, then Corollary \ref{prop:compressed-h} asserts
  $$ I(f^u_V h) = 0. $$
  Meanwhile, Proposition \ref{prop:non-vanishing} says
  $$ I(f^u_V h) = \underbrace{I(f^u_V, \tilde{\gamma})}_{\neq 0} h_{\tilde{G}}(\tilde{\gamma}), $$
  where $\tilde{\gamma}$ is any inverse image of the $u$-component of $\mathring{\gamma}$. Hence $h_{\tilde{G}}(\tilde{\gamma})=0$. Since $\mathring{T}(\mathring{F}) \to T(F)$ has dense image, we conclude that $h_{\tilde{G}}=0$ by the continuity of orbital integrals.
\end{proof}

\bibliographystyle{abbrv}
\bibliography{spectrans}

\begin{thebibliography}{10}

\bibitem{Ad98}
J.~Adams.
\newblock Lifting of characters on orthogonal and metaplectic groups.
\newblock {\em Duke Math. J.}, 92(1):129--178, 1998.

\bibitem{Ar88-TF2}
J.~Arthur.
\newblock The invariant trace formula. {II}. {G}lobal theory.
\newblock {\em J. Amer. Math. Soc.}, 1(3):501--554, 1988.

\bibitem{Ar89-unip}
J.~Arthur.
\newblock Unipotent automorphic representations: conjectures.
\newblock {\em Ast\'erisque}, (171-172):13--71, 1989.
\newblock Orbites unipotentes et repr{\'e}sentations, II.

\bibitem{Ar93}
J.~Arthur.
\newblock On elliptic tempered characters.
\newblock {\em Acta Math.}, 171(1):73--138, 1993.

\bibitem{Ar94}
J.~Arthur.
\newblock On the {F}ourier transforms of weighted orbital integrals.
\newblock {\em J. Reine Angew. Math.}, 452:163--217, 1994.

\bibitem{Ar96}
J.~Arthur.
\newblock On local character relations.
\newblock {\em Selecta Math. (N.S.)}, 2(4):501--579, 1996.

\bibitem{Ar02}
J.~Arthur.
\newblock A stable trace formula. {I}. {G}eneral expansions.
\newblock {\em J. Inst. Math. Jussieu}, 1(2):175--277, 2002.

\bibitem{ArGerm}
J.~Arthur.
\newblock Germ expansions for real groups.
\newblock \url{http://www.math.toronto.edu/arthur}, 2004.

\bibitem{Ar06}
J.~Arthur.
\newblock A note on {$L$}-packets.
\newblock {\em Pure Appl. Math. Q.}, 2(1, Special Issue: In honor of John H.
  Coates. Part 1):199--217, 2006.

\bibitem{Ar13}
J.~Arthur.
\newblock {\em The endoscopic classification of representations}, volume~61 of
  {\em American Mathematical Society Colloquium Publications}.
\newblock American Mathematical Society, Providence, RI, 2013.
\newblock Orthogonal and symplectic groups.

\bibitem{BM00}
D.~Barbasch and A.~Moy.
\newblock A new proof of the {H}owe conjecture.
\newblock {\em J. Amer. Math. Soc.}, 13(3):639--650 (electronic), 2000.

\bibitem{Bo79}
A.~Borel.
\newblock Automorphic {$L$}-functions.
\newblock In {\em Automorphic forms, representations and {$L$}-functions
  ({P}roc. {S}ympos. {P}ure {M}ath., {O}regon {S}tate {U}niv., {C}orvallis,
  {O}re., 1977), {P}art 2}, Proc. Sympos. Pure Math., XXXIII, pages 27--61.
  Amer. Math. Soc., Providence, R.I., 1979.

\bibitem{Bor98}
M.~Borovoi.
\newblock Abelian {G}alois cohomology of reductive groups.
\newblock {\em Mem. Amer. Math. Soc.}, 132(626):viii+50, 1998.

\bibitem{Bo94b}
A.~Bouaziz.
\newblock Int\'egrales orbitales sur les groupes de {L}ie r\'eductifs.
\newblock {\em Ann. Sci. \'Ecole Norm. Sup. (4)}, 27(5):573--609, 1994.

\bibitem{C11}
P.-H. Chaudouard.
\newblock Le transfert lisse des int\'egrales orbitales d'apr\`es
  {W}aldspurger.
\newblock In {\em On the stabilization of the trace formula}, volume~1 of {\em
  Stab. Trace Formula Shimura Var. Arith. Appl.}, pages 145--180. Int. Press,
  Somerville, MA, 2011.

\bibitem{CD84}
L.~Clozel and P.~Delorme.
\newblock Le th\'eor\`eme de {P}aley-{W}iener invariant pour les groupes de
  {L}ie r\'eductifs.
\newblock {\em Invent. Math.}, 77(3):427--453, 1984.

\bibitem{CHLN11}
L.~Clozel, M.~Harris, J.-P. Labesse, and B.-C. Ng{\^o}, editors.
\newblock {\em On the stabilization of the trace formula}, volume~1 of {\em
  Stabilization of the Trace Formula, Shimura Varieties, and Arithmetic
  Applications}.
\newblock International Press, Somerville, MA, 2011.

\bibitem{Du75}
M.~Duflo.
\newblock Repr\'esentations irr\'eductibles des groupes semi-simples complexes.
\newblock In {\em Analyse harmonique sur les groupes de {L}ie ({S}\'em.,
  {N}ancy-{S}trasbourg, 1973--75)}, pages 26--88. Lecture Notes in Math., Vol.
  497. Springer, Berlin, 1975.

\bibitem{FLM11}
T.~Finis, E.~Lapid, and W.~M{\"u}ller.
\newblock On the spectral side of {A}rthur's trace formula---absolute
  convergence.
\newblock {\em Ann. of Math. (2)}, 174(1):173--195, 2011.

\bibitem{He00}
G.~Henniart.
\newblock Une preuve simple des conjectures de {L}anglands pour {${\rm GL}(n)$}
  sur un corps {$p$}-adique.
\newblock {\em Invent. Math.}, 139(2):439--455, 2000.

\bibitem{HS12}
K.~Hiraga and H.~Saito.
\newblock On {$L$}-packets for inner forms of {$SL_n$}.
\newblock {\em Mem. Amer. Math. Soc.}, 215(1013):vi+97, 2012.

\bibitem{KV95}
A.~W. Knapp and D.~A. Vogan, Jr.
\newblock {\em Cohomological induction and unitary representations}, volume~45
  of {\em Princeton Mathematical Series}.
\newblock Princeton University Press, Princeton, NJ, 1995.

\bibitem{KZ82-1}
A.~W. Knapp and G.~J. Zuckerman.
\newblock Classification of irreducible tempered representations of semisimple
  groups.
\newblock {\em Ann. of Math. (2)}, 116(2):389--455, 1982.

\bibitem{KZ82-2}
A.~W. Knapp and G.~J. Zuckerman.
\newblock Classification of irreducible tempered representations of semisimple
  groups. {II}.
\newblock {\em Ann. of Math. (2)}, 116(3):457--501, 1982.

\bibitem{Ko82}
R.~E. Kottwitz.
\newblock Rational conjugacy classes in reductive groups.
\newblock {\em Duke Math. J.}, 49(4):785--806, 1982.

\bibitem{Ko86}
R.~E. Kottwitz.
\newblock Stable trace formula: elliptic singular terms.
\newblock {\em Math. Ann.}, 275(3):365--399, 1986.

\bibitem{KR00}
R.~E. Kottwitz and J.~D. Rogawski.
\newblock The distributions in the invariant trace formula are supported on
  characters.
\newblock {\em Canad. J. Math.}, 52(4):804--814, 2000.

\bibitem{KS99}
R.~E. Kottwitz and D.~Shelstad.
\newblock Foundations of twisted endoscopy.
\newblock {\em Ast\'erisque}, (255):vi+190, 1999.

\bibitem{Lab99}
J.-P. Labesse.
\newblock Cohomologie, stabilisation et changement de base.
\newblock {\em Ast\'erisque}, (257):vi+161, 1999.
\newblock Appendix A by Laurent Clozel and Labesse, and Appendix B by Lawrence
  Breen.

\bibitem{Li11}
W.-W. Li.
\newblock Transfert d'int\'egrales orbitales pour le groupe m\'etaplectique.
\newblock {\em Compos. Math.}, 147(2):524--590, 2011.

\bibitem{Li12b}
W.-W. Li.
\newblock La formule des traces pour les rev\^etements de groupes r\'eductifs
  connexes. {II}. {A}nalyse harmonique locale.
\newblock {\em Ann. Sci. \'Ec. Norm. Sup\'er. (4)}, 45(5):787--859 (2013),
  2012.

\bibitem{Li12a}
W.-W. Li.
\newblock Le lemme fondamental pond\'er\'e pour le groupe m\'etaplectique.
\newblock {\em Canad. J. Math.}, 64(3):497--543, 2012.

\bibitem{Li13}
W.-W. Li.
\newblock La formule des traces pour les rev\^etements de groupes r\'eductifs
  connexes {III}: {L}e d\'eveloppement spectral fin.
\newblock {\em Math. Ann.}, 356(3):1029--1064, 2013.

\bibitem{Li14a}
W.-W. Li.
\newblock La formule des traces pour les rev\^etements de groupes r\'eductifs
  connexes. {I}. {L}e d\'eveloppement g\'eom\'etrique fin.
\newblock {\em J. Reine Angew. Math.}, (686):37--109, 2014.

\bibitem{Li14b}
W.-W. {Li}.
\newblock {La formule des traces pour les rev\^etements de groupes r\'eductifs
  connexes. IV. Distributions invariantes.}
\newblock {\em {Ann. Inst. Fourier}}, 64(6):2379--2448, 2014.

\bibitem{Li15}
W.-W. Li.
\newblock La formule des traces stable pour le groupe m\'etaplectique: les
  termes elliptiques.
\newblock {\em Invent. Math.}, 202(2):743--838, 2015.

\bibitem{MVW87}
C.~M{\oe}glin, M.-F. Vign{\'e}ras, and J.-L. Waldspurger.
\newblock {\em Correspondances de {H}owe sur un corps {$p$}-adique}, volume
  1291 of {\em Lecture Notes in Mathematics}.
\newblock Springer-Verlag, Berlin, 1987.

\bibitem{Re98}
D.~Renard.
\newblock Transfert d'int\'egrales orbitales entre {${\rm Mp}(2n,{\bf R})$} et
  {${\rm SO}(n+1,n)$}.
\newblock {\em Duke Math. J.}, 95(2):425--450, 1998.

\bibitem{Re99}
D.~Renard.
\newblock Endoscopy for {${\rm Mp}(2n,{\bf R})$}.
\newblock {\em Amer. J. Math.}, 121(6):1215--1243, 1999.

\bibitem{Sh79}
D.~Shelstad.
\newblock Characters and inner forms of a quasi-split group over {${\bf R}$}.
\newblock {\em Compositio Math.}, 39(1):11--45, 1979.

\bibitem{Sh82}
D.~Shelstad.
\newblock {$L$}-indistinguishability for real groups.
\newblock {\em Math. Ann.}, 259(3):385--430, 1982.

\bibitem{Sh08}
D.~Shelstad.
\newblock Tempered endoscopy for real groups. {III}. {I}nversion of transfer
  and {$L$}-packet structure.
\newblock {\em Represent. Theory}, 12:369--402, 2008.

\bibitem{Sh10}
D.~Shelstad.
\newblock Tempered endoscopy for real groups. {II}. {S}pectral transfer
  factors.
\newblock In {\em Automorphic forms and the {L}anglands program}, volume~9 of
  {\em Adv. Lect. Math. (ALM)}, pages 236--276. Int. Press, Somerville, MA,
  2010.

\bibitem{Tr67}
F.~Tr{\`e}ves.
\newblock {\em Topological vector spaces, distributions and kernels}.
\newblock Academic Press, New York, 1967.

\bibitem{Vi81}
M.-F. Vign{\'e}ras.
\newblock Caract\'erisation des int\'egrales orbitales sur un groupe r\'eductif
  {$p$}-adique.
\newblock {\em J. Fac. Sci. Univ. Tokyo Sect. IA Math.}, 28(3):945--961 (1982),
  1981.

\bibitem{Vo79}
D.~A. Vogan, Jr.
\newblock The algebraic structure of the representation of semisimple {L}ie
  groups. {I}.
\newblock {\em Ann. of Math. (2)}, 109(1):1--60, 1979.

\bibitem{Wa88}
J.-L. Waldspurger.
\newblock Repr\'esentation m\'etaplectique et conjectures de {H}owe.
\newblock {\em Ast\'erisque}, (152-153):3, 85--99 (1988), 1987.
\newblock S{\'e}minaire Bourbaki, Vol. 1986/87.

\bibitem{Wa91}
J.-L. Waldspurger.
\newblock Correspondances de {S}himura et quaternions.
\newblock {\em Forum Math.}, 3(3):219--307, 1991.

\bibitem{Wa08}
J.-L. Waldspurger.
\newblock L'endoscopie tordue n'est pas si tordue.
\newblock {\em Mem. Amer. Math. Soc.}, 194(908):x+261, 2008.

\bibitem{Wa13-4}
J.-L. Waldspurger.
\newblock Pr\'eparation \`a la stabilisation de la formule des traces tordue
  {IV} : transfert spectral archim\'{e}dien, 2013.
\newblock \url{http://www.math.jussieu.fr/~waldspur}.

\bibitem{Wa14-1}
J.-L. Waldspurger.
\newblock Stabilisation de la formule des traces tordue {I} : endoscopie tordue
  sur un corps local, 2014.
\newblock \url{http://www.math.jussieu.fr/~waldspur}.

\end{thebibliography}

\begin{flushleft}
  Wen-Wei Li \\
  E-mail address: \href{mailto:wwli@math.ac.cn}{\texttt{wwli@math.ac.cn}}
  Academy of Mathematics and Systems Science, Chinese Academy of Sciences, \\
  55, Zhongguancun donglu, 100190 Beijing, People's Republic of China. \\ \vspace{0.5em}
  University of Chinese Academy of Sciences, \\
  19A, Yuquan lu, 100049 Beijing, People's Republic of China.
\end{flushleft}

\printindex

\end{document}